\documentclass[a4paper,reqno]{amsart}

\usepackage[%
backend=bibtex,%
style=alphabetic,%
natbib=false,%
isbn=false, 
doi=false, 
url=false, 
eprint=true, 
giveninits=true, 
backref=false, 
maxbibnames=99,%
]{biblatex}%

\renewbibmacro{in:}{} 

\usepackage{amssymb,amsmath}%
\usepackage{binarytree}%
\usepackage{doi}%
\usepackage{eucal}%
\usepackage{graphicx}%
\usepackage{mathscinet}%
\usepackage{mathtools}%
\usepackage{tikz-cd}%
\usepackage{xparse}%

\usepackage{wasysym}
\makeatletter%
\DeclareFontFamily{OMX}{MnSymbolE}{}%
\DeclareSymbolFont{MnLargeSymbols}{OMX}{MnSymbolE}{m}{n}%
\SetSymbolFont{MnLargeSymbols}{bold}{OMX}{MnSymbolE}{b}{n}%
\DeclareFontShape{OMX}{MnSymbolE}{m}{n}{%
  <-6>  MnSymbolE5%
  <6-7>  MnSymbolE6%
  <7-8>  MnSymbolE7%
  <8-9>  MnSymbolE8%
  <9-10> MnSymbolE9%
  <10-12> MnSymbolE10%
  <12->   MnSymbolE12%
}{}%
\DeclareFontShape{OMX}{MnSymbolE}{b}{n}{%
  <-6>  MnSymbolE-Bold5%
  <6-7>  MnSymbolE-Bold6%
  <7-8>  MnSymbolE-Bold7%
  <8-9>  MnSymbolE-Bold8%
  <9-10> MnSymbolE-Bold9%
  <10-12> MnSymbolE-Bold10%
  <12->   MnSymbolE-Bold12%
}{}%
\let\llangle\@undefined%
\let\rrangle\@undefined%
\DeclareMathDelimiter{\llangle}{\mathopen}
{MnLargeSymbols}{'164}{MnLargeSymbols}{'164}%
\DeclareMathDelimiter{\rrangle}{\mathclose}
{MnLargeSymbols}{'171}{MnLargeSymbols}{'171}%
\makeatother%

\usepackage{hyperref}%
\hypersetup{%
  colorlinks,%
  linkcolor={red!80!black},%
  citecolor={blue!80!black},%
  urlcolor={blue!80!black}%
}%

\usepackage{thmtools}%
\usepackage{cleveref}%

\numberwithin{equation}{section}%

\newcounter{thmintro}%

\declaretheorem[style=theorem,sibling=equation]{corollary}%
\declaretheorem[style=theorem,sibling=equation]{lemma}%
\declaretheorem[style=theorem,sibling=equation]{proposition}%
\declaretheorem[style=theorem,sibling=equation]{theorem}%
\declaretheorem[style=theorem,numbered=no, name=Theorem]{theorem*}%
\declaretheorem[style=theorem,sibling=thmintro,name=Theorem, Refname={Theorem,Theorems}]{theorem intro}%
\declaretheorem[style=theorem,numbered=no,name=Corollary]{corollary intro}%
\declaretheorem[style=definition,sibling=equation]{definition}%
\declaretheorem[style=definition,numbered=no,name=Definition]{definition*}%
\declaretheorem[style=definition,sibling=equation,name={Definition-Proposition}]{defprop}%
\declaretheorem[style=definition,sibling=equation,Refname={Notation,Notations}]{notation}%
\declaretheorem[style=definition,sibling=equation]{construction}%
\declaretheorem[style=definition,sibling=equation]{variant}%
\declaretheorem[style=definition,sibling=equation,Refname={Setting,Settings}]{setting}%
\declaretheorem[style=definition,numbered=no]{acknowledgements}%
\declaretheorem[style=definition,numbered=no]{conventions}%
\declaretheorem[style=definition,numbered=no,name={Financial Support}]{financialsupport}%
\declaretheorem[style=remark,sibling=equation]{remark}%
\declaretheorem[style=remark,sibling=equation,numbered=no,name=Remark]{uremark}%
\declaretheorem[style=remark,sibling=equation]{example}%

\AddToHook{env/theorem/begin}{     \crefalias{equation}{theorem}}
\AddToHook{env/theorem intro/begin}{     \crefalias{equation}{theorem intro}}
\AddToHook{env/corollary intro/begin}{     \crefalias{equation}{corollary intro}}
\AddToHook{env/example/begin}{     \crefalias{equation}{example}}
\AddToHook{env/proposition/begin}{    \crefalias{equation}{proposition}}
\AddToHook{env/lemma/begin}{     \crefalias{equation}{lemma}}
\AddToHook{env/corollary/begin}{     \crefalias{equation}{corollary}}
\AddToHook{env/remark/begin}{    \crefalias{equation}{remark}}
\AddToHook{env/definition/begin}{    \crefalias{equation}{definition}}
\AddToHook{env/defprop/begin}{    \crefalias{equation}{defprop}}
\AddToHook{env/notation/begin}{    \crefalias{equation}{notation}}
\AddToHook{env/construction/begin}{    \crefalias{equation}{construction}}
\AddToHook{env/setting/begin}{    \crefalias{equation}{setting}}
\AddToHook{env/variant/begin}{    \crefalias{equation}{variant}}

\newcommand{\kk}{\mathbf{k}}%
\newcommand{\ZZ}{\mathbb{Z}}%
\DeclareMathOperator{\chark}{char}

\newcommand{\op}{\mathrm{op}}%

\NewDocumentCommand{\id}{sO{}}{\IfBooleanTF{#1}{1}{\mathrm{id}}_{#2}}%

\DeclareMathOperator{\img}{im}%

\NewDocumentCommand{\set}{mo}{\{#1\IfValueTF{#2}{\,\mid\,#2}{}\}}%

\NewDocumentCommand{\add}{s m}{\operatorname{add}\IfBooleanTF{#1}{(#2)}{#2}}
\NewDocumentCommand{\proj}{s m}{\operatorname{proj}\IfBooleanTF{#1}{(#2)}{#2}}
\NewDocumentCommand{\mmod}{s m}{\operatorname{mod}\IfBooleanTF{#1}{(#2)}{#2}} 
\NewDocumentCommand{\smod}{s m}{\underline{\operatorname{mod}}(#2)}%

\NewDocumentCommand{\Hom}{O{}mm}{\operatorname{Hom}_{#1}(#2,#3)}%
\NewDocumentCommand{\sHom}{O{}mm}{\underline{\operatorname{Hom}}_{#1}(#2,#3)}%
\NewDocumentCommand{\RHom}{sO{}D<>{\bullet,*}mm}{\mathbb{R}\operatorname{Hom}_{#2}\IfBooleanT{#1}{^{#3}}(#4,#5)}%
\NewDocumentCommand{\dgHom}{O{}mmO{}}{\mathbf{Hom}_{#1}^{#4}(#2,#3)}%
\NewDocumentCommand{\Ext}{O{}O{\bullet}mm}{\operatorname{Ext}_{#1}^{#2}(#3,#4)}
\NewDocumentCommand{\Lotimes}{O{}}{\otimes_{#1}^{\mathbb{L}}}

\NewDocumentCommand{\DerCat}{O{}m}{\operatorname{D}\IfValueT{#1}{^{\mathrm{#1}}}(#2)}%

\newcommand{\category}[1]{\mathcal{#1}}%
\newcommand{\C}{\category{C}}%
\newcommand{\T}{\category{T}}%

\NewDocumentCommand{\ev}{O{}}{\operatorname{ev}_{#1}} 

\NewDocumentCommand{\s}{sO{}m}{\mathbf{s}^{#2}\IfBooleanTF{#1}{(#3)}{#3}}
\NewDocumentCommand{\vdeg}{m}{\lvert #1\rvert_{\mathrm{v}}} 
\NewDocumentCommand{\hdeg}{m}{\lvert #1\rvert_{\mathrm{h}}} 

\NewDocumentCommand{\pairing}{D<>{}}{\langle #1\rangle}

\NewDocumentCommand{\BB}{mO{\bullet}}{\operatorname{B}(#1)_{#2}}

\RenewDocumentCommand{\H}{O{\bullet}m}{\operatorname{H}^{#1}(#2)}%
\NewDocumentCommand{\Hclass}{m}{\{\!\!\{ {#1} \}\!\!\}}

\NewDocumentCommand{\HC}{sO{\bullet}D<>{*}mod<>}{\operatorname{C}\IfBooleanTF{#1}{^{#2}}{^{#2,#3}}(#4\IfValueT{#5}{,#5})\IfValueT{#6}{[#6]}}
\NewDocumentCommand{\nHC}{sO{\bullet}D<>{*}mod<>}{\overline{\operatorname{C}}\IfBooleanTF{#1}{^{#2}}{^{#2,#3}}(#4\IfValueT{#5}{,#5})\IfValueT{#6}{[#6]}}
\NewDocumentCommand{\HH}{sO{\bullet}D<>{*}mod<>}{\operatorname{HH}\IfBooleanTF{#1}{^{#2}}{^{#2,#3}}(#4\IfValueT{#5}{,#5})\IfValueT{#6}{[#6]}}

\NewDocumentCommand{\HoHC}{sO{\bullet}D<>{*}mod<>}{\operatorname{C}\IfBooleanTF{#1}{_{#2}}{_{#2,#3}}(#4\IfValueT{#5}{,#5})\IfValueT{#6}{[#6]}}
\NewDocumentCommand{\nHoHC}{sO{\bullet}D<>{*}mod<>}{\overline{\operatorname{C}}\IfBooleanTF{#1}{_{#2}}{_{#2,#3}}(#4\IfValueT{#5}{,#5})\IfValueT{#6}{[#6]}}
\NewDocumentCommand{\HoHH}{sO{\bullet}D<>{*}mod<>}{\operatorname{HH}\IfBooleanTF{#1}{_{#2}}{_{#2,#3}}(#4\IfValueT{#5}{,#5})\IfValueT{#6}{[#6]}}

\newcommand{\Hd}{d_{\mathrm{Hoch}}}
\newcommand{\HoHd}{b}

\newcommand{\BV}{\Delta} 

\NewDocumentCommand{\BimHC}{sD!!{A^e}O{\bullet}O{\star}D<>{*}mod<>}{\operatorname{C}_{#2}\IfBooleanTF{#1}{^{#3,#4,#5}}{^{#3,#5}}(#6\IfValueTF{#7}{,#7}{,#6})\IfValueT{#8}{[#8]}}
\NewDocumentCommand{\BimHH}{sD!!{A^e}O{\bullet}O{\star}D<>{*}mod<>}{\operatorname{Ext}_{#2}\IfBooleanTF{#1}{^{#3,#4,#5}}{^{#3,#5}}(#6\IfValueTF{#7}{,#7}{,#6})\IfValueT{#8}{[#8]}}

\newcommand{\dBim}{d_{\mathrm{Bim}}}
\newcommand{\dv}{d_{\mathrm{v}}}
\renewcommand{\dh}{d_{\mathrm{h}}}

\NewDocumentCommand{\RelBimHC}{sD!!{A^e}O{\bullet}O{\star}D<>{*}mmd<>}{\operatorname{C}\IfBooleanTF{#1}{^{#3,#4,#5}}{^{#3,#5}}(#6\IfValueTF{#7}{\,|\,#7}{,#6})\IfValueT{#8}{[#8]}}
\NewDocumentCommand{\RelBimHH}{sD!!{A^e}O{\bullet}O{\star}D<>{*}mmd<>}{\operatorname{HH}\IfBooleanTF{#1}{^{#3,#4,#5}}{^{#3,#5}}(#6\IfValueTF{#7}{\,|\,#7}{,#6})\IfValueT{#8}{[#8]}}
\NewDocumentCommand{\dRelBim}{O{A}D<>{M}}{d_{#1|#2}}

\NewDocumentCommand{\HMC}{sO{\bullet}D<>{*}mod<>}{\operatorname{CM}\IfBooleanTF{#1}{^{#2}}{^{#2,#3}}(#4\IfValueTF{#5}{,#5}{,\Hclass{\Astr<d+2>[#4]}})\IfValueT{#6}{[#6]}}
\NewDocumentCommand{\HMH}{sO{\bullet}D<>{*}mod<>}{\operatorname{HM}\IfBooleanTF{#1}{^{#2}}{^{#2,#3}}(#4\IfValueTF{#5}{,#5}{,\Hclass{\Astr<d+2>[#4]}})\IfValueT{#6}{[#6]}}

\newcommand{\HMd}{d_{\mathrm{HM}}}

\NewDocumentCommand{\BimHMC}{sD!!{A}O{\bullet}O{\star}D<>{*}mO{d+2}d<>}{\operatorname{CM}_{#2^e}\IfBooleanTF{#1}{^{#3,#4,#5}}{^{#3,#5}}(#6,\Hclass{\Astr<#7>[\IfValueTF{#8}{#8}{#2\ltimes{#6}}]})}
\NewDocumentCommand{\BimHMH}{sD!!{A}O{\bullet}O{\star}D<>{*}mO{d+2}d<>}{\operatorname{EM}_{#2^e}\IfBooleanTF{#1}{^{#3,#4,#5}}{^{#3,#5}}(#6,\Hclass{\Astr<#7>[\IfValueTF{#8}{#8}{#2\ltimes{#6}}]})}

\newcommand{\BimHMd}{d_{\mathrm{BimHM}}}

\NewDocumentCommand{\YonedaProd}{mm}{{#1}\circ{#2}}

\NewDocumentCommand{\preLie}{mm}{{#1}\bullet{#2}}
\NewDocumentCommand{\cupp}{mm}{{#1}\cdot{#2}}

\NewDocumentCommand{\braces}{mm}{{#1}\{#2\}}
\DeclareMathOperator{\Sq}{Sq}

\NewDocumentCommand{\Astr}{sD<>{}O{A}}{\IfBooleanTF{#1}{\overline{m}}{m}_{#2}^{#3}}
\NewDocumentCommand{\AstrAux}{sD<>{}O{A}}{\widetilde{m}_{#2}^{#3}}

\NewDocumentCommand{\AllBims}{s}{\operatorname{Bim}_{A_{\infty}}\IfBooleanT{#1}{^{\vee}}}
\NewDocumentCommand{\AllAlgs}{s}{\operatorname{Alg}_{A_{\infty}}\IfBooleanT{#1}{^{\vee}}}

\NewDocumentCommand{\graft}{mm}{{#1}\vee{#2}}
\NewDocumentCommand{\PBT}{O{n}}{\operatorname{PBT}_{#1}}

\NewDocumentCommand{\Map}{O{}mm}{\operatorname{Map}_{#1}(#2,#3)}
\NewDocumentCommand{\dgOp}{}{\operatorname{dgOp}}
\NewDocumentCommand{\opA}{O{\infty}}{\mathbb{A}_{#1}}
\NewDocumentCommand{\opEnd}{m}{\mathcal{E}(#1)}
\NewDocumentCommand{\opLinEnd}{mm}{\mathcal{E}(#1,#2)}
\NewDocumentCommand{\Str}{O{\opA}D<>{\Astr}m}{\operatorname{Str}_{#1,#2}(#3)}
\NewDocumentCommand{\dSS}{O{}}{\mathrm{d}_{#1}}

\NewDocumentCommand{\gLambda}{O{\sigma}D<>{d}}{\Lambda(#1,#2)}
\newcommand{\bfLambda}{\mathbf{\Lambda}}
\NewDocumentCommand{\Aut}{sm}{\operatorname{Aut}(#2)}
\NewDocumentCommand{\Out}{sm}{\operatorname{Out}(#2)}
\NewDocumentCommand{\InnAut}{sm}{\operatorname{Inn}(#2)}
\NewDocumentCommand{\Pic}{sm}{\operatorname{Pic}(#2)}
\NewDocumentCommand{\twBim}{sO{1}mO{1}}{\IfBooleanTF{#1}{{}_{#2}{#3}}{{}_{#2}{#3}_{#4}}}

\begin{document}

\title[The Derived Auslander--Iyama Correspondence II]%
{The Derived Auslander--Iyama Correspondence II: Bimodule Calabi--Yau Structures}%

\author[G.~Jasso]{Gustavo Jasso}%

\address[G.~Jasso]{%
  Mathematisches Institut, %
  Universität zu Köln, %
  Weyertal 86-90, %
  50931 Köln, %
  Germany}%
\email{gjasso@math.uni-koeln.de}%
\urladdr{https://gustavo.jasso.info}%

\author[F.~Muro]{Fernando Muro}%

\address[F.~Muro]{%
  Universidad de Sevilla, %
  Facultad de Matemáticas, %
  Departamento de Álgebra, %
  Calle Tarfia s/n, %
  41012 Sevilla, %
  Spain%
}%
\email{fmuro@us.es}%
\urladdr{https://personal.us.es/fmuro/}%



\begin{abstract}
  Let $d$ be a positive integer. In a previous article we established a bijective
correspondence between the following classes of objects, considered up to the
appropriate notion of equivalence: differential graded (dg) algebras with
finite-dimensional $0$-th cohomology such that the canonical generator of their
perfect derived category is a basic $d\ZZ$-cluster tilting object, and basic
Frobenius algebras that are twisted $(d+2)$-periodic as bimodules. In this
article, we prove a variant of our general correspondence for bimodule right
Calabi--Yau dg algebras. A novel ingredient is a new cohomology theory which
contains obstructions to the existence and uniqueness of minimal
$A_\infty$-bimodule structures on a graded bimodule. As an application of our
results, we obtain, to our knowledge, the first example of an algebraic
triangulated category with a triangulated Calabi--Yau structure that cannot be
lifted to a bimodule right Calabi--Yau structure on any of its dg enhancements.


\end{abstract}

\maketitle

\setcounter{tocdepth}{1}
\tableofcontents

\section{Introduction}

\subsection{Calabi--Yau triangulated categories}

We work over a field $\kk$. Let $\T$ be an algebraic\footnote{A triangulated
  category is \emph{algebraic} if it is equivalent to the homotopy category of a
  pre-triangulated dg category~\cite[Section~3.6]{Kel06}.} triangulated category
with finite-dimensional morphism spaces and split idempotents. Following
Kontsevich~\cite{Kon98}, the triangulated category $\T$ is
\emph{$n$-Calabi--Yau}, $n\in\ZZ$, if there exist natural isomorphisms
\[
  \T(y,x[n])\stackrel{\sim}{\longrightarrow}D\T(x,y),\qquad{x,y\in\T},
\]
where $V\mapsto DV$ denotes the passage to the linear dual, that exhibit the
$n$-fold power of the shift functor as a (graded) Serre functor in the sense
of~\cite{BK89}.\footnote{See~\cite[Section~2.6]{Kel08} and Van den
  Bergh's~\cite[Appendix~A]{Boc08} for the precise definition of a graded Serre
  functor as well as a discussion of sign issues.} The bounded derived category
of coherent sheaves on a smooth projective Calabi--Yau variety of dimension $n$
is an $n$-Calabi--Yau triangulated category, hence the terminology. Not least
due to Kontsevich's celebrated Homological Mirror Symmetry
Conjecture~\cite{Kon95}, Calabi--Yau triangulated categories also play an
important role in symplectic geometry. In representation theory of quivers and
algebras, prime examples of Calabi--Yau triangulated categories are (higher)
cluster categories~\cite{BMR+06,Kel05,Ami09,Guo11,FKQ24} and stable categories
of suitable self-injective algebras including preprojective algebras of Dynkin
quivers and higher-dimensional generalisations thereof~\cite{Ar96,Cra00,IO13}.
We refer the reader to the survey articles~\cite{Kel08,Kel10} for further
context and examples.

Recall that a differential graded (dg) algebra $\Gamma$ is \emph{homologically
  smooth} if it is compact as a dg $\Gamma$-bimodule. Following Kontsevich and
Ginzburg~\cite{Gin06}, such a dg algebra $\Gamma$ is \emph{bimodule left
  $n$-Calabi--Yau}, $n\in\ZZ$, if there exists an isomorphism
\[
  \RHom[\Gamma^e]{\Gamma}{\Gamma^e}\cong\Gamma[-n],\qquad
  \Gamma^e\coloneqq\Gamma\otimes\Gamma^\op,
\]
in the derived category of dg $\Gamma$-bimodules, and it is a general fact that
the finite-dimensional (=perfectly valued) derived category of $\Gamma$, that is
the full subcategory of the unbounded derived category $\DerCat{\Gamma}$ spanned
by the dg $\Gamma$-modules whose cohomology has finite total dimension, is an
$n$-Calabi--Yau triangulated category in this case~\cite[Lemma~4.1]{Kel08}. For
example, the dg algebra $C_*(\Omega X)$ of chains on the loop space of an
$n$-dimensional compact orientable manifold is bimodule left
$n$-Calabi--Yau~\cite{Lur09c,CG15} (see also~\cite[Theorem~5.4]{BD19}). More
generally, every homologically smooth dg algebra $B$ admits a (left)
$n$-Calabi--Yau completion $\mathbf{\Pi}_n(B)$, $n\in\ZZ$, that, as the name
suggests, has the bimodule left $n$-Calabi--Yau
property~\cite{Kel11,Yeu16,BCS24}. Keller's Calabi--Yau completion
simultaneously generalises the construction of the preprojective algebra of a
quiver~\cite{GP79,BGL87} and that of the Ginzburg dg algebra of a quiver with
potential~\cite{Gin06} if one takes into account deformations. For example, up
to quasi-isomorphism, preprojective algebras of extended Dynkin quivers arise as
$2$-Calabi--Yau completions, and there is a suitable generalisation of this
important fact to higher dimensions~\cite{HIO14}. It is worth mentioning that
(deformed) Calabi--Yau completions also arise in symplectic geometry, see for
example~\cite{EL17a,LU24}. In algebro-geometric terms, the passage to the
Calabi--Yau completion corresponds to the passage from a smooth algebraic
variety to the total space of its (shifted) cotangent bundle~\cite{BCS24}. In
\Cref{subsec:examples} we recall some examples that motivate our work in which
Calabi--Yau completions play a prominent role.

In this article, a sequel to~\cite{JKM22}, we investigate the question of how
the main theorem therein interacts with bimodule Calabi--Yau properties. We are
concerned with the following alternative notion of a Calabi--Yau dg algebra,
which is also due to Kontsevich~\cite{Kon93} (see also~\cite{KS09,BD19}). Let
$A$ be a dg algebra whose cohomology is degree-wise finite-dimensional. We say
that $A$ is \emph{bimodule right $n$-Calabi--Yau}, $n\in\ZZ$, if there exists an
isomorphism
\[
  A[n]\cong DA,\qquad DA\coloneqq\dgHom[\kk]{A}{\kk},
\]
in the derived category of dg $A$-bimodules, where $DA$ is the (differential
graded) linear dual of the diagonal $A$-bimodule. In this case, the perfect
derived category $\DerCat[c]{A}$, that is the full subcategory of $\DerCat{A}$
spanned by the compact objects, is $n$-Calabi--Yau as a triangulated category.
Under a suitable refinement, the bimodule right Calabi--Yau property can be
regarded as Koszul dual to the bimodule left Calabi--Yau property, see for
example~\cite[Proposition~10.4]{KW21} and the more recent~\cite{BCL25}. Under
Koszul duality, Keller's Calabi--Yau completion of a homologically smooth dg
algebra corresponds to Segal's cyclic completion of dg algebras with
finite-dimensional cohomology~\cite{Seg08}, which can be understood as a derived
analogue of the passage from a finite-dimensional algebra $\Lambda$ to its
trivial extension ${\Lambda\ltimes D\Lambda}$. This ties with the fact that the
bimodule right Calabi--Yau property can be understood as a derived version of
the notion of a symmetric algebra.\footnote{Recall that a (necessarily
  finite-dimensional) algebra is \emph{symmetric} if $\Lambda\cong D\Lambda$ as
  a $\Lambda$-bimodule.} In representation theory of algebras, cyclic
completions are central to the construction of higher cluster categories of
finite-dimensional algebras of finite global dimension~\cite{Kel05,Ami09,Guo11}.
Finally, we mention that bimodule right Calabi--Yau algebras play a central role
in Kontsevich and Soibelman's approach to Donaldson--Thomas
invariants~\cite{KS08} and are also abundant in symplectic geometry in the form
of certain flavours of Fukaya categories~\cite{Fuk10} (see also~\cite{CL11}).

\subsection{Main results}

In this article, we investigate the subtle relationship between Calabi--Yau
structures at the triangulated and dg-enhanced levels. The following is one of
our main results. To our knowledge, this result gives the first example of a
triangulated Calabi--Yau structure that cannot be lifted to a bimodule right
Calabi--Yau structure.

\begin{theorem*}[{see~\Cref{thm:non-enhanceable-CY} and~\Cref{rmk:A-dual-numbers-periodicity}}]
  Let $\kk$ be a field with $\chark(\kk)=2$ and consider the dg algebra
  \begin{equation}
    \label{eq:A-intro}
    A\coloneqq\frac{\kk[e,t^{\pm1}]\langle h\rangle}{(h^2,ht+th,he+eh+1)},\qquad
    |e|=|h|=0,\quad |t|=1,
  \end{equation}
  endowed with the differential
  \[
    d(e)=0,\qquad d(t)=0,\qquad d(h)=e^2t,
  \]
  whose cohomology is isomorphic to the graded Laurent polynomial algebra
  \[
    \bfLambda=\Lambda[\imath^{\pm1}],\qquad |\imath|=-1,
  \]
  where $\Lambda=\kk[\varepsilon]/(\varepsilon^2)$ is the algebra of dual
  numbers.\footnote{This dg algebra was first considered
    in~\cite[Remark~8]{MSS07}.} Then, the nondegenerate symmetric associative
  pairing
  \[
    \pairing<-,->_{a,b}\colon\Lambda\times\Lambda\longrightarrow\kk,\qquad
    \pairing<1,1>_{a,b}=a,\qquad\pairing<1,\varepsilon>_{a,b}=b\neq0,
  \]
  defines a triangulated $0$-Calabi--Yau structure on $\DerCat[c]{A}$ that can
  be lifted to a bimodule right $0$-Calabi--Yau structure on $A$ if and only if $a=0$.
\end{theorem*}

The previous theorem is established through a careful analysis of the obstructions to
lifting graded-bimodule isomorphisms on cohomology to quasi-isomorphisms of dg
bimodules (see~\Cref{thm:CY-equivalence} below). In fact, this
obstruction theory can be developed and applied in a broader setting, as we
now explain.

The central concept underlying this article is that of
a cluster tilting object, whose definition is due to Iyama~\cite{Iya07} in the
abelian setting and Iyama and Yoshino~\cite{IY08} in the triangulated setting,
see also~\cite{GKO13,IJ17} for the `periodic' variant below. For ${d=2}$,
cluster tilting objects play a fundamental role in the additive categorification
programme of Fomin--Zelevinsky's cluster algebras~\cite{BMR+06,Ami09,Kel08a}
and, in general, are central to Iyama's higher-dimensional version of
Auslander--Reiten Theory~\cite{Iya07,Iya07a,Iya11}. We refer the reader to the
introduction to our previous article~\cite{JKM22} for a more thorough discussion
of our motivation.

\begin{definition*}[Iyama--Yoshino] Let $d\geq1$. A (basic) object $C\in\T$ is
  \emph{$d\ZZ$-cluster tilting}\footnote{If the triangulated category $\T$ is
    not necessarily $\operatorname{Hom}$-finite, it is necessary to impose the
    additional condition that the subcategory $\add*{A}\subseteq\DerCat[c]{A}$
    is functorially finite in the sense of~\cite{AS80}.} if
  \[
    \forall i\not\in d\ZZ,\qquad \T(C,C[i])=0,
  \]
  and the following equalities of full subcategories of $\T$ hold:
  \begin{align*}
    \add*{C}&=\set{X\in\T}[\forall 0<i<d,\ \T(X,C[i])=0]\\&=\set{Y\in\T}[\forall 0<i<d,\ \T(C,Y[i])=0],
  \end{align*}
  where $\add*{C}\subseteq\T$ denotes the closure of $C$ under the formation of
  finite direct sums and retracts in $\T$. In the presence of the second
  condition, the first vanishing condition is equivalent to the existence of an
  isomorphism $C\cong C[d]$ in $\T$.
\end{definition*}
If $d=1$ then a $1\ZZ$-cluster tilting object is simply an additive generator of
$\T$. By definition, a triangulated category $\T$ is \emph{additively finite} if
and only if it admits an additive generator. Thus, the existence of a
$1\ZZ$-cluster tilting object in $\T$ places us precisely in the setup of the
second-named author's~\cite{Mur22}. For example, the perfect derived category of
the dg algebra in \eqref{eq:A-intro} is additively finite. More generally,
although it is not obvious from the definition, a $d\ZZ$-cluster tilting object
must generate $\T$ as a triangulated category with split
idempotents~\cite[Theorem~3.1]{IY08}. Since the ambient category $\T$ is assumed
to be algebraic, the presence of a $d\ZZ$-cluster tilting object $C\in\T$
therefore implies the existence of a dg algebra $A$ and an equivalence of
triangulated categories~\cite{Kel94}
\[
  \DerCat[c]{A}\stackrel{\sim}{\longrightarrow}\T,\qquad A\longmapsto C.
\]
The above equivalence induces an isomorphism of graded algebras
\[
  \textstyle\H{A}\cong\bigoplus_{di\in d\ZZ}\T(C,C[di]),\qquad g*f\coloneqq
  g[dj]\circ f,\qquad |f|=dj;
\]
since the object $C\in\T$ is $d\ZZ$-cluster tilting, these graded algebras are
`$d$-sparse' in the sense that they are concentrated in degrees $d\ZZ$. The main
results in our previous works,~\cite{Mur20b} for $d=1$ and
\cite[Theorem~A]{JKM22} in general, show that---remarkably---the dg algebra $A$
above is uniquely determined up to quasi-isomorphism by the finite-dimensional
algebra $\H[0]{A}\cong\T(C,C)$ and the $\H[0]{A}$-bimodule
$\H[-d]{A}\cong\T(C,C[-d])$. Therefore, it is natural to ask whether the
property of the dg algebra $A$ being bimodule right Calabi--Yau can be detected
at the level of the graded algebra $\H{A}$ as well. For example, an evident
necessary condition is the existence of an isomorphism of graded
$\H{A}$-bimodules\footnote{Here and throughout the article, we write $V\mapsto V(1)$
  for the shift of a graded vector space, and reserve the notation $V\mapsto
  V[1]$ for cochain complexes/differential graded vector spaces.}
\[
  \H{A}(n)\stackrel{\sim}{\longrightarrow}D\H{A},
\]
which in cohomological degree $0$ implies the existence of an isomorphism of
$\H[0]{A}$-bimodules
\[
  \T(C,C[n])=\H[n]{A}\stackrel{\sim}{\longrightarrow}D\H[0]{A}=D\T(C,C).
\]
Before stating the first main result in this article, we remind the reader of
the Hochschild cohomology
\[
  \HH{\bfLambda}=\HH{\bfLambda}[\bfLambda]\coloneqq\BimHH!\bfLambda^e!{\bfLambda},
\]
defined for any graded algebra $\bf\Lambda$; in particular we can take
$\bfLambda=\H{A}$.

\setcounter{thmintro}{2}
\begin{theorem intro}
  \label{thm:CY-equivalence}
  Let $\kk$ be a perfect field. Fix integers $d\geq1$ and $m\in\ZZ$ and
  set\footnote{The dg algebras in \Cref{thm:CY-equivalence} have their
    cohomology concentrated in degrees $d\ZZ$, and hence can only have the
    bimodule right $n$-Calabi--Yau property for $n\in d\ZZ$.} $n\coloneqq md$.
  Let $A$ be a dg algebra such that $\H{A}$ is degree-wise
  finite-dimensional\footnote{Equivalently, the perfect derived category
    $\DerCat[c]{A}$ is $\operatorname{Hom}$-finite.} and whose $0$-th cohomology
  is a basic finite-dimensional algebra. Suppose $A\in\DerCat[c]{A}$ is a
  $d\ZZ$-cluster tilting object. Then, the following statements are equivalent:
  \begin{enumerate}
  \item\label{it:thm:CY-equivalence:CY} The dg algebra $A$ is bimodule right
    $n$-Calabi--Yau.
  \item\label{it:thm:CY-equivalence:BV} There exists an isomorphism of graded
    $\H{A}$-bimodules
    \[
      \varphi\colon\H{A}(n)\stackrel{\sim}{\longrightarrow}D\H{A}
    \]
    and
    \[
      0=\BV(\Hclass{\Astr<d+2>[A]})\in\HH[d+1]<-d>{\H{A}}
    \]
    where $\BV=\BV_\varphi$ is the Batalin--Vilkovisky operator (\Cref{def:BV})
    and $\Hclass{\Astr<d+2>}$ is the universal Massey product of length $d+2$ of
    $A$ (\Cref{def:UMP-dgAs}).
  \end{enumerate}
  Moreover, a graded $\H{A}$-bimodule isomorphism
  $\varphi\colon\H{A}(n)\stackrel{\sim}{\to}D\H{A}$ satisfies the vanishing
  condition in statement \eqref{it:thm:CY-equivalence:BV} if and only if it is
  induced by a bimodule right $n$-Calabi--Yau structure on $A$.
\end{theorem intro}

The universal Massey product of length $d+2$ is an invariant associated with any
dg algebra with $d$-sparse cohomology. This invariant also plays a crucial role in the
proof of the main theorem in~\cite{Mur22} and in the proof of~\cite[Theorem~A]{JKM22}. The
vanishing condition in \Cref{thm:CY-equivalence}\eqref{it:thm:CY-equivalence:BV}
is of obstruction-theoretic nature. The Batalin--Vilkovisky (BV) operator is
obtained, as explained in \cite{Men09}, by dualising Connes' boundary operator
$B$ on the (normalised) Hochschild chain complex of $\H{A}$ by means of the graded bimodule
isomorphism $\varphi$. BV operators have been considered in relation to
smooth Calabi--Yau algebras already in Ginzburg's seminal preprint~\cite{Gin06} (see
also \cite{CEG07}) and have been further investigated in this and other closely
related contexts by several authors, see for
example~\cite{Lam10,CCEY21,CYZ16,LZZ16,Vol16} among many others. Furthermore,
the well-known relationship between the $B$ operator and the ISB sequence of
Connes relates the BV operator to cyclic and negative cyclic homology, see for
example~\cite[Section~8]{CYZ16}, and the latter cohomology theories are known to
control the deformation theory of smooth Calabi--Yau
algebras~\cite{PS95,dTdVVdB18}. 
From this point of view, the vanishing condition in
\Cref{thm:CY-equivalence}\eqref{it:thm:CY-equivalence:BV} seems natural, and
is in fact satisfied by every bimodule right Calabi--Yau dg algebra with
$d$-sparse cohomology (\Cref{prop:CY-BV}). Under the hypothesis
of~\Cref{thm:CY-equivalence}, provided that a graded $\H{A}$-bimodule
isomorphism $\H{A}(n)\cong D\H{A}$ exists, we do not know if the condition
\[
  \BV_\varphi(\Hclass{\Astr<d+2>[A]})=0
\]
is always satisfied for a suitable choice of bimodule isomorphism $\varphi$. In
\Cref{prop:BV-non-zero} we provide an explicit example that shows that, in
general, the above vanishing does depend on the choice of $\varphi$. In any
case, it is interesting that this vanishing condition is the only obstruction to
the validity of the bimodule Calabi--Yau property in our context. Finally, we
mention that, in symplectic geometry, the obstruction for the existence of an
\emph{exact} left Calabi--Yau structure on the wrapped Fukaya category of a
nondegenerate Liouville manifold can also be formulated in terms of the
BV operator~\cite[Proposition~4]{Li24}.

\Cref{thm:CY-equivalence} is a consequence of a more precise statement. Let
$\Lambda$ be a basic Frobenius algebra and $I$ an invertible $\Lambda$-bimodule.
Since $\Lambda$ is basic, we may choose an algebra automorphism
${\sigma\colon\Lambda\stackrel{\sim}{\to}\Lambda}$, unique up to inner
automorphisms, such that $I\cong\twBim{\Lambda}[\sigma]$ as
$\Lambda$-bimodules.\footnote{More precisely, the map
  \[
    \Out{\Lambda}\stackrel{}{\longrightarrow}\Pic{\Lambda},\qquad[\sigma]\longmapsto[\twBim{\Lambda}[\sigma]],
  \]
  is a group isomorphism, where
  $\Out{\Lambda}\coloneqq\Aut{\Lambda}/\InnAut{\Lambda}$ is the outer
  automorphism group of $\Lambda$ and $\Pic{\Lambda}$ is the Picard group of
  invertible $\Lambda$-bimodules~\cite[Proposition~3.8]{Bol84}. } We introduce
the graded algebra
\[
  \textstyle\bfLambda=\gLambda\coloneqq\bigoplus_{di\in
    d\ZZ}\twBim[\sigma^i]{\Lambda},\qquad x*y\coloneqq\sigma^j(x)\cdot y,\quad
  |y|=dj,
\]
where $(x,y)\mapsto x\cdot y$ denotes the product in $\Lambda$. By construction,
$\bfLambda$ is concentrated in degrees $d\ZZ$ and it is easy to see that
$\bfLambda$ can be described as the $\sigma$-twisted Laurent polynomial algebra
with coefficients in $\Lambda$ in a variable of degree $-d$. We are chiefly
interested in the case where there exists an isomorphism
\[
  \Omega_{\Lambda^e}^{d+2}(\Lambda)\cong\twBim{\Lambda}[\sigma]
\]
in the stable category of $\Lambda$-bimodules, so that $\Lambda$ is
\emph{twisted $(d+2)$-periodic with respect to $\sigma$}, see~\cite{ES08} for a
survey and \cite{CDIM20,JKM22,KLW24} and the references therein for more
information on this class of algebras. The inclusion of the {degree $0$}
component $j\colon\Lambda=\bfLambda^0\hookrightarrow\bfLambda$ induces a
restriction map
\[
  j^*\colon\HH{\bfLambda}[\bfLambda]\longrightarrow\HH{\Lambda}[\bfLambda]\cong\Ext[\Lambda^e][\bullet]{\Lambda}{\bfLambda^*}.
\]
We are ready to state the second main result in this article, which is a
Calabi--Yau variant of~\cite[Theorem~A]{JKM22}.

\begin{theorem intro}[Calabi--Yau Auslander--Iyama Correspondence]
  \label{thm:CY-correspondence}
  Let $\kk$ be a perfect field. Fix integers $d\geq1$ and $m\in\ZZ$ and set
  $n\coloneqq md$. There are bijective correspondences between the following:
  \begin{enumerate}
  \item\label{it:thm:CY-correspondence:dgAs} Quasi-isomorphism classes of dg
    algebras $A$ with the following properties:
    \begin{enumerate}
    \item\label{it:thm:CY-correspondence:dgAs:it:H0} The $0$-th cohomology $\H[0]{A}$ is a basic
      finite-dimensional algebra.
    \item\label{it:thm:CY-correspondence:dgAs:it:dZ-CT} The free dg $A$-module $A\in\DerCat[c]{A}$ is a $d\ZZ$-cluster tilting
      object.
    \end{enumerate}
    Moreover, we require the dg algebra $A$ to satisfy the following property:
    \begin{enumerate}
      \setcounter{enumii}{2}
    \item\label{it:thm:CY-correspondence:dgAs:it:CY} The dg algebra $A$ is bimodule right $n$-Calabi--Yau.
    \end{enumerate}
  \item\label{it:thm:CY-correspondence:gLambda} Equivalence classes of pairs
    $(\Lambda,I)$ consisting of
    \begin{enumerate}
    \item\label{it:thm:CY-correspondence:gLambda:it:tp} a basic Frobenius algebra $\Lambda$ that is twisted $(d+2)$-periodic
      and
    \item\label{it:thm:CY-correspondence:gLambda:it:I} an invertible $\Lambda$-bimodule $I$ such that $\Omega_{\Lambda^e}^{d+2}(\Lambda)\cong I$
      in the stable category of $\Lambda$-bimodules.
    \end{enumerate}
    Moreover, we require the pair $(\Lambda,I)$ to satisfy the following
    property:
    \begin{enumerate}
      \setcounter{enumii}{2}
    \item\label{it:thm:CY-correspondence:gLambda:it:CY}
      There exists an
      isomorphism of graded $\bfLambda$-bimodules\footnote{Recall that the
        induced graded bimodule structures on shifts and linear duals involve additional
      signs. In the special case $n=d$, it suffices to require the existence of an isomorphism of
      $\Lambda$-bimodules
      \[
        I\cong\twBim[\sigma]{\Lambda}\stackrel{\sim}{\longrightarrow}D\Lambda
      \]
      which is to say that $\sigma$ is a Nakayama automorphism for the
      Frobenius algebra $\Lambda$, see
      \Cref{coro:graded-bimodule-iso-vs-pairing-trivial} for the precise statement.}
      \[
        \varphi\colon\bfLambda(n)\stackrel{\sim}{\longrightarrow}D{\bfLambda},
      \]
      where $\sigma\colon\Lambda\stackrel{\sim}{\to}\Lambda$ is an algebra
      automorphism such that $I\cong\twBim{\Lambda}[\sigma]$ as
      $\Lambda$-bimodules. Moreover, given any exact sequence of
      $\Lambda$-bimodules
      \[
        \eta\colon\quad 0\to I\to P_{d+2}\to\cdots\to P_3\to P_2\to
        P_1\to\Lambda\to 0
      \]
      with projective(-injective) middle terms, we have
      \[
        0=\BV(\Hclass{\Astr<d+2>[\eta]})\in\HH[d+1]<-d>{\bfLambda},
      \]
      where $\BV=\BV_\varphi$ is the Batalin--Vilkovisky operator
      (\Cref{def:BV}), and $\Hclass{\Astr<d+2>[\eta]}\in\HH[d+2]<-d>{\bfLambda}$ is
      the unique class whose Gerstenhaber square vanishes, that is
      \[
        0=\Sq(\Hclass{\Astr<d+2>[\eta]})\in\HH{\bfLambda},
      \]
      and such that
      \[
        j^*\Hclass{\Astr<d+2>[\eta]}=\Hclass{\eta}\in\HH[d+2]<-d>{\Lambda}[\bfLambda]\cong\Ext[\Lambda^e][d+2]{\Lambda}{I}.
      \]
    \end{enumerate}
  \end{enumerate}
  The correspondence
  $\eqref{it:thm:CY-correspondence:dgAs}\to\eqref{it:thm:CY-correspondence:gLambda}$
  is given by $A\mapsto(H^0(A),H^{-d}(A))$. Moreover, a graded
  $\bfLambda$-bimodule isomorphism
  $\varphi\colon\bfLambda(n)\stackrel{\sim}{\to}D{\bfLambda}$ satisfies the
  vanishing condition in statement
  \eqref{it:thm:CY-correspondence:gLambda:it:CY} for some $\eta$ if and only if
  it is induced by a bimodule right $n$-Calabi--Yau structure on any choice of
  dg algebra $A$ satisfying the conditions in statement
  \eqref{it:thm:CY-correspondence:dgAs}.
\end{theorem intro}

In the context of \Cref{thm:CY-correspondence}, two pairs $(\Lambda,I)$ and
$(\Lambda',I')$ as in~\eqref{it:thm:CY-correspondence:gLambda} are equivalent if
there exists an algebra isomorphism
$\varphi\colon\Lambda\stackrel{\sim}{\to}\Lambda'$ such that $\varphi^*(I')\cong
I$. We also mention that some of the conditions in
\eqref{it:thm:CY-correspondence:gLambda:it:CY} can be formulated purely in terms of
$\Lambda$, see \Cref{prop:graded-bimodule-iso-vs-pairing}.

Removing the bimodule Calabi--Yau properties
\eqref{it:thm:CY-correspondence:dgAs:it:CY} and
\eqref{it:thm:CY-correspondence:gLambda:it:CY} from \Cref{thm:CY-correspondence}
yields precisely~\cite[Theorem~A]{JKM22}, which is the main theorem therein.
Hence the theorem follows from~\cite[Theorem~A]{JKM22} as soon as we verify that
the correspondence is compatible with the
Calabi--Yau properties \eqref{it:thm:CY-correspondence:dgAs:it:CY} and
\eqref{it:thm:CY-correspondence:gLambda:it:CY}. This is non-trivial. As we
explain below, our proof of~\Cref{thm:CY-correspondence} makes use of techniques
of homotopy theory, in particular of an extension to
$A_\infty$-bimodules~\cite{JM25} of the enhanced obstruction theory for
$A_\infty$-algebras developed by the second-named author in~\cite{Mur20b} and,
in this way, it has some formal similarities with our proof
of~\cite[Theorem~A]{JKM22}.

\begin{uremark}
  There is a version of \Cref{thm:CY-correspondence}, which follows immediately
  from~\cite[Theorem~A]{JKM22}, where items \eqref{it:thm:CY-correspondence:dgAs:it:CY} and
\eqref{it:thm:CY-correspondence:gLambda:it:CY} are replaced by
  the following weaker conditions (recall that $n=md$):
  \begin{itemize}
    \item[(1c')] The additive category $\add*{A}\subseteq\DerCat[c]{A}$ is $n$-Calabi--Yau in the
      sense that there are natural isomorphisms of vector spaces
      \[
        \Hom[A]{Q}{P[n]}\stackrel{\sim}{\longrightarrow}D\Hom[A]{P}{Q},\qquad P,Q\in\add*{A}.
      \]
      Here, $\add*{A}$ denotes the smallest additive subcategory of $\DerCat[c]{A}$
      that contains $A$ and that is closed under retracts. In other words, the
      $n$-fold shift is a Serre functor on the additive category $\add*{A}$.
    \item[(2c')] There exists an isomorphism of $\Lambda$-bimodules
      \[
        \varphi\colon\twBim[\sigma^m]{\Lambda}\stackrel{\sim}{\longrightarrow}D\Lambda,
      \]
      where $\sigma\colon\Lambda\stackrel{\sim}{\longrightarrow}\Lambda$ is an
      algebra automorphism such that $I\cong\twBim{\Lambda}[\sigma]$ as
      $\Lambda$-bimodules.
    \end{itemize}
    Indeed, this follows from the equivalence of pairs
    \[
      (\add*{A},[-d])\stackrel{\sim}{\longrightarrow}(\proj*{\Lambda},-\otimes_{\Lambda}\twBim{\Lambda}[\sigma])
    \]
    given by the Yoneda functor
    \[
      \Hom[A]{A}{-}\colon\add*{A}\stackrel{\sim}{\longrightarrow}\proj*{\Lambda},
    \]
    where $\Lambda=\H[0]{A}$ (see also~\cite[Corollary~4.5.20]{JKM22}). There is
    a further version of \Cref{thm:CY-correspondence}, which also follows immediately
    from~\cite[Theorem~A]{JKM22}, where items \eqref{it:thm:CY-correspondence:dgAs:it:CY} and
\eqref{it:thm:CY-correspondence:gLambda:it:CY} are now replaced by
    the following conditions:
    \begin{itemize}
    \item[(1c'')] There are natural isomorphisms of vector spaces
      \[
        \Hom[A]{Q}{P[n]}\stackrel{\sim}{\longrightarrow}D\Hom[A]{P}{Q},\qquad P,Q\in\add*{A},
      \]
      that exhibit the $n$-fold shift as a graded Serre functor on the
      $d$-sparse graded category associated with the pair $(\add*{A},[d])$,
      see~\cite[Appendix~A]{Boc08}. Equivalently, the $n$-fold shift is a (graded)
      Serre functor on the $(d+2)$-angulated category $(\add*{A},[d],\pentagon)$,
      where $\pentagon$ is the standard $(d+2)$-angulation obtained
      from~\cite[Theorem~1]{GKO13}.
    \item[(2c'')] There exists an isomorphism of graded $\bfLambda$-bimodules
      \[
        \varphi\colon\bfLambda(n)\stackrel{\sim}{\longrightarrow}D{\bfLambda},
      \]
      where $\sigma\colon\Lambda\stackrel{\sim}{\to}\Lambda$ is an algebra
      automorphism such that $I\cong\twBim{\Lambda}[\sigma]$ as
      $\Lambda$-bimodules.
    \end{itemize}
    Thus, the main difficulty in establishing \Cref{thm:CY-correspondence} is that of
    promoting a graded bimodule isomorphism as
    in~\eqref{it:thm:CY-correspondence:gLambda:it:CY}
    to a \emph{quasi-isomorphism} of dg bimodules as
    in~\eqref{it:thm:CY-correspondence:dgAs:it:CY}.
\end{uremark}

\subsection{Examples}
\label{subsec:examples}

\Cref{thm:CY-equivalence,thm:CY-correspondence} are motivated by the following
examples. For $d=1$, \Cref{thm:CY-equivalence,thm:CY-correspondence} can be
regarded as statements concerning additively-finite Calabi--Yau triangulated
categories. Prominent examples of such categories are the $m$-Calabi--Yau
cluster categories associated with Dynkin quivers, see~\cite{BMR+06} for the case
$m=2$ and \cite[Corollary~1]{Kel05} for the case $m\in\ZZ$. It is also worth
noting that, over an algebraically closed field, the weak\footnote{That is, the
  $n$-fold shift functor is a Serre functor but not necessarily in the graded
  sense.} $1$-Calabi--Yau additively-finite triangulated categories have been
classified by Amiot~\cite[Theorem 9.5]{Ami07}, leveraging
\cite[Theorem~1.2]{BES07}.

The aforementioned examples generalise as follows. Let $d\geq1$. A (basic)
finite-dimensional algebra $H$ of global dimension at most $d$ is
\emph{$d$-representation-finite ($d$-hereditary)} if there exists a $d$-cluster
tilting $H$-module $M\in\mmod*{H}$~\cite{IO11}. The $d$-Calabi--Yau
Amiot--Guo--Keller (AGK) cluster category of $H$,
\[
  \DerCat[c]{\mathbf{\Pi}}\stackrel{\pi}{\longrightarrow}\C_d(H)\coloneq\DerCat[c]{\mathbf{\Pi}}/\DerCat[fd]{\mathbf{\Pi}},\qquad
  \mathbf{\Pi}\coloneq\mathbf{\Pi}_{d+1}(H),
\]
is a $d$-Calabi--Yau triangulated category and
$\pi\mathbf{\Pi}\in\C(\mathbf{\Pi})$ is a $d\ZZ$-cluster tilting
object~\cite{Ami09,Guo11,Kel05,IO13,IY20}, see also~\cite[Section~6.3]{JKM22}.
More generally, the $md$-Calabi--Yau AGK cluster category $\C_{md}(H)$, $m\geq1$, admits a $d\ZZ$-cluster tilting
object~\cite[Section~5]{OT12}.\footnote{In~\cite{OT12} the authors treat the
  case $m=2$ but, relying on the results in~\cite{Guo11,IY20}, their methods
  generalise to the case $m\geq1$~\cite{OWJas17,JKb}.}

\subsection{Applications}

\subsubsection{Contractibility of curves in Calabi--Yau threefolds}

Our work has potential applications in algebraic geometry. Consider the
polynomial algebra $\kk[u^{-1}]$ with the variable $u^{-1}$ placed in
cohomological degree $-2$. A dg algebra is said to be
\emph{$\kk[u^{-1}]$-enhanced} if it is isomorphic to a dg $\kk[u^{-1}]$-algebra
in the homotopy category of dg $\kk$-algebras. Hua and Keller conjecture that
the derived deformation (dg) algebra $\Gamma=\Gamma_C^Y$ associated with an nc
rigid rational curve $C\subset Y$ in a quasi-projective Calabi--Yau threefold is
$\kk[u^{-1}]$-enhanced if and only if the curve is
contractible~\cite[Conjecture~6.8]{HK24} (we refer the reader to
\emph{loc.~cit.}~for an explanation of the terminology). Moreover, they prove
the necessity of their conjecture by showing that, when the curve $C$ contracts,
the dg algebra $\Gamma$ is quasi-isomorphic to the connective cover
$\tau^{\leq0}A$ of a dg algebra $A$ that satisfies the conditions in
\Cref{thm:CY-equivalence} for $d=2$ and $m=1$ and such that the compact derived
category $\DerCat[c]{A}$ is quasi-equivalent to the dg category of maximal
Cohen--Macaulay modules over the ring of formal functions of the singularity
underlying the contraction, see~\cite[Proposition~6.9]{HK24} and Keller's
appendix to~\cite{JKM22} and \cite{JKM24}, where the role of the main theorem in
\cite{JKM22} in relation to the proof of the Donovan--Wemyss
Conjecture~\cite{DW16} is also explained.\footnote{A proof of the conjecture that does
not rely on~\cite[Theorem~A]{JKM22} has appeared recently in~\cite{KLW24}.} The
key point is that the latter dg category of maximal Cohen--Macaulay modules is
linear over the algebra of Laurent polynomials $\kk[u,u^{-1}]$ as a consequence
of a famous theorem of Eisenbud~\cite{Eis80}. Therefore, it is interesting to
find sufficient criteria for a dg algebra to be $\kk[u,u^{-1}]$-enhanced.

A necessary condition for a dg algebra $A$ to be $\kk[u,u^{-1}]$-enhanced is the
existence of an isomorphism $A[2]\cong A$ in the derived category of dg
$A$-bimodules. Suppose given a dg algebra $A$ that satisfies the assumptions in
\Cref{thm:CY-equivalence} for $d=2$ and $m=1$; in particular,
\[
  A[2]\cong DA
\]
in the derived category of dg $A$-bimodules. Suppose now that there exists a
further isomorphism
\[
  \H{A}(2)\cong \H{A}.
\]
If the additional conditions in \Cref{thm:CY-equivalence} are satisfied when
applied to the pair $(\H[0]{A},\H[-2]{A})$ for $d=2$ and $m=0$, the theorem
yields the existence of an isomorphism $A\cong DA$ in the derived category of dg
$A$-modules and therefore
\[
  A[2]\cong DA\cong A;
\]
in particular the perfect derived category $\DerCat[c]{A}$ is $2$-periodic. It
would therefore be interesting to determine if such a dg algebra $A$ is
necessarily $\kk[u,u^{-1}]$-enhanced.

\subsubsection{Frobenius Jacobian algebras}

Another potential application of our work is to a representation-theoretic
characterisation of the Frobenius Jacobian algebras. Suppose that $\kk$ is an
algebraically closed field of characteristic $0$. Let $Q$ be a finite quiver,
and recall that a potential $W$ on $Q$ is the datum of a (possibly infinite)
linear combination of cycles of length at least $2$ in $Q$. To the pair $(Q,W)$,
Derksen, Weyman and Zelevinsky~\cite{DWZ08} associate the Jacobian algebra
$J(Q,W)$. When $J(Q,W)$ is finite-dimensional, it is also
$1$-Iwanaga--Gorenstein, for $J(Q,W)$ can be realised as the endomorphism
algebra of a $2$-cluster tilting object in a $2$-Calabi--Yau triangulated
category~\cite{Ami09,KR07}. Moreover, $J(Q,W)$ is Frobenius (hence
$0$-Iwanaga--Gorenstein) if and only if it is the endomorphism algebra of a
$2\ZZ$-cluster tilting object~\cite{HI11}. Consequently, Frobenius Jacobian
algebras are twisted $4$-periodic with respect to the Nakayama automorphism $\nu
$ of the algebra, which moreover satisfies the conditions in
\Cref{prop:graded-bimodule-iso-vs-pairing} with $d=2$ and $m=1$ with respect to
a suitable choice of Frobenius pairing. Furthermore, as a consequence of
\Cref{thm:CY-equivalence}, the vanishing condition
\[
  0=\BV(\Hclass{\Astr<4>[\eta]})\in\HH[3]<-2>{\bfLambda}[\bfLambda]
\]
also holds, where $\Lambda=J(Q,W)$ is the Jacobian algebra and $\eta$ is a
suitable exact sequence of $\Lambda$-bimodules that witnesses the existence of a
stable $\Lambda$-bimodule isomorphism
$\Omega_{\Lambda^e}^4(\Lambda)\simeq\twBim{\Lambda}[\nu]$. In view of the work of
Keller and Liu~\cite{KL23a,KL23b,KL23c} and \Cref{thm:CY-correspondence}, it is
tempting to conjecture that these properties almost characterise this class of
algebras, see also~\cite[Section~6.3]{JKM22} and \Cref{subsec:future_work}.

\subsection{Future work}
\label{subsec:future_work}

Bimodule Calabi--Yau properties for dg algebras admit refinements to what are
called Calabi--Yau \emph{structures}~\cite{KS09}, see~\cite{BD19,KW21} for
interesting applications of these refined notions in representation theory of
finite-dimensional algebras and elsewhere in mathematics. In future work we hope
to investigate whether the dg algebras $A$ in
\Cref{thm:CY-correspondence}\eqref{it:thm:CY-correspondence:dgAs} admit a
\emph{right Calabi--Yau structure}, that is a class $\widetilde{\varphi}\in
D\mathrm{HC}_{-n}(A)$ in the linear dual of the cyclic homology of $A$ whose
image
\[
  \varphi\in D\mathrm{HH}_{-n}(A)\cong\Hom[\DerCat{A^e}]{A[n]}{DA}
\]
in the Hochschild homology of $A$ induces an isomorphism
\[{\varphi\colon A[n]\stackrel{\sim}{\longrightarrow}DA}\] in the derived
category of dg $A$-bimodules. Right Calabi--Yau structures are known to exist in
numerous examples. Their existence is relevant not least due to the
recent proof by Keller and Liu~\cite{KL23a,KL23c} of a modified version of
Amiot's conjecture~\cite{Ami11} on the characterisation of the $2$-Calabi--Yau
cluster categories associated with quivers with potential (and of a
higher-dimensional generalisation thereof). Indeed, in addition to working in
the pseudo-compact setting, the main modification to the conjecture is an
additional assumption of existence of a right Calabi--Yau structure at the level
of enhancements---rather than the mere Calabi--Yau property at the level of
triangulated categories.

\subsection{Outline of the proof of \Cref{thm:CY-correspondence}}

Assume that the correspondence in \Cref{thm:CY-correspondence} is well
defined. It is then injective, for it is the restriction of the bijective
correspondence in~\cite[Theorem~A]{JKM22}. Hence, it is enough to show that the
inverse correspondence in the latter theorem restricts to a map
$\eqref{it:thm:CY-correspondence:gLambda}\to\eqref{it:thm:CY-correspondence:dgAs}$
under the additional assumptions in \Cref{thm:CY-correspondence}. Notice,
however, that the right Calabi--Yau property is a statement about dg
$A$-bimodules rather than about the dg algebra $A$ itself. For this reason, we
need to appeal to an enhanced obstruction theory for the existence and
uniqueness of $A_\infty$-bimodule structures that is developed in~\cite{JM25}
very much along the lines of its precursor for $A_\infty$-algebras introduced by
the second-named author in~\cite{Mur20b}. In particular, we use a novel
cohomology theory introduced in~\cite{JM25} called Massey bimodule cohomology
(\Cref{def:RelBimHMC}). This permits us to establish the following Calabi--Yau
analogue of \cite[Theorem~B]{JKM22}, from which \Cref{thm:CY-correspondence}
follows using computations similar to those carried out in
\cite{Mur22,JKM22}.

\begin{theorem intro}
  \label{thm:CY-Kadeikshvili}
  Let $A$ be a dg algebra whose cohomology is concentrated in degrees $d\ZZ$,
  $d\geq1$. Suppose that the following Massey bimodule cohomology
  satisfies\footnote{Here and throughout the article, we write $A(0)$ to
    emphasise that we wish to treat $A$ as a bimodule rather than as an
    algebra.}
  \[
    \BimHMH!\H{A}![p+1]<-p>{\H{A}}<A\ltimes A(0)>=0,\qquad p>d.
  \]
  Suppose also that there exists an isomorphism of graded $\H{A}$-bimodules
  \[
    \varphi\colon \H{A}(n)\stackrel{\sim}{\longrightarrow}D\H{A}
  \]
  for some $n\in d\ZZ$. If we have
  \[
    0=\BV(\Hclass{\Astr<d+2>})\in\HH[d+1]<-d>{\H{A}}=\BimHH[d+1]<-d>{\H{A}},
  \]
  then there exists an isomorphism $\tilde{\varphi}\colon A[n]\simeq DA$ in the
  derived category of dg $A$-bimodules that induces $\varphi$ in cohomology. In
  particular, $A$ is bimodule right $n$-Calabi--Yau.
\end{theorem intro}

The class $\Hclass{\Astr<d+2>[A\ltimes A(0)]}$, that we call the \emph{bimodule
  universal Massey product of length $d+2$}, can be defined for any pair $(A,M)$
consisting of a dg algebra and a dg bimodule over it whose cohomologies are
concentrated in degrees $d\ZZ$. This new obstruction class lives in a cohomology
theory, called bimodule Hochschild cohomology, that is of independent
interest. \Cref{thm:CY-Kadeikshvili} can be regarded as a Kadeishvili-type
theorem: Given the vanishing of certain
obstructions of Hochschild type, it guarantees the existence of a quasi-isomorphism $A[n]\simeq
DA$ as soon as there is an isomorphism at the cohomological level. For the sake of comparison,
recall that Kadeishvili's Formality Theorem~\cite{Kad88} says that, given the
vanishing of certain obstructions of Hochschild type, two dg algebras whose
cohomologies are isomorphic as graded algebras are already quasi-isomorphic as
dg algebras. Ultimately, \Cref{thm:CY-Kadeikshvili} is obtained as an
application of the more general~\cite[Theorem~6.2.7]{JM25}, in which the relevant obstruction is
expressed as an equality of \emph{bimodule} universal Massey products. While the
necessary obstruction theory is developed in~\cite{JM25}, one of the main
technical contributions in this article is to show that, in the case of the
diagonal bimodule, this obstruction is equivalent to the vanishing of the BV
operator on the \emph{algebra} universal Massey product. This is achieved by
establishing an explicit isomorphism
\begin{align*}
  \RelBimHH{\H{A}}{\H{A}}&\stackrel{\sim}{\longrightarrow}\HH{\H{A}}{[\varepsilon]}/(\varepsilon^2),\qquad |\varepsilon|=(1,0),
\end{align*}
where the source denotes the bimodule Hochschild cohomology of the
diagonal bimodule of the graded algebra $\H{A}$, see~\Cref{coro:kappa} and the discussion below.

Given a \emph{graded} algebra $A$ and a \emph{graded} $A$-bimodule
$M$, the bimodule Hochschild cochain complex $\RelBimHC{A}{M}$ is part of a
standard triangle
\[
  \BimHC{M}<-1>\longrightarrow\RelBimHC{A}{M}\longrightarrow\HC{A}\stackrel{\delta}{\longrightarrow}\BimHC{M}
\]
of differential bigraded vector spaces. Here, $\HC{A}$ is the familiar
Hochschild complex of $A$ and
\[
  \BimHC{M}\cong\Hom[A^e]{\BB{A}\otimes_AM\otimes_A\BB{A}}{M}
\]
is the cochain complex that computes the bigraded vector space $\BimHH{M}$ of
graded-bimodule self-extensions; here, $\BB{A}$ is the bar resolution of $A$. The connecting map in the above triangle
\begin{align*}
  \HC{A}&\stackrel{\delta}{\longrightarrow}\BimHC{M}\\
  c&\longmapsto\cupp{\id[M]}{c}-\cupp{c}{\id[M]}
\end{align*}
is the commutator with the identity map of $M$. Here, we use the natural left and
right actions of $\HC{A}$ on $\BimHC{M}$ induced by the usual monoidal structure
on the derived category of $A$-bimodules. The
upshot is the existence of long exact sequences of graded vector spaces
\[
  \begin{tikzcd}[column sep=small,row sep=small]
    \cdots\rar{\delta}&\BimHH[n-1]{M}\rar{i}&\RelBimHH[n]{A}{M}\rar{p}&\HH[n]{A}\ar[out=-10,in=170,overlay]{dll}[description]{\delta}\\
    &\BimHH[n]{M}\rar{i}&\RelBimHH[n+1]{A}{M}\rar{p}&\HH[n+1]{A}\rar{\delta}&\cdots
  \end{tikzcd}
\]
in which the connecting maps vanish precisely when the graded $\HH{A}$-bimodule
$\BimHH{M}$ is graded-symmetric, see \Cref{prop:RelBimHC-all}. For example, this
vanishing occurs when $M=A$ is the diagonal bimodule since $\HH{A}$ is
graded-commutative with respect to the cup product. In fact, in the latter case
there is an explicit isomorphism of Gerstenhaber algebras (with respect to the
total degree)
\begin{equation*}
  \RelBimHH{A}{A}\stackrel{\sim}{\longrightarrow}\HH{A}{[\varepsilon]}/(\varepsilon^2),
\end{equation*}
where $\varepsilon$ is a central element of bidegree $(1,0)$, see
\Cref{coro:kappa}. Ultimately, this isomorphism is what permits us to relate the
obstruction to the validity of the bimodule right Calabi--Yau property in
\Cref{thm:CY-Kadeikshvili}, and then also in
\Cref{thm:CY-equivalence,thm:CY-correspondence}, to the BV operator, see
\Cref{prop:CY-Ai-aux}.

Returning to the case of dg algebras and dg bimodules with $d$-sparse
cohomology, the bimodule universal Massey product
\[
  \Hclass{\Astr<d+2>[A\ltimes M]}\in\RelBimHH[d+2]<-d>{\H{A}}{\H{M}}
\]
is defined in terms of any minimal model of the pair $(A,M)$; that is, by
definition, it is given by a minimal model of $A$ as an $A_\infty$-algebra and a
compatible minimal model of $M$ as an $A_\infty$-bimodule. For this reason, in
order to prove \Cref{thm:CY-Kadeikshvili}, we first prove a variant for minimal
$A_\infty$-algebras (\Cref{thm:CY-Kadeishvili-Ai}), for which we are required to
study $A_\infty$-bimodules in some detail as well.

\subsection{Structure of the article}

The proofs of
\Cref{thm:CY-equivalence,thm:CY-correspondence,thm:CY-Kadeikshvili} require
substantial general preliminaries. We hope that we have made the article
more accessible by including these preliminaries in detail. In
\Cref{sec:preliminaries} we recall basic aspects of the theory of dg algebras
that are needed throughout the article. In \Cref{sec:Hochschild-algebras} we
recall the definition and fundamental algebraic structure on the Hochschild
cohomology of a graded algebra. In \Cref{sec:Hochschild-bimodules} we recall
from \cite{JM25} the definition and algebraic structure of the bimodule
Hochschild cohomology associated with a graded bimodule, and study the bimodule
Hochschild cohomology of the diagonal bimodule in some detail, since this
is our main object of interest. In \Cref{sec:Ai-stuff}, we recall various aspects
of the theory of $A_\infty$-algebras and $A_\infty$-bimodules that are needed to
prove the crucial \Cref{thm:CY-Kadeikshvili} and its $A_\infty$-variant
\Cref{thm:CY-Kadeishvili-Ai}. In \Cref{sec:Kadeishvili}, we recall from
\cite{Mur20b,JKM22,JM25} the definition of the Hochschild--Massey cochain
complex and its variant for bimodules as well as important results concerning
them, and use these to prove \Cref{thm:CY-Kadeishvili-Ai} and
\Cref{thm:CY-Kadeikshvili}. In \Cref{sec:Thm:CY} we give our proofs of
\Cref{thm:CY-equivalence,thm:CY-correspondence} using
\Cref{thm:CY-Kadeikshvili}---it is only in this section that the specific
context of \Cref{thm:CY-equivalence,thm:CY-correspondence} becomes relevant.
Finally, in \Cref{sec:the-example} we give our example of a non-enhanceable
triangulated Calabi--Yau structure.

\begin{conventions}
  We work over an arbitrary field $\kk$. We say that an object $c$ in a
  $\kk$-linear category in which the Krull--Remak--Schmidt Theorem holds (for
  example, in an additive category with finite-dimensional morphism spaces and
  split idempotents \cite[Cor.~4.4]{Kra15}) is \emph{basic} if in any
  decomposition $c=c_1\oplus c_2\oplus\cdots\oplus c_n$ into indecomposable
  objects, there is an isomorphism $c_i\cong c_j$ if and only if $i=j$. A
  finite-dimensional algebra $\Lambda$ is \emph{basic} if the regular
  representation of $\Lambda$ is basic as an object of its category of
  finite-dimensional (projective) modules. The Jacobson radical of a
  finite-dimensional algebra $\Lambda$ is denoted by $J_\Lambda$; our main
  results require the ground field to be perfect or, more generally, that the
  semisimple algebra $\Lambda/J_\Lambda$ is separable over $\kk$. Finally, we
  compose morphisms in a category as functions: the composite of morphisms
  $f\colon x\to y$ and $g\colon y\to z$ is the morphism $gf=g\circ f\colon x\to
  z$.
\end{conventions}

\begin{table}
  \begin{center}
    \bgroup \def\arraystretch{1.33}%
    \begin{tabular}{|r | p{23em} l|}
      \hline
      $c_1\bullet_i c_2$&Infinitesimal composition of multilinear maps&\eqref{eq:infinitesimal_composition}\\
      $\braces{c_1}{c_2}$&Binary brace operation&\eqref{eq:binary_brace}\\
      $\braces{m}{x,y}$&Binary multiplication brace operation&\eqref{eq:generic-cup_product}\\
      \hline
      $\HC{A}$&Hochschild cochain complex of $A$&\eqref{eq:HC}\\
      $\preLie{c_1}{c_2}$&Pre-Lie product&\eqref{eq:pre-Lie_product-HC}\\
      $[c_1,c_2]$&Gerstenhaber bracket&\eqref{eq:Gerstenhaber_bracket-HC}\\
      $\Sq(c)$&Gerstenhaber square&\eqref{eq:Gerstenhaber_square-HC}\\
      $\Hd$&Hochschild differential&\eqref{eq:Hd}\\
      $\HH{A}$&Hochschild cohomology of $A$&\eqref{eq:HH}\\
      $\cupp{c_1}{c_2}$&Cup product&\eqref{eq:cup_product-HC}\\
      \hline
      $\BimHC{M}$&Bimodule cochain complex of $M$&\eqref{eq:BimHC}\\
      $\BimHH{M}$&Cohomology of $\BimHC{M}$&\eqref{eq:BimHH}\\
      $\dBim=\dh+\dv$&Bimodule cochain complex differential&\eqref{eq:dBim}\\
      $\dh$&Horizontal bimodule differential&\eqref{eq:dh}\\
      $\dv$&Vertical bimodule differential&\eqref{eq:dv}\\
      $\YonedaProd{c_1}{c_2}$&Yoneda product on $\BimHC{M}$&\eqref{eq:Yoneda_product-BimHC}\\
      $[c_1,c_2]$&Lie bracket on $\BimHC{M}$&\eqref{eq:Lie_bracket-BimHC}\\
      \hline
      $\RelBimHC{A}{M}$&Bimodule Hochschild cochain complex of $M$&\eqref{eq:RelBimHC}\\
      $\preLie{c_1}{c_2}$&Bimodule pre-Lie product&\eqref{eq:pre-Lie_product-RelBimHC}\\
      $[c_1,c_2]$&Bimodule Gerstenhaber bracket&\eqref{eq:Gerstenhaber_bracket-RelBimHC}\\
      $\Sq(c)$&Bimodule Gerstenhaber square&\eqref{eq:Gerstenhaber_square-RelBimHC}\\
      $m_2^{A\ltimes M}$&Brace algebra multiplication on $\RelBimHC{A}{M}$&\eqref{eq:square-zero_product-AM}\\
      $\dRelBim$&Bimodule Hochschild differential&\eqref{eq:Hd-RelBimHC}\\
      $\RelBimHH{A}{M}$&Bimodule Hochschild cohomology of $M$&\eqref{eq:RelBimHH}\\
      $\cupp{c_1}{c_2}$&Bimodule cup product&\eqref{eq:cup_product-RelBimHC}\\
      $\delta$&Connecting homomorphism ${\HC{A}\to\BimHC{M}}$ &\eqref{eq:delta}\\
      \hline
      $\HMC{A}$&Hochschild--Massey cochain complex of $A$&\eqref{eq:HMC}\\
      $\HMd$&Hochschild--Massey differential&\eqref{eq:HMd}\\
      $\HMH{A}$&Hochschild--Massey cohomology of $A$&\eqref{eq:HMH}\\
      \hline
      $\BimHMC{M}$&Massey bimodule cochain complex of $M$&\eqref{eq:RelBimHMC}\\
      $\BimHMd$&Massey bimodule differential&\eqref{eq:BimHMd}\\
      $\BimHMH{M}$&Massey bimodule cohomology of $M$&\eqref{eq:RelBimHMH}\\
      \hline
    \end{tabular}
    \egroup
  \end{center}\vspace{1em}

  \caption{The various Hochschild-type cochain complexes considered in this
    article, with their cohomologies and operations. Here, $A$ is a graded
    algebra and $M$ is a graded $A$-bimodule. In the two bottom blocks, $A$ and
    $M$ are assumed to be concentrated in degrees $d\ZZ$, $d\geq1$, and are
    endowed with minimal $A_\infty$-structures.}
  \label{table:Hochschild}
\end{table}

\section{Differential graded algebras}
\label{sec:preliminaries}

To fix notation and sign conventions, in this preliminary section we
recall the basic aspects of the representation theory of differential graded
algebras. All the material in this section is standard and therefore our
exposition is rather brief; our main references are~\cite{Kel94,Kel06}.

\subsection{Differential graded vector spaces}

\begin{definition}
  A \emph{differential graded (dg) vector space} is a pair $V=(V^\sharp,d_V)$,
  consisting of a graded vector space $V^\sharp=\coprod_{i\in\ZZ}V^i$ and a
  linear map ${d=d_V\colon V^\sharp\to V^\sharp}$ such that $d(V^i)\subseteq
  V^{i+1}$ and $d\circ d=0$. A \emph{morphism of dg vector spaces} $f\colon V\to
  W$ is a linear map such that $d_W\circ f=f\circ d_W$. The collection of dg
  vector spaces and their morphisms forms an abelian category in the evident way.
\end{definition}

\begin{notation}
  Let $V$ be a dg vector space. For a homogeneous element $v\in V^i$ we write
  $|v|\coloneqq i$. Most formulas involving elements of a dg vector space are
  given in terms of their homogeneous elements, even if this assumption is not
  made explicit in the text. This abuse of language should not lead to
  confusion.
\end{notation}

\begin{remark}
  There is an evident equivalence between the category of cochain complexes of
  vector spaces and that of dg vector spaces that associates to a dg vector
  space $V$ the cochain complex $(V^i,d|_{V^i})_{i\in\ZZ}$; in particular, a
  morphism of dg vector spaces $f\colon V\to W$ is uniquely determined by its
  restrictions
  \[
    f^i\colon V^i\longrightarrow W^i,\qquad i\in\ZZ.
  \]
  In the sequel we treat this equivalence of categories as an identification and
  pass between the two notions as the context requires.
\end{remark}

The category of dg vector spaces has a symmetric monoidal structure that we now
recall.

\begin{definition}
  The tensor product of two dg vector spaces $V$ and $W$ is the dg vector space
  $V\otimes W$ with homogeneous components
  \[
    (V\otimes W)^k\coloneqq\coprod_{i+j=k}V^i\otimes W^j,\qquad k\in\ZZ,
  \]
  and the differential
  \[
    d_{V\otimes W}(v\otimes w)\coloneqq d_v(v)\otimes w+(-1)^{|v|}v\otimes
    d_W(w).
  \]
  The braiding---responsible for the Koszul sign rule---is given by
  \[
    V\otimes W\stackrel{\sim}{\longrightarrow}W\otimes V,\qquad v\otimes
    w\mapsto (-1)^{|v||w|}w\otimes v.
  \]
  The tensor product of two morphisms of dg vector spaces
  \[ {f\colon V_1\longrightarrow W_1}\qquad\text{and}\qquad{g\colon
      V_2\longrightarrow W_2}
  \]
  is the morphism
  \[
    (f\otimes g)(v_1\otimes v_2)\coloneqq f(v_1)\otimes g(v_2),\qquad v_1\in V_1,\ v_2\in
    V_2.
  \]
\end{definition}

This symmetric monoidal structure is closed, with the following internal
$\operatorname{Hom}$.

\begin{definition}
  Let $V$ and $W$ be dg vector spaces. We define $\dgHom[\kk]{V}{W}$ to be the
  dg vector space with the homogeneous components
  \[
    \dgHom[\kk]{V}{W}[i]\coloneqq\prod_{j\in\ZZ}\Hom[\kk]{V^j}{W^{j+i}},\qquad
    i\in\ZZ,
  \]
  and the differential
  \begin{equation}
    \label{eq:partial}
    \partial_{V,W}(f)\coloneqq d_W\circ f-(-1)^{|f|}f\circ d_V,\qquad f\in\dgHom[\kk]{V}{W}.
  \end{equation}
  A morphism $f\in\dgHom[\kk]{V}{W}[i]$ is said to be \emph{homogeneous} of
  degree $|f|\coloneqq i$. The \emph{evaluation map}
  \begin{equation}
    \label{eq:evaluation-map}
    \ev=\ev[V,W]\colon \dgHom[\kk]{V}{W}\otimes V\longrightarrow W,\qquad f\otimes
    v\longmapsto f(v),
  \end{equation}
  is a morphism of dg vector spaces and is the canonical choice of counit for an
  adjunction $(-\otimes V)\dashv \dgHom[\kk]{V}{-}$.
\end{definition}

\begin{remark}
  \label{rmk:extra-signs-tensor-maps}
  Suppose given morphisms of vector spaces
  \[ {f\colon V_1\longrightarrow W_1}\qquad\text{and}\qquad{g\colon
      V_2\longrightarrow W_2}
  \]
  that are homogeneous but not necessarily of degree $0$. The Koszul sign
  rule implies that additional signs arise when evaluating the map
  $f\otimes g$ in homogeneous elementary tensors:
  \[
    (f\otimes g)(v_1\otimes v_2)\coloneqq(-1)^{|g||v|}f(v_2)\otimes g(v_2),\qquad v_1\in
    V_1,\ v_2\in V_2.
  \]
\end{remark}

\begin{remark}
  Graded vector spaces, viewed as dg vector spaces with vanishing differential,
  form a full subcategory of that of dg vector spaces, and the above closed
  symmetric monoidal structure restricts to this full subcategory.
\end{remark}

\begin{example}
  \label{ex:DV}
  Let $V$ be a dg vector space. The \emph{($\kk$-)linear dual} of $V$ is the dg
  vector space
  \[
    DV\coloneqq\dgHom[\kk]{V}{\kk}.
  \]
  It has the homogeneous components
  \[
    (DV)^i=\Hom[\kk]{V^{-i}}{\kk},\qquad i\in\ZZ,
  \]
  and the differential (notice the minus sign)
  \[
    \partial(f)=\partial_{V,\kk}(f)=-(-1)^{|f|}f\circ d_V,\qquad i\in\ZZ.
  \]
  The additional minus sign in the differential, which stems from
  \eqref{eq:partial}, is necessary when considering bimodule structures on $V$
  and its linear dual (see also \Cref{ex:dual map}).
\end{example}

\begin{example}
  \label{ex:dual map}
  Let $V$ and $W$ be dg vector spaces. Then, the map
  \begin{align*}
    D\colon\dgHom[\kk]{V}{W}\longrightarrow\dgHom[\kk]{DW}{DV}\\
    f\longmapsto (f^*\colon g\longmapsto (-1)^{|f||g|}gf),
  \end{align*}
  is a morphism of dg vector spaces.
\end{example}

\begin{definition}
  Let $V$ be a dg vector space. Given $n\in\ZZ$, we let $V[n]$ be the dg vector
  space with the homogeneous components
  \[
    V[n]\coloneqq V^{i+n},\qquad i\in\ZZ,
  \]
  and the differential $d_{V[n]}\coloneqq(-1)^nd_V$. Given a further dg vector
  space $W$ and a morphism ${f\in\dgHom[\kk]{V}{W}[i]}$, we define the morphism
  $f[n]\in\dgHom[\kk]{V[n]}{W[n]}[i]$ by the formula
  \[
    f[n](v)\coloneqq (-1)^{ni}f(v),\qquad v\in V.
  \]
\end{definition}

\begin{notation}
  \label{not:s}
  We make use the following standard notational device\footnote{Rigorously, we
    let $\s[n]\coloneqq 1_\kk\in\kk[n]$ and identify $V[n]$ with the tensor
    product $\kk[n]\otimes V$.} that makes it easier to keep track of degrees
  and Koszul signs when shifts of dg vector spaces are involved: For a
  homogeneous element $v\in V^{i}$, we write $\s[n]v\in V[n]$ for the
  corresponding homogeneous element of degree $|\s[n]v|=i-n$. In particular, when
  computing Koszul signs we may regard $\s[n]$ as a formal symbol of degree
  $-n$. With this convention in mind, the differential of $V[n]$ is given by
  \[
    d_{V[n]}(\s[n]v)=(-1)^{n}\s[n]d_V(v),\qquad v\in V.
  \]
  Similarly, for $f\in\dgHom[\kk]{V}{W}[i]$, the morphism
  $f[n]\in\dgHom[\kk]{V[n]}{W[n]}[i]$ is given by the formula
  \[
    f[n](\s[n]v)=(-1)^{n|f|}\s[n]f(v)=(-1)^{ni}\s[n]f(v),\qquad v\in V.
  \]
\end{notation}

\begin{example}
  \label{ex:UVnW}
  Let $U$, $V$, and $W$ be dg vector spaces and $n\in\ZZ$. Then, the map
  \begin{align*}
    \varphi\colon U\otimes V[n]\otimes
    W&\stackrel{\sim}{\longrightarrow}(U\otimes V\otimes W)[n]\\
    u\otimes\s[n]v\otimes w,&\longmapsto (-1)^{n|u|}\s[n](u\otimes v\otimes w),
  \end{align*}
  is an isomorphism of dg vector spaces. Moreover, for homogeneous morphisms of
  dg vector spaces (not necessarily of degree $0$)
  \[
    f\colon U_1\longrightarrow U_2,\qquad g\colon V_1\longrightarrow
    V_2,\qquad\text{and}\qquad h\colon W_1\longrightarrow W_2,
  \]
  the following diagram commutes:
  \[
    \begin{tikzcd}
      U_1\otimes V_1[n]\otimes W_1\rar{\varphi}\dar[swap]{f\otimes g[n]\otimes
        h}&(U_1\otimes V_1\otimes W_1)[n]\dar{(f\otimes g\otimes h)[n]}\\
      U_2\otimes V_2[n]\otimes W_2\rar{\varphi}&(U_2\otimes V_2\otimes W_2)[n].
    \end{tikzcd}
  \]
\end{example}

\begin{example}
  \label{ex:VWn}
  Let $V$ and $W$ be dg vector spaces and $n\in\ZZ$. A morphism
  ${f\in\dgHom[\kk]{V}{W}}$ can be regarded as an element
  ${f\s[n]\in\dgHom[\kk]{V[-n]}{W}}$ of degree $|f|-n$ via the formula
  \[
    f\s[n](\s[-n]v)\coloneqq f(v),\qquad v\in V.
  \]
  Moreover, the morphism
  \[
    \varphi\colon\dgHom[\kk]{V}{W}{[n]}\stackrel{\sim}{\longrightarrow}\dgHom[\kk]{V[-n]}{W},\qquad
    \s[n]f\longmapsto(-1)^{n|f|}f\s[n],
  \]
  is an isomorphism of dg vector spaces.
\end{example}

\begin{definition}
  Let $V$ be a dg vector space. Recall that the \emph{cohomology} of $V$ is the
  graded vector space
  \[
    \H{V}\coloneqq\ker d_V/\img d_V.
  \]
  A morphism between dg vector spaces $f\colon V\to W$ is a
  \emph{quasi-isomorphism} if the induced map
  \[
    \H{f}\colon\H{V}\longrightarrow\H{W},\qquad [v]\longmapsto[f(v)],
  \]
  is an isomorphism of graded vector spaces.
\end{definition}

\begin{remark}
  Since we work over a field, a morphism between dg vector spaces $f\colon V\to
  W$ is a quasi-isomorphism if and only if it is a \emph{homotopy equivalence},
  that is if there exists a morphism of dg vector spaces $g\colon W\to V$ such
  that
  \begin{align*} [g\circ
    f]=[\id[V]]&\in\H[0]{\dgHom[\kk]{V}{V}}\intertext{and}[f\circ
    g]=[\id[W]]&\in\H[0]{\dgHom[\kk]{W}{W}}.
  \end{align*}
\end{remark}

\subsection{Differential graded algebras}

A \emph{differential graded (dg) algebra} is a dg vector space $A$ endowed with
an associative and unital multiplication law
\[
  A\otimes A\longrightarrow A,\qquad x\otimes y\longmapsto x\cdot y=xy,
\]
given by a morphism of dg vector spaces. Explicitly, the multiplication law
satisfies
\[
  |x\cdot y|=|x|+|y|,\qquad x,y\in A,
\]
as well as the \emph{graded Leibniz rule}
\[
  d(x\cdot y)=d(x)\cdot y+(-1)^{|x|}x\cdot d(y),\qquad x,y\in A.
\]
The cohomology of a dg algebra $A$ is a graded algebra (dg algebra with
vanishing differential) with the induced multiplication law
\[
  \H{A}\otimes\H{A}\longrightarrow\H{A},\qquad [x]\otimes[y]\longmapsto[xy].
\]
A \emph{morphism of dg algebras} $f\colon A\to B$ is a morphism of dg vector
spaces such that $f(1_A)=1_B$ and
\[
  f(x\cdot y)=f(x)\cdot f(y),\qquad x,y\in A.
\]
Such a morphism of dg algebras induces a morphism of graded algebras
\[
  \H{f}\colon\H{A}\longrightarrow\H{B},\qquad [x]\longmapsto[f(x)].
\]
A morphism of dg algebras $f$ is a \emph{quasi-isomorphism} if $\H{f}$ is an
isomorphism of graded algebras (equivalently, an isomorphism of graded vector
spaces).

\subsection{Differential graded modules}

Let $A$ be a dg algebra.

\begin{definition}
  A \emph{differential graded (dg) right $A$-module} is a dg vector space $M$
  endowed with an action map
  \[
    M\otimes A\longrightarrow M,\qquad m\otimes x\longmapsto m\cdot x,
  \]
  given by a morphism of dg vector spaces and such that, for all $m\in M$ and
  for all $x,y\in A$, the following equalities are satisfied:
  \begin{align*}
    m\cdot 1&=m,& (m\cdot x)\cdot y&=m\cdot(xy).
  \end{align*}
\end{definition}

\begin{remark}
  There is an evident notion of dg \emph{left} $A$-module and all of the above
  discussion applies \emph{mutatis mutandis} to these. Furthermore, left dg
  $A$-modules identify with right dg $A^\op$-modules, where $A^\op$ is the dg
  algebra with the same underlying dg vector space as $A$ but with the opposite
  multiplication law:
  \[
    A^\op\otimes A^\op\longrightarrow A^\op,\qquad x\otimes y\longmapsto
    (-1)^{|x||y|}yx.
  \]
  For this reason we omit the adjective `right' in the sequel.
\end{remark}

\begin{example}
  The multiplication of the dg algebra $A$ endows it with the structure of a dg
  $A$-module that we refer to as the \emph{regular dg $A$-module}.
\end{example}

\begin{example}
  Let $M$ be a dg $A$-module and $n\in\ZZ$. The dg vector space $M[n]$ is a dg
  $A$-module with the action map
  \[
    M[n]\otimes A\longrightarrow M[n],\qquad \s[n](m)\otimes
    x\longmapsto\s[n](mx).
  \]
\end{example}

\begin{remark}
  The cohomology of a dg $A$-module $M$ has the structure of an $\H{A}$-module
  with the induced action map
  \[
    \H{M}\otimes\H{A}\longrightarrow\H{M},\qquad [m]\otimes[x]\longmapsto[m\cdot
    x].
  \]
\end{remark}

\begin{definition}
  A \emph{morphism of dg $A$-modules} $f\colon M\to N$ is a morphism of dg
  vector spaces such that, for all $x\in A$ and for all $m\in M$,
  \[
    f(m\cdot x)=f(m)\cdot x.
  \]
  Such a morphism of dg $A$-modules induces a morphism of graded $\H{A}$-modules
  \[
    \H{f}\colon\H{M}\longrightarrow\H{N},\qquad[m]\longmapsto[f(m)].
  \]
  A morphism of dg $A$-modules $f$ is a \emph{quasi-isomorphism} if $\H{f}$ is
  an isomorphism of graded $\H{A}$-modules (equivalently, an isomorphism of
  graded vector spaces).
\end{definition}

\begin{definition}
  The \emph{derived category} $\DerCat{A}$ of dg $A$-modules, defined as the
  localisation of the category of dg $A$-modules at the class of
  quasi-isomorphisms, is a (locally small) compactly generated triangulated
  $\kk$-category. The \emph{perfect derived category $\DerCat[c]{A}$ of $A$} is,
  by definition, the smallest triangulated subcategory of $\DerCat{A}$ that
  contains $A$ and is closed under retracts; as the notation indicates, it coincides
  with the full subcategory of $\DerCat{A}$ spanned by the compact
  objects~\cite{Kel94}.
\end{definition}

\subsection{Differential graded bimodules}

Let $A$ be a dg algebra.

\begin{definition}
  A \emph{differential graded (dg) $A$-bimodule} is a dg vector space $M$
  endowed with an action map
  \[
    A\otimes M\otimes A\longrightarrow A,\qquad x\otimes m\otimes y\longmapsto
    x\cdot m\cdot y,
  \]
  given by a morphism of dg vector spaces and such that, for all $m\in M$ and
  for all $w,x,y,z\in A$, the following equalities are satisfied:
  \begin{align*}
    (1\cdot m\cdot 1)&=m,& w\cdot (x\cdot m\cdot y)\cdot z&=(wx)\cdot m\cdot (yz).
  \end{align*}
\end{definition}

\begin{remark}
  It is customary to identify dg $A$-bimodules with dg $(A\otimes
  A^{\op})$-modules as follows: Given a dg $A$-bimodule $M$, define a right
  action of $A\otimes A^{\op}$ on (the underlying dg vector space of) $M$ via
  \[
    m\cdot(x\otimes y)\coloneqq (-1)^{|x||m|}x\cdot m\cdot y,\qquad m\in M,
    x,y\in A.
  \]
  On the one hand, this identification is convenient because it allows us to extend all
  notions available for right dg modules to dg bimodules. On the other hand,
  since we are also interested in working with minimal $A_\infty$-bimodules later
  in this article, this identification is inconvenient since the tensor product
  of $A_\infty$-algebras is an unwieldy beast, see for
  example~\cite{Lod11,MTTV21}.
\end{remark}

The following dg bimodules play a crucial role in this article.

\begin{example}
  \label{ex:diagonal_bimodule}
  The multiplication of the dg algebra $A$ endows it with the structure of a dg
  $A$-bimodule that we refer to as the \emph{diagonal dg $A$-bimodule}.
\end{example}

\begin{example}
  \label{ex:DA}
  Let $M$ be a dg $A$-bimodule. The dg vector space $DM=\dgHom[\kk]{M}{\kk}$ is
  a dg $A$-bimodule with the action map
  \[
    (x\cdot f\cdot y)(m)\coloneqq (-1)^{|x|(|f|+|y|+|m|)}f(y\cdot m\cdot
    x),\qquad x,y\in A,\ m\in M, f\in DM.
  \]
  Notice that, since we work over a field, there is a canonical isomorphism of
  graded $\H{A}$-bimodules $\H{DM}\cong D\H{M}$ that we treat as an
  identification in the sequel.
\end{example}

\begin{example}
  \label{ex:Mn}
  Let $M$ be a dg $A$-bimodule and $n\in\ZZ$. The dg vector space $M[n]$ is a dg
  $A$-bimodule with the action
  \begin{equation}
    \label{eq:Mn}
    x\cdot\s[n]{m}\cdot y\coloneqq(-1)^{n|x|}\s[n](x\cdot m\cdot y),\qquad
    x,y\in A,\ m\in M.
  \end{equation}
  In other words, the action map
  \[
    \mu_{1,1}^{M[n]}\colon A\otimes M[n]\otimes A\longrightarrow M[n]
  \]
  is defined, in terms of the action map
  \[
    \mu_{1,1}^{M}\colon A\otimes M\otimes A\longrightarrow M
  \]
  by the commutativity of the following diagram, in which the horizontal map
  is the isomorphism of dg vector spaces from \Cref{ex:UVnW}:
  \begin{equation}
    \begin{tikzcd}
      A\otimes M[n]\otimes A\dar\rar{\sim}\dar{\mu_{1,1}^M}&(A\otimes M\otimes A)[n]\dar{\s[n]\mu_{1,1}^M}\\
      M[n]\rar[equals]&M[n].
    \end{tikzcd}
  \end{equation}
\end{example}

\begin{example}
  Let $M$ be a dg $A$-bimodule and $n\in\ZZ$. As a special case of \Cref{ex:Mn},
  we see that the dg $A$-bimodule $(DM)[n]$ is endowed with the action
  \begin{align*}
    (x\cdot \s[n]f\cdot y)(m)&=(-1)^{n|x|}(\s[n](x\cdot f\cdot y))(m)\\
                             &=(-1)^{|x|(n+|f|+|y|+|m|)}f(ymx),
  \end{align*}
  where $f\in DM$ and $x,y\in A$ and $m\in M$.
\end{example}

\begin{example}
  \label{ex:DAn}
  Let $M$ be a dg $A$-bimodule and $n\in\ZZ$. Recall from \Cref{ex:VWn} that the
  map
  \begin{align*}
    \varphi\colon (DM)[n]\stackrel{\sim}{\longrightarrow}D(M[-n]),\qquad \s[n]f\longmapsto(-1)^{n|f|}f\s[n]
  \end{align*}
  is an isomorphism of dg vector spaces, where $f\in DM$. It is straightforward
  to verify that this isomorphism is in fact an isomorphism of dg $A$-bimodules.
\end{example}

\begin{remark}
  The cohomology of a dg $A$-bimodule $M$ has the structure of a graded
  $\H{A}$-bimodule with the induced action map
  \[
    \H{A}\otimes\H{M}\otimes\H{A}\longrightarrow\H{M},\qquad
    [x]\otimes[m]\otimes[y]\longmapsto[x\cdot m\cdot y].
  \]
\end{remark}

\begin{definition}
  A \emph{morphism of dg $A$-bimodules} $f\colon M\to N$ is a morphism of dg
  vector spaces such that, for all $x,y\in A$ and for all $m\in M$,
  \[
    f(x\cdot m\cdot y)=x\cdot f(m)\cdot y.
  \]
  Such a morphism of dg $A$-bimodules induces a morphism of graded
  $\H{A}$-bimodules
  \[
    \H{f}\colon\H{M}\longrightarrow\H{N},\qquad [m]\longmapsto[f(m)].
  \]
  A morphism of dg $A$-bimodules $f$ is a \emph{quasi-isomorphism} if $\H{f}$ is
  an isomorphism of graded $\H{A}$-bimodules (equivalently, an isomorphism of
  graded vector spaces).
\end{definition}

\begin{definition}
  The \emph{derived category of dg $A$-bimodules} is defined in the obvious way,
  as the localisation of the category of dg $A$-bimodules at the class of
  quasi-isomorphisms.
\end{definition}

\begin{remark}
  Identifying dg $A$-bimodules with (right) dg $A\otimes A^\op$-modules, we see
  that the derived category of dg $A$-bimodules is also a (locally small)
  compactly generated triangulated $\kk$-category. However, we do not make
  explicit use of this fact in the sequel.
\end{remark}

\subsection{Bimodule right Calabi--Yau structures}

The following is the central concept of interest in this article.

\begin{definition}[{\cite{Kon93,KS09}}]
  \label{def:right-CY-as-bimodule}
  Let $A$ be a dg algebra whose cohomology is degree-wise finite dimensional and
  $n\in\ZZ$.
  \begin{enumerate}
  \item A \emph{bimodule right $n$-Calabi--Yau structure} on $A$ is an
    isomorphism
    \[
      \varphi\colon A[n]\stackrel{\sim}{\longrightarrow}DA
    \]
    in the derived category of dg $A$-bimodules.
  \item We say that $A$ is \emph{right $n$-Calabi--Yau as a bimodule} if it
    admits a bimodule right $n$-Calabi--Yau structure.
  \end{enumerate}
\end{definition}

\begin{remark}
  \label{rmk:right-CY-as-bimodule}
  Let $A$ be a dg algebra whose cohomology is degree-wise finite-dimensional. A
  bimodule right $n$-Calabi--Yau structure $\varphi\colon
  A[n]\stackrel{\sim}{\to}DA$ induces an isomorphism of graded $\H{A}$-bimodules
  \[
    \H{\varphi}\colon\H{A}[-n]\stackrel{\sim}{\longrightarrow}\H{DA}\cong D\H{A},
  \]
  The isomorphisms of graded vector spaces
  \[
    \begin{tikzcd}[ampersand replacement=\&] H\rar{\varphi[-n]}\&(DH)[-n]\cong
      D(H[n])\&D(DH)\lar[swap]{\varphi^*}
    \end{tikzcd}
  \]
  induce an isomorphism of vector spaces
  \[
    H^i\cong(D(DH))^i=D((DH)^{-i})=D(D(H^i)),\qquad i\in\ZZ.
  \]
  which shows that the assumption that $H=\H{A}$ is degree-wise
  finite-dimensional in \Cref{def:right-CY-as-bimodule} is necessary.
\end{remark}

\begin{remark}
  We warn the reader that bimodule right Calabi--Yau structures are more often
  considered for \emph{proper} dg algebras, that is dg algebras whose cohomology
  has finite total dimension, an assumption that is never satisfied by a
  non-zero dg algebra as in \Cref{thm:CY-equivalence,thm:CY-correspondence}.
\end{remark}

\section{Hochschild cohomology of graded algebras}
\label{sec:Hochschild-algebras}

In this section we recall, paying particular attention to the signs involved,
the definition and algebraic structure present in the Hochschild
co\-ho\-mo\-lo\-gy of graded algebras.

\subsection{Differential graded Lie algebras}

\begin{definition}
  A \emph{differential graded (dg) Lie algebra} is a dg vector space $L$
  endowed with a \emph{Lie bracket}
  \[
    L\otimes L\longrightarrow L,\qquad x\otimes y\longmapsto [x,y],
  \]
  that is given by a morphism of dg vector spaces and such that the following
  axioms are satisfied\footnote{Recall that we make no assumption
    on the characteristic of the ground field.}
  \begin{align*}
    [x,x]&=0&|x|&\text{ is even},\\
    [x,[x,x]]&=0&|x|&\text{ is odd},\\
    [x,y]&=-(-1)^{|x||y|}[y,x],\\
    [x,[y,z]]&=[[x,y],z]+(-1)^{|x||y|}[y,[x,z]].
  \end{align*}
  The last equality is called the \emph{graded Jacobi identity}.
\end{definition}

\begin{remark}
  By definition, the Lie bracket in a dg Lie algebra $L$ satisfies
  \[
    |[x,y]|=|x|+|y|,\qquad x,y\in L,
  \]
  as well as the graded Leibniz rule
  \[
    d([x,y])=[d(x),y]+(-1)^{|x|}[x,d(y)],\qquad x,y\in L.
  \]
\end{remark}

\begin{example}
  \label{ex:dga-to-dgla}
  Let $A$ be a dg algebra. Then, the graded commutator
  \[
    A\otimes A\longrightarrow A,\qquad x\otimes y\longmapsto xy-(-1)^{|x||y|}yx,
  \]
  endows (the underlying dg vector space of) $A$ with the structure of a dg Lie
  algebra.
\end{example}

\begin{remark}
  The cohomology of a dg Lie algebra $L$ has the structure of a graded Lie
  algebra (dg Lie algebra with vanishing differential) with the induced Lie
  bracket
  \[
    \H{L}\otimes\H{L}\longrightarrow\H{L},\qquad
    [x]\otimes[y]\longmapsto\Big[[x,y]\Big].
  \]
\end{remark}

\begin{definition}
  Let $L_1=(L_1,[-,-]_1)$ and $L_2=(L_2,[-,-]_2)$ be dg Lie algebras. A
  \emph{morphism of dg Lie algebras} $f\colon L_1\to L_2$ is a morphism of dg
  vector spaces that preserves the Lie bracket:
  \[
    f([x,y]_1)=[f(x),f(y)]_2,\qquad x,y\in L_1.
  \]
  In this case, the induced morphism of graded vector spaces
  \[
    \H{f}\colon\H{L_1}\longrightarrow\H{L_2}
  \]
  is a morphism of graded Lie algebras. Such a morphism $f$ is a
  \emph{quasi-isomorphism} if $\H{f}$ is an isomorphism of graded Lie algebras
  (equivalently, an isomorphism of graded vector spaces).
\end{definition}

We also need the following variant of the notion of a dg Lie algebra.

\begin{definition}
  A \emph{shifted differential graded (dg) Lie algebra} is a dg vector space
  $L$ equipped with a degree $-1$ Lie bracket
  \[
    L\otimes L\longrightarrow L,\qquad x\otimes y\longmapsto[x,y],
  \]
  such that the map
  \[
    L[1]\otimes L[1]\longrightarrow L[1]\qquad \s x\otimes \s y\longmapsto [\s
    x,\s y]\coloneqq\s{[x,y]},
  \]
  which is a degree $0$ morphism of dg vector spaces\footnote{Notice that
    \[
      |[\s x,\s y]|=|\s[x,y]|=(|x|+|y|-1)-1=(|x|-1)+(|y|-1)=|\s x|+|\s y|.
    \]} endows the dg vector space $L[1]$ with the structure of a dg Lie
  algebra. Explicitly,
  \begin{itemize}
  \item the following variant of the graded Leibniz rule is satisfied
    \[
      d([x,y])=[d(x),y]+(-1)^{|\s x|}[x,d(y)],\qquad x,y\in L,
    \]
    where $|\s x|=|x|-1$ is the shifted degree,
  \item and the following equations are satisfied for all $x,y,z\in L$:
    \begin{align}
      \label{eq:shifted_Lie_algebra}
      [x,x]&=0&|x|&\text{ is odd},\nonumber\\
      [x,[x,x]]&=0&|x|&\text{ is even},\nonumber\\
      [x,y]&=-(-1)^{|\s x||\s y|}[y,x],\\
      [x,[y,z]]&=[[x,y],z]+(-1)^{|\s x||\s y|}[y,[x,z]]\nonumber.
    \end{align}
  \end{itemize}
\end{definition}

\subsection{Gerstenhaber algebras}

\begin{definition}[{\cite{Ger63}}]
  \label{def:Gerstenhaber-algebra}
  A \emph{Gerstenhaber algebra} is a graded vector space $G$ equipped with the
  following structures:
  \begin{itemize}
  \item The graded vector space $G$ is a shifted graded Lie algebra with degree
    $-1$ Lie bracket
    \[
      G\otimes G\longrightarrow G,\qquad x\otimes y\longmapsto[x,y].
    \]
  \item The graded vector space $G$ is endowed with a graded-commutative (not
    necessarily unital) algebra structure, that is $G$ is equipped with an
    associative multiplication map
    \[
      G\otimes G\longrightarrow G,\qquad x\otimes y\longmapsto x\cdot y,
    \]
    given by a morphism of graded vector spaces and such that
    \[
      x\cdot y=(-1)^{|x||y|}y\cdot x,\qquad x,y\in G.
    \]
  \end{itemize}
  These structures must be compatible in the sense that \emph{Gerstenhaber
    relation} must be satisfied:
  \begin{equation}
    \label{eq:Gerstenhaber_relation}
    [x,y\cdot z]=[x,y]\cdot z+(-1)^{|\s x||y|}y\cdot[x,z],\qquad x,y,z\in G,
  \end{equation}
  where $|\s x|=|x|-1$ is the shifted degree. If $\chark(\kk)=2$, a Gerstenhaber
  algebra must be equipped with an additional map called \emph{Gerstenhaber
    square}
  \[
    \Sq\colon G\longrightarrow G,\qquad x\longmapsto \Sq(x),
  \]
  that satisfies
  \[
    |\Sq(x)|=2|x|-1,\qquad x\in G.
  \]
  In general, this map is not linear nor is an abelian group homomorphism. The
  Gerstenhaber square must satisfy the following equations for all $x,y\in G$:
  \begin{align}
    \label{eq:Gerstenhaber_square-relations}
    \Sq(x+y)&=\Sq(x)+\Sq(y)+[x,y],\nonumber\\
    \Sq(x\cdot y)&=\Sq(x)\cdot y^2+x\cdot[x,y]\cdot y+x^2\cdot\Sq(y),\\
    [\Sq(x),y]&=[x,[x,y]].\nonumber
  \end{align}
\end{definition}

\begin{remark}
  \label{rmk:Gerstenhaber_square-char}
  If $\chark(\kk)\neq 2$, we define the Gerstenhaber square operation in a
  Gerstenhaber algebra in even degrees as follows:
  \[
    \Sq\colon G^{2n}\longrightarrow G^{4n-1},\qquad
    x\longmapsto\tfrac{1}{2}[x,x].
  \]
  This operation satisfies the required equations in
  \Cref{def:Gerstenhaber-algebra} whenever these make sense.
\end{remark}

\subsection{Differential bigraded vector spaces}

\begin{definition}
  A \emph{bigraded vector space} is a tuple of vector spaces
  \[
    V^{\bullet,*}=(V^{p,q})_{(p,q)\in\ZZ^2}.
  \]
  A bihomogeneous element $v\in V^{p,q}$ has the \emph{bidegree} $(p,q)$ and the
  \emph{total degree} $|v|\coloneqq p+q$. We also write $\hdeg{v}\coloneqq p$
  for the \emph{horizontal degree} and $\vdeg{v}\coloneqq q$ for the
  \emph{vertical degree}, so that $|v|=\hdeg{v}+\vdeg{v}$. A \emph{morphism}
  $f\colon V^{\bullet,*}\to V^{\bullet,*}$ \emph{of bidegree $(r,s)\in\ZZ^2$}
  between bigraded vector spaces is a tuple of linear maps
  \[
    (V^{p,q}\to W^{p+r,q+s})_{(p,q)\in\ZZ^2}.
  \]
  Bigraded vector spaces and bihomogeneous morphisms of bidegree $(0,0)$ form an
  abelian category in the obvious way.
\end{definition}

\begin{definition}
  A \emph{differential bigraded vector space} is a pair $(V^{\bullet,*},d_V)$
  consisting of a bigraded vector space and a bidegree $(1,0)$ differential
  $d_V\colon V^{\bullet,*}\to V^{\bullet,*}$, that is $d_V\circ d_V=0$ and
  \[
    d_V(V^{p,q})\subseteq V^{p+1,q},\qquad (p,q)\in\ZZ^2.
  \]
  Thus, $(V^{\bullet,q},d_V)$ is a cochain complex of vector spaces for each
  $q\in\ZZ$. A \emph{morphism of differential bigraded vector spaces} is a
  bidegree $(0,0)$ morphism of bigraded vector spaces ${f\colon V^{\bullet,*}\to
    W^{\bullet,*}}$ such that $d_W\circ f=f\circ d_V$.
\end{definition}

\begin{definition}
  The \emph{cohomology} of a differential bigraded vector space
  $(V^{\bullet,*},d_V)$ is the bigraded vector space
  \[
    \H[\bullet,*]{V^{\bullet,*},d_V}\coloneqq(\H[p,q]{V^{\bullet,*},d_V})_{(p,q)\in\ZZ^2},
  \]
  where
  \[
    \H[p,q]{V^{\bullet,*},d_V}\coloneqq\H[p]{V^{\bullet,q},d_V},\qquad
    (p,q)\in\ZZ^2.
  \]
  A morphism $f\colon(V^{\bullet,*},d_V)\to(W^{\bullet,*},d_W)$ of differential
  bigraded vector spaces is a \emph{quasi-isomorphism} if the morphisms of
  cochain complexes of vector spaces
  \[
    f\colon(V^{\bullet,q},d_V)\longrightarrow(W^{\bullet,q},d_W),\qquad q\in\ZZ,
  \]
  are quasi-isomorphisms; equivalently, the induced map
  \[
    \H[\bullet,*]{f}\colon\H[p,q]{V^{\bullet,*},d_V}\longrightarrow\H[p,q]{W^{\bullet,*},d_W}
  \]
  is an isomorphism of bigraded vector spaces.
\end{definition}


\begin{definition}
  \label{def:hv-shifts}
  Let $(V^{\bullet,*},d_V)$ be a differential bigraded vector space. Given
  ${n\in\ZZ}$, we let $(V^{\bullet,*}[n],d_{V[n]})$ be the differential bigraded
  vector space with \emph{horizontally-shifted} bihomogeneous components
  \[
    V[n]^{p,q}\coloneqq V^{p+n,q},\qquad (p,q)\in\ZZ^2,
  \]
  and which is endowed with the bidegree $(1,0)$ differential $d_{V[n]}\coloneqq
  (-1)^nd_V$. Notice that
  \[
    \H[p,q]{V^{\bullet,*}[n],d_{V[n]}}=\H[p+n,q]{V^{\bullet,*},d_V},\qquad
    (p,q)\in\ZZ^2.
  \]
  Similarly, we let $(V^{\bullet,*}(n)),d_{V(n)}$ be the differential bigraded
  vector space with vertically-shifted bihomogeneous components
  \[
    V(n)^{p,q}\coloneqq V^{p,q+n}\qquad (p,q)\in\ZZ^2,
  \]
  and which is endowed with the differential $d_{V(n)}=d_V$. Notice that
  \[
    \H[p,q]{V^{\bullet,*}(n),d_{V(n)}}=\H[p,q+n]{V^{\bullet,*},d_V},\qquad
    (p,q)\in\ZZ^2.
  \]
\end{definition}

\begin{remark}
  We identify graded vector spaces with bigraded vector spaces concentrated in
  horizontal degree $0$. Hence, to a graded vector space $V$ we associate the
  bigraded vector space with the bihomogeneous components
  \[
    V^{0,j}\coloneqq V^j,\qquad j\in\ZZ,
  \]
  and $V^{i,*}=0$ for $i\neq 0$. Under this identification, the shift of graded
  vector spaces corresponds to the \emph{vertical} shift of bigraded vector
  spaces. In particular, when dealing with graded and bigraded vector spaces in
  tandem, we write $V\mapsto V(1)$ for the (vertical-)shift operation on graded
  vector spaces.
\end{remark}

\subsection{Brace algebras}

Our primary source of Gerstenhaber algebras are brace algebras with additional
structure, see~\cite{Kad88,Get93,GV95}. We recall the variant of the definition
that is compatible with our preferred sign conventions.

\begin{definition}[{\cite{GV95}}]
  \label{def:brace_algebra}
  A \emph{brace algebra} is a bigraded graded vector space $B^{\bullet,*}$ that
  is endowed with a system of bidegree $(-n,0)$ \emph{brace operations}
  \begin{align*}
    B^{p_0,q_0}\otimes B^{p_1,q_1}\otimes\cdots\otimes B^{p_n,q_n}&\longrightarrow B^{p_0+p_1+\cdots+p_n-n,q_0+q_1+\cdots+q_n}&n\geq1\\
    x_0\otimes x_1\otimes\cdots\otimes x_n&\longmapsto\braces{x_0}{x_1\dots,x_n},
  \end{align*}
  which vanish when $p_0<n$, and which satisfy the \emph{brace relation}
  \begin{multline}
    \label{eq:brace_relation}
    \braces{\braces{x}{y_1,\dots,y_p}}{z_1,\dots,z_q}=\\\sum_{}(-1)^{\maltese}\braces{x}{z_1,\dots,z_{i_1},\braces{y_1}{z_{i_1+1},\dots,z_{j_1}},\dots,\braces{y_p}{z_{i_p+1},\dots,z_{j_p}},z_{j_p+1},\dots,z_q};
  \end{multline}
  here, the sum ranges over all interlaced tuples
  \[
    0\leq i_1\leq j_1\leq\cdots\leq i_p\leq j_p\leq q,
  \]
  and the sign is determined by the
  Koszul sign rule with respect to the shifted total degree:
  \[
    \textstyle\maltese\coloneqq\sum_{s=1}^p\sum_{t=1}^{i_s}|\s y_s||\s
    z_t|=\sum_{s=1}^p\sum_{t=1}^{i_s}(|y_s|-1)(|z_t|-1).
  \]
\end{definition}

\begin{remark}
  The prototypical examples of brace algebras are obtained from graded
  (non-symmetric) operads by combining their infinitesimal composition
  operations in a natural way, see~\cite[Section~1.1]{GV95} for details.
  Although all the brace algebras that interest us are of this form, in this
  article we try to keep the use of operadic language to a minimum as it is not
  strictly necessary for our purposes.
\end{remark}

The following result summarises the algebraic structures that we wish to extract
from a brace algebra, compare with~\cite[Sections~1.1 and~1.2]{GV95}.

\begin{defprop}
  \label{defprop:braces-Gerstenhaber}
  Let $B$ be a brace algebra.
  \begin{enumerate}
  \item\label{eq:braces-Gerstenhaber-preLie} The bidegree $(-1,0)$ \emph{pre-Lie
      product} is the binary operation
    \begin{align*}
      B^{p,q}\otimes B^{s,t}\longrightarrow B^{p+s-1,q+t}\qquad x\otimes y\longmapsto \preLie{x}{y}\coloneqq\braces{x}{y}.
    \end{align*}
    If $\chark(\kk)=2$ or $p+q$ is even, we also consider the \emph{Gerstenhaber
      square} operation
    \begin{align*}
      \Sq\colon B^{p,q}&\longrightarrow B^{2p-1,2q},\qquad x\longmapsto\Sq(x)\coloneqq\preLie{x}{x}.
    \end{align*}
  \item\label{eq:braces-Gerstenhaber-bracket} The graded commutator of the
    pre-Lie product, called the \emph{Gerstenhaber Lie bracket},
    \begin{align*}
      B^{p,q}\otimes B^{s,t}&\longrightarrow B^{p+s-1,q+t}\\
      x\otimes y&\longmapsto[x,y]\coloneqq\preLie{x}{y}-(-1)^{|\s x||\s y|}\preLie{y}{x},
    \end{align*}
    endows (the underlying graded vector space of) $B$ with the structure of a
    shifted graded Lie algebra with respect to the total degree. Of course, the
    Gerstenhaber Lie bracket has bidegree $(-1,0)$.
  \item\label{eq:braces-Gerstenhaber-multiplication} A \emph{multiplication} on
    $B$ is an element $m\in B^{2,0}$ such that
    \[
      \preLie{m}{m}=\braces{m}{m}\stackrel{!}{=}0.
    \]
    In the presence of a multiplication we consider the following additional
    structure on $B$:
    \begin{itemize}
    \item The bidegree $(0,0)$ \emph{cup product} is the binary operation
      \begin{align*}
        B^{p,q}\otimes B^{s,t}&\longrightarrow B^{p+s,q+t}\\
        x\otimes y&\longmapsto\cupp{x}{y}\coloneqq(-1)^{|\s x|}\braces{m}{x,y}.
      \end{align*}
      Endowed with the cup product, $B$ is a graded (non-unital) algebra with
      respect to the total degree.
    \item The bidegree $(1,0)$ map
      \begin{align*}
        d\colon B^{p,q}&\longrightarrow B^{p+1,q}\\
        x&\longmapsto
           d(x)\coloneqq[m,x]=\preLie{m}{x}-(-1)^{|\s c|}\preLie{x}{m},
      \end{align*}
      is a differential: $d\circ d=0$. Consequently, $(B^{\bullet,*},d)$ is a
      differential bigraded vector space.
    \item With the above differential, $(B^{\bullet,*},d)$ is a shifted dg Lie
      algebra with the Gerstenhaber Lie bracket and a (non-unital) dg algebra
      with the cup product, in both cases with respect to the total degree.
    \item With the induced operations, the cohomology of $(B^{\bullet,*},d)$,
      which itself is naturally bigraded, is a Gerstenhaber algebra with respect
      to the total degree. In particular, the cup product is graded-commutative
      with respect to the total degree at the cohomological level (but in
      general not at the cochain level):
      \[
        \cupp{[x]}{[y]}=(-1)^{|x||y|}\cupp{[y]}{[x]},\qquad x,y\in
        B^{\bullet,*}.
      \]
    \end{itemize}
  \end{enumerate}
\end{defprop}

\begin{remark}
  A multiplication on a brace algebra is equivalent to the data of a homotopy
  Gerstenhaber algebra in the sense of~\cite[Definition~2]{GV95}, see Theorem~3
  in~\emph{op.~cit.}
\end{remark}

\begin{remark}
  The Gerstenhaber algebra structure on the cohomology of a brace algebra $B$
  depends only on the binary brace defining the pre-Lie product and on braces of
  the form $\braces{m}{x,y}$, where $m\in B^{2,0}$ is a multiplication. For the
  purposes of this article, other braces of higher arity and the brace
  relation are only needed to prove the claims in
  \Cref{defprop:braces-Gerstenhaber}. For this reason we do not define the full
  brace algebra structure on the Hochschild(-type) cochain complexes recalled in
  later sections.
\end{remark}

\begin{example}
  For illustration purposes, let us show that the cup product in a brace algebra
  $B^{\bullet,*}$ with multiplication $m\in B^{2,0}$ is associative: Firstly,
  since $\braces{m}{m}=0$ and $|\s m|=2-1=1$, the brace relation yields
  \[
    0=\braces{\braces{m}{m}}{x,y,z}\stackrel{\eqref{eq:brace_relation}}{=}\braces{m}{\braces{m}{x,y},z}+(-1)^{|\s
      x|}\braces{m}{x,\braces{m}{y,z}}.
  \]
  Therefore, since $|x\cdot y|=|x|+|y|$,
  \begin{align*}
    \cupp{(\cupp{x}{y})}{z}-\cupp{x}{(\cupp{y}{z})}&=(-1)^{|\s x|+|\s(\cupp{x}{y})|}\braces{m}{\braces{m}{x,y},z}-(-1)^{|\s x|+|\s y|}\braces{m}{x,\braces{m}{y,z}}\\
                                                   &=(-1)^{|y|}\Big(\braces{m}{\braces{m}{x,y},z}+(-1)^{|\s
                                                     x|}\braces{m}{x,\braces{m}{y,z}}\Big)\\
                                                   &=0.
  \end{align*}
  Thus $\cupp{(\cupp{x}{y})}{z}=\cupp{x}{(\cupp{y}{z})}$, as required.
\end{example}


\subsection{Interlude: Infinitesimal compositions}

In order to simplify some formulas that arise when dealing with Hochschild-type
cochain complexes, it is convenient to introduce a minimal amount of operadic
notation. Consider homogeneous operations
\begin{align*}
  c_1\colon X_1\otimes\cdots\otimes X_p&\longrightarrow X,&\hdeg{c_1}&\coloneqq p,&\vdeg{c_1}&=q,&|c_1|&\coloneqq p+q\\
  c_2\colon Y_1\otimes\cdots\otimes Y_s&\longrightarrow Y,&\hdeg{c_2}&\coloneqq s,&\vdeg{c_2}&=t,&|c_2|&\coloneqq s+t,
\end{align*}
where $X_1,\dots,X_p,X$ and $Y_1,\dots,Y_s,Y$ are graded vector spaces, and the
vertical degrees are the degrees of the maps as morphisms of graded vector
spaces.

\begin{notation}
  For each $1\leq i\leq p$ such that $X_i=Y$, we define the \emph{(suspended) infinitesimal
    composition}
  \[
    c_1\bullet_ic_2\colon X_1\otimes\cdots\otimes
    X_{i-1}\otimes(Y_1\otimes\cdots\otimes Y_s)\otimes
    X_{i+1}\otimes\cdots\otimes X_p\longrightarrow X
  \]
  by the formula
  \begin{multline}
    \label{eq:infinitesimal_composition}
    (c_1\bullet_ic_2)(x_1,\dots,x_{i-1},y_1,\dots,y_s,x_{i+1},\dots,x_p)\coloneqq\\
    (-1)^{\maltese_i}c_1(x_1,\dots,x_{i-1},c_2(y_1,\dots,y_s),x_{i+1},\dots,x_p),
  \end{multline}
  with the sign given by
  \begin{align*}
    \maltese_i&\coloneqq\textstyle(\hdeg{c_2}-1)(\hdeg{c_1}-i)+\vdeg{c_2}\left(\hdeg{c_1}-1+\sum_{j=1}^{i-1}|x_j|\right)\\
              &=\textstyle(s-1)(p-i)+t\left(p-1+\sum_{j=1}^{i-1}|x_j|\right),
  \end{align*}
  If $X_i\neq Y$ we set $c_1\bullet_ic_2\coloneqq0$.
\end{notation}

\begin{remark}
  Infinitesimal composition is associative in the sense that
  \begin{align}
    \label{eq:infinitesimal_composition-associative}
    \begin{split}
      c_1\bullet_i (c_2\bullet_ j c_3)&=(c_1\bullet_i c_2)\bullet_{i+j-1}c_3\\
      (c_1\bullet_i c_2)\bullet_j c_3&=(-1)^{\maltese}(c_1\bullet_j c_3)\bullet_{i+n-1}c_2,\qquad j<i,\ n={\hdeg{c_3}},
    \end{split}
  \end{align}
  whenever these are defined, where the sign is given by the following integer
  modulo $2$
  \[
    \maltese=(\vdeg{c_2}+1-\hdeg{c_2})(\vdeg{c_3}+1-\hdeg{c_3})=|\s c_2||\s
    c_3|.
  \]
\end{remark}

\begin{remark}
  The previous formula is motivated by the definition of the infinitesimal
  compositions in the linear endomorphism operad of a pair of vector spaces
  introduced in~\cite{BM09}. The extra signs---those not given by the Koszul
  sign rule---stem from the implicit application of the so-called operadic
  suspension, compare with~\cite[Definition~2.4]{Mur16}.
\end{remark}
  
\begin{notation}
  Following~\cite{GV95}, we define
  \begin{equation}
    \label{eq:binary_brace}
    \braces{c_1}{c_2}\coloneqq\sum_{i=1}^pc_1\bullet_ic_2.
  \end{equation}
\end{notation}

\begin{remark}
  Notice that the number of inputs of $\braces{c_1}{c_2}$ is $p+s-1$, as
  expected from \Cref{def:brace_algebra}.
\end{remark}

Suppose given a further homogeneous operation
\begin{align*}
  m\colon W_1\otimes W_2&\longrightarrow W,&\hdeg{m}&\coloneqq 2,&\vdeg{m}&=0,&|m|\coloneqq 2.
\end{align*}
With this additional datum we introduce the following notation.

\begin{notation}
  For $c\coloneqq c_1$ we set $d(c)\coloneqq\braces{m}{c}-(-1)^{|\s
    c|}\braces{c}{m}$. Explicitly,
  \begin{align}
    \label{eq:generic-Hd}
    \begin{split}
      d(c)(x_1,\dots,x_{p+1})&=(-1)^{q|x_1|+q}m(x_1,c(x_2,\dots,x_{p+1}))\\
                             &\phantom{=}+\sum_{i=1}^p(-1)^{i+q}c(x_1,\dots,m(x_i,x_{i+1}),\dots,x_{p+1})\\&\phantom{=}+(-1)^{p+1+q}m(c(x_1,\dots,x_p),x_{p+1});
    \end{split}
  \end{align}
  compare with the definition of the differential in a brace algebra with
  multiplication given in \Cref{defprop:braces-Gerstenhaber}.
\end{notation}

\begin{notation}
  We also define
  \begin{equation}
    \label{eq:multiplication_brace}
    \braces{m}{c_1,c_2}\coloneqq (m\bullet_1c_1)\bullet_{p+1} c_2=(-1)^{|c_1||c_2|}(m\bullet_2 c_2)\bullet_1 c_1,
  \end{equation}
  so that, if $W_1=X$ and $W_2=Y$,
  \[
    \braces{m}{c_1,c_2}(x_1,\dots,x_p,y_1,\dots,y_q)\coloneqq(-1)^{\maltese}m(c_1(x_1,\dots,x_p),c_2(y_1,\dots,y_q)),
  \]
  with the sign given by
  \begin{align*}
    \maltese&\coloneqq\textstyle(|c_1|-1)+\vdeg{c_2}\left(\vdeg{c_1}+\sum_{i=1}^{\hdeg{c_1}}|x_i|\right)\\
            &=\textstyle(p+q-1)+t\left(p+\sum_{i=1}^p|x_i|\right).
  \end{align*}
\end{notation}

\begin{remark}
  The number of inputs of $\braces{m}{c_1,c_2}$ is $p+s$, as expected from
  \Cref{def:brace_algebra}. Also,
  \begin{align}
    \label{eq:generic-cup_product}
    (-1)^{|\s c_1|}\braces{m}{c_1,c_2}=(-1)^{t\left(p+\sum_{i=1}^p|x_i|\right)}m(c_1(x_1,\dots,x_p),c_2(y_1,\dots,y_q)),
  \end{align}
  where $|\s c_1|\coloneqq|c_1|-1$, compare with the definition of the cup
  product in a brace algebra with multiplication given in
  \Cref{defprop:braces-Gerstenhaber}.
\end{remark}

\subsection{The Hochschild cochain complex}
\label{subsec:Hochschild-algebras}

Let $A$ be a graded algebra.

\begin{definition}
  \label{def:HC}
  The \emph{Hochschild cochain complex of $A$} is the bigraded vector space
  $\HC{A}$ with the bihomogeneous components
  \begin{equation}
    \label{eq:HC}
    \HC[p]<q>{A}\coloneqq\dgHom[\kk]{A^{\otimes p}}{A}[q],\qquad p\geq 0,\
    q\in\ZZ,
  \end{equation}
  As is customary, given a cochain $c\in\HC[p]<q>{A}$, we refer to the
  horizontal degree $\hdeg{c}=p$ as the \emph{Hochschild degree} and to the
  vertical degree $\vdeg{c}=q$ as the \emph{internal degree}.
\end{definition}

\begin{notation}
  We set
  \begin{align*}
    \HC*[p]{A}&\coloneqq\prod_{q\in\ZZ}\HC[p]<q>{A}=\dgHom[\kk]{A^{\otimes p}}{A},\qquad p\geq0.
  \end{align*}
\end{notation}

The Hochschild cochain complex of $A$ is a brace algebra (see
\Cref{rmk:HCA-operad}), and hence we may apply
\Cref{defprop:braces-Gerstenhaber} to describe its algebraic structure.

Firstly, the bidegree $(-1,0)$ \emph{pre-Lie product} is given by
\begin{align}
  \label{eq:pre-Lie_product-HC}
  \begin{split}
    \HC[p]<q>{A}\otimes \HC[s]<t>{A}\longrightarrow \HC[p+s-1]<q+t>{A}\\
    c_1\otimes c_2\longmapsto \preLie{c_1}{c_2}\coloneqq\braces{c_1}{c_2},
  \end{split}
\end{align}
see \Cref{eq:binary_brace} for the definition of $\braces{c_1}{c_2}$. It follows
from \Cref{defprop:braces-Gerstenhaber} that $\HC{A}$ is a shifted graded Lie
algebra with the bidegree $(-1,0)$ \emph{Gerstenhaber (Lie) bracket}
\begin{equation}
  \label{eq:Gerstenhaber_bracket-HC}
  [c_1,c_2]\coloneqq\preLie{c_1}{c_2}-(-1)^{(p+q-1)(s+t-1)}\preLie{c_2}{c_1},
\end{equation}
with respect to the total degree. If $\chark(\kk)=2$ or $p+q$ is even, we also
consider the \emph{Gerstenhaber square} operation
\begin{equation}
  \label{eq:Gerstenhaber_square-HC}
  \Sq\colon\HC[p]<q>{A}\longrightarrow\HC[2p-1]<2q>{A},\qquad c\longmapsto\Sq(c)\coloneqq\preLie{c}{c}.
\end{equation}

The Hochschild cochain complex $\HC{A}$ is also endowed with the multiplication
\[
  \Astr<2>\in\HC[2]<0>{A}=\dgHom[\kk]{A^{\otimes 2}}{A}[0]=\Hom[\kk]{A^{\otimes
      2}}{A}
\]
given by the multiplication operation
\[
  \Astr<2>\colon A\otimes A\longrightarrow A,\qquad x\otimes y\longmapsto xy,
\]
of the graded algebra $A$. Using that $\Astr<2>$ has the bidegree $(2,0)$, the
associativity of the multiplication yields
\begin{align*}
  (\preLie{\Astr<2>}{\Astr<2>})(x_1,x_2,x_3)&=-\Astr<2>(\Astr<2>(x_1,x_2),x_3)+\Astr<2>(x_1,\Astr<2>(x_2,x_3))=0,
\end{align*}
so that $\Astr<2>\in\HC[2]<0>{A}$ is a brace algebra multiplication in the sense
of \Cref{defprop:braces-Gerstenhaber}. Consequently, $\HC{A}$ is endowed with
the bidegree $(1,0)$ \emph{Hochschild differential}
\begin{align}
  \label{eq:Hd}
  \begin{split}
    \Hd\colon\HC[p]<q>{A}&\longrightarrow\HC[p+1]<q>{A}\\
    c&\longmapsto[\Astr<2>,c]=\preLie{\Astr<2>}{c}-(-1)^{p+q-1}\preLie{c}{\Astr<2>},
  \end{split}
\end{align}
and hence $\HC{A}$ is a differential bigraded vector space. Explicitly, using
\Cref{eq:generic-Hd} we see that the action of the Hochschild differential on a
cochain $c\in\HC[p]<q>{A}$ is given by the formula
\begin{align}
  \label{eq:Hd-full}
  \begin{split}
    \Hd(c)(x_1,\dots x_{p+1})&\coloneqq(-1)^{q|x_1|+q}x_1c(x_2,\dots,x_{p+1})\\&\phantom{=}+\sum_{i=1}^p(-1)^{i+q}c(x_1,\dots,x_ix_{i+1},\dots,x_{p+1})\\&\phantom{=}+(-1)^{p+1+q}c(x_1,\dots,x_p)x_{p+1}.
  \end{split}
\end{align}

\begin{definition}
  \label{def:HHA}
  The \emph{Hochschild cohomology of $A$} is the bigraded vector space
  \begin{equation}
    \label{eq:HH}
    \HH{A}\coloneqq\H[\bullet,*]{\HC{A},\Hd}.
  \end{equation}
\end{definition}

\begin{notation}
  We set
  \begin{align*}
    \HH*[p]{A}&\coloneqq\prod_{q\in\ZZ}\HH[p]<q>{A}[A],\qquad p\geq0.
  \end{align*}
\end{notation}

\begin{remark}
  Endowed with the Hochschild differential, $\HC{A}$ is a shifted dg Lie algebra
  with the Gerstenhaber bracket, and $\HH{A}$ is a shifted graded Lie algebra,
  both with respect to the total degree.
\end{remark}

Appealing again to \Cref{defprop:braces-Gerstenhaber}, the multiplication
$\Astr<2>\in\HC{A}$ yields the bidegree $(0,0)$ \emph{cup product}
\begin{align}
  \label{eq:cup_product-HC}
  \begin{split}
    \HC[p]<q>{A}\otimes\HC[s]<t>{A}&\longrightarrow\HC[p+s]<q+t>{A}\\
    c_1\otimes c_2&\longmapsto\cupp{c_1}{c_2}\coloneqq(-1)^{|\s c_1|}\braces{\Astr<2>}{c_1,c_2},
  \end{split}
\end{align}
so that $\HC{A}$ is also a dg algebra and $\HH{A}$ is a Gerstenhaber algebra with
the induced operations, both with respect to the total degree. Explicitly, we
see from \Cref{eq:generic-cup_product} that
\begin{multline}
  \label{eq:cup_product}
  (\cupp{c_1}{c_2})(x_1,\dots,x_{p+s})\coloneqq\\(-1)^{t(p+\sum_{i=1}^p|x_i|)}c_1(x_1,\dots,x_p)c_2(x_{p+1},\dots,x_{p+s}).  
\end{multline}

\begin{remark}
  \label{rmk:HCA-operad}
  The brace algebra structure on $\HC{A}$ is obtained from the \emph{shifted}
  endomorphism graded operad of $A$, which is isomorphic to the endomorphism
  graded operad of the shifted graded vector space $A[1]$ by means of a
  non-trivial isomorphism, see for example~\cite[Remarks~2.5 and 2.7]{Mur16}.
\end{remark}

\begin{remark}
  \label{rmk:functoriality-isos-HC}
  The Hochschild cochain complex is functorial under isomorphisms of graded
  algebras: Given an isomorphism of graded algebras $\varphi\colon
  A\stackrel{\sim}{\to} B$, the isomorphism of graded vector spaces
  \[
    \Phi_{p,q}\colon\HC[p]<q>{A}\stackrel{\sim}{\longrightarrow}\HC[p]<q>{B},\qquad
    c\longmapsto c^\varphi,
  \]
  where
  \[
    c^\varphi(b_1,\dots,b_p)\coloneqq
    \varphi(c(\varphi^{-1}(b_1)),\dots,\varphi^{-1}(b_p)),
  \]
  assemble into an isomorphism of differential bigraded vector spaces
  \[
    \Phi\colon\HC{A}\stackrel{\sim}{\longrightarrow}\HC{B}
  \]
  that is strictly compatible with all of the algebraic structure described
  above. In particular, the induced isomorphism of bigraded vector spaces
  \[
    \H[\bullet,*]{\Phi}\colon\HH{A}\stackrel{\sim}{\longrightarrow}\HH{B}
  \]
  is an isomorphism of Gerstenhaber algebras with respect to the total degree.
\end{remark}

\begin{remark}
  The Hochschild cohomology $\HH{A}$ can be interpreted in conceptual terms as
  \[
    \HH[p]<q>{A}\cong\H[p]{\RHom*[A^e]<\bullet,q>{A}{A}}=\Ext[A^e][p,q]{A}{A},\quad
    p\geq0,\ q\in\ZZ,
  \]
  where we regard $A$ as an object of the derived category of graded
  $A$-bimodules\footnote{Which should not be confused with the derived category
    of dg $A$-bimodules, with $A$ viewed as a dg algebra with vanishing
    differential.}. Here,
  \[
    \RHom*[A^e]<\bullet,q>{A}{A}\coloneqq\RHom[A^e]{A}{A(q)},\qquad q\in\ZZ,
  \]
  is the derived $\operatorname{Hom}$ functor in the derived category of graded
  $A$-bimodules and
  \[
    \Ext[A^e][p,q]{A}{A}\coloneqq\Ext[A^e][p]{A}{A(q)},\quad p\geq0,\ q\in\ZZ.
  \]
  The above identifications stem from isomorphisms of cochain complexes of
  vector spaces
  \[
    \HC<q>{A}\cong\Hom[A^e]{\BB{A}}{A(q)},\qquad q\in\ZZ,
  \]
  where $\BB{A}$ is the bar resolution of $A$, which is a graded projective
  resolution of the latter as a graded $A$-bimodule. In terms of this
  identification, the cup product on $\HH{A}$ corresponds to the (bigraded)
  composition of endomorphisms in the derived category of graded $A$-bimodules
  (equivalently, to the Yoneda product of extensions). On the other hand, to our
  knowledge, there is no known interpretation of the Gerstenhaber bracket or of
  the Gerstenhaber square in similar terms, although both admit an
  interpretation in terms of `loop operations' in the category of extensions,
  see~\cite{Sch98a} where the case of ungraded algebras is treated.
\end{remark}

\begin{remark}
  \label{rmk:HC-gvsp}
  It is also useful to consider the Hochschild cochain complex when $A$ is a
  plain graded vector space, rather than a graded algebra. In this case the
  pre-Lie product, the Gerstenhaber bracket and the Gerstenhaber square are
  still defined, by the same formulas, as these only depend on the infinitesimal
  composition of multilinear maps. In this case the name `Hochschild cochain
  complex' is a misnomer for it does not have a differential; notwithstanding,
  we do not introduce any alternative terminology to distinguish this case.
\end{remark}

\section{Bimodule Hochschild cohomology}
\label{sec:Hochschild-bimodules}

In this section we recall the definition of the bimodule Hochschild cohomology of a
graded bimodule that is introduced in~\cite{JM25}. We also study the effect of
the passage to shifts and linear duals on this cohomology, and study the
case of the diagonal bimodule in more detail.

\subsection{The bimodule Hochschild cochain complex}
\label{subsec:RelBimHC}

Let $A$ be a graded algebra and $M$ a graded $A$-bimodule. We begin with an
auxiliary definition.

\begin{definition}
  The \emph{bimodule cochain complex} $\BimHC{M}$ is the bigraded vector space
  with the bihomogeneous components
  \begin{equation}
    \label{eq:BimHC}
    \BimHC[n]<r>{M}\coloneqq\bigoplus_{p+q=n}\BimHC*[p][q]<r>{M},\qquad n\geq0,\
    r\in\ZZ,
  \end{equation}
  where
  \[
    \BimHC*[p][q]<r>{M}\coloneqq\dgHom[\kk]{A^{\otimes p}\otimes M\otimes
      A^{\otimes q}}{M}[r].
  \]
\end{definition}

\begin{remark}
  \label{rmk:RelBimHC-arity}
  The horizontal degree $p+q$ of a cochain $c\in\BimHC*[p][q]<r>{M}$ does
  \emph{not} agree with the `arity degree' $p+1+q$.
\end{remark}

\begin{remark}
  Endowed with a suitable bidegree $(1,0)$ differential, $\BimHC[n]<r>{M}$ is a
  differential bigraded vector space whose cohomology
  \begin{equation}
    \label{eq:BimHH}
    \H[\bullet,*]{\BimHC{M}}=\BimHH{M}
  \end{equation}
  computes the bigraded vector space of graded-bimodule self-extensions of $M$.
  This conceptual interpretation stems from isomorphisms of cochain complexes of
  vector spaces
  \[
    \BimHC<r>{M}\cong\dgHom[A^e]{\BB{A}\otimes_AM\otimes_A\BB{A}}{M(r)},\qquad
    r\in\ZZ,
  \]
  where $\BB{A}$ is the bar resolution of $A$. Moreover, $\BimHH{M}$ is endowed
  with an associative cup product that corresponds to the Yoneda product of
  extensions, as well as with a natural graded $\HH{A}$-bimodule structure
  induced, roughly speaking, by the monoidal structure on the derived category
  of graded $A$-bimodules given by the derived tensor product over $A$. To keep
  the exposition as brief as possible, we do not recall these structures
  separately but rather deduce them from the algebraic structure of the bimodule
  Hochschild cochain complex of $M$, see \Cref{defprop:BimHC}.
\end{remark}

\begin{definition}
  \label{def:RelBimHC}
  The \emph{bimodule Hochschild cochain complex} is the bigraded vector space
  $\RelBimHC{A}{M}$ with the bihomogeneous components
  \begin{equation}
    \label{eq:RelBimHC}
    \RelBimHC[n]<r>{A}{M}\coloneqq\HC[n]<r>{A}\oplus\BimHC[n-1]<r>{M},\qquad n\geq0,\ r\in\ZZ.
  \end{equation}
\end{definition}

\begin{remark}
  Let $n\geq0$. Viewed as a homogeneous element of $\RelBimHC{A}{M}$, the
  horizontal degree of a cochain $c\in\BimHC*[p][q]<r>{M}$ with $p+q=n-1$ is
  given by the expected `arity degree' $n=p+1+q$, compare with
  \Cref{rmk:RelBimHC-arity}.
\end{remark}

\begin{notation}
  We set
  \[
    \HC*[n]{A\,|\,M}\coloneqq\prod_{r\in\ZZ}\RelBimHC[n]<r>{A}{M},\qquad n\geq0.
  \]
\end{notation}

Similar to $\HC{A}$, the bimodule Hochschild cochain complex $\RelBimHC{A}{M}$
is a brace algebra with multiplication (see \Cref{rmk:RelBimHC-operad}). It
is endowed with the bidegree $(-1,0)$ \emph{bimodule pre-Lie product}
\begin{align}
  \label{eq:pre-Lie_product-RelBimHC}
  \begin{split}
    \RelBimHC{A}{M}\otimes\RelBimHC{A}{M}&\longrightarrow\RelBimHC[\bullet-1]{A}{M}\\
    c_1\otimes c_2&\longrightarrow\preLie{c_1}{c_2}\coloneqq\braces{c_1}{c_2},
  \end{split}
\end{align}
where $c_1,c_2$ lie in either $\HC{A}$ or $\BimHC[\bullet-1]{M}$, see
\Cref{eq:binary_brace} for the definition of $\braces{c_1}{c_2}$. It is useful
to describe explicitly the bimodule pre-Lie product in the four natural cases:
\begin{itemize}
\item If $c_1,c_2\in\HC{A}$, then $\preLie{c_1}{c_2}\in\HC[\bullet-1]{A}$ is
  given by the pre-Lie product in the Hochschild cochain complex $\HC{A}$
  defined in \Cref{eq:pre-Lie_product-HC}.
\item If $c_1\in\HC{A}$ and $c_2\in\BimHC[\bullet-1]{M}$, then
  $\preLie{c_1}{c_2}=0$.
\item If $c_1\in\BimHC*[p][q]<r>{M}$ and $c_2\in\HC[s]<t>{A}$, then
  \[
    \preLie{c_1}{c_2}\in\BimHC*[p+s-1][q]<r>{M}\oplus\BimHC*[p][q+s-1]<r>{M}
  \]
  is the cochain
  \begin{equation}
    \label{eq:pre-Lie_product-RelBimHC-split}
    \preLie{c_1}{c_2}=\underbrace{\,\sum_{i=1}^pc_1\bullet_ic_2\,}_{\in\BimHC*[p+s-1][q]<r>{M}}+\underbrace{\,\sum_{j=1}^qc_1\bullet_{p+1+j} c_2\,}_{\in\BimHC*[p][q+s-1]<r>{M}}.
  \end{equation}
\item Finally, if $c_1\in\BimHC*[p][q]<r>{M}$ and $c_2\in\BimHC*[s][t]<u>{M}$,
  then
  \[
    \preLie{c_1}{c_2}= c_1\bullet_{p+1}c_2\in\BimHC*[p+s][q+t]<r+u>{M};
  \]
  explicitly, using \Cref{eq:infinitesimal_composition} we see that
  \begin{multline}
    \label{eq:associative-preLie-product}
    (\preLie{c_1}{c_2})(x_1,\dots,x_{p+s},m,y_1,\dots,y_{q+t})\coloneqq\\
    (-1)^{\maltese}c_1(x_1\dots,x_p,c_2(x_{p+1},\dots,x_{p+s},m,y_1,\dots,y_t),y_{1+t},\dots,y_{q+t}),
  \end{multline}
  with the sign given by
  \[
    \textstyle\maltese\coloneqq(s+t)q+u\left(p+q+\sum_{i=1}^p|x_i|\right).
  \]
\end{itemize}

\begin{remark}
  The pre-Lie product as defined by \Cref{eq:pre-Lie_product-RelBimHC} is
  ultimately determined by the infinitesimal composition operation defined by
  \Cref{eq:infinitesimal_composition}. Consequently, there is an inconsistency
  in the special case when $M=A$ is the diagonal bimodule, since infinitesimal
  compositions of the form $c_1\bullet_ic_2$ with $c_1\in\HC{A}$ and
  ${c_2\in\BimHC{A}}$ do not automatically vanish, which is meant to be the case
  in this context. This discrepancy can be fixed by replacing the diagonal
  bimodule by a weighted copy, say $A\cdot \iota$ where $\iota$ is a formal
  symbol, so that $A\neq A\cdot\iota$. Having clarified this important subtlety,
  in the sequel we do not make this distinction explicit in order to not
  overload the notation.
\end{remark}

We now describe the algebraic structure on the Hochschild bimodule complex
$\RelBimHC{A}{M}$ as prescribed by \Cref{defprop:braces-Gerstenhaber}. Firstly,
the bidegree $(-1,0)$ \emph{bimodule Gerstenhaber bracket}
\begin{align}
  \label{eq:Gerstenhaber_bracket-RelBimHC}
  \begin{split}
    \RelBimHC{A}{M}\otimes\RelBimHC{A}{M}&\longrightarrow\RelBimHC[\bullet-1]{A}{M}\\
    c_1\otimes c_2&\longmapsto [c_1,c_2]\coloneqq\preLie{c_1}{c_2}-(-1)^{|\s c_1||\s c_2|}\preLie{c_2}{c_1},
  \end{split}
\end{align}
endows $\RelBimHC{A}{M}$ with the structure of a shifted graded Lie algebra with
respect to the total degree. If $\chark(\kk)=2$ or $n+r$ is even, we also
consider the \emph{bimodule Gerstenhaber square} operation
\begin{equation}
  \label{eq:Gerstenhaber_square-RelBimHC}
  \Sq\colon\RelBimHC[n]<r>{A}{M}\longrightarrow\RelBimHC[2n-1]<2r>{A}{M},\qquad c\longmapsto\Sq(c)\coloneqq\preLie{c}{c}.
\end{equation}

\begin{remark}
  It follows from the definition of the pre-Lie product on $\RelBimHC{A}{M}$
  that $\BimHC{M}\subseteq\RelBimHC{A}{M}<1>$ is a graded Lie ideal (defined in
  the obvious way) and hence also a graded Lie subalgebra. Similarly,
  $\HC{A}<1>\subseteq\RelBimHC{A}{M}<1>$ is a graded Lie subalgebra.
\end{remark}

\begin{example}
  \label{ex:bracket-1M}
  Let $c\in\HC{A}$. For every $\iota\in\HC[0]{M}$, we have
  \[
    \preLie{c}{\iota}=0=\preLie{\iota}{c}
  \]
  and hence
  \[ [c,\iota]=\preLie{c}{\iota}-(-1)^{|\s c|}\preLie{\iota}{c}.
  \]
  In particular, $[c,\id[M]]=0$.
\end{example}

The bimodule Hochschild cochain complex $\RelBimHC{A}{M}$ is also endowed with
the brace algebra multiplication
\begin{equation}
  \label{eq:square-zero_product-AM}
  \Astr<2>[A\ltimes M]\coloneqq
  \Astr<2>+(\Astr<1,0>[M]+\Astr<0,1>[M])\in\HC[2]<0>{A}\oplus(\BimHC*[1][0]<0>{M}\oplus\BimHC*[0][1]<0>{M}),
\end{equation}
where $\Astr<2>$ is the multiplication of $A$ and
\begin{align*}
  \Astr<1,0>[M]\colon A\otimes M&\longrightarrow M,\qquad x\otimes m\longmapsto xm,\intertext{and}
                                  \Astr<0,1>[M]\colon M\otimes A&\longrightarrow M,\qquad m\otimes y\longmapsto my,
\end{align*}
are the action maps of the graded $A$-bimodule $M$. Equivalently, $m_2^{A\ltimes
  M}$ is the square-zero product on the graded vector space $A\oplus M$:
\begin{equation}
  (x,m)(x',m')\coloneqq (xx',xm'+mx'),\qquad x,x'\in A,\ m,m'\in M.
\end{equation}
Similar to the case of the multiplication $\Astr<2>\in\HC{A}$, the fact that
$\Astr<2>[A\ltimes M]\in\RelBimHC{A}{M}$ is a brace-algebra multiplication is a
consequence of the associativity of the square-zero product. We obtain the
bidegree $(1,0)$ \emph{bimodule Hochschild differential}
\begin{equation}
  \label{eq:Hd-RelBimHC}
  \dRelBim\colon\RelBimHC{A}{M}\longrightarrow\RelBimHC[\bullet+1]{A}{M},\qquad
  c\longmapsto \dRelBim(c)\coloneqq[\Astr<2>[A\ltimes M],c],
\end{equation}
and hence $\RelBimHC{A}{M}$ is a differential bigraded vector space.

\begin{definition}
  The \emph{bimodule Hochschild cohomology of $M$} is the cohomology of the
  bimodule Hochschild cochain complex:
  \begin{equation}
    \label{eq:RelBimHH}
    \RelBimHH{A}{M}\coloneqq\H[\bullet,*]{\RelBimHC{A}{M}}. 
  \end{equation}
\end{definition}

\begin{notation}
  We set
  \[
    \HH*[n]{A\,|\,M}\coloneqq\prod_{r\in\ZZ}\RelBimHH[n]<r>{A}{M},\qquad n\geq0.
  \]
\end{notation}

\begin{remark}
  Endowed with the bimodule Hochschild differential, $\RelBimHC{A}{M}$ is a
  shifted dg Lie algebra and $\RelBimHH{A}{M}$ is a shifted graded algebra, both
  with respect to the total degree.
\end{remark}

Finally, appealing once again to \Cref{defprop:braces-Gerstenhaber}, the
multiplication $\Astr<2>[A\ltimes M]\in\RelBimHC[2]<0>{A}{M}$ also induces the
bidegree $(0,0)$ \emph{bimodule cup product}
\begin{align}
  \label{eq:cup_product-RelBimHC}
  \begin{split}
    \RelBimHC{A}{M}\otimes\RelBimHC{A}{M}&\longrightarrow\RelBimHC{A}{M}\\
    c_1\otimes c_2&\longmapsto\cupp{c_1}{c_2}\coloneqq(-1)^{|\s c_1|}\braces{\Astr<2>[A\ltimes M]}{c_1,c_2},
  \end{split}
\end{align}
where $c_1,c_2$ lie in either $\HC{A}$ or $\BimHC[\bullet-1]{M}$. Keeping
\Cref{eq:multiplication_brace} in mind, let us describe the cochain
\[
  \braces{m_2^{A\ltimes
      M}}{c_1,c_2}=\braces{\Astr<2>}{c_1,c_2}+\braces{\Astr<1,0>[M]}{c_1,c_2}+\braces{\Astr<0,1>[M]}{c_1,c_2}
\]
in the four natural cases:
\begin{itemize}
\item If $c_1\in\HC[p]<q>{A}$ and $c_2\in\HC[s]<t>{A}$, then
  \begin{align*}
    \braces{\Astr<2>}{c_1,c_2}&=(\Astr<2>\bullet_1c_1)\bullet_{p+1}c_2=(-1)^{\vdeg{c_1}\vdeg{c_2}}(\Astr<2>\bullet_2 c_2)\bullet_1 c_1\\
    \braces{\Astr<1,0>[M]}{c_1,c_2}&=0\\
    \braces{\Astr<0,1>[M]}{c_1,c_2}&=0.
  \end{align*}
  Consequently, the cup product $\cupp{c_1}{c_2}$ is the cup product in the
  Hochschild cochain complex $\HC{A}$ defined in \Cref{eq:cup_product}.
\item If $c_1\in\HC[p]<q>{A}$ and $c_2\in\BimHC*[s][t]<u>{M}$, then
  \begin{align}
    \label{eq:cup_product-RelBimHC-HC-BimHC}
    \begin{split}
      \braces{\Astr<2>}{c_1,c_2}&=0\\
      \braces{\Astr<1,0>[M]}{c_1,c_2}&=(\Astr<1,0>[M]\bullet_1c_1)\bullet_{p+1}c_2\\
      \braces{\Astr<0,1>[M]}{c_1,c_2}&=0.
    \end{split}
  \end{align}
  Consequently, the cup product $\cupp{c_1}{c_2}$ lies in
  $\BimHC*[p+s][t]<q+u>{M}$ in this case and, using
  \Cref{eq:generic-cup_product}, we see that it is given explicitly by
  \begin{multline*}
    (\cupp{c_1}{c_2})(x_1,\dots,x_{p+s},m,y_1\dots,y_t)\coloneqq\\
    (-1)^{u\left(p+\sum_{i=1}^p|x_i|\right)}\underbrace{c_1(x_1,\dots,x_p)}_{\in
      A}\underbrace{c_2(x_{p+1},\dots,x_{p+s},m,y_1\dots,y_t)}_{\in M}.
  \end{multline*}
\item If $c_1\in\BimHC*[p][q]<r>{M}$ and $c_2\in\HC[s]<t>{A}$, then
  \begin{align}
    \label{eq:cup_product-RelBimHC-BimHC-HC}
    \begin{split}
      \braces{\Astr<2>}{c_1,c_2}&=0\\
      \braces{\Astr<1,0>[M]}{c_1,c_2}&=0\\
      \braces{\Astr<0,1>[M]}{c_1,c_2}&=(\Astr<0,1>[M]\bullet_1c_1)\bullet_{p+1+q+1}c_2
    \end{split}
  \end{align}
  Consequently, the cup product $\cupp{c_1}{c_2}$ lies in
  $\BimHC*[p][q+s]<r+t>{M}$ in this case and, using
  \Cref{eq:generic-cup_product}, we see that it is given explicitly by
  \begin{multline*}
    (\cupp{c_1}{c_2})(x_1,\dots,x_p,m,y_1\dots,y_{q+s})\coloneqq\\
    (-1)^{\maltese}\underbrace{c_1(x_1,\dots,x_p,m,y_1,\dots,y_q)}_{\in
      M}\underbrace{c_2(y_{q+1},\dots,y_{q+s})}_{\in A},
  \end{multline*}
  with the sign given by
  \[
    \textstyle\maltese\coloneqq
    t\left(p+1+q+\sum_{i=1}^p|x_i|+|m|+\sum_{j=1}^q|y_j|\right).
  \]
\item Finally, if $c_1\in\BimHC*[p][q]<r>{M}$ and $c_2\in\BimHC*[s][t]<u>{M}$,
  then
  \begin{align}
    \label{eq:BimHC-square-zero}
    \begin{split}
      \braces{\Astr<2>}{c_1,c_2}&=0\\
      \braces{\Astr<1,0>[M]}{c_1,c_2}&=0\\
      \braces{\Astr<0,1>[M]}{c_1,c_2}&=0.
    \end{split}
  \end{align}
  Consequently, $\cupp{c_1}{c_2}=0$ in this case.
\end{itemize}

With this structure, $\RelBimHC{A}{M}$ is a dg algebra and its cohomology
$\RelBimHH{A}{M}$ is a Gerstenhaber algebra, in both cases with respect to the
total degree. Notice also that $\BimHC{M}<-1>\subseteq\RelBimHC{A}{M}$ is a
square-zero dg ideal.

\begin{remark}
  \label{rmk:RelBimHC-operad}
  The brace algebra structure on $\RelBimHC{A}{M}$ is obtained from the
  infinitesimal composition operations of the shifted linear graded endomorphism
  operad~\cite{BM09} of the pair $(A,M)$, which is isomorphic to the linear
  endomorphism graded operad of the pair $(A,M(1))$.
\end{remark}

\begin{remark}
  \label{rmk:RelBimHC-under-isos}
  The bimodule Hochschild cochain complex $\RelBimHC{A}{M}$ is functorial under
  isomorphisms of graded $A$-bimodules. More precisely, given an isomorphism of
  graded $A$-bimodules $\varphi\colon M\stackrel{\sim}{\longrightarrow}N$ the
  maps
  \begin{align*}
    \Phi_{p,q}^{\varphi}\colon\BimHC*[p][q]<r>{M}&\stackrel{\sim}{\longrightarrow}\BimHC*[p][q]<r>{N},\qquad p,q\geq0,\ r\in\ZZ,\\
    c&\longmapsto c^\varphi,
  \end{align*}
  where, for $x_1,\dots,x_p,y_1,\dots,y_q\in A$ and $n\in N$,
  \[
    c^\varphi(x_1,\dots,x_p,n,y_1,\dots,y_q)\coloneqq\varphi(c(x_1,\dots,x_p,\varphi^{-1}(n),y_1,\dots,y_q)),
  \]
  assemble into an isomorphism of bigraded vector spaces
  \[
    \Phi\coloneqq\id[\HC{A}]\oplus\Phi^{\varphi}\colon\RelBimHC{A}{M}\stackrel{\sim}{\longrightarrow}\RelBimHC{A}{N}
  \]
  that is strictly compatible with all of the algebraic structure described
  above. Observe that the cochain $c^\varphi$ is determined by the commutativity
  of the following diagram of homogeneous morphisms of graded vector spaces:
  \begin{equation}
    \label{eq:Phi}
    \begin{tikzcd}
      A^{\otimes p}\otimes M\otimes A^{\otimes q}\dar[swap]{\id\otimes\varphi\otimes\id}\rar{c}&M\dar{\varphi}\\
      A^{\otimes p}\otimes N\otimes A^{\otimes q}\rar{c^{\varphi}}&N.
    \end{tikzcd}
  \end{equation}
\end{remark}

\begin{remark}
  \label{rmk:RelBimHC-gvsp}
  It is also useful to consider the bimodule Hochschild cochain complex when $A$
  and $M$ are plain graded vector spaces. In this case the pre-Lie product,
  the Gerstenhaber bracket and the Gerstenhaber square are still defined, by the
  same formulas, as they only depend on the infinitesimal composition of
  multilinear maps (compare with \Cref{rmk:HC-gvsp}).
\end{remark}

We now describe in more detail the algebraic structure on the bimodule cochain
complex $\BimHC{M}$ and on the bimodule Hochschild cochain complex
$\RelBimHC{A}{M}$. Firstly, notice that the bimodule Hochschild differential
decomposes as
\begin{equation}
  \label{eq:Hd-RelBimHC-sum}
  \dRelBim(c)=[\Astr<2>,c]+[\Astr<1,0>,c]+[\Astr<0,1>,c],\qquad c\in\RelBimHC{A}{M}.
\end{equation}
Thus, if $c\in\HC{A}$, then $[\Astr<2>,c]=\Hd(c)$ and
\begin{align}
  \label{eq:Hd-RelBimHC-sum-HC}
  \begin{split}
    [\Astr<1,0>,c]&=\preLie{\Astr<1,0>[M]}{c}-(-1)^{|\s c|}\preLie{c}{\Astr<1,0>[M]}=\Astr<1,0>[M]\bullet_1c,\\
    [\Astr<0,1>,c]&=\preLie{\Astr<0,1>[M]}{c}-(-1)^{|\s
                    c|}\preLie{c}{\Astr<0,1>[M]}=\Astr<0,1>[M]\bullet_2c.
  \end{split}
\end{align}
Moreover,
\begin{align*}
  \Astr<1,0>[M]\bullet_1c&=(\Astr<1,0>[M]\bullet_1c)\bullet_{p+1}\id[M]=\braces{\Astr<1,0>[M]}{c,\id[M]}=(-1)^{|\s c|}\cupp{c}{\id[M]},\\
  \Astr<0,1>[M]\bullet_2c&=(\Astr<0,1>[M]\bullet_1\id[M])\bullet_2 c=\braces{\Astr<0,1>[M]}{\id[M],c}=\cupp{\id[M]}{c},
\end{align*}
where we use that the shifted total degree of $\id[M]\in\BimHC*[0][0]<0>{M}$ is
\[
  |\s \id[M]|=|\id[M]|-1=1-1=0.
\]
Finally, if we identify
\[
  \BimHC{M}=(\BimHC{M}<-1>)[1],
\]
we can consider the map
\begin{align}
  \label{eq:delta}
  \begin{split}
    \delta\colon\HC{A}&\longrightarrow\BimHC{M}\\
    c&\longmapsto\cupp{\s\id[M]}{c}-\cupp{c}{\s\id[M]}=\s(\cupp{\id[M]}{c}-(-1)^{|c|}\cupp{c}{\id[M]}),
  \end{split}
\end{align}
where we use the shifted graded $\HC{A}$-bimodule structure on $\BimHC{M}$. On
the other hand, if $c\in\BimHC[\bullet-1]{M}$, then
\begin{align}
  \label{eq:dRelBim-sum}
  \begin{split}
    [\Astr<2>,c]&=\preLie{\Astr<2>}{c}-(-1)^{|\s c|}\preLie{c}{\Astr<2>}=(-1)^{|c|}\preLie{c}{\Astr<2>}\in\BimHC{M}\\
    [\Astr<1,0>[M],c]&=\preLie{\Astr<1,0>[M]}{c}-(-1)^{|\s c|}\preLie{c}{\Astr<1,0>[M]}\\
                &=\Astr<1,0>[M]\bullet_2 c-(-1)^{|\s c|}c\bullet_{p+1}\Astr<1,0>[M]\in\BimHC{M}\\
    [\Astr<0,1>[M],c]&=\preLie{\Astr<0,1>[M]}{c}-(-1)^{|\s
                       c|}\preLie{c}{\Astr<0,1>[M]}\\
                &=\Astr<0,1>[M]\bullet_1c-(-1)^{|\s
                  c|}c\bullet_{p+1}\Astr<0,1>[M]\in\BimHC{M}.
  \end{split}
\end{align}
It follows that the bimodule Hochschild differential $\dRelBim$ has the form
\begin{align}
  \label{eq:dRelBim-matrix}
  \left(\begin{smallmatrix}
    \Hd&0\\
    \delta&\dBim
  \end{smallmatrix}\right)\colon\HC{A}\oplus\BimHC[\bullet-1]{M}\longrightarrow\HC[\bullet+1]{A}\oplus\BimHC[\bullet]{M},
\end{align}
where the bidegree $(1,0)$ \emph{bimodule differential}
\begin{equation*}
  \label{eq:dBim1}
  \dBim\colon\BimHC[\bullet-1]{M}\longrightarrow\BimHC{M}
\end{equation*}
is the restriction of the bimodule Hochschild differential $\dRelBim$ to its
differential bigraded subspace $\BimHC[\bullet-1]{M}\subseteq\RelBimHC{A}{M}$.
Explicitly, using \Cref{eq:generic-Hd} we see that
\begin{equation}
  \label{eq:dBim}
  \dBim(c)=\dh(c)+\dv(c),\qquad c\in\BimHC*[p][q]<r>{M}
\end{equation}
with the \emph{horizontal bimodule differential}
\begin{equation*}
  \dh(c)\colon\BimHC*[p][q]<r>{M}\longrightarrow\BimHC*[p+1][q]<r>{M}
\end{equation*}
given by
\begin{align}
  \label{eq:dh}
  \begin{split}
    \dh(c)(x_1,\dots,x_{p+1},&m,y_1,\dots,y_q)\coloneqq(-1)^{r|x_1|+r}x_1c(x_2,\dots,x_{p+1},m,y_1,\dots,y_q)\\
                             &\phantom{=}+\sum_{i=1}^p(-1)^{i+r}c(x_1,\dots,x_ix_{i+1},\dots,x_{p+1},m,y_1,\dots,y_q)\\
                             &+(-1)^{p+1+r}c(x_1,\dots,x_p,x_{p+1}m,y_1,\dots,y_q),
  \end{split}
\end{align}  
and with the \emph{vertical bimodule differential}
\begin{equation*}
  \dv(c)\colon\BimHC*[p][q]<r>{M}\longrightarrow\BimHC*[p][q+1]<r>{M}
\end{equation*}
given by
\begin{align}
  \label{eq:dv}
  \begin{split}
    \dv(c)(x_1,\dots,x_p,&m,y_1,\dots,y_{q+1})\coloneqq(-1)^{p+1+r}c(x_1,\dots,x_p,my_1,y_2\dots,y_{q+1})\\
                         &\phantom{=}+\sum_{j=1}^q(-1)^{p+1+j+r}c(x_1,\dots,x_p,m,y_1,\dots,y_jy_{j+1},\dots,y_{q+1})\\
                         &+(-1)^{p+q+r}c(x_1,\dots,x_p,m,y_1,\dots,y_q)y_{q+1}.
  \end{split}
\end{align}

\begin{remark}
  \label{rmk:transpose-bimodule-complex}
  Let $r\in\ZZ$. The cochain complex $\BimHC[\bullet-1]<r>{M}[N]$ is the
  totalisation of the apparent double cochain complex with anti-commuting
  differentials:
  \[
    \dv\dh+\dh\dv=0.
  \]
\end{remark}

The following proposition summarises certain computational aspects of bimodule
Hochschild cohomology that can be derived from the previous discussion. Although
not all of these are needed in the sequel, we list some of them here to aid the
reader's intuition and refer to ~\cite[Section~3]{JM25} for a more
thorough discussion.

\begin{proposition}[{\cite[Remark~3.3.9, Lemma~3.3.12, Proposition~3.4.8, and
    Corollary~3.4.9]{JM25}}]
  \label{prop:RelBimHC-all}
  Let $r\in\ZZ$. The following statements hold:
  \begin{enumerate}
  \item As a cochain complex, $\RelBimHC<r>{A}{M}$ is the standard mapping
    cocone of the cochain map
    \[
      \delta\colon\HC<r>{A}\longrightarrow\BimHC<r>{M},\qquad
      c\longmapsto\cupp{\s\id[M]}{c}-\cupp{c}{\s\id[M]}.
    \]
    Moreover, the graded algebra structure on $\RelBimHC{A}{M}$ is precisely
    that of the square-zero extension of $\HC{A}$ by the $\HC{A}$-bimodule
    $\BimHC{M}<-1>$.
  \item There is an exact triangle in the derived category of dg vector spaces
    \[
      \begin{tikzcd}[column sep=small]
        \BimHC{M}<-1>\rar{i}&\RelBimHC{A}{M}\rar{p}&\HC{A}\rar{\delta}&\BimHC{M}.
      \end{tikzcd}
    \]
    Moreover,
    \begin{itemize}
    \item The map $p[1]\colon\RelBimHC{A}{M}<1>\to\HC{A}<1>$ is a quotient map
      of dg Lie algebras, and hence
      $i[1]\colon\BimHC{M}<-1>\hookrightarrow\RelBimHC{A}{M}$ is the inclusion
      of a dg Lie ideal and also of a dg Lie subalgebra.
    \item The map $p\colon\RelBimHC{A}{M}\to\HC{A}$ is a quotient map of dg
      algebras, and hence the map
      $i\colon\BimHC{M}<-1>\hookrightarrow\RelBimHC{A}{M}$ is the inclusion of a
      square-zero dg ideal and also of a dg $\HC{A}$-bimodule.
    \end{itemize}
  \item There is a long exact sequence of vector spaces
    \[
      \begin{tikzcd}[column sep=small,row sep=small]
        \cdots\rar{\delta}&\BimHH[n-1]<r>{M}\rar{i}&\RelBimHH[n]<r>{A}{M}\rar{p}&\HH[n]<r>{A}\ar[out=-10,in=170,overlay]{dll}[description]{\delta}\\
        &\BimHH[n]<r>{M}\rar{i}&\RelBimHH[n+1]<r>{A}{M}\rar{p}&\HH[n+1]<r>{A}\rar{\delta}&\cdots
      \end{tikzcd}
    \]
    Moreover:
    \begin{itemize}
    \item The map $p\colon\RelBimHH{A}{M}\to\HH{A}$ is a morphism of
      Gerstenhaber algebras (defined in the obvious way).
    \item The map $i\colon\BimHH{M}<-1>\to\RelBimHH{A}{M}$ is a morphism of
      graded $\RelBimHH{A}{M}$-bimodules, where $\BimHH{M}$ is viewed as a
      graded $\RelBimHC{A}{M}$-bimodule via the quotient map $p$ (equivalently,
      through the shifted dg $\RelBimHH{A}{M}$-bimodule structure).
    \item The connecting homomorphism $\delta$ vanishes if and only if
      $\BimHH{M}$ is symmetric as a graded $\HH{A}$-bimodule:
      \[
        \cupp{\s\id[M]}{c}=\cupp{c}{\s\id[M]}\in\BimHH{M},\qquad c\in\HH{A},
      \]
      where we use that $|\s\id[M]|=|\id[M]|-1=0$ since
      $\id[M]\in\BimHC*[0][0]<0>{M}$.
    \end{itemize}
  \end{enumerate}
\end{proposition}

We also record the following result that is of independent interest.

\begin{defprop}[{\cite[Lemma~3.2.12]{JM25}}]
  \label{defprop:BimHC}
  The following statements hold:
  \begin{enumerate}
  \item Endowed with the shifted bimodule differential $\dBim[1]$ and the
    \emph{Yoneda product}\footnote{Here, we identify $\BimHC{M}$ with a
      differential bigraded subspace of $\RelBimHC{A}{M}<1>$. Notice also that
      \[
        |\s c_1\circ\s c_2|=|\s(\preLie{c_1}{c_2})|=|c_1|+|c_2|-2=|\s c_1|+|\s c_2|,
      \]
    so that the Yoneda product is indeed graded.}
    \begin{align}
      \label{eq:Yoneda_product-BimHC}
      \begin{split}
        \BimHC{M}&\longrightarrow\BimHC{M}\\
        \s c_1\otimes\s c_2&\longmapsto\YonedaProd{\s c_1}{\s c_2}\coloneqq\s(\preLie{c_1}{c_2}),
      \end{split}
    \end{align}
    see \Cref{eq:associative-preLie-product}, the bimodule cochain complex
    $\BimHC{M}$ is a dg algebra with respect to the total degree.
  \item The graded commutator
    \begin{equation}
      \label{eq:Lie_bracket-BimHC}
      [\s c_1,\s c_2]=\YonedaProd{\s c_1}{\s c_2}-(-1)^{|\s c_1||\s
        c_2|}\YonedaProd{\s c_2}{\s c_1}
    \end{equation}
    of the Yoneda product agrees with the restriction of the bimodule
    Gerstenhaber bracket to the dg Lie ideal
    $\BimHC{M}\subseteq\RelBimHC{A}{M}<1>$.
  \item Endowed with the graded $\HC{A}$-bimodule structure given by the shifted
    bimodule cup product
    \begin{align*}
      \HC{A}\otimes\BimHC{M}\otimes\HC{A}&\longrightarrow\BimHC{M}\\
      c_1\otimes\s c_2\otimes c_3&\longmapsto(-1)^{|c_1|}\s(\cupp{c_1}{\cupp{c_2}{c_3}}),
    \end{align*}
    the bimodule cochain complex $\BimHC{M}$ is a dg $\HC{A}$-bimodule with
    respect to the total degree. Consequently, the bigraded vector space of
    bimodule self-extensions $\BimHH{M}$ is a graded $\HH{A}$-bimodule with
    respect to the total degree.
  \end{enumerate}
\end{defprop}

\begin{remark}
  \label{rmk:tensor_action}
  The action of $\HH{A}$ on the bigraded bimodule self-extension space
  $\BimHC{M}$ described in \Cref{defprop:BimHC} can be understood in conceptual
  terms as follows. Recall that the derived category of graded $A$-bimodules is
  a monoidal category with respect to the derived tensor product $(M,N)\mapsto
  M\Lotimes[A]N$. The tensor product of morphisms induces maps
  \begin{align*}
    \RHom[A^e]{A}{A}\otimes\RHom[A^e]{M}{M}&\longrightarrow\RHom[A^e]{A\Lotimes[A]M}{A\Lotimes[A]M}\\
    \RHom[A^e]{M}{M}\otimes\RHom[A^e]{A}{A}&\longrightarrow\RHom[A^e]{M\Lotimes[A]A}{M\Lotimes[A]A}.
  \end{align*}
  Applying the structural isomorphisms $A\Lotimes[A]{M}\cong M$ and
  $M\Lotimes[A]{A}\cong M$, we obtain maps
  \begin{align*}
    \RHom[A^e]{A}{A}\otimes\RHom[A^e]{M}{M}&\longrightarrow\RHom[A^e]{M}{M}\\
    \RHom[A^e]{M}{M}\otimes\RHom[A^e]{A}{A}&\longrightarrow\RHom[A^e]{M}{M}
  \end{align*}
  that induce the required action maps
  \begin{align*}
    \HH{A}\otimes\BimHH{M}[M]&\longrightarrow\BimHH{M}[M]\\
    \BimHH{M}[M]\otimes\HH{A}&\longrightarrow\BimHH{M}[M].
  \end{align*}
  Since the derived tensor product over $A$ is not symmetric in general, the
  left and right actions of $\HH{A}$ on $\BimHH{M}[M]$ might be quite different;
  more precisely $\BimHH{M}[M]$ need not be a symmetric $\HH{A}$-bimodule (compare
  \Cref{prop:RelBimHC-all}). It is also worth mentioning that $L_\infty$-algebra
  (=strong homotopy Lie algebra) structures on mapping cones have been
  considered previously in the literature, see~\cite{FM07} and the references
  therein.
\end{remark}

\begin{remark}
  \label{rmk:BimHC-MN}
  Let $M$ and $N$ be graded $A$-bimodules. The bimodule cochain complex admits
  an evident generalisation
  \[
    \BimHC[n]<r>{M}[N]\coloneqq\bigoplus_{p+q=n}\dgHom[\kk]{A^{\otimes p}\otimes
      M\otimes A^{\otimes q}}{N},\qquad n\geq 1,\ r\in\ZZ,
  \]
  whose cohomology is the bigraded vector space of graded-bimodule extensions,
  \[
    \H[\bullet,*]{\BimHC{M}[N]}=\BimHH{M}[N].
  \]
  The formula for the bimodule differential is entirely analogous, and so are
  the formulas for the left and right actions of the Hochschild cochain complex
  $\HC{A}$. In particular, $\BimHH{M}[N]$ is a graded $\HH{A}$-bimodule. Notice,
  however, that $\BimHC{M}[N]$ does not admit a Yoneda product when $M\not\cong
  N$, although there is a Yoneda product of the form
  \[
    \BimHC{M}[N]\otimes\BimHC{L}[M]\longrightarrow\BimHC{L}[N],
  \]
  where $L$ is a further graded $A$-bimodule, which corresponds to the bigraded
  composition of morphisms in the derived category of graded $A$-bimodules. As
  with the bimodule Hochschild cochain complex, we also consider the bimodule
  complex when $A$, $M$ and $N$ plain graded vector spaces, in which case
  there is no natural differential.
\end{remark}

\subsection{Passage to vertical shifts of a graded module}

Let $A$ be a graded algebra and $M$ a graded $A$-bimodule. Fix $n\in\ZZ$. We
wish to show that the bimodule Hochschild cochain complex is invariant under the
passage $M\mapsto M(n)$ from $M$ to its (vertical) shift. For $p,q\geq0$ and
$r\in\ZZ$, we introduce the isomorphism of vector spaces
\begin{align*}
  \Psi_{p,q,r}^{(n)}\colon\BimHC*[p][q]<r>{M}&\stackrel{\sim}{\longrightarrow}\BimHC*[p][q]<r>{M(n)}\\
  c&\longmapsto c^{(n)},
\end{align*}
where, for $x_1\dots,x_p,y_1,\dots,y_q\in A$ and $m\in M$,
\begin{multline}
  \label{eq:cn}
  c^{(n)}(x_1,\dots,x_p,\s[n]m,y_1,\dots
  y_q)\coloneqq\\(-1)^{n(\vdeg{c}+\sum_{i=1}^p|x_i|)}\s[n]c(x_1\dots,x_p,m,y_1,\dots,y_q).
\end{multline}
In other words, the cochain $c^{(n)}\in\BimHC*[p][q]<r>{M(n)}$ is defined by the
commutativity of the following diagram, in which the top horizontal map is an
instance of the isomorphism of dg vector spaces in \Cref{ex:UVnW}:
\[
  \begin{tikzcd}
    A^{\otimes p}\otimes M(n)\otimes A^{\otimes
      q}\rar{\sim}\dar{c^{(n)}}&(A^{\otimes p}\otimes M\otimes
    A^{\otimes q})(n)\dar{c(n)}\\
    M(n)&M(n)\lar[equals]
  \end{tikzcd}
\]
We then define the map
\begin{equation}
  \label{eq:Psi}
  \Psi\coloneqq\id[\HC{A}]\oplus\Psi^{(n)}\colon\RelBimHC{A}{M}\longrightarrow\RelBimHC{A}{M(n)},
\end{equation}
where
\[
  \Psi^{(n)}\colon\BimHC[m-1]<r>{M}\longrightarrow\BimHC[m-1]<r>{M(n)},\qquad
  m\geq0,\ r\in\ZZ,
\]
acts on a cochain $c=(c_{p,q})_{p+1+q=m}\in\BimHC[m-1]<r>{M}$ by
\[
  \Psi^{(n)}(c)\coloneqq(c_{p,q}^{(n)})_{p+1+q=m}.
\]

\begin{remark}
  \label{rmk:Psi}
  For a cochain $c\in\BimHC*[0][0]{M}=\dgHom[\kk]{M}{M}$, we have
  $c^{(n)}=c[n]$, see \Cref{not:s}.
\end{remark}

\begin{proposition}
  \label{prop:shifting-cochains}
  Let $n\in\ZZ$. The isomorphism of bigraded vector spaces
  \begin{align*}
    \Psi\colon\RelBimHC{A}{M}\stackrel{\sim}{\longrightarrow}\RelBimHC{A}{M(n)},
  \end{align*}
  is strictly compatible with the following algebraic structures on both source
  and target:\footnote{In fact, $\Psi$ is an isomorphism of brace algebras with
    multiplication.} the bimodule pre-Lie product, the bimodule Gerstenhaber
  bracket, the bimodule Gerstenhaber square, the brace algebra multiplication,
  the bimodule Hochschild differential, and the bimodule cup product.
  Consequently, the induced map
  \[
    \H[\bullet,*]{\Psi}\colon\RelBimHH{A}{M}\stackrel{\sim}{\longrightarrow}\RelBimHH{A}{M(n)}
  \]
  is an isomorphism of Gerstenhaber algebras with respect to the total degree.
\end{proposition}
\begin{proof}
  Firstly, let $c_1,c_2$ lie in either $\HC{A}$ or $\BimHC[\bullet-1]{M}$. A
  straightforward case-by-case computation shows that
  \[
    \Psi(c_1\bullet_i c_2)=\Psi(c_1)\bullet_i\Psi(c_2),\qquad 1\leq i\leq\hdeg{c_1}.
  \]
  Since the algebraic structure on the bimodule Hochschild cochain complex is
  defined completely in terms of the infinitesimal compositions, it only remains
  to show that $\Psi$ preserves the brace algebra multiplication:
  \[
    \Psi(\Astr<2>[A\ltimes M])=m_2^{A\ltimes M(n)}.
  \]
  Indeed,
  \[
    \Psi(\Astr<2>[A\ltimes
    M])\stackrel{\eqref{eq:square-zero_product-AM}}{=}\Psi(\Astr<2>+\Astr<1,0>[M]+\Astr<0,1>[M])\stackrel{\eqref{eq:Phi}}{=}\Astr<2>+\Psi(\Astr<1,0>[M])+\Psi(\Astr<0,1>[M]).
  \]
  Moreover, since $\vdeg{\Astr<1,0>[M]}=0$, for $x\in A$ and $m\in M$,
  \begin{align*}
    \Psi(\Astr<1,0>[M])(x,\s[n]m)&\stackrel{\eqref{eq:Phi}}{=}(-1)^{n|x|}\s[n]\Astr<1,0>[M](x,m)\\
                                 &=(-1)^{n|x|}\s[n]xm\\
                                 &\stackrel{\eqref{eq:Mn}}{=}m_{1,0}^{M(n)}(x,\s[n]m).
  \end{align*}
  Similarly, $\Psi(\Astr<0,1>[M])=m_{0,1}^{M(n)}$. Therefore,
  \[
    \Psi(\Astr<2>[A\ltimes
    M])=\Astr<2>+m_{1,0}^{M(n)}+m_{0,1}^{M(n)}\stackrel{\eqref{eq:square-zero_product-AM}}{=}m_2^{A\ltimes
      M(n)}.
  \]
  This finishes the proof.
\end{proof}

\begin{remark}
  \label{rmk:shifting-cochains}
  \Cref{prop:shifting-cochains} remains valid when $A$ and $M$ are plain graded
  vector spaces, keeping in mind that the brace algebra multiplication,
  Hochschild differential and cup product are not defined in this case.
\end{remark}

\begin{variant}
  \label{variant:shifting-cochains}
  Let $A$, $M$ and $N$ be graded vector spaces (for example a graded algebra and
  two graded bimodules over it). For $p,q\geq0$ we introduce the isomorphism of
  vector spaces
  \[
    \Psi_{p,q,r}^{(n)}\colon\dgHom[\kk]{A^{\otimes p}\otimes M\otimes A^{\otimes
        q}}{N}[r]\longrightarrow\dgHom[\kk]{A^{\otimes p}\otimes M(n)\otimes
      A^{\otimes q}}{N(n)}[r],
  \]
  defined using \Cref{eq:cn}. We then have
  \[
    \Psi(c_1\bullet_i c_2)=\Psi(c_1)\bullet_i\Psi(c_2),\qquad 1\leq i\leq\hdeg{c_1},
  \]
  whenever both infinitesimal compositions vanish, but also in the following
  cases:
  \begin{itemize}
  \item $c_1\in\dgHom[\kk]{A^{\otimes p}\otimes M\otimes A^{\otimes q}}{N}[r]$
    and $c_2\in\HC{A}$.
  \item $c_1\in\dgHom[\kk]{A^{\otimes p}\otimes M\otimes A^{\otimes q}}{N}[r]$
    and $c_2\in\BimHC{M}$.
  \item $c_1\in\BimHC{N}$ and $c_2\in\dgHom[\kk]{A^{\otimes p}\otimes M\otimes
      A^{\otimes q}}{N}[r]$.
  \end{itemize}
\end{variant}

\subsection{Passage to the linear dual of a graded bimodule}

Let $A$ be a graded algebra and $M$ a graded $A$-bimodule. We now analyse how
the bimodule cochain complex interacts with the passage from $M$ to its
linear dual $DM$. For ${p,q\geq0}$ and $r\in\ZZ$, we introduce the morphism of
vector spaces (notice the transposition $(p,q,r)\mapsto(q,p,r)$ of the first two
components)
\begin{align*}
  \Theta_{p,q,r}^{D}\colon\BimHC*[p][q]<r>{M}&\longrightarrow\BimHC*[q][p]<r>{DM},\qquad
                                               p,q\geq 0,\ r\in\ZZ,\\
  c&\longmapsto c^{D},
\end{align*}
where, for $x_1\dots,x_p,y_1,\dots,y_q\in A$, $f\in DM$ and $m\in M$,
\begin{equation}
  \label{eq:cD}
  c^{D}(y_1,\dots
  y_q,f,x_1,\dots,x_p)(m)\coloneqq(-1)^{\maltese}f(c(x_1,\dots,x_p,m,y_1,\dots
  y_q)),
\end{equation}
and
\[
  \textstyle
  \maltese\coloneqq(p+1)(q+1)+\vdeg{c}|f|+(\sum_{j=1}^q|y_j|)(|f|+\sum_{i=1}^p|x_i|+|m|).
\]
In other words, the cochain $c^D\in\BimHC*[q][p]<r>{DM}$ is defined by the
commutativity, up to the sign $(-1)^{(p+1)(q+1)}$, of the following diagram of
homogeneous morphisms of graded vector spaces, in which the top horizontal
isomorphism uses the braiding corresponding to the permutation
$(\begin{smallmatrix}1&2&3&4\\4&1&2&3)\end{smallmatrix}$:
\[
  \begin{tikzcd}[column sep=small] A^{\otimes q}\otimes DM\otimes A^{\otimes
      p}\otimes M\ar{rr}{\sim}\ar{d}{c^{D}\otimes\id}&&DM\otimes A^{\otimes
      p}\otimes M\otimes
    A^{\otimes q}\dar{\id\otimes c}\\
    DM\otimes M\rar{\ev}&\kk&DM\otimes M\lar[swap]{\ev}
  \end{tikzcd}
\]
We then define the map
\begin{equation}
  \label{eq:Theta}
  \Theta\coloneqq\id[\HC{A}]\oplus\Theta^D\colon\RelBimHC{A}{M}\longrightarrow\RelBimHC{A}{DM},
\end{equation}
where
\[
  \Theta^D\colon\BimHC[n-1]<r>{M}\longrightarrow\BimHC[n-1]<r>{DM},\qquad
  n\geq0,\ r\in\ZZ,
\]
acts on a cochain $c=(c_{p,q})_{p+1+q=n}\in\BimHC[n-1]<r>{M}$ by
\[
  \Theta^D(c)\coloneqq(c_{p,q}^D)_{q+1+p=n}.
\]

\begin{remark}
  \label{rmk:Theta}
  For a cochain $c\in\BimHC*[0][0]{M}=\dgHom[\kk]{M}{M}$, we have $c^{D}=-c^*$
  (notice the minus sign), see \Cref{ex:dual map}. For example, if $d\colon M\to
  M$ is a degree $1$ differential, for $f\in DM$,
  \[
    d^D(f)(m)=-(-1)^{|f|}f(d(m)),\qquad m\in M,
  \]
  compare with \Cref{ex:DV}.
\end{remark}

\begin{lemma}
  \label{lemma:Theta}
  The morphism of bigraded vector spaces
  \begin{align*}
    \Theta\colon\RelBimHC{A}{M}\longrightarrow\RelBimHC{A}{DM},
  \end{align*}
  satisfies the following compatibilities with respect to the infinitesimal
  compositions:
  \begin{enumerate}
  \item\label{eq:Theta-HC-HC} If $c_1,c_2\in\HC{A}$, then
    \[
      \Theta(c_1\bullet_ic_2)=c_1\bullet_ic_2=\Theta(c_1)\bullet_i\Theta(c_2),\qquad
      1\leq i\leq\hdeg{c_1}.
    \]
  \item\label{eq:Theta-HC-BimHC} If $c_1\in\HC{A}$ and
    $c_2\in\BimHC[\bullet-1]{M}$, then
    \[
      \Theta(c_1\bullet_ic_2)=0=\Theta(c_1)\bullet_i\Theta(c_2),\qquad 1\leq
      i\leq\hdeg{c_1}.
    \]
  \item\label{eq:Theta-BimHC-HC} If $c_1\in\BimHC*[p][q]{M}$ and $c_2\in\HC{A}$,
    then
    \[
      \Theta(c_1\bullet_{p+1+j}c_2)=\Theta(c_1)\bullet_j\Theta(c_2),\qquad 1\leq
      j\leq q,
    \]
    and
    \[
      \Theta(c_1\bullet_ic_2)=\Theta(c_1)\bullet_{q+1+i}\Theta(c_2),\qquad 1\leq
      i\leq p.
    \]
  \item\label{eq:Theta-BimHC-BimHC} If $c_1\in\BimHC*[p][q]<r>{M}$ and
    $c_2\in\BimHC*[s][t]<u>{M}$, then
    \[
      \Theta(c_1\bullet_{p+1}c_2)=-(-1)^{|\s c_1||\s
        c_2|}\Theta(c_2)\bullet_{t+1}\Theta(c_1),
    \]
    where $|\s c_1|=p+q+r-1$ is the shifted total degree, and similarly for $|\s
    c_2|$.
  \end{enumerate}
\end{lemma}
\begin{proof}
  The first two statements are obvious since $\Theta$ acts as the identity on
  the subspace $\HC{A}\subseteq\RelBimHC{A}{M}$, see \Cref{eq:Theta}.

  We continue by establishing statement \eqref{eq:Theta-BimHC-HC}. Let
  $c_1\in\BimHC*[p][q]<r>{M}$ and $c_2\in\HC[s]<t>{A}$ and $1\leq j\leq q$. We
  compute, for $x_1,\dots,x_p,y_1,\dots,y_{q+s}\in A$, $f\in DM$ and $m\in M$,
  \begin{align*}
    (\Theta&(c_1)\bullet_jc_2)(y_1,\dots,y_{q+s},f,x_1,\dots,x_p)(m)\\&\stackrel{\eqref{eq:infinitesimal_composition}}{=}(-1)^{\maltese_1}c_1^D(y_1,\dots,y_{j-1},c_2(y_j,\dots,y_{j+s-1}),y_{j+s},\dots,y_{q+s},f,x_1,\dots,x_p)(m)\\
           &\stackrel{\eqref{eq:cD}}{=}(-1)^{\maltese_1+\maltese_2}f(c_1(x_1,\dots,x_p,m,y_1,\dots,y_{j-1},c_2(y_j,\dots,y_{j+s-1}),y_{j+s},\dots,y_{q+s})).
  \end{align*}
  where the signs are given by
  \begin{align*}
    \maltese_1&\stackrel{\eqref{eq:infinitesimal_composition}}{=}\textstyle(s-1)(p+1+q-j)+\vdeg{c_2}\left(p+q+\sum_{i=1}^{j-1}|y_i|\right)\\
              &=\textstyle(s-1)(p+1)+(s-1)(q-j)+\vdeg{c_2}\left(p+q+\sum_{i=1}^{j-1}|y_i|\right)\\
    \maltese_2&\stackrel{\eqref{eq:cD}}{=}\textstyle(p+1)(q+1)+|f|\vdeg{c_1}+\left(\vdeg{c_2}+\sum_{k=1}^{q+s}|y_k|\right)\left(|f|+\sum_{\ell=1}^{p}|x_\ell|+|m|\right).
  \end{align*}
  On the other hand,
  \begin{align*}
    \Theta&(c_1\bullet_{p+1+j}c_2)(y_1,\dots,y_{q+s},f,x_1,\dots,x_p)(m)\\
          &\stackrel{\eqref{eq:cD}}{=}(-1)^{\maltese_3}f((c_1\bullet_{p+1+j}c_2)(x_1,\dots,x_p,m,y_1,\dots,y_{q+s}))\\
          &\stackrel{\eqref{eq:infinitesimal_composition}}{=}(-1)^{\maltese_3+\maltese_4}f(c_1(x_1,\dots,x_p,m,y_1,\dots,y_{j-i},c_2(y_j,\dots,y_{j+s-1}),y_{j+s},\dots,y_{q+s})),
  \end{align*}
  where the signs are given by
  \begin{align*}
    \maltese_3&\stackrel{\eqref{eq:cD}}{=}\textstyle(p+1)(q+s)+|f|(\vdeg{c_1}+\vdeg{c_2})+\left(\sum_{k=1}^{q+s}|y_k|\right)\left(|f|+\sum_{\ell=1}^p|x_\ell|+|m|\right)\\
    \maltese_4&\stackrel{\eqref{eq:infinitesimal_composition}}{=}\textstyle(s-1)(q-j)+\vdeg{c_2}\left(p+q+\sum_{\ell=1}^{p}|x_\ell|+|m|+\sum_{i=1}^{j-i}|y_i|\right).
  \end{align*}
  Comparing the signs we see that
  \[
    \Theta(c_1\bullet_{p+1+j}c_2)=\Theta(c_1)\bullet_jc_2=\Theta(c_1)\bullet_j\Theta(c_2).
  \]
  An entirely analogous computation shows that
  \[
    \Theta(c_1\bullet_ic_2)=\Theta(c_1)\bullet_{q+1+i}\Theta(c_2),\qquad 1\leq i\leq
    p.
  \]

  We finish the proof by establishing statement \eqref{eq:Theta-BimHC-BimHC}.
  Let $c_1\in\BimHC*[p][q]<r>{M}$ and $c_2\in\BimHC*[s][t]<u>{M}$. For
  $x_1,\dots,x_{p+s},y_1\dots,y_{q+t}\in A$, $f\in DM$ and $m\in M$ we compute
  \begin{align*}
    \Theta&(c_1\bullet_{p+1}c_2)(y_1\dots,y_{q+t},f,x_1,\dots,x_{p+s})(m)\\
          &\stackrel{\eqref{eq:cD}}{=}(-1)^{\maltese_5}f((c_1\bullet_{p+1}c_2)(x_1,\dots,x_{p+s},m,y_1,\dots,y_{q+t}))\\
          &\stackrel{\eqref{eq:infinitesimal_composition}}{=}(-1)^{\maltese_5+\maltese_6}f(c_1(x_1,\dots,x_p,c_2(x_{p+1}\dots,x_{p+s},m,y_1\dots,y_t),y_{1+t},\dots,y_{q+t})),
  \end{align*}
  where the signs are given by
  \begin{align*}
    \maltese_5&\stackrel{\eqref{eq:cD}}{=}(q+t+1)(p+s+1)+\textstyle|f|(\vdeg{c_1}+\vdeg{c_2})+\left(\sum_{j=1}^{q+t}|y_j|\right)\left(|f|+\sum_{i=1}^{p+s}|x_i|+|m|\right)\\
    \maltese_6&\stackrel{\eqref{eq:infinitesimal_composition}}{=}\textstyle(s+t)q+\vdeg{c_2}\left(p+q+\sum_{k=1}^p|x_k|\right).
  \end{align*}
  On the other hand,
  \begin{align*}
    (\Theta&(c_2)\bullet_{t+1}\Theta(c_1))(y_1\dots,y_{q+t},f,x_1,\dots,x_{p+s})(m)\\
           &\stackrel{\eqref{eq:infinitesimal_composition}}{=}(-1)^{\maltese_7}c_2^D(y_1\dots,y_t,c_1^D(y_{1+t},\dots,y_{q+t},f,x_1,\dots,x_p),x_{p+1}\dots,x_{p+s})(m)\\
           &\stackrel{\eqref{eq:cD}}{=}(-1)^{\maltese_7+\maltese_8}c_1^D(y_{1+t},\dots,y_{q+t},f,x_1,\dots,x_p)(c_2(x_{p+1}\dots,x_{p+s},m,y_1\dots,y_t))\\
           &\stackrel{\eqref{eq:cD}}{=}(-1)^{\maltese_7+\maltese_8+\maltese_9}f(c_1(x_1,\dots,x_p,c_2(x_{p+1}\dots,x_{p+s},m,y_1\dots,y_t),y_{1+t},\dots,y_{q+t})),
  \end{align*}
  where the signs are given by
  \begin{align*}
    \maltese_7&\stackrel{\eqref{eq:infinitesimal_composition}}{=}\textstyle(p+q)s+\vdeg{c_1}\left(s+t+\sum_{i=1}^{t}|y_i|\right)\\
    \maltese_8&\stackrel{\eqref{eq:cD}}{=}\textstyle(s+1)(t+1)+\vdeg{c_2}\left(\vdeg{c_1}+\sum_{j=1}^{q}|y_{j+t}|+|f|+\sum_{k=1}^p|x_k|\right)\\&\textstyle\phantom{=}+\left(\sum_{i=1}^t|y_i|\right)\left(\vdeg{c_1}+\sum_{j=1}^{q}|y_{j+t}|+|f|+\sum_{\ell=1}^{p+s}|x_\ell|+|m|\right)\\
    \maltese_9&\stackrel{\eqref{eq:cD}}{=}\textstyle(p+1)(q+1)+\vdeg{c_1}|f|\\&\phantom{=}\textstyle+\left(\sum_{j=1}^{q}|y_{j+t}|\right)\left(|f|+\vdeg{c_2}+\sum_{\ell=1}^{p+s}|x_\ell|+|m|+\sum_{i=1}^{t}|y_i|\right).
  \end{align*}
  Comparing the signs we see that
  \[
    \Theta(c_1\bullet_{p+1}c_2)=-(-1)^{\maltese}\Theta(c_2)\bullet_{t+1}\Theta(c_1),
  \]
  with,
  \begin{align*}
    \maltese&=\vdeg{c_1}\vdeg{c_2}+\vdeg{c_1}(s+t)+\vdeg{c_2}(p+q)+(p+q)(s+t)\\
            &=\vdeg{c_1}\vdeg{c_2}+\vdeg{c_1}(\hdeg{c_2}-1)+\vdeg{c_2}(\hdeg{c_1}-1)+(\hdeg{c_1}-1)(\hdeg{c_2}-1)\\
            &=(\hdeg{c_1}+\vdeg{c_1}-1)(\hdeg{c_2}+\vdeg{c_2}-1)\\
            &=|\s c_1||\s c_2|.
  \end{align*}
  This finishes the proof.
\end{proof}

\begin{proposition}
  \label{prop:dualising-cochains}
  The morphism of bigraded vector spaces
  \begin{align*}
    \Theta\colon\RelBimHC{A}{M}\longrightarrow\RelBimHC{A}{DM},
  \end{align*}
  is a morphism of differential bigraded vector spaces. Moreover, $\Theta$ is
  strictly compatible with the brace algebra multiplication, the bimodule
  Gerstenhaber bracket and the bimodule Gerstenhaber square. In particular,
  $\Theta$ induces a morphism of shifted graded Lie algebras
  \[
    \H[\bullet,*]{\Theta}\colon\RelBimHH{A}{M}\longrightarrow\RelBimHH{A}{DM}
  \]
  with respect to the total degree that also preserves the induced Gerstenhaber
  square operation.
\end{proposition}
\begin{proof}
  We verify first that $\Theta$ preserves the brace algebra multiplication.
  Indeed,
  \begin{align*}
    \Theta(\Astr<2>[A\ltimes M])\stackrel{\eqref{eq:square-zero_product-AM}}{=}\Theta(\Astr<2>+\Astr<1,0>[M]+\Astr<0,1>[M])\stackrel{\eqref{eq:Theta}}{=}\Astr<2>+\Theta(\Astr<1,0>[M])+\Theta(m_{0,1}^{M}).
  \end{align*}
  Moreover, since $\Astr<1,0>[M]\in\BimHC*[1][0]<0>{M}$, for $x\in A$, $f\in DM$
  and $m\in M$,
  \begin{align*}
    \Theta(\Astr<1,0>[M])(f,x)(m)\stackrel{\eqref{eq:cD}}{=}f(\Astr<1,0>[M](x,m))=f(xm)\stackrel{\eqref{eq:Mn}}{=}m_{0,1}^{DM}(f,x)(m).
  \end{align*}
  Similarly, keeping in mind that $\vdeg{\Astr<0,1>[M]}=0$, for $y\in A$, $f\in
  DM$ and $m\in M$ we have
  \begin{align*}
    \Theta(\Astr<0,1>[M])(y,f)(m)&\stackrel{\eqref{eq:cD}}{=}(-1)^{|y|(|f|+|m|)}f(\Astr<0,1>[M](m,y))\\&=(-1)^{|y|(|f|+|m|)}f(my)\\&\stackrel{\eqref{eq:Mn}}{=}m_{0,1}^{DM}(y,f)(m).
  \end{align*}
  Therefore,
  \[
    \Theta(\Astr<2>[A\ltimes
    M])=\Astr<2>+m_{0,1}^{DM}+m_{1,0}^{DM}\stackrel{\eqref{eq:square-zero_product-AM}}{=}\Astr<2>[A\ltimes
    DM],
  \]
  as required.

  We continue the proof by showing that $\Theta$ is strictly compatible with the
  bimodule Gerstenhaber bracket and the bimodule Gerstenhaber square. Since
  $\Theta$ acts as the identity on $\HC{A}\subseteq\RelBimHC{A}{M}$, see
  \Cref{eq:Theta}, it is obvious that
  \begin{align*}
    \Theta([c_1,c_2])&=[\Theta(c_1),\Theta(c_2)]&c_1,c_2&\in\HC{A},\\
    \Theta(\Sq(c))&=\Sq(\Theta(c))&c&\in\HC{A}.
  \end{align*}
  We now treat the two remaining cases (recall that the Lie bracket is graded
  anti-symmetric with respect to the shifted total degree):
  \begin{itemize}
  \item For $c_1\in\BimHC*[p][q]{M}$ and $c_2\in\HC{A}$,
    \begin{align*}
      \Theta([c_1,c_2])&\stackrel{\eqref{eq:Gerstenhaber_bracket-RelBimHC}}{=}\Theta(\preLie{c_1}{c_2}-(-1)^{|\s c_1||\s c_2|}\underbrace{\preLie{c_2}{c_1}}_{=0})\\
                       &\stackrel{\eqref{eq:pre-Lie_product-RelBimHC-split}}{=}\textstyle\Theta(\sum_{i=1}^pc_1\bullet_ic_2+\sum_{j=1}^qc_1\bullet_{p+1+q}c_2)\\
                       &\stackrel{\ref{lemma:Theta}\eqref{eq:Theta-BimHC-HC}}{=}\textstyle\sum_{i=1}^p\Theta(c_1)\bullet_{q+1+i}\Theta(c_2)+\sum_{j=1}^q\Theta(c_1)\bullet_{j}\Theta(c_2)\\
                       &\stackrel{\eqref{eq:pre-Lie_product-RelBimHC-split}}{=}\preLie{\Theta(c_1)}{\Theta(c_2)}-(-1)^{|\s c_1||\s c_2|}\underbrace{\preLie{\Theta(c_2)}{\Theta(c_1)}}_{=0}\\
                       &\stackrel{\eqref{eq:Gerstenhaber_bracket-RelBimHC}}{=}[\Theta(c_1),\Theta(c_2)].
    \end{align*}
  \item For $c_1\in\BimHC*[p][q]{M}$ and $c_2\in\BimHC*[s][t]{M}$,
    \begin{align*}
      \Theta([c_1,c_2])&\stackrel{\eqref{eq:Gerstenhaber_bracket-RelBimHC}}{=}\Theta(\preLie{c_1}{c_2}-(-1)^{|\s c_1||\s c_2|}\preLie{c_2}{c_1})\\
                       &\stackrel{\eqref{eq:pre-Lie_product-RelBimHC}}{=}\Theta(c_1\bullet_{p+1}c_2-(-1)^{|\s c_1||\s c_2|}c_2\bullet_{s+1}c_1)\\
                       &\stackrel{\ref{lemma:Theta}\eqref{eq:Theta-BimHC-BimHC}}{=}-(-1)^{|\s c_1||\s c_2|}\Theta(c_2)\bullet_{t+1}\Theta(c_1)+\Theta(c_1)\bullet_{q+1}\Theta(c_2)\\
                       &\stackrel{\eqref{eq:pre-Lie_product-RelBimHC}}{=}\preLie{\Theta(c_1)}{\Theta(c_2)}-(-1)^{|\s c_1||\s c_2|}\preLie{\Theta(c_2)}{\Theta(c_1)}\\
                       &\stackrel{\eqref{eq:Gerstenhaber_bracket-RelBimHC}}{=}[\Theta(c_1),\Theta(c_2)].
    \end{align*}
  \item For $c\in\BimHC*[p][q]{M}$, when $\chark(\kk)=2$ or $|c|=|\s c|+1$ is
    even,
    \begin{align*}
      \Theta(\Sq(c))&\stackrel{\eqref{eq:Gerstenhaber_square-RelBimHC}}{=}\Theta(\preLie{c}{c})\\
                    &\stackrel{\eqref{eq:pre-Lie_product-RelBimHC}}{=}\Theta(c\bullet_{p+1}c)\\
                    &\stackrel{\ref{lemma:Theta}\eqref{eq:Theta-BimHC-BimHC}}{=}-(-1)^{|\s c|}\Theta(c)\bullet_{q+1}\Theta(c)\\
                    &=\Theta(c)\bullet_{q+1}\Theta(c)\\
                    &\stackrel{\eqref{eq:pre-Lie_product-RelBimHC}}{=}\preLie{\Theta(c)}{\Theta(c)}\\
                    &\stackrel{\eqref{eq:Gerstenhaber_square-RelBimHC}}{=}\Sq(\Theta(c)).
    \end{align*}
  \end{itemize}
 
  Finally, we prove that $\Theta$ commutes with the bimodule Hochschild
  differential. Indeed, for $c\in\RelBimHC{A}{M}$,
  \begin{align*}
    \Theta(\dRelBim(c))&\stackrel{\eqref{eq:Hd-RelBimHC}}{=}\Theta([\Astr<2>[A\ltimes M],c])\\
                       &=[\Theta(\Astr<2>[A\ltimes M]),\Theta(c)]\\
                       &=[\Astr<2>[A\ltimes DM],\Theta(c)]\\
                       &\stackrel{\eqref{eq:Hd-RelBimHC}}{=}\dRelBim<DM>(\Theta(c)).
  \end{align*}
  This finishes the proof.
\end{proof}

\begin{remark}
  \label{rmk:dualising-cochains}
  \Cref{lemma:Theta} and \Cref{prop:dualising-cochains} remain valid when $A$
  and $M$ are plain graded vector spaces, keeping in mind that the brace algebra
  multiplication, Hochschild differential and cup product are not defined in this
  case.
\end{remark}

\begin{variant}
  \label{variant:dualising-cochains}
  Let $A$, $M$ and $N$ be graded vector spaces (for example a graded algebra and
  two graded bimodules over it). For $p,q\geq0$ we introduce the morphism
  \[
    \Theta_{p,q,r}^{D}\colon\dgHom[\kk]{A^{\otimes p}\otimes M\otimes A^{\otimes
        q}}{N}[r]\longrightarrow\dgHom[\kk]{A^{\otimes p}\otimes DN\otimes
      A^{\otimes q}}{DM}[r],
  \]
  defined using \Cref{eq:cD}. Then, the obvious variants of the compatibilities
  in \Cref{lemma:Theta} hold, compare with \Cref{variant:shifting-cochains}.
\end{variant}

\subsection{Bimodules isomorphic to their vertically-shifted linear dual}
\label{interlude:shifted-duals}

Let $A$ be a graded algebra and $M$ a graded $A$-bimodule. We are interested in
the case when $M(n)\cong DM$ for some $n\in\ZZ$. We begin with a basic
observation concerning such graded bimodules. We refer the reader
to~\cite[Sec.~2.6]{Kel08}, Van den Bergh's~\cite[Appendix~A]{Boc08}
and~\cite{RRZ17} for closely related discussions concerning Serre functors on
triangulated categories in the sense of~\cite{BK89}.

\begin{proposition}
  \label{prop:pairing-MnDM}
  Suppose that there exists an isomorphism of graded $A$-bimodules
  \[
    \varphi\colon M(n)\stackrel{\sim}{\longrightarrow} DM
  \]
  for some $n\in\ZZ$. Then, the induced nondegenerate bilinear pairing
  \[
    \pairing<-,->=\pairing<-,->_\varphi\colon M(n)\times
    M\longrightarrow\kk,\qquad (\s[n]m_1,m_2)\longmapsto\varphi(\s[n]m_1)(m_2)
  \]
  satisfies the following equation for all $x,y\in A$ and $ m_1,m_2\in M$:
  \begin{align*}
    \pairing<\s[n]xm_1y,m_2>=(-1)^{|x|(|m_1|+|y|+|m_2|)}\pairing<\s[n]m_1,ym_2x>.
  \end{align*}
\end{proposition}
\begin{proof}
  Indeed,
  \begin{align*}
    \pairing<\s[n]xm_1y,m_2>&=\varphi(\s[n]xm_1y)(m_2)\\
                            &=(-1)^{n|x|}\varphi(x\s[n]m_1y)(m_2)\\\
                            &=(-1)^{n|x|+|x|(|m_1|-n+|y|+|m_2|)}\varphi(\s[n]m_1)(ym_2x)\\
                            &=(-1)^{|x|(|m_1|+|y|+|m_2|)}\pairing<\s[n]m_1,ym_2x>.\qedhere
  \end{align*}    
\end{proof}

\begin{remark}
  In \Cref{prop:pairing-MnDM}, notice that whenever $|m_1|+|m_2|\neq n$ we must
  have
  \[
    \pairing<\s[n]m_1,m_2>=\varphi(\s[n]m_1)(m_2)=0.
  \]
  Indeed, $\s[n]m_1\in M(n)^{|m_1|-n}$ and hence
  \[
    \varphi(\s[n]m_1)\in (DM)^{|m_1|-n}=D(M^{n-|m_1|}).
  \]
\end{remark}

\begin{corollary}
  \label{cor:graded-symmetric-pairing}
  Suppose that there exists an isomorphism of graded $A$-bimodules
  \[
    \varphi\colon A(n)\stackrel{\sim}{\longrightarrow} DA
  \]
  for some $n\in\ZZ$. Then, the induced nondegenerate bilinear pairing
  \[
    \pairing<-,->=\pairing<-,->_\varphi\colon A(n)\times
    A\longrightarrow\kk,\qquad (\s[n]x,y)\longmapsto\varphi(\s[n]x)(y)
  \]
  is associative,
  \begin{align*}
    \pairing<\s[n]xy,z>&=\pairing<\s[n]x,yz>&x,y,z&\in A,\intertext{and graded-symmetric with respect to the degree in $A$,}
                                                    \pairing<\s[n]x,y>&=(-1)^{|x||y|}\pairing<\s[n]y,x>&x,y&\in A.
  \end{align*}
\end{corollary}
\begin{proof}
  Both claims follow readily from \Cref{prop:pairing-MnDM}. Associativity is
  clear and to exhibit the graded-symmetry we compute
  \[
    \pairing<\s[n]x,y>=\pairing<\s[n]x\cdot
    1,y>=(-1)^{|x|(|1|+|y|)}\pairing<\s[n]1,yx>=(-1)^{|x||y|}\pairing<\s[n]y,x>.\qedhere
  \]
\end{proof}

\begin{remark}
  \label{rmk:phi-psi}
  Suppose that there exists an isomorphism of graded $A$-bimodules
  \[
    \varphi\colon M(n)\stackrel{\sim}{\longrightarrow} DM
  \]
  for some $n\in\ZZ$. The reader can easily verify, using
  \Cref{ex:DAn} or by direct computation, that the map
  \[
    \psi\colon M\stackrel{\sim}{\longrightarrow} D(M(n)),\qquad m\longmapsto
    (-1)^{n|m|}\varphi(\s[n]m)\s[-n],
  \]
  is also an isomorphism of graded $A$-bimodules. By adjunction, we obtain a
  nondegenerate bilinear pairing
  \[
    \pairing<-,->_{\psi}\colon M\times
    M(n)\longrightarrow\kk,\qquad(m_1,m_2)\longmapsto\psi(m_1)(\s[n]m_2),
  \]
  that by construction satisfies
  \[
    \pairing<m_1,\s[n]m_2>_{\psi}=(-1)^{n|m_1|}\varphi(\s[n]m_1)(m_2)=(-1)^{n|m_1|}\pairing<\s[n]m_1,m_2>_{\varphi},
  \]
  so that the values of both pairings agree whenever $n$ is even (note that the
  construction of these two pairings only needs that $M$ is a graded vector
  space). When $M=A$ is the diagonal $A$-bimodule, if $n$ is odd the new pairing
  is only associative up to a sign:
  \[
    \pairing<xy,\s[n]z>_{\psi}=(-1)^{|x|+|y|}\pairing<\s[n]xy,z>_{\varphi}\stackrel{\ref{cor:graded-symmetric-pairing}}{=}(-1)^{|x|+|y|}\pairing<\s[n]x,yz>_\varphi=(-1)^{|y|}\pairing<x,\s[n]yz>_\psi.
  \]
  Still for $n$ odd, the new pairing is anti-symmetric with respect to the total
  degree,
  \[
    \pairing<x,\s[n] y>_{\psi}+(-1)^{|\s x||\s y|}\pairing<y,\s[n]x>_{\psi}=0,
  \]
  compare with~\cite[Section~2]{Cho08}. Indeed,
  \begin{align*}
    \pairing<x,\s[n] y>_{\psi}&=(-1)^{|x|}\pairing<\s[n]x,y>_{\varphi}\\
                              &\stackrel{\ref{cor:graded-symmetric-pairing}}{=}(-1)^{|x|+|x||y|}\pairing<\s[n]y,x>_{\varphi}\\
                              &=(-1)^{|x|+|x||y|+|y|}\pairing<y,\s[n]x>_{\psi}\\
                              &=(-1)^{(|x|+1)(|y|+1)-1}\pairing<y,\s[n]x>_{\psi}\\
                              &=-(-1)^{|\s x||\s y|}\pairing<y,\s[n]x>_{\psi}.
  \end{align*}
\end{remark}

\begin{defprop}
  \label{defprop:upsilon}
  Suppose that there exists an isomorphism of graded $A$-bimodules
  \[
    \varphi\colon M(n)\stackrel{\sim}{\longrightarrow}DM
  \]
  for some $n\in\ZZ$, and consider the isomorphism
  \[
    \psi\colon M\stackrel{\sim}{\longrightarrow} D(M(n)),\qquad m\longmapsto
    m^*\coloneqq(-1)^{n|m|}\varphi(\s[n]m)\s[-n],
  \]
  defined in \Cref{rmk:phi-psi}. We introduce the morphism of differential
  bigraded vector spaces
  \[
    \begin{tikzcd}
      \RelBimHC{A}{M}\dar[dotted]{\Upsilon}\rar{\Psi}&\RelBimHC{A}{M(n)}\dar{\Theta}\\
      \RelBimHC{A}{M}&\RelBimHC{A}{D(M(n))}\lar{\Phi}
    \end{tikzcd}
  \]
  obtained as the composite of the morphisms defined by
  \Cref{eq:Phi,eq:Theta,eq:Psi} (the latter for the inverse isomorphism
  $\psi^{-1}$). Explicitly, $\Upsilon$ acts as the identity on the subspace
  $\HC{A}\subseteq\RelBimHC{A}{M}$ and, to a cochain $c\in\BimHC*[p][q]<r>{M}$
  it associates the cochain (notice the transposition $(p,q,r)\mapsto(q,p,r)$ of
  the first two components)
  \[
    \Upsilon(c)\in\BimHC*[q][p]<r>{M}
  \]
  that is uniquely determined by the requirement that the following equality
  holds for all homogeneous elements $x_1,\dots,x_p,y_1,\dots,y_q\in A$ and
  $m,m'\in M$:
  \begin{multline}
    \label{eq:Upsilon}
    \pairing<\s[n]\Upsilon(c)(y_1,\dots,y_q,m,x_1,\dots,x_p),m'>_{\varphi}\\=(-1)^{(p+1)(q+1)+\maltese}\pairing<\s[n]c(x_1,\dots,x_p,m',y_1,\dots,y_q),m>_{\varphi},
  \end{multline}
  where
  \[
    \maltese\coloneqq\textstyle(|m|+\sum_{j=1}^q|y_j|)(|m'|+\sum_{i=1}^p|x_i|).
  \]
  The morphism $\Upsilon$ is strictly compatible with the Gerstenhaber bracket and
  the Gerstenhaber square operation.
\end{defprop}
\begin{proof}
  Indeed, for $c\in\BimHC*[p][q]<r>{M}$ the cochain
  $\Upsilon(c)\in\BimHC*[q][p]<r>{M}$ is determined by the requirement that
  \begin{align*}
    \pairing<\Upsilon(c)(y_1\dots,y_q,m,x_1,\dots,x_p),\s[n]m'>_{\psi}=(c^{(n)})^D(y_1\dots,y_q,m^*,x_1,\dots,x_p)(\s[n]m').
  \end{align*}
  The left-hand side can be written instead as
  \begin{align*}
    \pairing<\Upsilon(c)(y_1\dots,y_q,m,x_1,\dots,x_p),\s[n]m'>_{\psi}=(-1)^{\maltese_0}\pairing<\s[n]\Upsilon(c)(y_1\dots,y_q,m,x_1,\dots,x_p),m'>_{\varphi},
  \end{align*}
  where
  \[
    \maltese_0=n(\vdeg{c}+\sum_{i=1}^p|x_i|+|m|+\sum_{j=1}^q|y_j|).
  \]
  The right-hand side is given by
  \begin{align*}
    (c^{(n)})^D&(y_1\dots,y_q,m^*,x_1,\dots,x_p)(\s[n]m')\\
               &=(-1)^{\maltese_1}m^*(c^{(n)}(x_1,\dots,x_p,\s[n]m',y_1,\dots,y_q))\\
               &=(-1)^{\maltese_1+\maltese_2}m^*(\s[n]c(x_1,\dots,x_p,m',y_1,\dots,y_q))\\
               &=(-1)^{\maltese_1+\maltese_2}\pairing<m,\s[n]c(x_1,\dots,x_p,m',y_1,\dots,y_q)>_{\psi}\\
               &=(-1)^{\maltese_1+\maltese_2+n|m|}\pairing<\s[n]m,c(x_1,\dots,x_p,m',y_1,\dots,y_q)>_{\varphi}\\
               &=(-1)^{\maltese_1+\maltese_2+\maltese_3+n|m|}\pairing<\s[n]c(x_1,\dots,x_p,m',y_1,\dots,y_q),m>_{\varphi},
  \end{align*}
  where
  \begin{align*}
    \maltese_1&=\textstyle(p+1)(q+1)+\vdeg{c}|m|+\left(\sum_{j=1}^q|y_j|\right)\left(|m|+\sum_{i=1}^p|x_i|-n+|m'|\right)\\
    \maltese_2&=\textstyle n\left(\vdeg{c}+\sum_{i=1}^p|x_i|\right)\\
    \maltese_3&=\textstyle|m|\left(\vdeg{c}+\sum_{i=1}^p|x_i|+|m'|+\sum_{j=1}^q|y_j|\right).
  \end{align*}
  The claim follows since, modulo $2$,
  \[
    \maltese=\maltese_0+\maltese_1+\maltese_2+\maltese_3+n|m|.
  \]
  The strict compatibility between $\Upsilon$ and the Gerstenhaber bracket and
  the Gerstenhaber square operation follows immediately from
  \Cref{rmk:RelBimHC-under-isos} and
  \Cref{prop:shifting-cochains,prop:dualising-cochains}. This finishes the
  proof.
\end{proof}

\begin{remark}
  \Cref{defprop:upsilon} remains valid when $A$ and $M$ are graded vector spaces
  and $\varphi$ is only an isomorphism of such, see also
  \Cref{rmk:shifting-cochains,rmk:dualising-cochains}.
\end{remark}

\subsection{The bimodule Hochschild cohomology of the diagonal bimodule}

Let $A$ be a graded algebra. In this article we are chiefly interested in the
bimodule Hochschild cochain complex $\RelBimHC{A}{A}$, which turns out to admit
a very natural description.

\begin{definition}
  We introduce the bigraded vector space
  \[
    \HC{A}<\varepsilon>\coloneqq\HC{A}\oplus(\HC{A}\cdot\varepsilon),
  \]
  where $\varepsilon$ is a bidegree $(1,0)$ element, so that we may identify
  \[
    \HC{A}\cdot\varepsilon=\HC{A}<-1>
  \]
  as bigraded vector spaces. We endow $\HC{A}{[\varepsilon]}$ with the unique
  bidegree $(1,0)$ differential $d$ that extends the Hochschild differential on
  $\HC{A}$ and satisfies $d(\varepsilon)=0$ and
  \[
    d(\cupp{c}{\varepsilon})=\cupp{\Hd(c)}{\varepsilon},\qquad c\in\HC{A}.
  \]
  We also extend the dg algebra structure on $\HC{A}$, given by the cup product
  defined by \Cref{eq:cup_product-HC}, to $\HC{A}<\varepsilon>$ by stipulating
  that $\varepsilon^2=0$ and that $\varepsilon$ is central with respect to the
  total degree:
  \[
    \cupp{\varepsilon}{c}=(-1)^{|c|}\cupp{c}{\varepsilon},\qquad c\in\HC{A}.
  \]
  In this way, $\HC{A}<\varepsilon>$ becomes a dg algebra with respect to the
  total degree. Moreover, by construction, the cohomology of
  $\HC{A}{[\varepsilon]}$ is the graded algebra
  \[
    \H[\bullet,*]{\HC{A}<\varepsilon>}=\HH{A}<\varepsilon>/(\varepsilon^2),\qquad
    |\varepsilon|=1
  \]
  of dual numbers with coefficients in $\HH{A}$, where $ \varepsilon$ is a
  central element.
\end{definition}

We define a morphism of bigraded vector spaces
\[
  \kappa\colon\RelBimHC{A}{A}\longrightarrow\HC{A}<\varepsilon>
\]
as follows: For $c\in\HC{A}$ we set $\kappa(c)\coloneqq c$ and, for
$c\in\BimHC*[p][q]<r>{A}$ with $n=p+1+q$ we set
\[
  \kappa(c)\coloneqq\cupp{(c\bullet_{p+1}1_A)}{\varepsilon}\in\HC[n-1]{A}\cdot\varepsilon,
\]
where $1_A\in\HC[0]<0>{A}=\Hom[\kk]{\kk}{A}$ is the unit of the graded algebra
$A$. Explicitly,
\begin{equation}
  \label{eq:kappa}
  \kappa(c)(x_1,\dots,x_{p+q})\stackrel{\eqref{eq:infinitesimal_composition}}{=}(-1)^qc_i(x_1,\dots,x_p,1_A,x_{p+1},\dots,x_{p+q}).
\end{equation}

\begin{example}
  \label{ex:kappa-id}
  For the identity morphism $\id[A]\in\BimHC*[0][0]<0>{A}$,
  \[
    \kappa(\id[A])=\cupp{1_A}{\varepsilon}.
  \]
\end{example}

\begin{remark}
  The following description of the morphism $\kappa$ is also useful: For a
  cochain
  \[
    c=(c_i)_{i=1}^n\in\BimHC[n-1]{A}=\bigoplus_{i=1}^n\BimHC*[i-1][n-i]{A}
  \]
  we have
  \begin{equation}
    \label{eq:kappa-sum}
    \kappa(c)(x_1,\dots,x_{n-1})=\sum_{i=1}^n(-1)^{n-i}c_i(x_1,\dots,x_{i-1},1_A,x_i,\dots,x_{n-1}).
  \end{equation}
\end{remark}

\begin{proposition}
  \label{prop:RelBimHC-A}
  The morphism of bigraded vector spaces
  \[
    \kappa\colon\RelBimHC{A}{A}\longrightarrow\HC{A}<\varepsilon>
  \]
  is a quasi-isomorphism of dg algebras with respect to the total degree, so
  that it induces an isomorphism of graded algebras
  \[
    \H[\bullet,*]{\kappa}\colon\RelBimHH{A}{A}\stackrel{\sim}{\longrightarrow}\HH{A}<\varepsilon>.
  \]
\end{proposition}
\begin{proof}
  We first prove that $\kappa$ is a morphism of differential bigraded vector
  spaces. Indeed, for a cochain $c\in\HC[n]<r>{A}$,
  \begin{align*}
    \dRelBim<A>(c)&\stackrel{\eqref{eq:Hd-RelBimHC-sum}}{=}[\Astr<2>,c]+[\Astr<1,0>,c]+[\Astr<0,1>,c]\\
                  &\stackrel{\eqref{eq:Hd-RelBimHC-sum-HC}}{=}\underbrace{[\Astr<2>,c]}_{\in\HC{A}}+\underbrace{\Astr<1,0>\bullet_1 c}_{\in\BimHC*[n][0]{A}}+\underbrace{\Astr<0,1>\bullet_2c}_{\in\BimHC*[0][n]{A}}
  \end{align*}
  and hence
  \begin{align*}
    \kappa(\dRelBim<A>(c))&=[m_2,c]+\underbrace{\kappa(\Astr<1,0>\bullet_1c)+\kappa(\Astr<0,1>\bullet_2c)}_{=0?}=\Hd(\kappa(c)),
  \end{align*}
  provided that $\kappa(\Astr<1,0>\bullet_1c)+\kappa(\Astr<0,1>\bullet_2c)=0$,
  which is the case since
  \begin{align*}
    \kappa(\Astr<1,0>\bullet_1c)(x_1,\dots,x_n)\stackrel{\eqref{eq:kappa}}{=}(\Astr<1,0>\bullet_1c)(x_1,\dots,x_n,1_A)
  \end{align*}
  and
  \begin{align*}
    \kappa(\Astr<0,1>\bullet_2c)(x_1,\dots,x_n)\stackrel{\eqref{eq:kappa}}{=}-(\Astr<0,1>\bullet_2c)(x_1,\dots,x_n)1_A.
  \end{align*}
  We now consider the case $c\in\BimHC*[p][q]<r>{A}$ with $p+1+q=n$. In this
  case
  \begin{align*}
    \dRelBim<A>(c)=\underbrace{\dh(c)}_{\BimHC*[p+1][q]<r>{A}}+\underbrace{\dv(c)}_{\BimHC*[p][q+1]<r>{A}},
  \end{align*}
  see \Cref{eq:dRelBim-sum}. Using \Cref{eq:kappa} while keeping in mind that
  $\dh(c)\in\BimHC*[p+1][q]{A}$, we first compute
  \begin{align*}
    \kappa(\dh(c))&(x_1,\dots,x_{n+1})\stackrel{\eqref{eq:dh}}{=}(-1)^{q}\Big(-1)^{r|x_1|+r}x_1c(x_2,\dots,x_{p+1},1_A,x_{p+2},\dots,x_{n+1})\\
                  &\phantom{=}+\sum_{i=1}^p(-1)^{i+r}c(x_1,\dots,x_ix_{i+1},\dots,x_{p+1},1_A,x_{p+2},\dots,x_{n+1})\\
                  &+(-1)^{p+1+r}c(x_1,\dots,x_p,x_{p+1}1_A,x_{p+2},\dots,x_{n+1})\Big),
  \end{align*}
  and, since $\dv(c)\in\BimHC*[p][q+1]{A}$,
  \begin{align*}
    \kappa(\dv(c))&(x_1,\dots,x_{n+1})\stackrel{\eqref{eq:dv}}{=}(-1)^{q+1}\Big((-1)^{p+1+r}c(x_1,\dots,x_p,1_Ax_{p+1},\dots,x_{n+1})\\
                  &\phantom{=}+\sum_{j=1}^q(-1)^{p+1+j+r}c(x_1,\dots,x_p,1_A,x_{p+1},\dots,x_{p+j}x_{p+j+1},\dots,x_{n+1})\\
                  &+(-1)^{p+q+r}c(x_1,\dots,x_p,1_A,x_{p+1},\dots,x_n)x_{n+1}\Big),
  \end{align*}
  so that
  \begin{align*}
    \kappa(\dRelBim<A>(c))&(x_1,\dots,x_{n+1})=(-1)^{r|x_1|+r+q}x_1c(x_2,\dots,x_{p+1},1_A,x_{p+2}\dots,x_{n+1})\\
                          &\phantom{=}+\sum_{i=1}^p(-1)^{i+r+q}c(x_1,\dots,x_ix_{i+1},\dots,x_{p+1}1_A,x_{p+2},\dots,x_{n+1})\\
                          &\phantom{=}+\sum_{j=1}^q(-1)^{p+j+r+q}c(x_1,\dots,x_p1_A,x_{p+2},\dots,x_{n+1})\\
                          &\phantom{=}(-1)^{p+1+r}c(x_1,\dots,x_p,1_A,x_{p+1},\dots,x_n)x_{n+1}.
  \end{align*}
  On the other hand, since $n=p+1+q$,
  \begin{align*}
    \Hd(\kappa(c))&(x_1,\dots,x_{n+1})\stackrel{\eqref{eq:Hd-full}}{=}(-1)^{r|x_1|+r}x_1\kappa(c)(x_2,\dots,x_{n+1})\\
                  &\phantom{=}+\sum_{i=1}^{n}(-1)^{i+r}\kappa(c)(x_1,\dots,x_ix_{i+1},\dots,x_{n+1})\\
                  &\phantom{=}+(-1)^{n+1+r}\kappa(c)(x_1,\dots,x_n)x_{n+1}\\
                  &\stackrel{\eqref{eq:kappa}}{=}(-1)^{r|x_1|+r+q}x_1c(x_2,\dots,x_{p+1},1_A,x_{p+2},\dots,x_{n+1})\\
                  &\phantom{=}+\sum_{i=1}^{p}(-1)^{i+r+q}\kappa(c)(x_1,\dots,x_ix_{i+1},\dots,x_{p+1},1_A,x_{p+2},\dots,x_{n+1})\\
                  &\phantom{=}+\sum_{j=1}^{q}(-1)^{p+j+r+q}\kappa(c)(x_1,\dots,x_p,1_A,x_{p+1},\dots,x_{p+j}x_{p+j+1}\dots,x_{n+1})\\
                  &\phantom{=}+(-1)^{p+1+r}c(x_1,\dots,x_p,1_A,x_{p+1},\dots,x_n)x_{n+1}.
  \end{align*}
  It follows that $\kappa(\dRelBim<A>(c))=\Hd(\kappa(c))$, as required.

  We now prove that $\kappa$ is compatible with the corresponding algebra
  structures. It is obvious that
  \[
    \kappa(\cupp{c_1}{c_2})=\cupp{\kappa(c_1)}{\kappa(c_2)},\qquad
    c_1,c_2\in\HC{A}
  \]
  since $\kappa$ acts as the identity on $\HC{A}\subseteq\RelBimHC{A}{A}$.
  Similarly,
  \[
    \kappa(\underbrace{\cupp{c_1}{c_2}}_{=0})=0=\cupp{\kappa(c_1)}{\kappa(c_2)},\qquad
    c_1,c_2\in\BimHC{A};
  \]
  the vanishing of the expression on the right-hand side is clear since
  $\varepsilon^2=0$, while the vanishing of the expression on the left-hand side
  follows from~\Cref{eq:BimHC-square-zero}. We now consider the remaining cases:
  \begin{itemize}
  \item If $c_1\in\HC[p]<q>{A}$ and $c_2\in\BimHC*[s][t]<u>{A}$, then
    $\cupp{c_1}{c_2}\in\BimHC*[p+s][q]<u>{A}$ and hence
    \begin{align*}
      \cupp{\kappa(c_1)}{\kappa(c_2)}&\stackrel{\eqref{eq:kappa}}{=}\cupp{c_1}{(\cupp{(c_2\bullet_{s+1}1_A)}{\varepsilon})}\\
                                     &=\cupp{(\cupp{c_1}{(c_2\bullet_{s+1}1_A)})}{\varepsilon}\\
                                     &\stackrel{\eqref{eq:cup_product-RelBimHC-HC-BimHC}}{=}\cupp{((\Astr<2>\bullet_1c_1)\bullet_{p+1}(c_2\bullet_{s+1}1_A))}{\varepsilon}\\
                                     &\stackrel{\eqref{eq:infinitesimal_composition-associative}}{=}\cupp{((\Astr<1,0>\bullet_1c_1)\bullet_{p+1}c_2)\bullet_{p+s+1}1_A)}{\varepsilon}\\
                                     &\stackrel{\eqref{eq:cup_product-RelBimHC-HC-BimHC}}{=}\cupp{((\cupp{c_1}{c_2})\bullet_{p+s+1}1_A)}{\varepsilon}\\
                                     &\stackrel{\eqref{eq:kappa}}{=}\kappa(\cupp{c_1}{c_2}),
    \end{align*}
    where in the fourth equality we use that $\Astr<1,0>=\Astr<2>$ as morphisms
    of graded vector spaces.
  \item If $c_1\in\BimHC*[p][q]<r>{A}$ and $c_2\in\HC[s]<t>{A}$, then
    $\cupp{c_1}{c_2}\in\BimHC*[p][q+s]<r+t>{A}$ and hence
    \begin{align*}
      \cupp{\kappa(c_1)}{\kappa(c_2)}&\stackrel{\eqref{eq:kappa}}{=}\cupp{(\cupp{(c_1\bullet_{p+1}1_A)}{\varepsilon})}{c_2}\\
                                     &=(-1)^{s+t}\cupp{(\cupp{(c_1\bullet_{p+1}1_A)}{c_2})}{\varepsilon}\\
                                     &\stackrel{\eqref{eq:cup_product-RelBimHC-BimHC-HC}}{=}(-1)^{s+t+p+q-1}\cupp{((\Astr<2>\bullet_1(c_1\bullet_{p+1} 1_A))\bullet_{p+q+1}c_2)}{\varepsilon}\\
                                     &\stackrel{\eqref{eq:infinitesimal_composition-associative}}{=}(-1)^{s+t+p+q-1+(s+t-1)}\cupp{(((\Astr<2>\bullet_1c_1)\bullet_{p+1+q+1}c_2)\bullet_{p+1}1_A)}{\varepsilon}\\
                                     &=(-1)^{p+q}\cupp{(((\Astr<0,1>\bullet_1c_1)\bullet_{p+1+q+1}c_2)\bullet_{p+1}1_A)}{\varepsilon}\\
                                     &\stackrel{\eqref{eq:cup_product-RelBimHC-BimHC-HC}}{=}\cupp{((\cupp{c_1}{c_2})\bullet_{p+1}1_A)}{\varepsilon}\\
                                     &\stackrel{\eqref{eq:kappa}}{=}\kappa(\cupp{c_1}{c_2});
    \end{align*}
    here, in the fourth equality we use that $1_A\in\HC[0]<0>{A}$ and in the
    fifth equality we use that $\Astr<1,0>=\Astr<2>$ as morphisms of graded
    vector spaces.
  \end{itemize}

  In order to show that $\kappa$ is a quasi-isomorphism we argue as follows.
  Recall that the bar resolution of $A$ is the chain complex of graded
  $A$-bimodules with components
  \[
    \BB{A}[n]\coloneqq A\otimes A^{\otimes n}\otimes A,\qquad n\geq0
  \]
  and the degree $-1$ differential
  \[
    d(x_0\otimes\cdots\otimes
    x_{n+1})\coloneqq\sum_{i=0}^n(-1)^ix_0\otimes\cdots\otimes
    x_ix_{i+1}\otimes\cdots x_{n+1}.
  \]
  The augmentation map
  \[
    \BB{A}[0]\longrightarrow A,\qquad x_0\otimes x_1\longmapsto x_0x_1,
  \]
  exhibits $\BB{A}$ as a graded projective resolution of the diagonal bimodule.
  The bar resolution is also equipped with the comultiplication map
  \[
    \BB{A}\longrightarrow\BB{A}\otimes_A\BB{A}
  \]
  whose components
  \[
    \BB{A}[n]\longrightarrow\bigoplus_{i=0}^n\BB{A}[i]\otimes_A\BB{A}[n-i],
  \]
  are given by
  \[
    (x_0\otimes\cdots\otimes
    x_{n+1})\longmapsto\sum_{i=0}^n(a_0\otimes\cdots\otimes a_i\otimes
    1_A)\otimes(1_A\otimes a_{i+1}\otimes\cdots\otimes a_{n+1}).
  \]
  Moreover,
  \[
    \BB{A}\otimes_A\BB{A}\cong\BB{A}\otimes_AA\otimes_A\BB{A}
  \]
  is also a graded projective resolution of the diagonal bimodule, as witnessed
  by the augmentation map
  \[
    \BB{A}[0]\otimes_AA\otimes_A\BB{A}[0]\longrightarrow A,\qquad (x_0\otimes
    x_1)\otimes y\otimes (z_0\otimes z_1)\longmapsto(x_0x_1yz_0z_1).
  \]
  Since the comultiplication map is a morphism of chain complexes of projective
  graded $A$-bimodules and is compatible with the corresponding augmentation, it
  is a chain homotopy equivalence. It follows that the induced map
  \begin{align*}
    \widetilde{\kappa}\colon\Hom[A^e]{\BB{A}\otimes_AA\otimes_A\BB{A}}{A}&\longrightarrow\Hom[A^e]{\BB{A}}{A},
  \end{align*}
  given explicitly by
  \[
    \kappa(c)(x_0,\dots,x_{n+1})=\sum_{i=0}^n
    c(x_0,\dots,x_i,1_A,x_{i+1},\dots,x_{n+1})
  \]
  for $c\in\HC[n]{A}$, is also a homotopy equivalence. Under the identifications
  \begin{align*}
    \BimHC{A}&\cong\Hom[A^e]{\BB{A}\otimes_AA\otimes_A\BB{A}}{A}
               \intertext{and}
               \HC{A}&\cong\Hom[A^e]{\BB{A}}{A}
  \end{align*}
  induced by the $\otimes$-$\operatorname{Hom}$ adjunction, the homotopy
  equivalence $\widetilde{\kappa}$ corresponds to the restriction of $\kappa[1]$
  to $\BimHC{A}\subseteq\RelBimHC{A}{A}$, and hence the latter map is also a
  homotopy equivalence. We finish the proof by observing that the following
  diagram of morphisms of differential bigraded vector spaces commutes:
  \[
    \begin{tikzcd}
      \RelBimHC{A}{A}\dar{\kappa}\rar{p}&\HC{A}\dar{\id[\HC{A}]}\rar{\delta}&\BimHC{A}\dar{\kappa[1]}\\
      \HC{A}<\varepsilon>\rar{p}&\HC{A}\rar{0}&\HC{A}
    \end{tikzcd}
  \]
  Indeed, the leftmost square commutes by the definition of $\kappa$ since the
  horizontal maps are the canonical projections. That the rightmost square
  commutes is a consequence of the fact that $\kappa$ is a morphism of
  differential bigraded vector spaces. The induced long exact sequence in
  cohomology yields the desired fact that $\kappa$ is a quasi-isomorphism,
  compare with \Cref{prop:RelBimHC-all}.
\end{proof}

\begin{corollary}
  \label{coro:kappa}
  The isomorphism of graded algebras
  \[
    \H[\bullet,*]{\kappa}\colon\RelBimHH{A}{A}\stackrel{\sim}{\longrightarrow}\HH{A}<\varepsilon>/(\varepsilon^2)
  \]
  induces a shifted Lie bracket and a Gerstenhaber square operation on
  $\HH{A}<\varepsilon>$, and hence an isomorphic Gerstenhaber algebra structure.
  Moreover, the following statements hold:
  \begin{enumerate}
  \item The induced shifted Lie bracket extends the Gerstenhaber bracket on the
    shifted graded Lie subalgebra $\HH{A}\subseteq\HC{A}<\varepsilon>$ and is
    uniquely determined by the following additional requirements:
    \begin{align*} [\varepsilon,\varepsilon]&=0&[c,\varepsilon]&=0&c&\in\HH{A}.
    \end{align*}
    From these it follows that the following equations must hold:
    \begin{align*}
      [c_1,\cupp{c_2}{\varepsilon}]&=\cupp{[c_1,c_2]}{\varepsilon}&c_1,c_2&\in\HH{A},\\
      [\cupp{c_1}{\varepsilon},\cupp{c_2}{\varepsilon}]&=0&c_1,c_2&\in\HH{A}.
    \end{align*}
  \item If $\chark(\kk)=2$, then the induced Gerstenhaber square on
    $\HH{A}<\varepsilon>$ extends the Gerstenhaber square on $\HH{A}$ and is
    more generally given by the formula
    \[
      \Sq(c_1+\cupp{c_2}{\varepsilon})=\Sq(c_1)+\cupp{(c_2^2+[c_1,c_2])}{\varepsilon},\qquad
      c_1,c_2\in\HH{A}.
    \]
    If $\chark(\kk)\neq2$, then the Gerstenhaber square is defined in even
    degrees as in~\Cref{rmk:Gerstenhaber_square-char}.
  \end{enumerate}
\end{corollary}
\begin{proof}
  Since $\kappa$ acts as the identity on the shifted graded Lie subalgebra
  $\HH{A}\subseteq\RelBimHH{A}{A}$, where the bimodule Gerstenhaber bracket is
  given by the usual Gerstenhaber bracket, we see that the transported shifted
  Lie bracket must also extend the Gerstenhaber bracket on
  $\HH{A}\subseteq\HH{A}<\varepsilon>$. In order to establish the remaining
  constraints on the transported shifted Lie bracket, recall from
  \Cref{ex:kappa-id} that $\kappa(\id[A])=\cupp{1_A}{\varepsilon}$. Hence,
  \[
    [\varepsilon,\varepsilon]=[\kappa(\id[A]),\kappa(\id[A])]=\kappa([\id[A],\id[A]])=0,
  \]
  since $[\id[A],\id[A]]=0$, for $\id[A]\in\BimHC*[0][0]<0>{A}$ has odd degree
  $|\id[A]|=1$ given by its arity. Similarly, for all $c\in\HH{A}$ we have
  $[c,\id[A]]=0$ (see \Cref{ex:bracket-1M}), and hence
  \[
    [c,\varepsilon]=[c,\cupp{1_A}{\varepsilon}]=[\kappa(c),\kappa(\id[A])]=\kappa([c,\id[A]])=0,\qquad
    c_1,c_2\in\HH{A}.
  \]
  Consequently, by the Gerstenhaber relation,
  \begin{align*}
    [c_1,\cupp{c_2}{\varepsilon}]&\stackrel{\eqref{eq:Gerstenhaber_relation}}{=}\cupp{[c_1,c_2]}{\varepsilon}\pm\underbrace{\cupp{c_2}{[c_1,\varepsilon]}}_{=0}=\cupp{[c_1,c_2]}{\varepsilon}.
  \end{align*}
  Similarly,
  \begin{align*}
    [\cupp{c_1}{\varepsilon},\cupp{c_2}{\varepsilon}]&\stackrel{\eqref{eq:Gerstenhaber_relation}}{=}\cupp{[\cupp{c_1}{\varepsilon},c_2]}{\varepsilon}\pm\cupp{c_2}{[\cupp{c_1}{\varepsilon},\varepsilon]}\\
                                                     &\stackrel{\eqref{eq:Gerstenhaber_relation}}{=}\underbrace{\cupp{(\pm\cupp{[c_1,c_2]}{\varepsilon})}{\varepsilon}}_{=0}\pm\cupp{c_2}{(\cupp{\underbrace{[\varepsilon,c_1]}_{=0}}{\varepsilon}\pm\cupp{c_1}{\underbrace{[\varepsilon,\varepsilon]}_{=0}})}=0.
  \end{align*}

  When $\chark(\kk)=2$, the transported Gerstenhaber square on
  $\HH{A}{[\varepsilon]}/(\varepsilon^2)$ must extend that of $\HH{A}$ since $\kappa$
  acts as the identity on $\HH{A}\subseteq\RelBimHH{A}{A}$. Moreover, since in
  $\RelBimHH{A}{A}$ we have
  \[
    \Sq(\id[A])=\preLie{\id[A]}{\id[A]}=\id[A]\bullet_1\id[A]=\id[A],
  \]
  we must have
  \[
    \Sq(\varepsilon)=\Sq(\kappa(\id[A]))=\kappa(\Sq(\id[A]))=\kappa(\id[A])=1_A\cdot\varepsilon=\varepsilon.
  \]
  Finally, the relations satisfied by the Gerstenhaber square in
  $\RelBimHH{A}{A}$ yield, for $c\in\HC{A}$,
  \[
    \Sq(\cupp{c}{\varepsilon})=\Sq(c)\cdot\underbrace{\varepsilon^2}_{=0}+\cupp{c}{\cupp{\underbrace{[c,\varepsilon]}_{=0}}{\varepsilon}}+\cupp{c^2}{\Sq(\varepsilon)}=\cupp{c^2}{\varepsilon}
  \]
  and then also, for $c_1,c_2\in\HC{A}$,
  \begin{align*}
    \Sq(c_1+\cupp{c_2}{\varepsilon})&=\Sq(c_1)+\Sq(\cupp{c_1}{\varepsilon})+[c_1,\cupp{c_2}{\varepsilon}]\\&=\Sq(c_1)+\cupp{c_2^2}{\varepsilon}+\cupp{[c_1,c_2]}{\varepsilon}\\
                                    &=\Sq(c_1)+\cupp{(c_2^2+[c_1,c_2])}{\varepsilon}.
  \end{align*}
  This finishes the proof.
\end{proof}

\begin{remark}
  \label{rmk:BV-motivation}
  We record the following formula for later use. Suppose that there exists an
  isomorphism of graded $A$-bimodules
  \[
    \varphi\colon A(n)\stackrel{\sim}{\longrightarrow}DA
  \]
  for some $n\in\ZZ$. Given a cochain
  \[
    c=(c_i)_{i=1}^m\in\BimHC[m-1]{A}=\bigoplus_{i=1}^m\BimHC*[i-1][m-i]{A},
  \]
  we have
  \[
    \Upsilon(c)_i=\Upsilon(c_{m-i+1})\in\BimHC*[i-1][m-i]{A}
  \]
  and hence
  \[
    \kappa(\Upsilon(c))(x_1,\dots,x_{m-1})\stackrel{\eqref{eq:kappa-sum}}{=}\sum_{i=1}^m(-1)^{m-i}\kappa(\Upsilon(c)_i)(x_1,\dots,x_{i-1},1_A,x_i,\dots,x_{m-1})
  \]
  Consequently, $\kappa(\Upsilon(c))$ is uniquely determined by the requirement
  that, for all $x_1,\dots,x_m\in A$,
  \begin{align*}
    \pairing<\s[n]&\kappa(\Upsilon(c))(x_1,\dots,x_{m-1}),x_m>_{\varphi}\\
                  &=\sum_{i=1}^m(-1)^{m-i}\pairing<\s[n]\Upsilon(c)_i(x_1,\dots,x_{i-1},1_A,x_i,\dots,x_{m-1}),x_m>_\varphi\\
                                     &\stackrel{\eqref{eq:Upsilon}}{=}(-1)^m\sum_{i=1}^m(-1)^{i(m-1)+\maltese}\pairing<\s[n]c_{m-i+1}(x_i,\dots,x_{m-1},x_m,x_1,\dots,x_{i-1}),1_A>_\varphi.
  \end{align*}
  where
  \[
    \maltese=\textstyle\left(\sum_{j=1}^{i-1}|x_j|\right)\Big(\sum_{k=i}^m|x_k|\Big).
  \]
\end{remark}

\Cref{rmk:BV-motivation} motivates the introduction of the following operator in
our context. We shall return to it in later sections.

\begin{definition}
  \label{def:BV}
  Suppose that there exists an isomorphism of graded $A$-bimodules
  \[
    \varphi\colon A(n)\stackrel{\sim}{\longrightarrow}DA
  \]
  for some $n\in\ZZ$. The \emph{Batalin--Vilkovisky (BV) operator} is the
  bidegree $(-1,0)$ operator
  \[
    \BV=\BV_\varphi\colon\HC{A}\longrightarrow\HC[\bullet-1]{A}
  \]
  uniquely characterised by the property that, for a cochain $c\in\HC[m]<r>{A}$,
  the cochain $\BV(c)\in\HC[m-1]<r>{A}$ satisfies
  \begin{multline*}
    \pairing<\s[n]\BV(c)(x_1,\dots,x_{m-1}),x_m>_{\varphi}=\\\textstyle(-1)^{m+r+1}\sum_{i=1}^{m}(-1)^{i(m-1)+\maltese_i}\pairing<\s[n]c(x_i,\dots,x_m,x_1,\dots,x_{i-1}),1_A>_{\varphi},
  \end{multline*}
  for all $x_1,\dots,x_m\in A$, where
  \[
    \maltese_i=\textstyle\left(\sum_{j=1}^{i-1}|x_j|\right)\Big(\sum_{k=i}^m|x_k|\Big).
  \]
\end{definition}

\begin{remark}
  \label{rmk:BV-under-algebra-isos}
  Suppose that there exists an isomorphism of graded $A$-bimodules
  \[
    \varphi\colon A(n)\stackrel{\sim}{\longrightarrow}DA
  \]
  for some $n\in\ZZ$. Given an isomorphism of graded algebras
  \[
    \gamma\colon B\stackrel{\sim}{\longrightarrow}A,
  \]
  we obtain an induced isomorphism of graded $B$-bimodules
  \[
    \psi\colon B(n)\stackrel{\sim}{\longrightarrow}DB,\qquad \s[n]b\longmapsto\varphi(\s[n]\gamma(b))(\gamma(-))
  \]
  whose associated pairing is therefore
  \[
    \pairing<\s[n]b_1,b_2>_\psi=\pairing<\s[n]\gamma(b_1),\gamma(b_2)>_\varphi,\qquad b_1,b_2\in B.
  \]
  It is then straightforward to verify that the following diagram commutes:
  \[
    \begin{tikzcd}
      \HC{B}[B]\dar[swap]{\BV_\psi}\rar{(-)^\gamma}&\HC{A}[A]\dar{\BV_\varphi}\\
      \HC[\bullet-1]{B}[B]\rar{(-)^\gamma}&\HC[\bullet-1]{A}[A],
    \end{tikzcd}
  \]
  where the map $c\mapsto c^\gamma$ is defined in \Cref{rmk:functoriality-isos-HC}.
\end{remark}

\begin{theorem}[{\cite[Theorem~18]{Men09}}]
The Batalin--Vilkovisky operator introduced in \Cref{def:BV} satisfies the identity 
  \[
    [x,y]=(-1)^{|x|}(\BV(x\cdot y)-\BV(x)\cdot
    y-(-1)^{|x|}x\cdot\Delta(y)),\qquad x,y\in\HH{A}[A].
  \]
\end{theorem}

Menichi~\cite{Men09} works with differential graded algebras, not just graded algebras. 
The latter can be regarded as differential graded algebras with trivial differential, 
so the results in \cite{Men09} apply. 
In the graded case, Hochschild cohomology is bigraded, see \Cref{def:HHA},
and the Batalin--Vilkovisky operator (\Cref{def:BV}) has bidegree $(1,0)$.
Menichi~\cite{Men09} defines the Batalin--Vilkovisky operator in the Hochschild cohomology
of a differential graded algebra by dualising Connes' boundary operator 
on the (normalised) Hochschild chain complex, but he does not give a formula for it.
Below we show that \Cref{def:BV} is the outcome of such a dualisation in the graded algebra case.

\begin{definition}
  The \emph{Hochschild chain complex} of $A$ is the differential bigraded vector
  space with the components
  \[
    \HoHC[m]<r>{A}[A]=(A\otimes A^{\otimes m})^r,\qquad m\geq0,\quad r\in\ZZ,
  \]
  endowed with the differential of homological bidegree $(-1,0)$
  \begin{equation}
    \label{eq:HoHd}
    \HoHd\colon\HoHC[m]{A}[A]\longrightarrow\HoHC[m-1]{A}[A]
  \end{equation}
  given by
  \begin{align*}
    \HoHd(x_0\otimes x_1\otimes\cdots\otimes x_m)&=x_0x_1\otimes\cdots\otimes x_m\\
                                                 &\phantom{=}+\sum_{i=1}^{m-1}(-1)^ix_0\otimes x_1\otimes\cdots\otimes x_ix_{i+1}\otimes\cdots\otimes x_m\\
                                                 &\phantom{=}+(-1)^{m+\maltese_1}x_mx_0\otimes x_1\otimes\cdots\otimes x_{m-1},\notag
  \end{align*}
  where $x_0,x_1,\dots,x_m\in A$ and
  \[
    \textstyle\maltese_1\coloneqq|x_m|\left(\sum_{i=0}^{m-1}|x_i|\right)
  \]
  Its homology $\HoHH{A}[A]$ is the \emph{Hochschild homology} of $A$.
\end{definition}

Suppose that there exists an isomorphism of graded $A$-bimodules
\[
  \varphi\colon A(n)\stackrel{\sim}{\longrightarrow}DA
\]
for some $n\in\ZZ$, which we fix until the end of this section. We introduce the
isomorphisms of vector spaces
\begin{equation}
  \label{eq:Xi}
  \Xi=\Xi_\varphi\colon\HC[m]<r>{A}[A]\stackrel{\sim}{\longrightarrow}D(\HoHC[m]<r>{A}[A](n)),\qquad
  m\in\ZZ,\quad r\in\ZZ,
\end{equation}
where, for $c\in\HC[m]<r>{A}[A]$ and $x_0,x_1,\dots,x_m\in A$,
\[
  \Xi(c)(\s[n](x_0\otimes x_1\cdots\otimes
  x_m)\coloneqq(-1)^{nr+m+\maltese_2}\pairing<\s[n]c(x_1,\dots,x_m),x_0>_\varphi,
\]
and
\[
  \textstyle\maltese_2\coloneqq|x_0|\left(\sum_{i=1}^m|x_i|\right).
\]

\begin{proposition}
  \label{prop:Xi-iso}
  The isomorphisms of vector spaces \eqref{eq:Xi} assemble into an isomorphism
  of differential bigraded vector spaces
  \[
    \Xi=\Xi_\varphi\colon\HC{A}[A]\stackrel{\sim}{\longrightarrow}D(\HoHC{A}[A](n)).
  \]
\end{proposition}
\begin{proof}
  Indeed, for $c\in\HC[m]<r>{A}[A]$ and $x_0,x_1,\dots,x_m,x_{m+1}\in A$,
  \begin{align*}
    \Xi(\Hd(c))&(\s[n](x_0\otimes x_1\otimes\cdots\otimes x_{m+1}))\\
               &\stackrel{\eqref{eq:Xi}}{=}(-1)^{nr+m+1+\maltese_1}\pairing<\s[n]\Hd(c)(x_1,\dots,x_m,x_{m+1}),x_0>_\varphi\\
               &\stackrel{\eqref{eq:Hd-full}}{=}(-1)^{nr+m+1+\maltese_1}\Big((-1)^{r|x_1|+r}\pairing<\s[n]x_1c(x_2,\dots,x_m,x_{m+1}),x_0>_\varphi\\
               &\phantom{=}+\sum_{i=1}^m(-1)^{i+r}\pairing<\s[n]c(x_1,\dots,x_ix_{i+1},\dots,x_{m+1}),x_0>_\varphi\\
               &\phantom{=}+(-1)^{m+1+r}\pairing<\s[n]c(x_1,\dots,x_m)x_{m+1},x_0>_\varphi\Big)\\
               &=(-1)^{nr+m+1+r+\maltese_1}\Big((-1)^{r|x_1|}\pairing<\s[n]x_1c(x_2,\dots,x_m,x_{m+1}),x_0>_\varphi\\
               &\phantom{=}+\sum_{i=1}^m(-1)^{i}\pairing<\s[n]c(x_1,\dots,x_ix_{i+1},\dots,x_{m+1}),x_0>_\varphi\\
               &\phantom{=}+(-1)^{m+1}\pairing<\s[n]c(x_1,\dots,x_m)x_{m+1},x_0>_\varphi\Big),
  \end{align*}
  where
  \[
    \textstyle\maltese_1=|x_0|\left(\sum_{i=1}^{m+1}|x_i|\right).
  \]
  On the other hand, keeping in mind that $\Xi(c)$ is a functional of degree
  $r-n$ and using the formula
  \[
    b(n)^*(\Xi(c))\s[n]=(-1)^{r+1}\Xi(c)\s[n]b,
  \]
  we have
  \begin{align*}
    b(n)^*(\Xi(c))&(\s[n](x_0\otimes x_1\otimes\cdots\otimes x_{m+1}))\\
                  &=(-1)^{r+1}\Xi(c)(\s[n]b(x_0\otimes x_1\otimes\cdots\otimes x_{m+1}))\\
                  &\stackrel{\eqref{eq:HoHd}}{=}(-1)^{r+1}\Big(\Xi(c)(\s[n](x_0x_1\otimes\cdots\otimes x_{m+1}))\\
                  &\phantom{=}+\sum_{i=1}^m(-1)^i\Xi(c)(\s[n](x_0\otimes x_1\otimes\cdots\otimes x_ix_{i+1}\otimes\cdots\otimes x_{m+1}))\\
                  &\phantom{=}+(-1)^{m+1+\maltese_2}\Xi(c)(\s[n](x_{m+1}x_0\otimes x_1\otimes\cdots\otimes x_m))\Big)\\
                  &\stackrel{\eqref{eq:Xi}}{=}(-1)^{r+1+nr+m}\Big((-1)^{\maltese_3}\pairing<\s[n]c(x_2,\dots,x_{m+1}),x_0x_1>_\varphi
    \\&\phantom{=}+\sum_{i=1}^m(-1)^{i+\maltese_1}\pairing<\s[n] c(x_1,\dots,x_ix_{i+1},\dots,x_{m+1}),x_0>_\varphi\\
                  &\phantom{=}+(-1)^{m+1+\maltese_2+\maltese_4}\pairing<\s[n]c(x_1\otimes\cdots\otimes x_m),x_{m+1}x_0>_\varphi\Big)\\
                  &\stackrel{\eqref{cor:graded-symmetric-pairing}}{=}(-1)^{r+1+nr+m}\Big((-1)^{\maltese_3+\maltese_5}\pairing<\s[n]x_1c(x_2,\dots,x_{m+1}),x_0>_\varphi
    \\&\phantom{=}+\sum_{i=1}^m(-1)^{i+\maltese_1}\pairing<\s[n] c(x_1,\dots,x_ix_{i+1},\dots,x_{m+1}),x_0>_\varphi\\
                  &\phantom{=}+(-1)^{m+1+\maltese_2+\maltese_4}\pairing<\s[n]c(x_1\otimes\cdots\otimes x_m),x_{m+1}x_0>_\varphi\Big),
  \end{align*}
  where
  \begin{align*}
    \maltese_2&\textstyle=|x_{m+1}|\left(\sum_{j=0}^m|x_j|\right),&\qquad\maltese_3&\textstyle=(|x_0|+|x_1|)\left(\sum_{j=2}^{m+1}|x_j|\right),\\
    \maltese_4&\textstyle=(|x_{m+1}|+|x_0|)\left(\sum_{j=1}^m|x_j|\right),&\maltese_5&\textstyle=|x_1|\left(r+|x_0|+\sum_{j=2}^{m+2}|x_j|\right).
  \end{align*}
  The claim follows since, modulo $2$,
  \[
    \maltese_2+\maltese_4=\maltese_1\qquad\text{and}\qquad\maltese_3+\maltese_5=r|x_1|+\maltese_1.\qedhere
  \]
\end{proof}

\begin{definition}
  We let $\nHC{A}[A]\subseteq\HC{A}[A]$ be the subcomplex of \emph{normalised
    Hochschild cochains}, that is cochains $c\in\HC{A}[A]$ that vanish on all
  inputs containing the unit element of $A$:
  \[
    c(\dots,1_A,\dots)=0.
  \]
  It is well known that the canonical inclusion map
  $\nHC{A}[A]\hookrightarrow\HC{A}[A]$ is a quasi-isomorphism, see for
  example~\cite[\S1.5.7]{Lod98}. Similarly, we let $\overline{A}\coloneqq
  A/\kk1_A$ and consider the \emph{normalised Hochschild chain complex}
  \[
    \nHoHC[m]<r>{A}[A]\coloneqq (A\otimes\overline{A}^{\otimes m})^r,\qquad
    m\geq0,\quad r\in\ZZ,
  \]
  which is the quotient of $\HoHC{A}[A]$ by its subcomplex spanned by all
  elementary tensors of the form
  \[
    x\otimes x_1\otimes\cdots\otimes 1_A\otimes\cdots x_m,
  \]
  that is such that at least one of $x_1,\dots,x_m$ is the unit element of $A$.
  The canonical quotient map $\HoHC{A}[A]\twoheadrightarrow\nHoHC{A}[A]$ is a
  quasi-isomorphism, see for example~\cite[Proposition~1.1.15]{Lod98}.
\end{definition}

\begin{remark}
  \label{rmk:Xi-normalised}
  The isomorphism in \Cref{prop:Xi-iso} induces an isomorphism of differential
  bigraded vector spaces
  \[
    \Xi=\Xi_\varphi\colon\nHC{A}[A]\stackrel{\sim}{\longrightarrow}D(\nHoHC{A}[A](n)).
  \]
\end{remark}

Recall that \emph{Connes' (normalised) boundary operator} is the homological
bidegree (1,0) operator
\begin{equation}
  \label{eq:B}
  B\colon\nHoHC[m-1]{A}[A]\longrightarrow\nHoHC[m]{A}[A]
\end{equation}
given by
\[
  B(x_1\otimes\cdots\otimes
  x_m)\coloneqq\sum_{i=1}^m(-1)^{i(m-1)+\maltese_i}1_A\otimes
  x_i\otimes\cdots\otimes x_m\otimes x_1\otimes\cdots x_{i-1},
\]
where $x_1,\dots,x_m\in A$ and
\[
  \textstyle\maltese_i\coloneqq\left(\sum_{j=1}^{i-1}|x_j|\right)\Big(\sum_{k=i}^m|x_k|\Big).
\]
It satisfies $B^2=0$ and $Bb+bB=0$, see for example~\cite[Section~2.1]{Lod98}.

\begin{proposition}
  The following diagram commutes:
  \[
    \begin{tikzcd}
      \HC{A}[A]\rar{\Xi}\dar{\BV}&D(\HoHC[m]{A}[A](n))\dar{B(n)^*}\\
      \HC[\bullet-1]<r>{A}[A]\rar{\Xi}&D(\HoHC[\bullet-1]{A}[A](n)).
    \end{tikzcd}
  \]
  Consequently, $\BV^2=0$ and $\BV\Hd+\Hd\BV=0$ on the (quasi-isomorphic)
  subcomplex $\nHC{A}[A]\subseteq\HC{A}[A]$ of normalised Hochschild cochains.
  In particular, the Batalin--Vilkovisky operator descends to a bidegree
  $(-1,0)$ differential
  \[
    \BV\colon\HH{A}[A]\longrightarrow\HH[\bullet-1]{A}[A]
  \]
  that participates in a commutative diagram of the form
  \[
    \begin{tikzcd}
      \HH{A}[A]\rar{\sim}\dar{\BV}&D(\HoHH[m]{A}[A](n))\dar{B(n)^*}\\
      \HH[\bullet-1]<r>{A}[A]\rar{\sim}&D(\HoHH[\bullet-1]{A}[A](n)).
    \end{tikzcd}
  \]
\end{proposition}
\begin{proof}
  Indeed, for $c\in\HC[m]<r>{A}[A]$ and $x_1,\dots,x_m\in A$,
  \begin{align*}
    \Xi(\BV(c))&(\s[n](x_m\otimes x_1\otimes\cdots\otimes x_{m-1}))\\
               &\stackrel{\eqref{eq:Xi}}{=}(-1)^{nr+m-1+\maltese_1}\pairing<\s[n]\BV(c)(x_1\otimes\cdots\otimes x_{m-1}),x_m>_\varphi\\
               &\stackrel{\ref{def:BV}}{=}(-1)^{nr+m-1+\maltese_1+m+r+1}\\
               &\phantom{====}\sum_{i=1}^m(-1)^{i(m-1)+\maltese_i}\pairing<\s[n]c(x_i,\dots,x_m,x_1,\dots,x_{i-1}),1_A>_\varphi,
  \end{align*}
  where
  \[
    \textstyle\maltese=|x_m|\left(\sum_{j=1}^{m-1}|x_j|\right)\qquad\text{and}\qquad\maltese_i=\textstyle\left(\sum_{j=1}^{i-1}|x_j|\right)\Big(\sum_{k=i}^m|x_k|\Big).
  \]
  On the other hand, since $\Xi(c)$ is a functional of degree $r-n$, and keeping
  the formula
  \[
    B(n)^*(\Xi(c))\s[n]=(-1)^{|\Xi(c)|+n}\Xi(c)\s[n]B=(-1)^{m+r}\Xi(c)\s[n]B
  \]
  in mind, we have
  \begin{align*}
    B(n)^*(\Xi(c))&(\s[n](x_m\otimes x_1\otimes\cdots\otimes x_{m-1}))\\
                  &=(-1)^{m+r+\maltese}\Xi(c)(\s[n]B(x_1\otimes\cdots\otimes x_{m-1}\otimes x_m))\\
                  &\stackrel{\eqref{eq:B}}{=}(-1)^{m+r}\Big(\sum_{i=1}^m(-1)^{i(m-1)+\maltese_i}\\
                  &\phantom{====}\Xi(c)(\s[n](1_A\otimes x_i\otimes\cdots\otimes x_m\otimes x_1\otimes x_{i-1}))\Big)\\
                  &\stackrel{\eqref{eq:Xi}}{=}(-1)^{m+r+nr+m+\maltese}\Big(\sum_{i=1}^m(-1)^{i(m-1)+\maltese_i}\\
                  &\phantom{====}\pairing<\s[n]c(x_i,\dots,x_m,x_1,\dots,x_{i-1}),1_A>_\varphi\Big).
  \end{align*}
  Therefore, $\Xi\circ\BV=B(n)^*\circ\Xi$, as required. The remaining claims
  follow by transporting the identities $B^2=0$ and $Bb+bB=0$ along $\Xi$, which
  hold on $\nHoHC{A}[A]$, see also \Cref{rmk:Xi-normalised}.
\end{proof}

\begin{remark}
  In this article, we are exclusively interested in applying the
  Batalin--Vilkovisky operator to normalised Hochschild cochains that are part
  a strictly unital minimal $A_\infty$-structure on a suitable graded algebra,
  see the proof of \Cref{prop:CY-Ai-aux}.
\end{remark}

\section{\texorpdfstring{$A_\infty$-algebras and $A_\infty$-bimodules}{A-infinity
  algebras and A-infinity bimodules}}
\label{sec:Ai-stuff}

In this section we first review several important aspects concerning
$A_\infty$-algebras and $A_\infty$-bimodules that are needed to prove
\Cref{thm:CY-Kadeikshvili} and its variant for minimal $A_\infty$-algebras
(\Cref{thm:CY-Kadeishvili-Ai}).

\subsection{Prelude: More on infinitesimal compositions}

The following considerations concerning infinitesimal compositions are useful
for dealing with $A_\infty$-equations in a convenient way.

\begin{example}
  \label{ex:partial-via-braces}
  Let $(V,d_V)$ and $(W,d_W)$ be dg vector spaces. Then,
  \[
    \partial_{V^{\otimes
        n},W}(c)=\braces{d_W}{c}+(-1)^{n+r}\braces{c}{d_V},\qquad
    c\in\dgHom[\kk]{V^{\otimes n}}{W}[r].
  \]
  Indeed,
  \begin{align*}
    \braces{d_W}{c}+\braces{c}{d_V}&=d_W\bullet_1 c+(-1)^{n+r}\sum_{i=1}^nc\bullet_id_V\\
                                   &\stackrel{\eqref{eq:infinitesimal_composition}}{=}d_W\circ c+(-1)^{n+r}\sum_{i=1}^n(-1)^{n-1}c\circ(\id[V]^{\otimes(i-1)}\otimes d_V\otimes\id[V]^{\otimes(n-i)})\\
                                   &=d_W\circ c-(-1)^r c\circ\left(\sum_{i=1}^n\id[V]^{\otimes(i-1)}\otimes d_V\otimes\id[V]^{\otimes(n-i)}\right)\\
                                   &\stackrel{\eqref{eq:partial}}{=}d_W\circ c-(-1)^rc\circ d_{V^{\otimes n}}.
  \end{align*}
  In particular, if the total degree $n+r$ of $c$ is even, then
  \[
    \partial_{V^{\otimes n},W}(c)=\braces{d_W}{c}+\braces{c}{d_V}.
  \]
\end{example}

We introduce the following variant of the operation defined by
\Cref{eq:binary_brace}.

\begin{notation}
  \label{notation:completed-brace}
  Suppose given countable tuples
  \[
    c^{(1)}=(c^{(1)}_{1},c^{(1)}_{2},c^{(1)}_{3},\dots,c^{(1)}_{n},\dots)\qquad\text{and}\qquad
    c^{(2)}=(c^{(2)}_{1},c^{(2)}_{2},c^{(2)}_{4},\dots,c^{(2)}_{n},\dots),
  \]
  where
  \[
    c^{(1)}_{n}\colon X_{1,n}\otimes\cdots\otimes X_{n,n}\longrightarrow X_n
  \]
  and
  \[
    c^{(2)}_{n}\colon Y_{1,n}\otimes\cdots\otimes Y_{n,n}\longrightarrow Y_n
  \]
  are homogeneous morphisms of graded vector spaces. We define
  $\braces{c_1}{c_2}$ to be the tuple with the components
  \begin{equation}
    \label{eq:binary_brace-pq}
    \sum_{p+q-1=n}\braces{c_p^{(1)}}{c_q^{(2)}}=\sum_{p+q-1=n}\sum_{i=1}^pc_p^{(1)}\bullet_i c_q^{(2)},\qquad n\geq1,
  \end{equation}
  where the first sum ranges over all ordered decompositions $p+q-1=n$ with
  $p,q\geq1$ and the second sum ranges over all ordered decompositions $r+s+t=n$
  with $r,t\geq0$ and $s\geq1$. Notice that
  \begin{equation}
    \label{eq:binary_brace-rst}
    \sum_{p+q-1=n}\sum_{i=1}^pc_p^{(1)}\bullet_i
    c_q^{(2)}=\sum_{r+s+t=n}c^{(1)}_{r+1+t}\bullet_{r+1}c^{(2)}_{s},\qquad n\geq1,
  \end{equation}
  which is another convenient expression for $\braces{c_1^{(1)}}{c_2^{(2)}}$.
\end{notation}

\subsection{$A_\infty$-algebras and $A_\infty$-morphisms between them}

We recall the definition of an $A_\infty$-algebra, for which we adopt the sign
conventions in~\cite{Lef03}. We also recommend the articles~\cite{Kel01,Kel02a},
but warn the reader that the sign conventions used therein do not agree with
those in \cite{Lef03}. Finally, we remind the reader of \Cref{rmk:HC-gvsp} and
\Cref{notation:completed-brace}.

\begin{definition}
  An \emph{$A_\infty$-algebra} is a pair $(A,\Astr)$ consisting of a graded
  vector space $A$ and a cochain
  \[
    \Astr=(\Astr<1>,\Astr<2>,\Astr<3>,\dots,\Astr<n>,\dots)\in\prod_{n\geq1}\HC[n]<2-n>{A},
  \]
  called an \emph{$A_\infty$-algebra structure}, such that the
  $A_\infty$-equation is satisfied:
  \[
    \braces{\Astr}{\Astr}=0.
  \]
  We say that an $A_\infty$-algebra $(A,\Astr)$ is \emph{minimal} if $m_1=0$.
\end{definition}

\begin{remark}
  \label{rmk:strict_unitality}
  There are various notions of unitality for $A_\infty$-algebras,
  see~\cite[Section~I.2]{Sei08} for a discussion of this issue and its
  resolution. The notion of unitality that is most relevant for us is that of
  \emph{strict unitality}, that is the existence of a degree $0$ cocycle $1$
  which is a unit for the binary operation and such that
  \[
    \Astr<n>[](\dots,1,\dots)=0,\qquad n\geq3.
  \]
  We shall remind the reader of this assumption whenever it becomes relevant in
  \Cref{prop:CY-Ai-aux}, \Cref{thm:CY-Kadeishvili-Ai} and \Cref{rmk:BimHMC-HMC}.
\end{remark}

We now discuss alternative presentations of the $A_\infty$-equation(s). Our
motivation for this is two-fold: On the one hand, we wish to compare our sign
conventions with those in the literature; on the other hand, different
presentations of these equations suggest different interpretations that are
relevant for us in the sequel.

\begin{remark}
  Let $(A,\Astr)$ be an $A_\infty$-algebra. According to
  \Cref{eq:binary_brace-rst}, the components of the $A_\infty$-equation
  $\braces{\Astr}{\Astr}=0$ are given by
  \begin{multline}
    \label{eq:Lef03}
    0=\sum_{r+s+t=n}\Astr<r+1+t>\bullet_{r+1}\Astr<s>\\\stackrel{\eqref{eq:infinitesimal_composition}}{=}\sum_{r+s+t=n}(-1)^{rs+t}\Astr<r+1+t>\circ(\id[A]^{\otimes
      r}\otimes\Astr<s>\otimes\id[A]^{\otimes t}),
  \end{multline}
  where $n\geq1$ and the sum ranges over all ordered partitions $r+s+t=n$ with
  $s\geq1$ and $r,t\geq0$. This is the form of the $A_\infty$-equations used in
  \cite[Definition~1.2.1.1]{Lef03}.
\end{remark}

\begin{remark}
  \label{rmk:Tradler}
  Let $(A,\Astr)$ be an $A_\infty$-algebra. According to
  \Cref{eq:binary_brace-rst}, the components of the $A_\infty$-equation
  $\braces{\Astr}{\Astr}=0$ are
  \begin{align}
    \label{eq:A-infty-equations-operadic}
    \begin{split}
      0&=\sum_{p+q-1=n}\sum_{i=1}^{p}\Astr<p>\bullet_{i}\Astr<q>\\
       &\stackrel{\eqref{eq:infinitesimal_composition}}{=}\sum_{p+q-1=n}\sum_{i=1}^{p}(-1)^{p+q+i(q+1)}\Astr<p>\circ(\id[A]^{\otimes(i-1)}\otimes\Astr<q>\otimes \id[A]^{\otimes(n-i)})\\
       &=\sum_{p+q-1=n}\sum_{i=1}^{p}(-1)^{1+n+i(q+1)}\Astr<p>\circ(\id[A]^{\otimes(i-1)}\otimes\Astr<q>\otimes \id[A]^{\otimes(n-i)})\\
       &=\sum_{p+q-1=n}\sum_{i=1}^{p}(-1)^{(i-1)(q+1)+n-q}\Astr<p>\circ(\id[A]^{\otimes(i-1)}\otimes\Astr<q>\otimes \id[A]^{\otimes(n-i)}),\\
    \end{split}
  \end{align}
  where the first sum ranges over all ordered decompositions $p+q-1=n$ with
  $p,q\geq1$, see~\Cref{notation:completed-brace}. The last line in
  \eqref{eq:A-infty-equations-operadic} is the form of the $A_\infty$-equations
  used in~\cite[Proposition~1.4]{Tra08}.
\end{remark}

\begin{remark}
  \label{rmk:Ai-equations}
  Let $(A,\Astr)$ be an $A_\infty$-algebra. Let us review the meaning of the
  $A_\infty$-equations for small values of $n$. We remind the reader of
  \Cref{ex:partial-via-braces}.
  \begin{enumerate}
  \item The $A_\infty$-equation for $n=1$ is
    \[
      0=\Astr<1>\bullet_1\Astr<1>,
    \]
    so that $(A,\Astr<1>)$ is a dg vector space. Below we write
    $\partial=\partial_{A,A}$ for the differential of the dg vector space
    $\dgHom[\kk]{(A,\Astr<1>)}{(A,\Astr<1>)}$.
  \item The $A_\infty$-equation for $n=2$ is
    \[
      0=\Astr<1>\bullet_1\Astr<2>+\Astr<2>\bullet_1\Astr<1>+\Astr<2>\bullet_2\Astr<1>\stackrel{\eqref{ex:partial-via-braces}}{=}\partial(\Astr<2>),
    \]
    so that
    \[
      \Astr<2>\colon(A,\Astr<1>)\otimes(A,\Astr<1>)\longrightarrow(A,\Astr<1>),
    \]
    is a morphism of dg vector spaces.
  \item The $A_\infty$-equation for $n=3$ is
    \begin{align*}
      0&=\overbrace{\Astr<1>\bullet_1\Astr<3>+\Astr<3>\bullet_1\Astr<1>+\Astr<3>\bullet_2\Astr<1>+\Astr<3>\bullet_3\Astr<1>}^{\stackrel{\eqref{ex:partial-via-braces}}{=}\partial(\Astr<3>)}\\
       &\phantom{=}+\Astr<2>\bullet_1\Astr<2>+\Astr<2>\bullet_2\Astr<2>,
    \end{align*}
    which, using \eqref{eq:infinitesimal_composition}, \eqref{eq:binary_brace} and
    \Cref{ex:partial-via-braces}, can be written as
    \begin{align}
      \label{eq:Ai-associativity}
      0&=\partial(\Astr<3>)+\Astr<2>\circ(\id[A]\otimes\Astr<2>-\Astr<2>\otimes\id[A]),
    \end{align}
    so that $\H{A,\Astr<1>}$ is a (non-unital) associative graded algebra.
    Notice also that, if $\partial(\Astr<3>)=0$, then the binary multiplication
    $\Astr<2>$ is already associative (for example if $\Astr<1>=0$ or
    $\Astr<3>=0$). In particular, a minimal $A_\infty$-algebra structure is
    equipped with a (non-unital) associative multiplication.
  \end{enumerate}
  In general, using the formula in the penultimate line
  in~\Cref{eq:A-infty-equations-operadic} and \Cref{ex:partial-via-braces}, we
  see that the $A_\infty$-equation for $n\geq1$ can also be written as
  \begin{align}
    \label{eq:minimal-A-infty-equation-partial}
    \begin{split}
      \partial(\Astr<n>)&=-\sum_{p+q-1=n}\braces{\Astr<p>}{\Astr<q>}\\
                        &\stackrel{\eqref{eq:A-infty-equations-operadic}}{=}\sum_{p+q-1=n}\sum_{i=1}^p(-1)^{i(q+1)+n}\Astr<p>\circ(\id[A]^{\otimes(i-1)}\otimes\Astr<q>\otimes\id[A]^{\otimes(n-i)}),
    \end{split}
  \end{align}
  where the first sum ranges over all ordered decompositions $p+q-1=n$ with
  $p,q\geq 2$. Notice that \eqref{eq:minimal-A-infty-equation-partial} is the form of the
  $A_\infty$-equations used in~\cite{Mar06}.
\end{remark}

\begin{remark}
  It follows from \Cref{rmk:Ai-equations} that an $A_\infty$-algebra $(A,\Astr)$
  with $\Astr<n>=0$ for $n\geq3$ is a dg algebra. Conversely, a dg algebra gives
  rise to an $A_\infty$-algebra with $\Astr<n>=0$ for $n\geq3$. In this way we
  may and we will identify dg algebras with $A_\infty$-algebras whose operations
  of arity $3$ or higher vanish.
\end{remark}

We now discuss the $A_\infty$-equations in the case of minimal $A_\infty$-algebras.

\begin{notation}
  For a minimal $A_\infty$-algebra $(A,\Astr)$, we let
  \[
    \Astr*\coloneqq\Astr-\Astr<2>=(0,0,\Astr<3>,\Astr<4>,\dots,\Astr<n>,\dots)\in\prod_{n\geq1}\HC[n]<2-n>{A}.
  \]
  We also write $(A,\Astr*)=(A,\Astr)$ when we wish to view $A$ as being
  equipped with the associative algebra structure given by $\Astr<2>$, rather
  than as a mere graded vector space.
\end{notation}

\begin{remark}
  \label{rmk:minimal-A-infinity-equations}
  Let $(A,\Astr)$ be a minimal $A_\infty$-algebra. The $A_\infty$-equations for
  $n=1,2$ are trivial, and for $n\geq3$ they take the form
  \begin{equation}
    \label{eq:minimal-A-infinity-equations}
    0=\sum_{p+q-1=n}\braces{\Astr<p>}{\Astr<q>};
  \end{equation}
  here, the sum ranges over all ordered decompositions $p+q-1=n$ with $p,q,\geq
  2$. In particular, the $A_\infty$-equation for $n=3$ is simply the
  associativity of the binary operation:
  \[
    \braces{\Astr<2>}{\Astr<2>}\stackrel{\eqref{eq:Ai-associativity}}{=}\Astr<2>\circ(\id[A]\otimes\Astr<2>-\Astr<2>\otimes\id[A])=0.
  \]
\end{remark}

\begin{remark}
  Let $(A,m^A)$ be a minimal $A_\infty$-algebra. Observe that\footnote{Here we
    use the fact that the brace relations for the Hochschild cochain complex
    extend to its natural completion with respect to the arity degree.}
  \begin{align*}
    \braces{\Astr}{\Astr}=\Sq(\Astr)&=\Sq(\Astr<2>+\Astr*)\\
              &\stackrel{\eqref{eq:Gerstenhaber_square-relations}}{=}\underbrace{\Sq(\Astr<2>)}_{=0}+\Sq(\Astr*)+[\Astr<2>,\Astr*]\\
              &\stackrel{\eqref{eq:Hd}}{=}\Hd(\Astr*)+\Sq(\Astr*),
  \end{align*}
  where $\Sq(\Astr<2>)=0$ since the multiplication $\Astr<2>$ is associative. In
  summary, the equation of a minimal $A_\infty$-algebra is the
  \emph{Maurer--Cartan equation}\footnote{If $\chark(\kk)\neq2$,
  the Maurer--Cartan equation takes the more familiar form
    \[
      \Hd(\Astr*)+\tfrac{1}{2}[\Astr*,\Astr*]=0,
    \]
  see \Cref{rmk:Gerstenhaber_square-char}.}
\begin{equation}
  \label{eq:MC-equation}
    \Hd(\Astr*)+\Sq(\Astr*)=0,
  \end{equation}
  where the Hochschild differential is the one for the associative algebra
  $(A,\Astr<2>)$.
\end{remark}

\begin{definition}
  Let $(A,\Astr)$ and $(B,\Astr[B])$ be $A_\infty$-algebras. An
  \emph{$A_\infty$-morphism} $f\colon(A,\Astr)\to(B,\Astr[B])$ is a tuple
  \[
    (f_1,f_2,f_3,\dots,f_n,\dots)\in\prod_{n\geq1}\dgHom[\kk]{A^{\otimes n}}{B}[1-n]
  \]
  such that the \emph{$A_\infty$-morphism equations} hold for each $n\geq1$:
  \begin{equation}
    \label{eq:Ai-morphism}
    \sum_{p+q-1=n}\braces{f_p}{\Astr<q>[A]}=\sum_{r=1}^n\sum\braces{\Astr<r>[B]}{f_{i_1},\dots,f_{i_r}},
  \end{equation}
  where the sum on the left-hand side runs over all ordered decompositions
  $p+q-1=n$ with $p,q\geq 1$, the internal sum on the right-hand side ranges
  over all ordered decompositions $i_1+\cdots+i_r=n$, and
  \begin{multline}
    \label{eq:morphism-brace}
    \braces{\Astr<r>[B]}{f_{i_1},\dots,f_{i_r}}\coloneqq\\
    (\cdots((\Astr<r>[B]\bullet_{1}f_{i_1})\bullet_{i_1+1}f_{i_2})\cdots)\bullet_{i_1+i_2+\cdots+i_{r-1}+1}f_{i_r}.
  \end{multline} 
  The degree $0$ morphism $f_1\colon A\to B$ is called the \emph{linear part}
  of $f$. By the $A_\infty$-morphism equation for $n=1$,
  \[
    f_1\bullet_1\Astr<1>[A]=\Astr<1>[B]\bullet_{1}f_1,
  \]
  so that $f_1\colon(A,\Astr<1>)\to(B,\Astr<1>[B])$ is a morphism of dg vector
  spaces. An $A_\infty$-morphism $f$ as above is
  \begin{itemize}
  \item an \emph{$A_\infty$-isomorphism} if the linear part $f_1$ is an
    isomorphism of graded vector spaces.
  \item a \emph{gauge $A_\infty$-isomorphism} if the linear part $f_1$ is the
    identity (hence we must have $A=B$).
  \item an \emph{$A_\infty$-quasi-isomorphism} if the linear part $f_1$ is a
    quasi-isomorphism of dg vector spaces.
  \item \emph{strict} if $f_n=0$ for $n\geq2$.
  \end{itemize}
\end{definition}

\begin{remark}
  Every $A_\infty$-quasi-isomorphism between $A_\infty$-algebras admits an
  inverse up to a suitable notion of homotopy, see for
  example~\cite[Theorem~2.9]{Sei08}.
\end{remark}

\begin{remark}
  Every $A_\infty$-algebra (for example a dg algebra) is
  $A_\infty$-quasi-isomorphic to a minimal one, called a minimal model, which is
  unique up to non-unique $A_\infty$-isomorphism,
  see~\cite{Mer99,KS01,Mar04,Mar06} among many others. We review the
  construction of minimal models of dg algebras in \Cref{subsec:minimods}.
\end{remark}

\begin{remark}
  Let $f\colon(A,\Astr)\to(B,\Astr[B])$ be a strict $A_\infty$-morphism between
  $A_\infty$-algebras. In this case, the $A_\infty$-morphism equations simplify
  to
  \begin{align}
    \label{eq:strict_A-infty-morphism}
    \begin{split}
    f_1\circ\Astr<n>[A]&=(\cdots((\Astr<n>\bullet_1f_1)\bullet_2f_1)\cdots)\bullet_nf_1\\
                       &\stackrel{\eqref{eq:infinitesimal_composition}}{=}\Astr<n>[B]\circ(\underbrace{f_1\otimes\cdots\otimes f_1}_{n\text{ times}}),
    \end{split}
  \end{align}
  so that $f_1\colon A\to B$ is strictly compatible with the corresponding
  $A_\infty$-structures.
\end{remark}

\begin{remark}
  \label{rmk:perturbation}
  Let $(A,\Astr)$ be an $A_\infty$-algebra and $f_1\colon
  A\stackrel{\sim}{\longrightarrow}B$ an isomorphism of graded vector spaces. It
  follows from \Cref{eq:strict_A-infty-morphism} that $f_1$ determines a strict
  $A_\infty$-isomorphism if we endow $B$ with the $A_\infty$-algebra structure
  \[
    \Astr<n>[B]\coloneqq f_1\circ(\Astr<n>\circ(f_1^{-1}\otimes\cdots\otimes
    f_1^{-1})),\qquad n\geq1,
  \]
  which is of course uniquely determined.
  More generally, \emph{any} tuple of morphisms
  \[
    f=(f_1,f_2,f_3,\dots,f_n,\dots)\in\prod_{n\geq1}\dgHom[\kk]{A^{\otimes n}}{B}[1-n]
  \]
  such that $f_1$ is an isomorphism of graded vector spaces determines a unique
  $A_\infty$-algebra structure $\Astr[B]$ on $B$ such that
  $f\colon(A,\Astr)\to(B,\Astr[B])$ is an $A_\infty$-isomorphism (the claimed
  $A_\infty$-structure on $B$ is constructed inductively by taking the
  $A_\infty$-morphism equations as an Ansatz). This flexibility of
  $A_\infty$-algebra structures is both a curse and a blessing: On the positive
  side, the notion of $A_\infty$-algebra is invariant under homotopy equivalences
  of dg vector spaces~\cite{Mar04} while that of dg algebra is not; on the
  negative side, it is very difficult to recognise $A_\infty$-isomorphic
  $A_\infty$-algebras and to define useful invariants of these in terms of their
  higher operations.
\end{remark}


\begin{remark}
  \label{rmk:minimal-Ai-morphism-algebra-map}
  Let $f\colon (A,\Astr)\to(B,\Astr[B])$ be an $A_\infty$-morphism between
  minimal $A_\infty$-algebras. The $A_\infty$-morphism equation for $n=2$
  simplifies to
  \[
    f_1\bullet_1\Astr<2>\stackrel{\eqref{eq:binary_brace}}{=}\braces{f_1}{\Astr<2>}=\braces{\Astr<2>[B]}{f_1,f_1}\stackrel{\eqref{eq:morphism-brace}}{=}(\Astr<2>[B]\bullet_1f_1)\bullet_2f_1.
  \]
  Explicitly,
  \[
    f_1(\Astr<2>(x,y))=\Astr<2>[B](f_1(x),f_1(y)),\qquad x,y\in A,
  \]
  so that $f_1\colon A\to B$ is a morphism of (non-unital) associative algebras.
\end{remark}

In this article we are chiefly interested in $A_\infty$-algebras that are sparse
in the following sense.

\begin{definition}
  \label{def:sparse}
  Let $d\geq1$. A graded vector space $V$ is \emph{$d$-sparse} if $V^i\neq0$
  implies $i\in d\ZZ$ (this condition is of course vacuous for $d=1$). We say
  that an $A_\infty$-algebra is \emph{$d$-sparse} if its underlying graded
  vector space is $d$-sparse. Similarly, we say that an $A_\infty$-algebra (for
  example a dg algebra) is \emph{cohomologically $d$-sparse} if its cohomology
  (with respect to $\Astr<1>[]$) is a $d$-sparse graded vector space.
\end{definition}

\begin{remark}
  \label{rmk:sparse-operations}
  Let $(A,\Astr)$ be a $d$-sparse $A_\infty$-algebra. For degree reasons, if
  $n\geq 1$ and $\Astr<n>\neq0$, then we must have $n\in 2+d\ZZ$ (recall that
  $\Astr<n>$ has degree $2-n$), see for example~\cite[Proposition~4.3.4]{JKM22}.
  Thus, the only possibly non-zero operations in such a $d$-sparse
  $A_\infty$-algebra are $\Astr<1>$ (if $d=1$) and $\Astr<kd+2>$, $k\geq0$. In
  particular, $d$-sparse $A_\infty$-algebras with $d\geq2$ are necessarily
  minimal (but cohomologically $d$-sparse ones need not be).
\end{remark}

\begin{remark}
  \label{rmk:sparse-morphisms}
  Let $f\colon(A,\Astr)\to(B,\Astr[B])$ be an $A_\infty$-morphism between
  $d$-sparse algebras, $d\geq2$. For degree reasons, $f_n\neq0$ then we must
  have $n\in 1+d\ZZ$ (recall that $f_n$ has degree $1-n$), see for
  example~\cite[Proposition~4.4.4]{JKM22}. Thus, the only possible non-zero
  components of a morphism between such $d$-sparse algebras are $f_{kd+1}$,
  $k\geq0$.
\end{remark}

\subsection{$A_\infty$-bimodules and $A_\infty$-morphisms between them}

We remind the reader of \Cref{notation:completed-brace}.

\begin{definition}
  Let $A=(A,\Astr)$ be an $A_\infty$-algebra. An \emph{$A_\infty$-bimodule over
    $A$} is a pair $(M,\Astr[M])$ consisting of a graded vector space $M$ and a
  cochain
  \[
    \Astr[M]=(\Astr<1>[M],\Astr<2>[M],\Astr<3>[M],\dots,\Astr<n>[M],\dots)\in\prod_{n\geq1}\BimHC[n-1]<2-n>{M}
  \]
  with
  \[
    \Astr<n>[M]=(\Astr<p,q>[M]\colon A^{\otimes p}\otimes M\otimes A^{\otimes
      q}\longrightarrow M)_{p+1+q=n},
  \]
  called an \emph{$A_\infty$-bimodule structure}, such that the
  \emph{$A_\infty$-bimodule equation} holds:
  \[
    \braces{\Astr[M]}{\Astr}+\braces{\Astr[M]}{\Astr[M]}=0.
  \]
  Notice that, if $p+1+q=n$, then the cochain $\Astr<p,q>[M]$ has the vertical degree $2-n=1-p-q$.
\end{definition}

\begin{remark}
  Under the name `bipolydule', a slightly implicit definition of
  $A_\infty$-bimodule is given in~\cite[Section~2.5.1]{Lef03} where the
  corresponding derived category is also constructed. Explicit (equivalent)
  definitions of $A_\infty$-bimodule including the complete set of equations are
  given in \cite[Section~1.10]{Mar92} and \cite[Section~2]{Tra08}, see also
  \cite{GJ90} where the definition is given using the formalism of dg
  bicomodules over dg bicoalgebras (also used in \cite{Tra08}). An operadic
  approach to the derived category of $A_\infty$-bimodules can be extracted from
  the very general results in~\cite{BM09}. The more general notion of
  $A_\infty$-bimodule over an $A_\infty$-category is studied in~\cite{LM08}.
\end{remark}

\begin{remark}
  \label{rmk:Tradler-bim}
  Let $(A,\Astr)$ be an $A_\infty$-algebra and $(M,\Astr[M])$ an
  $A_\infty$-bimodule. According to the conventions in
  \Cref{notation:completed-brace}, the components of the $A_\infty$-bimodule
  equation can be written as
  \begin{align}
    \label{eq:Ai-bimodule}
    0=\sum_{p+q+s=n}\braces{\Astr<p,q>[M]}{\Astr<s>}+\sum_{p+q+s+t+1=n}\Astr<p,q>[M]\bullet_{p+1}\Astr<s,t>[M],\qquad n\geq1,
  \end{align}
  where the sum on the left term ranges over all ordered decompositions
  $p+q+s=n$ with $s\geq1$ and $p+q\geq1$, and the sum on the right term ranges
  over all ordered decompositions $p+q+s+t+1=n$ with $p,q,s,t\geq0$. The left
  term decomposes further as
  \[
    \sum_{p+q+s=n}\braces{\Astr<p,q>[M]}{\Astr<s>}\stackrel{\eqref{eq:infinitesimal_composition}}{=}\sum_{p+q+s=n}\sum_{i=1}^p\Astr<p,q>[M]\bullet_i\Astr<s>+\sum_{p+q+s=n}\sum_{j=1}^q\Astr<p,q>[M]\bullet_{p+1+j}\Astr<s>,
  \]
  compare with the form of the $A_\infty$-bimodule equations given
  in~\cite[Proposition~2.7]{Tra08}. Notice also that, for a fixed value of $n$,
  the $A_\infty$-bimodule equation \eqref{eq:Ai-bimodule} in fact encodes a
  system of $n$ linearly independent equations, for it takes place in the vector
  space
  \[
    \BimHC[n-1]<2-n>{A}[M]=\bigoplus_{i=1}^n\BimHC*[i-1][n-i]<2-n>{A}[M].
  \]
\end{remark}

In this article we are mainly interested in the following canonical
$A_\infty$-bimodule.

\begin{defprop}[{\cite[Lemma~4.1]{Tra08}}]
  \label{defprop:diaognal-Ai-bimodule}
  Let $(A,\Astr)$ be an $A_\infty$-algebra. The \emph{diagonal
    $A_\infty$-bimodule} $(A(0),\Astr[A[0]])$ is the $A_\infty$-bimodule with
  underlying graded vector space $A$ and $A_\infty$-bimodule structure with
  operations\footnote{We denote the diagonal $A_\infty$-bimodule structure by
    $\Astr[A[0]]$ to avoid confusing it with the given $A_\infty$-algebra
    structure $\Astr$ on $A$.}
  \[
    \Astr<n>[A[0]]=(\Astr<p,q>[A[0]])_{p+1+q=n}\in\BimHC[n-1]<2-n>{A},\qquad
    n\geq1,
  \]
  where
  \[
    \Astr<p,q>[A[0]]\coloneqq\Astr<p+1+q>[A],\qquad p+1+q=n.
  \]
\end{defprop}

We remind the reader of the following elementary fact: Given an algebra $A$ and
a vector space $M$, the set of $A$-bimodule structures on $M$ is in bijection
with the set of algebra structures on the vector space $A\oplus M$ that extend
the given algebra structure on $A$ and such that $M\subseteq A\oplus M$ is a
square-zero ideal. This observation admits the following generalisation to
$A_\infty$-algebras and $A_\infty$-bimodules.

\begin{proposition}
  \label{prop:Ai-bimodule-vs-Ai-algebra}
  Let $(A,\Astr)$ be an $A_\infty$-algebra and $M$ a graded vector space. The
  association $\Astr[M]\mapsto \Astr+\Astr[M]$ induces bijections between the
  following:
  \begin{enumerate}
  \item The set of $A_\infty$-bimodule structures $\Astr[M]$ over $(A,\Astr)$.
  \item The set of cochains
    \[
      \Astr[A\ltimes M]=\Astr+\Astr[M]\in\prod_{n\geq1}\RelBimHC[n]<2-n>{A}{M}
    \]
    such that $\braces{\Astr[A\ltimes M]}{\Astr[A\ltimes M]}=0$, that is such
    that $(A\oplus M,\Astr[A\ltimes M])$ is an $A_\infty$-algebra.
  \end{enumerate}
\end{proposition}
\begin{proof}
  Recall that $\braces{\Astr}{\Astr}=0$ since $\Astr$ is an $A_\infty$-algebra
  structure. Hence, given a cochain
  \[
    \Astr[M]\in\prod_{n\geq1}\RelBimHC[n]<2-n>{A}{M},
  \]
  for $\Astr[A\ltimes M]=\Astr+\Astr[M]$ we have
  \[
    \braces{\Astr[A\ltimes
      M]}{\Astr[A\ltimes M]}=\underbrace{\braces{\Astr}{\Astr}}_{=0}+\underbrace{\braces{\Astr}{\Astr[M]}}_{\stackrel{\eqref{eq:infinitesimal_composition}}{=}0}+\braces{\Astr[M]}{\Astr}+\braces{\Astr[M]}{\Astr[M]}.
  \]
  The claim follows from the above equation together with the observation that
  there is an inclusion of bigraded vector spaces
  $\RelBimHC{A}{M}\subseteq\HC{A\oplus M}$ that is compatible with the infinitesimal
  compositions.
\end{proof}

\Cref{prop:Ai-bimodule-vs-Ai-algebra} motivates the following notation.

\begin{notation}
  Let $(A,\Astr)$ be an $A_\infty$-algebra and $(M,\Astr[M])$ an
  $A_\infty$-bimodule. We set
  \[
    \Astr[A\ltimes
    M]\coloneqq\Astr+\Astr[M]\in\prod_{n\geq1}\RelBimHC[n][2-n]{A}{M}.
  \]
\end{notation}

\begin{remark}
  Let $(A,\Astr)$ be a minimal $A_\infty$-algebra and $(M,\Astr[M])$ an
  $A_\infty$-bimodule over it. Notice that the associated $A_\infty$-algebra
  $(A\oplus M,\Astr[A\ltimes M])$ is also minimal. It follows from
  \Cref{rmk:minimal-A-infinity-equations} that the $A_\infty$-bimodule equations
  for $n=1,2$ are trivial, and for $n=3$ they yield the fact that the
  square-zero product $\Astr<2>[A\ltimes M]$ is associative, so that $M$ is a
  (non-unital) graded $A$-bimodule in the associative sense.
\end{remark}

The upshot of \Cref{prop:Ai-bimodule-vs-Ai-algebra} is that we can reduce
several aspects of the theory of $A_\infty$-bimodules to that of
$A_\infty$-algebras (with appropriate modifications). For example, the notion of
morphism between $A_\infty$-bimodules can be formulated in the following manner.
We remind the reader of \Cref{rmk:BimHC-MN}.

\begin{definition}
  \label{def:Ai-bimodule-morphism}
  Let $(A,\Astr)$ be an $A_\infty$-algebra and $(M,\Astr[M])$ and $(N,\Astr[N])$
  two $A_\infty$-bimodules over it. An \emph{$A_\infty$-morphism of
    $A_\infty$-bimodules} ${f\colon(M,\Astr[M])\to(N,\Astr[N])}$ is a cochain
  \[
    f=(f_1,f_2,f_3,\dots,f_n,\dots)\in\prod_{n\geq 1}\BimHC[n-1]<1-n>{M}[N]
  \]
  such that
  \[
    \widetilde{f}\coloneqq\id[A]\oplus f\colon(A\oplus M,\Astr[A\ltimes M])\to(A\oplus N,\Astr[A\ltimes N])
  \]
  is a morphism of $A_\infty$-algebras; notice that the cochain $f_n$ is in fact a tuple
  \[
    f_n=(f_{p,q}\colon A^{\otimes p}\otimes M\otimes A^{\otimes
      q}\longrightarrow N)_{p+1+q=n},
  \]
  of homogeneous morphisms of vector spaces if (vertical) degree $1-n=-p-q$.
  The degree $0$ morphism $f_1=f_{0,0}\colon M\to N$ is called the \emph{linear part}
  of $f$. By the $A_\infty$-morphism equation for $n=1$,
  \[
    f_1\bullet_1\Astr<1>[M]=\Astr<1>[n]\bullet_{1}f_1,
  \]
  so that $f_1\colon(M,\Astr<1>[M])\to(N,\Astr<1>[N])$ is a morphism of dg vector
  spaces. An $A_\infty$-morphism of $A_\infty$-bimodules $f$ as above is
  \begin{itemize}
  \item an \emph{$A_\infty$-isomorphism} if the linear part $f_1$ is an
    isomorphism of graded vector spaces.
  \item a \emph{gauge $A_\infty$-isomorphism} if the linear part $f_1$ is the
    identity (hence we must have $M=N$).
  \item an \emph{$A_\infty$-quasi-isomorphism} if the linear part $f_1$ is a
    quasi-isomorphism of dg vector spaces.
  \item \emph{strict} if $f_n=0$ for $n\geq2$.
  \end{itemize}
\end{definition}

\begin{remark}
  Let $(A,\Astr)$ be a minimal $A_\infty$-algebra and
  $f\colon(M,\Astr[M])\to(N,\Astr[N])$ an $A_\infty$-morphism between
  minimal $A_\infty$-bimodules over it. It follows from
  \Cref{rmk:minimal-Ai-morphism-algebra-map} that
  \[
    \widetilde{f}=\id[A]\oplus f_1\colon(A\oplus M,\Astr[A\ltimes M])\longrightarrow(A\oplus
    N,\Astr[A\ltimes N])
  \]
  is a morphism of (non-unital) associative algebras. Equivalently,
  \[
    f_1\colon(M,\Astr[M])\longrightarrow(N,\Astr[N])
  \]
  is a morphism of (non-unital) associative graded $A$-bimodules.
\end{remark}

\begin{remark}
  Every $A_\infty$-bimodule is $A_\infty$-quasi-isomorphic to a minimal one,
  called a minimal model, which is unique up to non-unique
  $A_\infty$-isomorphism. This follows from the existence of the projective
  model structure on dg operads, since the bimodule Hochschild complex arises
  from the linear endomorphism operad from~\cite{BM09}. In \Cref{subsec:minimods}
  we explain how to use \Cref{prop:Ai-bimodule-vs-Ai-algebra} together with the
  well-known formulas for transferred $A_\infty$-algebra structures~\cite{Mar06}
  to obtain explicit formulas in the case of dg bimodules over dg algebras.
\end{remark}

In this article we are chiefly interested in sparse minimal
$A_\infty$-bimodules, compare with \Cref{def:sparse}.

\begin{definition}
  Let $d\geq1$. An $A_\infty$-bimodule over an $A_\infty$-algebra is
  \emph{$d$-sparse} if its underlying graded vector space is $d$-sparse (this
  condition is vacuous if $d=1$). More generally, an $A_\infty$-bimodule (for
  example a dg bimodule) is
  \emph{cohomologically $d$-sparse} if its cohomology (with respect to $m_1$) is
  a $d$-sparse graded vector space.
\end{definition}

\begin{remark}
  Clearly, \Cref{rmk:sparse-operations,rmk:sparse-morphisms} apply equally well
  to $d$-sparse $A_\infty$-bimodules and morphisms between them. In particular,
  the first possibly non-trivial higher operation in a $d$-sparse
  $A_\infty$-bimodule is
  \[
    \Astr<d+2>[]=(\Astr<p,q>[])_{p+1+q=d+2}.
  \]
\end{remark}

\subsection{Shifts and linear duals of $A_\infty$-bimodules}

We now turn our attention to shifts and linear duals of $A_\infty$-bimodules.
We recall the relevant definitions.

\begin{defprop}
  \label{defprop:shift_Ai-bimodule}
  Let $(A,\Astr)$ be an $A_\infty$-algebra and $(M,\Astr[M])$ an
  $A_\infty$-bimodule over it.
  Given $n\in\ZZ$, the graded vector space
  $M(n)$ carries a canonical $A_\infty$-bimodule structure with operations
  \[
    \Astr<k>[M(n)]\coloneqq(\Astr<p,q>[M(n)])_{p+1+q=k}\in\BimHC[k-1]<2-k>{M(n)},\qquad k\geq1,
  \]
  where
  \[
    \Astr<p,q>[M(n)]\coloneqq(\Astr<p,q>[M])^{(n)},\qquad p+1+q=k,
  \]
  is defined by the formula in \Cref{eq:cn}.
\end{defprop}
\begin{proof}
  The fact that $\Astr[M(n)]$ is an $A_\infty$-bimodule structure follows from
  \Cref{prop:shifting-cochains}, as the latter yields the fact that
  $\Astr[A\ltimes M(n)]=\Astr+\Astr[M(n)]$ is an $A_\infty$-algebra structure
  that satisfies the conditions in \Cref{prop:Ai-bimodule-vs-Ai-algebra}:
  \[
    \braces{\Astr[A\ltimes M(n)]}{\Astr[A\ltimes M(n)]}\stackrel{\eqref{eq:Psi}}{=}\braces{\Psi(\Astr[A\ltimes M])}{\Psi(\Astr[A\ltimes M])}\stackrel{\ref{prop:shifting-cochains}}{=}\Psi(\underbrace{\braces{\Astr[A\ltimes M]}{\Astr[A\ltimes M]}}_{=0})=0.\qedhere
  \]
\end{proof}

\begin{remark}
  Let $(A,\Astr)$ be an $A_\infty$-algebra and $(M,\Astr[M])$ a cohomologically
  $d$-sparse $A_\infty$-bimodule over it, $d\geq1$. Clearly, $(M(n),\Astr[M(n)])$ is
  cohomologically $d$-sparse if and only if $n\in d\ZZ$.
\end{remark}

The following basic result can be deduced, using \cite[Section~3.8]{LM08}, from
the description of $A_\infty$-bimodules as suitable $A_\infty$-functors with
values in the dg category of dg vector spaces given in \emph{op.~cit.} We give a
direct proof for the convenience of the reader.

\begin{proposition}
  \label{prop:shifting-Ai-morphisms}
  Let $(A,\Astr)$ be an $A_\infty$-algebra and $(M,\Astr[M])$ and $(N,\Astr[N])$
  two $A_\infty$-bimodules over it. Let $n\in\ZZ$ and suppose given an
  $A_\infty$-morphism of $A_\infty$-bimodules
  \[
    \varphi\colon(M,\Astr[M])\longrightarrow(N,\Astr[N]).
  \]
  Then, the degree $-p-q$ maps
  \[
    \varphi_{p,q}^{(n)}=\Psi(\varphi_{p,q})\colon A^{\otimes p}\otimes M(n)\otimes A^{\otimes
      q}\longrightarrow N(n),\qquad k\geq1,\qquad p+1+q=k,
  \]
  defined as in \Cref{variant:shifting-cochains}, assemble into an
  $A_\infty$-morphism of $A_\infty$-bimodules
  \[
    \varphi^{(n)}\colon(M(n),\Astr[M(n)])\longrightarrow(N(n),\Astr[N(n)]).
  \]
\end{proposition}
\begin{proof}
  We set
  \[
    \widetilde{\varphi}\coloneq\id[A]\oplus\varphi\colon(A\oplus
    M,\Astr[A\ltimes M])\longrightarrow(A\oplus N,\Astr[A\ltimes N]),
  \]
  which is an $A_\infty$-morphism of $A_\infty$-algebras by assumption.
  Similarly, we set
  \[
    \widetilde{\varphi}(n)\coloneq\id[A]\oplus\varphi(n)\colon(A\oplus
    M(n),\Astr[A\ltimes M(n)])\longrightarrow(A\oplus N(n),\Astr[A\ltimes
    N(n)]).
  \]
  Using the compatibilities stated in \Cref{variant:shifting-cochains}, we
  compute, for $n\geq1$,
  \begin{align*}
    \sum_{p+q-1=n}\braces{\widetilde{\varphi}_p(n)}{\Astr<q>[A\ltimes M(n)]}&=\sum_{p+q-1=n}\braces{\Psi(\widetilde{\varphi}_p)}{\Psi(\Astr<q>[A\ltimes M])}\\
                                                                            &=\sum_{p+q-1=n}\Psi(\braces{\widetilde{\varphi}_p}{\Astr<q>[A\ltimes M]})\\
                                                                            &\stackrel{\eqref{eq:Ai-morphism}}{=}\sum_{r=1}^n\sum\Psi(\braces{\Astr<r>[A\ltimes M]}{\widetilde{\varphi}_{i_1},\dots,\widetilde{\varphi}_{i_r}})\\
                                                                            &=\sum_{r=1}^n\sum\braces{\Psi(\Astr<r>[A\ltimes M])}{\Psi(\widetilde{\varphi}_{i_1}),\dots,\Psi(\widetilde{\varphi}_{i_r})}\\
                                                                            &=\sum_{r=1}^n\sum\braces{\Astr<r>[A\ltimes M(n)]}{\widetilde{\varphi}_{i_1}(n),\dots,\widetilde{\varphi}_{i_r}(n)},
  \end{align*}
  where all the non-indexed internal sums range over all ordered
  decompositions $i_1+\cdots+i_r=n$. This shows that $\widetilde{\varphi}(n)$ is
  a morphism of $A_\infty$-algebras, as required.
\end{proof}

\begin{defprop}[{\cite[Lemma~2.9]{Tra08}}]
  \label{defprop:dual_Ai-bimodule}
  Let $(A,\Astr)$ be an $A_\infty$-algebra and $(M,\Astr[M])$ an
  $A_\infty$-bimodule over it. The linear dual $DM=\dgHom[\kk]{M}{\kk}$ carries
  a canonical $A_\infty$-bimodule structure with operations
  \[
    \Astr<n>[DM]\coloneqq(\Astr<p,q>[DM])_{p+1+q=n}\in\BimHC[n-1]<2-n>{DM},\qquad
    n\geq1,
  \]
  where
  \[
    \Astr<p,q>[DM]\coloneqq(\Astr<q,p>[M])^{D},\qquad p+1+q=n,
  \]
  is defined by the formula in \Cref{eq:cD}. In particular, the differential
  $\Astr<1>[DM]$ on $DM$ is given, in terms of the differential $\Astr<1>[M]$ on
  $M$, by the expected formula
  \[
    \Astr<1>[DM](f)\stackrel{\eqref{eq:cD}}{=}(-1)^{(0+1)(0+1)+|f|\vdeg{\Astr<1>[M]}}f\circ\Astr<1>[M]=(-1)^{|f|+1}f\circ\Astr<1>[M]\stackrel{\ref{ex:DV}}{=}\partial_{M,\kk}(f).
  \]
\end{defprop}
\begin{proof}
  The fact that $\Astr[DM]$ is an $A_\infty$-bimodule structure follows from
  \Cref{prop:dualising-cochains}, as the latter yields the fact that
  $\Astr[A\ltimes DM]=\Astr+\Astr[DM]$ is an $A_\infty$-algebra structure that
  satisfies the conditions in \Cref{prop:Ai-bimodule-vs-Ai-algebra}:
  \[
    \braces{\Astr[A\ltimes DM]}{\Astr[A\ltimes
      DM]}\stackrel{\eqref{eq:Theta}}{=}\braces{\Theta(\Astr[A\ltimes
      M])}{\Theta(\Astr[A\ltimes
      M])}\stackrel{\ref{prop:dualising-cochains}}{=}\Theta(\underbrace{\braces{\Astr[A\ltimes
        M]}{\Astr[A\ltimes M]}}_{=0})=0.\qedhere
  \]
\end{proof}

\begin{remark}
  Let $(A,\Astr)$ be an $A_\infty$-algebra and $(M,\Astr[M])$ a cohomologically
  $d$-sparse $A_\infty$-bimodule over it, $d\geq1$. Clearly, $(DM,\Astr[DM])$ is
  also cohomologically $d$-sparse.
\end{remark}

\begin{remark}
  The attentive reader will notice, when comparing the formula in \Cref{eq:cD}
  that defines the cochain $\Astr<p,q>[DM]$ in \Cref{defprop:dual_Ai-bimodule}
  with the corresponding formula in~\cite[Lemma~2.9]{Tra08}, that the signs are
  off by the factor $(-1)^{(p+1)(q+1)}$. Indeed, this missing factor has been
  added in the erratum~\cite{Tra11} to \emph{op.~cit.}
\end{remark}

\begin{proposition}
  \label{prop:dualising-Ai-morphisms}
  Let $(A,\Astr)$ be an $A_\infty$-algebra and $(M,\Astr[M])$ and $(N,\Astr[N])$
  two $A_\infty$-bimodules over it. Let $n\in\ZZ$ and suppose given an
  $A_\infty$-morphism of $A_\infty$-bimodules
  \[
    \varphi\colon(M,\Astr[M])\longrightarrow(N,\Astr[N]).
  \]
  Then, the degree $-p-q$ maps
  \[
    \varphi_{p,q}^{D}=\Theta(\varphi_{p,q})\colon A^{\otimes p}\otimes DN\otimes A^{\otimes
      q}\longrightarrow DM,\qquad k\geq1,\qquad p+1+q=k,
  \]
  defined as in \Cref{variant:dualising-cochains}, assemble into an $A_\infty$-morphism of $A_\infty$-bimodules
  \[
    \varphi^{D}\colon(DN,\Astr[DN])\longrightarrow(DM,\Astr[DM]).
  \]
\end{proposition}
\begin{proof}
  Recall that $|\varphi|=1$ and $|\Astr[]|=2$. Using the compatibilities stated
  in \Cref{variant:dualising-cochains}, we compute, for $p,q\geq0$,
  \begin{align*}
    \braces{\Theta(\varphi_{p,q})}{\Astr<s>[A]}&\stackrel{\eqref{eq:binary_brace}}{=}\Theta(\braces{\varphi_{p,q}}{\Astr<s>[A]})\\
    \braces{\Theta(\varphi_{p,q})}{\Astr<t,s>[DN]}&\stackrel{\eqref{eq:binary_brace}}{=}\Theta(\varphi_{p,q})\bullet_{q+1}\Theta(\Astr<s,t>[N])\\&\stackrel{\ref{lemma:Theta}}{=}-\Theta(\braces{\Astr<t,s>[N]}{\varphi_{p,q}})\\
                                               &=-\Theta(\braces{\Astr<s,t>[N]}{\id[A],\stackrel{s}{\dots},\id[A],\varphi_{p,q},\id[A],\stackrel{t}{\dots},\id[A]}).
  \end{align*}
  Similarly,
  \begin{align*}
    \braces{\Theta(\Astr<t,s>[DM])}{\id[A],\stackrel{q}{\dots},\id[A],\Theta(\varphi_{p,q}),\id[A],\stackrel{s}{\dots},\id[A]}
    &=\braces{\Theta(\Astr<t,s>[DM])}{\Theta(\varphi_{p,q})}\\
    &=-\Theta(\braces{\varphi_{p,q}}{\Astr<s,t>[M]}).
  \end{align*}
  Thanks to the additional minus signs, we see that applying $\Theta$ to the
  $A_\infty$-morphism equation satisfied by $\varphi$ yields the
  $A_\infty$-morphism equation for $\varphi^D=\Theta(\varphi)$. This finishes
  the proof.
\end{proof}

We establish certain compatibilities between shifts and linear duals that are
needed in the sequel. These are rather elementary, but we did not find
suitable references in the literature. We remind the reader of
\Cref{defprop:upsilon}.

\begin{proposition}
  \label{prop:DM-shift}
  Let $(A,\Astr)$ be an $A_\infty$-algebra and $(M,\Astr[M])$ an
  $A_\infty$-bimodule over it. Fix $n\in\ZZ$. The following statements hold:
  \begin{enumerate}
  \item\label{it:DM-shift} The map
    \begin{equation}
      \label{eq:prop:DM-shift:it:DM-shift}
      \alpha\colon(DM)(n)\stackrel{\sim}{\longrightarrow}D(M(-n)),\qquad
      \s[n]f\longmapsto(-1)^{n|f|}f\s[n],
    \end{equation}
    see \Cref{ex:DAn}, yields a strict $A_\infty$-isomorphism of
    $A_\infty$-bimodules
    \[
      \alpha\colon((DM)(n),\Astr[(DM)[-n]])\stackrel{\sim}{\longrightarrow}(D(M(-n)),\Astr[D(M(-n))]).
    \]
  \item\label{it:upsilon} Suppose $(M,\Astr[M])$ is minimal and that there
    exists an isomorphism of graded $A$-bimodules
    \[
      \varphi\colon M(n)\stackrel{\sim}{\longrightarrow}DM.
    \]
    Then, the isomorphism of graded $A$-bimodules
    \begin{equation}
      \label{eq:prop:DM-shift:it:upsilon}
      \psi\colon M\stackrel{\sim}{\longrightarrow}D(M(n)),\qquad m\longmapsto
      (-1)^{n|m|}\varphi(\s[n]m)\s[-n],
    \end{equation}
    yields a strict $A_\infty$-isomorphism of $A_\infty$-bimodules
    \[
      \psi\colon(M,\Upsilon(\Astr[M]))\stackrel{\sim}{\longrightarrow}(D(M(n)),\Astr[D(M(n))]),
    \]
    compare with \Cref{rmk:phi-psi}.
  \end{enumerate}
\end{proposition}
\begin{proof}
  \eqref{it:DM-shift} Let $k\geq1$ and $p+1+q=k$. For
  $x_1,\dots,x_p,y_1,\dots,y_q\in A$, $f\in DM$ and $m\in M$, on the one hand,
  \begin{align*}
    \Astr<q,p>[D(M(-n))]&(y_1,\dots,y_q,\alpha(\s[n]f),x_1,\dots,x_p)(\s[-n]m)\\
                        &\stackrel{\eqref{eq:cD}}{=}(-1)^{\maltese_1}\alpha(\s[n]f)(\Astr<p,q>[M[-n]](x_1,\dots,x_p,\s[-n]m,y_1,\dots,y_q))\\
                        &=(-1)^{\maltese_1+n|f|}f(\s[n](\Astr<p,q>[M[-n]](x_1,\dots,x_p,\s[-n]m,y_1,\dots,y_q))\\
                        &\stackrel{\eqref{eq:cn}}{=}(-1)^{\maltese_1+n|f|+\maltese_2}f(\s[n](\s[-n]\Astr<p,q>[M](y_1,\dots,y_q,m,y_1,\dots,y_q)))\\
                        &=(-1)^{\maltese_3}f(\Astr<p,q>[M](x_1,\dots,x_p,m,y_1,\dots,y_q)),
  \end{align*}
  where
  \begin{align*}
    \maltese_1&=\textstyle(p+1)(q+1)+(|f|-n)(p+1+q)\\&\textstyle\phantom{=}+\left(\sum_{j=1}^q|y_j|\right)\left(\sum_{i=1}^p|x_i|+|f|+|m|\right)\\
    \maltese_2&=\textstyle n\left(p+1+q+\sum_{i=1}^p|x_i|\right)\\
    \maltese_3&=\textstyle (p+1)(q+1)+|f|(p+1+q)+n\left(|f|+\sum_{i=1}^p|x_i|\right)\\
              &\textstyle\phantom{=}+\left(\sum_{j=1}^q|y_j|\right)\left(|f|+\sum_{i=1}^p|x_i|+|m|\right).
  \end{align*}
  On the other hand,
  \begin{align*}
    \alpha(\Astr<q,p>[(DM)(n)])&(y_1,\dots,y_q,\s[n]f,x_1,\dots,x_p))(\s[-n]m)\\
                               &\stackrel{\eqref{eq:cn}}{=}(-1)^{\maltese_4}\alpha(\s[n]\Astr<q,p>[DM](y_1,\dots,y_q,f,x_1,\dots,x_p))(\s[-n]m)\\
                               &=(-1)^{\maltese_4+\maltese_5}(\Astr<q,p>[DM](y_1,\dots,y_q,f,x_1,\dots,x_p)\s[n])(\s[-n]m)\\
                               &=(-1)^{\maltese_4+\maltese_5}(\Astr<q,p>[DM](y_1,\dots,y_q,f,x_1,\dots,x_p))(m)\\
                               &\stackrel{\eqref{eq:cD}}{=}(-1)^{\maltese_4+\maltese_5+\maltese_6}f(\Astr<p,q>[M](x_1,\dots,x_p,m,y_1,\dots,y_q)),
  \end{align*}
  where
  \begin{align*}
    \maltese_4&=\textstyle n\left(p+1+q+\sum_{j=1}^q|y_j|\right)\\
    \maltese_5&=\textstyle n(p+1+q+\sum_{j=1}^q|y_j|+|f|+\sum_{i=1}^p|x_i|)\\
    \maltese_6&=\textstyle(p+1)(q+1)+|f|(p+1+q)\\
              &\phantom{=}\textstyle+\left(\sum_{j=1}^q|y_j|\right)\left(|f|+\sum_{i=1}^p|x_i|+|m|\right).
  \end{align*}
  Since $\maltese_3=\maltese_4+\maltese_5+\maltese_6$, the claim follows.

  \eqref{it:upsilon} Using and \Cref{defprop:shift_Ai-bimodule} and
  \Cref{defprop:dual_Ai-bimodule} to define the induced $A_\infty$-structure on
  the $n$-fold shift and the linear dual of $M$, respectively, this is immediate
  from the construction of $\Upsilon$ given in \Cref{defprop:upsilon}, compare
  with \Cref{rmk:RelBimHC-under-isos,rmk:perturbation}.
\end{proof}

\subsection{Restriction of scalars}

We now study the operation of restriction of scalars for $A_\infty$-bimodules and
its compatibility with the passage to shifts and linear duals. Let $(A,\Astr)$
and $(B,\Astr[B])$ be $A_\infty$-algebras and fix an $A_\infty$-morphism of
$A_\infty$-algebras
\[
  \varphi\colon(A,\Astr)\longrightarrow(B,\Astr[B]).
\]

\begin{defprop}[{\emph{cf.~}\cite[Section~6.2]{Kel01}}]
  \label{defprop:restriction-of-scalars}
  Let $(M,\Astr[M])$ be an $A_\infty$-bimodule over $(B,\Astr[B])$. For
  $p,q\geq0$ and $n\coloneqq p+1+q$, we define
  \[
    \Astr<p,q>[\varphi^*M]\coloneqq\sum_{r+1+s=n}\sum\braces{\Astr<r,s>[M]}{\varphi_{i_1},\dots,\varphi_{i_r},\id[M],\varphi_{j_1},\dots,\varphi_{j_s}},
  \]
  where the first sum ranges over all ordered decompositions $r+1+s=n$
  with $r,s\geq0$, and the second sum ranges over all ordered
  decompositions ${i_1+\cdots+i_r}=p$ and ${j_1+\cdots+j_s}=q$ with positive
  components. Then, the cochains
  \[
    \Astr<n>[\varphi^*M]\coloneqq(\Astr<p,q>[\varphi^*M])_{p+1+q=n}\in\bigoplus_{p+1+q=n}\BimHC[n-1]<2-n>{A}[M]
  \]
  assemble into an $A_\infty$-bimodule structure
  \[
    \Astr[\varphi^*M]\coloneqq(\Astr<1>[\varphi^*M],\Astr<2>[\varphi^*M],\Astr<3>[\varphi^*M],\dots,\Astr<n>[\varphi^*M],\dots)\in\prod_{n\geq1}\BimHC[n-1]<2-n>{A}[M]
  \]
  on $M$ over $(A,\Astr)$.
\end{defprop}
\begin{proof}
  The proof is straightforward, but somewhat tedious, and is left to the
  reader.\footnote{The proof can be streamlined by using the brace relation
    recalled in \Cref{eq:brace_relation}.}
  We only observe that the vertical degree of a cochain of the form
  \[
    \braces{\Astr<r,s>[M]}{\varphi_{i_1},\dots,\varphi_{i_r},\id[M],\varphi_{j_1},\dots,\varphi_{j_s}}
  \]
  is equal to the required vertical degree for $\Astr<p,q>[\varphi^*M]$; indeed, the
  vertical degree of such cochain is
  \begin{align*}
    \underbrace{\vdeg{\Astr<r,s>[M]}}_{=1-r-s}+\underbrace{\sum_{a=1}^r\vdeg{\varphi_{i_a}}}_{=r-p}+\underbrace{\sum_{b=1}^s\vdeg{\varphi_{j_b}}}_{=s-q}=1-p-q=\vdeg{\Astr<p,q>[\varphi^*M]}.
  \end{align*}\qedhere
\end{proof}

We remind the reader of the following elementary fact. Let $A$ and $B$ be
algebras and $M$ a $B$-bimodule. Then, given a morphism of algebras
$\varphi\colon A\to B$, there is a unique $A$-bimodule structure on $M$, denoted
$\varphi^*M$, such that the map
\[
  \widetilde{\varphi}=\varphi\oplus\id[M]\colon A\oplus\varphi^*M\longrightarrow B\oplus M
\]
is a morphism of algebras where both source and target are endowed with the
corresponding square-zero product. Indeed, for $x,y\in A$ and $m\in M$ one has $x\cdot m\cdot y\in M$ and therefore
\[
  x\cdot m\cdot y=\widetilde{\varphi}(x\cdot m\cdot
  y)=\widetilde{\varphi}(x)\cdot\widetilde{\varphi}(m)\cdot\widetilde{\varphi}(y)=\varphi(x)\cdot
  m\cdot\varphi(y).
\]
This observation admits the following generalisation to $A_\infty$-algebras and
$A_\infty$-bimodules.

\begin{proposition}
  \label{prop:restriction-of-scalars-as-algebra-map}
  Let $(M,\Astr[M])$ be an $A_\infty$-bimodule over $(B,\Astr[B])$. Then,
  $\Astr[\varphi^*M]$ is the unique $A_\infty$-bimodule structure on $M$ over
  $(A,\Astr)$ such that
  \[
    \widetilde{\varphi}\coloneqq\varphi\oplus\id[M]\colon(A\oplus
    M,\Astr[A\ltimes\varphi^*M])\longrightarrow(B\oplus M,\Astr[B\ltimes M])
  \]
  is an $A_\infty$-morphism of $A_\infty$-algebras.
\end{proposition}
\begin{proof}
  Let $\Astr[]$ be an $A_\infty$-bimodule structure on $M$ over $(A,\Astr)$.
  Consider now the candidate map
  \[
    \widetilde{\varphi}\coloneqq\varphi\oplus\id[M]\colon(A\oplus
    M,\Astr+\Astr[])\longrightarrow(B\oplus M,\Astr[B\ltimes M])
  \]
  Keeping in mind that $\varphi$ is an $A_\infty$-morphism of
  $A_\infty$-algebras by assumption and that $\widetilde{\varphi}_{0,0}=\id[M]$, the
  $A_\infty$-morphism equation for $\widetilde{\varphi}$ translates to the
  requirement that, for
  $p,q\geq0$ and $n\coloneqq p+1+q$,
  \[
    \Astr<p,q>[]=\braces{\widetilde{\varphi}_{0,0}}{\Astr<p,q>[]}=\sum_{r+1+s=n}\sum\braces{\Astr<r,s>[M]}{\varphi_{i_1},\dots,\varphi_{i_r},\id[M],\varphi_{j_1},\dots,\varphi_{j_s}},
  \]
  where the first sum ranges over all ordered decompositions $r+1+s=n$
  with $r,s\geq0$, and the second sum ranges over all ordered
  decompositions ${i_1+\cdots+i_r}=p$ and ${j_1+\cdots+j_s}=q$ with positive
  components. Uniqueness is then clear, and existence is precisely the statement
  of \Cref{defprop:restriction-of-scalars}.
\end{proof}

\Cref{prop:restriction-of-scalars-as-algebra-map} admits the following
variants.

\begin{proposition}
  \label{prop:bimaps-vs-algmaps}
  Let $(M,\Astr[M])$ be an $A_\infty$-bimodule over $(A,\Astr)$, and $(N,\Astr[N])$
  an $A_\infty$-bimodule over $(B,\Astr[B])$. Suppose given a cochain
  \[
    \psi=(\psi_1,\psi_2,\psi_3,\dots,\psi_n,\dots)
  \]
  with
  \[
    \psi_n=(\psi_{p,q}\colon A^{\otimes p}\otimes M\otimes A^{\otimes
      q}\longrightarrow N)_{p+1+q=n}.
  \]
  Then,
  \[
    \psi\colon(M,\Astr[M])\longrightarrow(N,\Astr[\varphi^*N])
  \]
  is an $A_\infty$-morphism of $A_\infty$-bimodules over $(A,\Astr)$ if and only
  if
  \[
    \varphi\oplus\psi\colon(A\oplus M,\Astr[A\ltimes M])\longrightarrow(B\oplus
    N,\Astr[B\ltimes N])
  \]
  is an $A_\infty$-morphism of $A_\infty$-algebras
\end{proposition}
\begin{proof}
  The claim follows by noticing that both statements correspond to the validity
  of the exact same set of $A_\infty$-morphism equations.
\end{proof}

We now turn our attention to the compatibility between restriction of scalars
and the passage to shifts and linear duals of $A_\infty$-bimodules.

\begin{proposition}
  \label{prop:compatibility-restriction}
  Let $(M,\Astr[M])$ be an $A_\infty$-bimodule over $(B,\Astr[B])$. Then
  \[
    (M(n),\Astr[\varphi^*(M)(n)])=(M(n),\Astr[\varphi^*(M(n))]),\qquad n\in\ZZ,
  \]
  and
  \[
    (DM,\Astr[D(\varphi^*M)])=(DM,\Astr[\varphi^*DM]).
  \]
\end{proposition}
\begin{proof}
  Both claims are immediate from the definition. We leave the details to the
  reader.
\end{proof}

\begin{proposition}
  \label{prop:n-well-def}
   Let $(M,\Astr[M])$ be an $A_\infty$-bimodule over $(A,\Astr)$, and $(N,\Astr[N])$
  an $A_\infty$-bimodule over $(B,\Astr[B])$.  Suppose given a cochain
  \[
    \psi=(\psi_1,\psi_2,\psi_3,\dots,\psi_n,\dots),
  \]
  with
  \[
    \psi_n=(\psi_{p,q}\colon A^{\otimes p}\otimes M\otimes A^{\otimes q}\longrightarrow N)_{p+1+q=n},
  \]
  that satisfies the equivalent conditions in \Cref{prop:bimaps-vs-algmaps}.
  Then,
  \[
     \psi^{(n)}=(\psi_1^{(n)},\psi_2^{(n)},\psi_3^{(n)},\dots,\psi_n^{(n)},\dots),
  \]
  also satisfies the equivalent conditions in \Cref{prop:bimaps-vs-algmaps}.
\end{proposition}
\begin{proof}
  Indeed, \Cref{prop:shifting-Ai-morphisms} shows that
  \[
    \psi^{(n)}\colon(M(n),\Astr[M(n)])\longrightarrow(N(n),\Astr[\varphi^*(N)(n)])\stackrel{\ref{prop:compatibility-restriction}}{=}(N(n),\Astr[\varphi^*(N(n))])
  \]
  is an $A_\infty$-morphism of $A_\infty$-bimodules.
\end{proof}

\begin{proposition}
  \label{prop:D-well-def}
  Let $(M,\Astr[M])$ be an $A_\infty$-bimodule over $(A,\Astr)$, and $(N,\Astr[N])$
  an $A_\infty$-bimodule over $(B,\Astr[B])$. Suppose given a cochain
  \[
    \psi=(\psi_1,\psi_2,\psi_3,\dots,\psi_n,\dots),
  \]
  with
  \[
    \psi_n=(\psi_{p,q}\colon A^{\otimes p}\otimes N\otimes A^{\otimes
      q}\longrightarrow M)_{p+1+q=n},
  \]
  such that \[
    \psi\colon(N,\Astr[\varphi^*N])\longrightarrow(M,\Astr[\varphi^*M])
  \]
  is an $A_\infty$-morphism of $A_\infty$-bimodules over $(A,\Astr)$. Then,
  \[
    \psi^{D}=(\psi_1^{D},\psi_2^{D},\psi_3^{D},\dots,\psi_n^{D},\dots),
  \]
  satisfies the equivalent conditions in \Cref{prop:bimaps-vs-algmaps}.
\end{proposition}
\begin{proof}
  Indeed, \Cref{prop:dualising-Ai-morphisms} shows that
  \[
    \psi^D\colon
    (DM,\Astr[DM])\longrightarrow(DN,\Astr[D\varphi^*(M)])\stackrel{\ref{prop:compatibility-restriction}}{=}(DN,\Astr[\varphi^*(DN)])
  \]
  is an $A_\infty$-morphism of $A_\infty$-bimodules.
\end{proof}

\begin{remark}
  The previous discussion can be understood in conceptual terms as follows.
  Consider the category $\AllBims$ of \emph{all} $A_\infty$-bimodules. Its
  objects are pairs $((A,\Astr),(M,\Astr[M]))$ consisting of an
  $A_\infty$-algebra $(A,\Astr)$ and an $A_\infty$-bimodule $(M,\Astr[M])$ over
  it. A morphism
  \[
    ((A,\Astr),(M,\Astr[M]))\longrightarrow((B,\Astr[B]),(N,\Astr[N]))
  \]
  between such pairs is itself a pair $(\varphi,\psi)$ consisting of an
  $A_\infty$-morphism of $A_\infty$-algebras
  \[
    \varphi\colon(A,\Astr)\longrightarrow(B,\Astr[B])
  \]
  and an $A_\infty$-morphism of $A_\infty$-bimodules
  \[
    \psi\colon(M,\Astr[M])\longrightarrow(N,\Astr[\varphi^*N])
  \]
  over $(A,\Astr)$. Composition is induced by the composition law of
  $A_\infty$-morphisms (which we did not define in this article). There is also
  a contravariant version $\AllBims*$ in which $\psi$ has instead the form
  \[
    \psi\colon(M,\Astr[M])\longleftarrow(N,\Astr[\varphi^*N]).
  \]
  Then, there is a globally-defined $n$-fold shift functor
  \begin{align*}
    [n]\colon\AllBims&\stackrel{\sim}{\longrightarrow}\AllBims\\
    ((A,\Astr),(M,\Astr[M]))&\longmapsto((A,\Astr),(M(n),\Astr[M(n)])),
  \end{align*}
  defined using \Cref{defprop:shift_Ai-bimodule} and
  \Cref{prop:shifting-Ai-morphisms}, and a globally-defined linear dual functor
  \begin{align*}
    D\colon\AllBims*&\longrightarrow\AllBims\\
    ((A,\Astr),(M,\Astr[M]))&\longmapsto((A,\Astr),(DM,\Astr[DM])),
  \end{align*}
  defined using \Cref{defprop:dual_Ai-bimodule} and
  \Cref{prop:dualising-Ai-morphisms}. We have effectively verified that these
  functors are well defined on objects and on morphisms, but we did not analyse
  the compatibility with the composition law as this is not strictly necessary
  for our purposes. As a side remark, we observe that $\AllBims$ and $\AllBims*$
  are the two variants of the Grothendieck construction of the $2$-functor
  \[
    \AllAlgs^{\op}\longrightarrow\operatorname{CAT}_\kk,\qquad
    (A,\Astr)\longmapsto\AllBims(A,\Astr),
  \]
  from the category of all $A_\infty$-algebras to the (very large) $2$-category
  of large $\kk$-categories, which associates to an $A_\infty$-algebra its
  category of $A_\infty$-bimodules and to an $A_\infty$-morphism of
  $A_\infty$-algebras the corresponding restriction of scalars.
\end{remark}

\subsection{Minimal models of dg algebras and dg bimodules}
\label{subsec:minimods}

We now review the construction of minimal models for dg algebras. Since explicit
formulas for obtaining minimal models of dg bimodules do not seem to be
available in the literature, we explain how to use
\Cref{prop:Ai-bimodule-vs-Ai-algebra} together with the well-known formulas for
the homotopy transfer of dg algebra structures to obtain them.\footnote{One can
  also establish the existence of minimal models of dg bimodules using Quillen model
  category techniques applied to the category of dg operads---we take this
  alternative approach in~\cite[Section~2]{JM25}.} We use the sign conventions in
\cite{Mar06}, as these are compatible with our conventions for $A_\infty$-structures
(see the last part of~\Cref{rmk:Ai-equations}).

\begin{definition}[{\cite{Kad82}}] Let $A$ be a dg algebra, viewed as an
  $A_\infty$-algebra $(A,\Astr)$ with vanishing higher operations $\Astr<n>$,
  $n\geq3$. A \emph{minimal model} for $A$ is a triple
  $(\H{A},\Astr[\H{A}],\varphi)$ consisting of
  \begin{itemize}
  \item a minimal $A_\infty$-algebra $(\H{A},\Astr[\H{A}])$
  \item together with an $A_\infty$-quasi-isomorphism of $A_\infty$-algebras
    \[
      \varphi\colon
      (\H{A},\Astr[\H{A}])\stackrel{\sim}{\longrightarrow}(A,\Astr)
    \]
    whose linear part $\varphi_1\colon\H{A}\to A$ is a \emph{cocycle selection
      map}, that is
    \[
      \forall x\in\H{A},\qquad [\varphi_1(x)]=x.
    \]
  \end{itemize}
\end{definition}

Since we work over a ground field, minimal models for dg algebras (and even for
$A_\infty$-algebras) always exist. The proof of this fact, known as the Homotopy
Transfer Theorem, goes back to Kadeishvili~\cite{Kad82} and there are beautiful
explicit formulas for the transferred $A_\infty$-algebra structure given in
terms of trees ~\cite{Mer99,KS01,Mar04,Mar06}. We begin by recalling the
necessary terminology.

\begin{definition}
  We define \emph{planar binary (rooted) trees (PBTs)} recursively as follows:
  The empty set is a PBT, which we denote by $()$. If $T_1$ and $T_2$ are PBTs,
  then
  \[
    T\coloneqq\graft{T_1}{T_2}\coloneqq(T_1,T_2)
  \]
  is a PBT, with the \emph{left subtree} $L(T)\coloneqq T_1$ and the \emph{right
    subtree $R(T)\coloneqq T_2$}; by convention, the empty PBT does not have
  (proper) subtrees. Given a PBT, say $T$, its \emph{arity (=number of leaves)}
  is defined recursively as
  \[
    a(T)\coloneqq\begin{cases}
      1& T=()\\
      a(L(T))+a(R(T))&T\neq(),
    \end{cases}
  \]
  and, when $T\neq()$, its \emph{signature} is
  \[
    \sigma(T)\coloneqq a(L(T))(a(R(T))+1).
  \]
  We denote the set of PBTs of arity $n$ by $\PBT[n]$.
\end{definition}

\begin{remark}
  Of course, PBTs have a natural pictorial presentation, with the tree
  $\graft{()}{()}$ depicted as
  
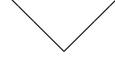
\[
    \begin{tikzpicture}
      \BinaryTree[local bounding box=INIT,label distance={2pt},scale=0.2]{%
        l,r}{3}
    \end{tikzpicture}
  \]
  and with a PBT with at least one non-empty subtree depicted by `grafting' its
  left and right subtrees onto the left and the right leaves of the above PBT,
  respectively. For example, the PBT $\graft{(\graft{()}{()})}{()}$ is depicted
  as
  
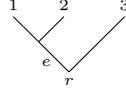
\[
    \begin{tikzpicture}
      \BinaryTree[label distance={2pt},math labels,scale=0.2,font={\tiny}]{%
        :r!l!l!l:1,!l:e!r!r:2,!r!r!r:3}{3}
    \end{tikzpicture}
  \]
  \emph{et cetera}. We could give rigorous definitions of `leaves', `internal
  vertices', `internal edges' and `root vertex' of a PBT, as well as of the
  canonical linear order on the leaves; since these concepts are quite
  intuitive, we do not do so here and trust that the reader will nonetheless be
  able to follow the forthcoming discussion. For example, the above tree has
  three leaves (no pun intended) ordered $1<2<3$, a single internal (unlabelled)
  vertex, one internal edge $e$ and the root vertex $r$. It is elementary to
  verify that the number of internal edges in a non-empty PBT, say $T$, is
  $a(T)-2$ (compare with the degree of an operation in an $A_\infty$-algebra
  structure).
\end{remark}

\begin{construction}
  \label{construction:Markl}
  Let $A$ be a dg algebra with binary operation $\Astr<2>$. Since we work over a
  field, we can find a \emph{strong deformation retraction} of $A$ onto $\H{A}$,
  that is a diagram
  \[
    \begin{tikzcd}
      A\rar[shift left]{p}\ar[loop left]{l}{h}&\H{A}\lar[shift left]{i}
    \end{tikzcd}
  \]
  of homogeneous morphisms of graded vector spaces of degrees
  \[
    |i|=0,\quad |p|=0,\quad |h|=-1,
  \]
  such that $pi=\id[\H{A}]$ and which also satisfy the following equations:
  \begin{align*}
    \partial(i)&=0,&\partial(p)&=0,&\partial(h)&=\id[A]-ip,&ph&=0,&hi&=0,&h^2&=0,
  \end{align*}
  see for example~\cite[Section~1.4]{Val14}. The last three equations are called
  \emph{side conditions} (\cite[Remark~4]{Mar06}) and can always be achieved
  after appropriately modifying the morphisms $i$, $p$, and $h$.
  Following~\cite{Mar06}, we define a minimal $A_\infty$-algebra structure on
  $\H{A}$ with operations
  \[
    \Astr<n>[\H{A}]\coloneqq\sum_{T\in\PBT[n]}p^A\bullet_1\left(\braces{\AstrAux<T>[A]}{i^A,\stackrel{n}{\dots},i^A}\right),\qquad
    n\geq2,
  \]
  where the auxiliary degree $2-n$ operations
  \[
    \AstrAux<T>[A]\colon A^{\otimes n}\longrightarrow A
  \]
  are defined recursively as follows:
  \[
    \AstrAux<T>[A]\coloneqq\begin{cases}
      \Astr<2>&L(T)=(),\ R(T)=()\\
      \Astr<2>[A]\bullet_1(h^A\bullet_1\AstrAux<L(T)>[A])&L(T)\neq(),\ R(T)=()\\
      \Astr<2>[A]\bullet_2(h^A\bullet_1\AstrAux<R(T)>[A])&L(T)=(),\ R(T)\neq()\\
      (\Astr<2>[A]\bullet_1(h^A\bullet_1\AstrAux<L(T)>[A]))\bullet_{a(L(T))+1}(h^A\bullet_1\AstrAux<R(T)>[A])&\text{otherwise}.
    \end{cases}
  \]
  A straightforward computation using \Cref{eq:infinitesimal_composition} shows
  that, in all four cases,
  \begin{align*}
    (p\bullet_1\left(\braces{\AstrAux<T>[A]}{i,\stackrel{n}{\dots},i}\right))=
    (-1)^{\sigma(T)}p\circ\Astr<2>\circ((h^A\circ\AstrAux<L(T)>[A])\otimes(h^A\circ\AstrAux<R(T)>[A]))\circ
    i^{\otimes n},
  \end{align*}
  where we set $\AstrAux<T>[A]\coloneqq\id[A]$ if $T=()$. Let us show this in
  the fourth case, which is the most complicated (but still trivial). Firstly,
  it follows immediately from \Cref{eq:infinitesimal_composition} that post
  composition with a unary operation does not contribute any signs, and neither
  does the pre-composition with unary operations of (vertical) degree $0$.
  Secondly, it also turns out that the infinitesimal composition
  \[
    \Astr<2>[A]\bullet_1(h^A\bullet_1\AstrAux<L(T)>[A])
  \]
  also has no contribution:
  \[
    (a(L(T))-1)(2-1)+(1-a(L(T)))(2-1)=0,
  \]
  where we use that $h^A\bullet_1\AstrAux<L(T)>[A]$ has degree
  $(2-a(L(T)))-1=1-a(L(T))$. Finally, pre-composition with
  $h^A\bullet_1\AstrAux<R(T)>[A]$ in position $a(L(T)+1)$ contributes the sign
  determined by the parity of the expression
  \begin{multline*}
    \Big(a(R)-1\Big)\underbrace{\Big(a(L(T))+1-(a(L(T))+1)\Big)}_{=0}\\+\Big(1-a(R(T))\Big)\Big(a(L(T))+1-1\Big)=(1-a(R(T)))a(L(T)),
  \end{multline*}
  which is equal to the signature $\sigma(T)$ modulo $2$.
\end{construction}

\begin{remark}
  At this juncture, is worth noting that, taken modulo $2$, the signature of a
  PBT is precisely the sign associated with the root vertex in
  \cite[Section~4]{Mar06}, so that our sign conventions agree with those in
  \emph{op.~cit.} We have chosen the above presentation as it relies only on our
  standing sign conventions concerning infinitesimal compositions.
\end{remark}

\begin{remark}
  \label{rmk:decorated_PBTs}
  \Cref{construction:Markl} is better visualised using PBTs that are decorated
  by the data of the chosen strong deformation retraction as follows: Leaves are
  decorated with $i^A$, internal edges are decorated with $h^A$, and internal
  vertices are understood as the application of the binary operation
  $\Astr<2>[A]$. Finally, the root vertex is decorated by $p^A$. The tree is
  then read as a flowchart for describing the corresponding summand of the
  operation. For example, the decorated PBT
  
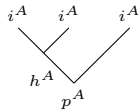
\[
    \begin{tikzpicture}
      \BinaryTree[label distance={2pt},math labels,scale=0.2,font={\tiny}]{%
        :p^A!l!l!l:i^A,!l:h^A!r!r:i^A,!r!r!r:i^A}{3}
    \end{tikzpicture}
  \]
  describes the operation
  \[
    p^A\circ(\Astr<2>[A]\bullet_1(h^A\bullet_1\Astr<2>[A]))\circ(i^A\otimes
    i^A\otimes i^A).
  \]
\end{remark}

\begin{theorem}[{\cite{Mar06}}]
  \label{thm:Markl}
  Let $A$ be a dg algebra. The system of operations $\Astr[\H{A}]$ described in
  \Cref{construction:Markl} endows the cohomology $\H{A}$ of $A$ with the
  structure of a minimal $A_\infty$-algebra. Moreover, the morphisms $i$ and $p$
  extend to mutually homotopy-inverse $A_\infty$-quasi-isomorphisms, so that
  $(\H{A},\Astr[\H{A}])$ yields a minimal model for $A$.
\end{theorem}

We now explain how to leverage \Cref{construction:Markl} to construct minimal
models of dg bimodules.

\begin{definition}
  \label{def:minimal-model-AM}
  Let $A$ be a dg algebra and $M$ a dg $A$-bimodule. A \emph{minimal model} of
  the pair $(A,M)$ consists of
  \begin{itemize}
  \item a minimal model $(\H{A},\Astr[\H{A}],\varphi)$ of $A$,
  \item a minimal $A_\infty$-bimodule $(\H{M},\Astr[\H{M}])$ over
    $(\H{A},\Astr[\H{A}],\varphi)$,
  \item together with a cochain
    \[
      \psi=(\psi_1,\psi_2,\psi_3,\dots,\psi_n,\dots)
    \]
    in
    \[
      \prod_{n\geq1}\bigoplus_{p+1+q=n}\dgHom[\kk]{A^{\otimes
          p}\otimes\H{M}\otimes A^{\otimes q}}{M}
    \]
    whose linear part $\psi_1\colon\H{M}\to M$ is a cocycle selection map and
    such that the component-wise direct sum
    \[
      \varphi\oplus\psi\colon(\H{A}\oplus\H{M},\Astr[\H{A}\ltimes\H{M}])\stackrel{\sim}{\longrightarrow}(A\oplus
      M,\Astr<2>[A\ltimes M])
    \]
    is an $A_\infty$-quasi-isomorphism of $A_\infty$-algebras, where
    $\Astr<2>[A\ltimes M]$ is the square-zero product on $A\oplus M$.
  \end{itemize}
\end{definition}

\begin{construction}
  \label{construction:Markl-bimodules}
  Let $A$ be a dg algebra and $M$ a dg $A$-bimodule. Choose strong deformation
  retractions
  \[
    \begin{tikzcd}
      A\rar[shift left]{p^A}\ar[loop left]{l}{h^A}&\H{A}\lar[shift left]{i^A}
    \end{tikzcd}\qquad\text{and}\qquad
    \begin{tikzcd}
      M\rar[shift left]{p^M}\ar[loop left]{l}{h^M}&\H{M}\lar[shift left]{i^M}
    \end{tikzcd}
  \]
  By forming their direct sum, these determine a strong deformation retraction
  \[
    \begin{tikzcd}
      A\oplus M\rar[shift left]{p^{A\ltimes M}}\ar[loop left]{l}{h^{A\ltimes
          M}}&\H{A}\oplus\H{M}\lar[shift left]{i^{A\ltimes M}},
    \end{tikzcd}
  \]
  where we now view $A\oplus M$ as a dg algebra with the square-zero product
  $\Astr<2>[A\ltimes M]$. \Cref{construction:Markl} then yields a minimal
  $A_\infty$-algebra structure $\Astr[\H{A}\ltimes\H{M}]$ on the graded vector
  space $\H{A}\oplus\H{M}$ which, by construction, satisfies the conditions in
  \Cref{prop:Ai-bimodule-vs-Ai-algebra}\footnote{Indeed, the square-zero product
    $\Astr<2>[A\ltimes M]$ exhibits $M\subseteq A\oplus M$ as a square-zero
    ideal.} and hence it encodes a minimal $A_\infty$-bimodule structure
  $\Astr[\H{M}]$ on $\H{M}$ over the minimal $A_\infty$-algebra
  $(\H{A},\Astr[\H{A}])$. Moreover, it follows immediately from the explicit
  formulas for extending
  \[
    i^{A\ltimes M}= i^A\oplus i^M\qquad\text{and}\qquad p^{A\ltimes
      M}=p^A\oplus p^M
  \]
  to mutually homotopy-inverse $A_\infty$-quasi-isomorphisms established
  in~\cite{Mar06} that in this way one obtains a minimal model of the pair
  $(A,M)$ in the sense of \Cref{def:minimal-model-AM}.
\end{construction}

\begin{remark}
  \label{rmk:decorated_PBTs-bimodules}
  \Cref{rmk:decorated_PBTs} of course also applies to the operations described
  in \Cref{construction:Markl-bimodules}, with an obvious caveat: Each PBT is
  now \emph{marked}, in the sense that exactly one of its leaves corresponds to
  an input from the dg $A$-bimodule $M$, and all possible markings should be
  considered. Furthermore, the decorations should involve data from the strong
  deformation retracts of $A$ and $M$ and either the binary operation
  $\Astr<2>[A]$ or one of the action maps $\Astr<1,0>[M]$ or $\Astr<0,1>[M]$,
  depending on whether the marked leaf flows into that part of the tree or not.
  Also, the root vertex now needs to be labelled by $p^M$ instead of $p^A$. For
  example, the marked decorated PBT
  
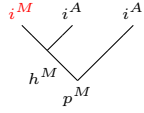
\[
    \begin{tikzpicture}
      \BinaryTree[local bounding box=INIT,label distance={2pt},math
      labels,scale=0.2,font={\tiny}]{%
        :p^M!l!l!l:{\color{red}i^M},!l:h^M!r!r:i^A,!r!r!r:i^A}{3}
    \end{tikzpicture}
  \]
  describes the operation
  \[
    p^M\circ(\underbrace{\Astr<0,1>[M]\bullet_1(h^M\bullet_1\Astr<0,1>[M])}_{=\AstrAux<T,0,2>[\H{M}]})\circ(i^M\otimes
    i^A\otimes i^A),
  \]
  which is a summand of $\Astr<0,2>[\H{M}]$. Notice also that the number of left
  and right inputs from $A$ (in this case $0$ and $2$) do not correspond to the
  arities of the left and right subtrees (in this case $2$ and $1$), as the
  former depend exclusively on the position of the marked leaf.
\end{remark}

\begin{example}
  \label{ex:minimod-diagonal}
  Let $A$ be a dg algebra. Choose a strong deformation retraction
  \[
    \begin{tikzcd}
      A\rar[shift left]{p^A}\ar[loop left]{l}{h^A}&\H{A}\lar[shift left]{i^A}
    \end{tikzcd}
  \]
  It is straightforward to verify, for example using the description in
  \Cref{rmk:decorated_PBTs-bimodules}, that \Cref{construction:Markl-bimodules}
  applied to the above data (taken twice) yields a minimal model of the diagonal
  $A$-bimodule that agrees with the $A_\infty$-bimodule structure from
  \Cref{defprop:diaognal-Ai-bimodule}. Given $n\in\ZZ$, using \Cref{prop:n-well-def}, we see that
  $(\H{A}(n),\Astr[A[n]])$ is part of a minimal model for the dg $A$-bimodule $A(n)$ and,
  using \Cref{prop:D-well-def}, we see that $(D\H{A},\Astr[DA])$ is part of a minimal model
  for the dg $A$-bimodule $DA$.
\end{example}

\section{Kadeishvili-type theorems for bimodule Calabi--Yau algebras}
\label{sec:Kadeishvili}

In this section, using one of the main results in~\cite{JM25},
\Cref{thm:B-bimodules}, we give a proof of \Cref{thm:CY-Kadeikshvili} and
its variant for minimal $A_\infty$-algebras (\Cref{thm:CY-Kadeishvili-Ai}). For
this, we first recall the notions of universal Massey products for
$A_\infty$-algebras and $A_\infty$-bimodules, as well as the companion
Hochschild--Massey cohomology theories.

\subsection{Universal Massey products for minimal $A_\infty$-algebras}

Let $(A,\Astr)$ be a minimal $A_\infty$-algebra. We make the following
observations:
The $A_\infty$-equation for $n=3$ is
\[
  0=\braces{\Astr<2>}{\Astr<2>}\stackrel{\eqref{eq:Gerstenhaber_square-HC}}{=}\Sq(\Astr<2>),
\]
which is simply the associativity of the binary operation,
see~\Cref{rmk:Ai-equations}. The $A_\infty$-equation for $n=4$ is
\[
  0=\braces{\Astr<2>}{\Astr<3>}+\braces{\Astr<3>}{\Astr<2>}\stackrel{\eqref{eq:Gerstenhaber_bracket-HC}}{=}[\Astr<2>,\Astr<3>]\stackrel{\eqref{eq:Hd}}{=}\Hd(\Astr<3>),
\]
so that $\Astr<3>\in\HC[3]<-1>{A}$ is a \emph{cocycle}. Moreover, the
$A_\infty$-equation for $n=5$ is
\begin{align*}
  0&=\braces{\Astr<2>}{\Astr<4>}+\braces{\Astr<4>}{\Astr<2>}+\braces{\Astr<3>}{\Astr<3>}\\
   &=[\Astr<2>,\Astr<4>]+\Sq(\Astr<3>)\\
   &=\Hd(\Astr<4>)+\Sq(\Astr<3>),
\end{align*}
so that
\[
  0=\Sq(\Hclass{\Astr<3>})\in\HH[5]<-2>{A}.
\]
More generally, assume that the minimal $A_\infty$-algebra $(A,\Astr)$ is
$d$-sparse, $d\geq1$. In this case the first possibly non-zero higher operation
is $\Astr<d+2>$, and the $A_\infty$-equation for $n=d+3$ is
\[
  0=\braces{\Astr<2>}{\Astr<d+2>}+\braces{\Astr<d+2>}{\Astr<2>}\stackrel{\eqref{eq:Gerstenhaber_bracket-HC}}{=}[\Astr<2>,\Astr<d+2>]\stackrel{\eqref{eq:Hd}}{=}\Hd(\Astr<d+2>),
\]
so that $\Astr<d+2>\in\HC[d+2]<-d>{A}$ is a cocycle. The next non-trivial
$A_\infty$-equation is that for $n=2d+3$, which yields
\begin{align*}
  0&=\braces{\Astr<2>}{\Astr<2d+2>}+\braces{\Astr<2d+2>}{\Astr<2>}+\braces{\Astr<d+2>}{\Astr<d+2>}\\
   &=[\Astr<2>,\Astr<2d+2>]+\Sq(\Astr<d+2>)\\
   &=\Hd(\Astr<2d+2>)+\Sq(\Astr<d+2>),
\end{align*}
so that
\[
  0=\Sq(\Hclass{\Astr<d+2>})\in\HH[2d+3]<-2d>{A}.
\]
These considerations suggest that the following definition is meaningful.

\begin{definition}
  \label{def:UMP}
  Let $(A,\Astr)$ be a $d$-sparse minimal $A_\infty$-algebra, $d\geq1$. The
  \emph{universal Massey product (UMP) of length $d+2$} is the Hochschild class
  \[
    \Hclass{\Astr<d+2>}\in\HH[d+2]<-d>{A}.
  \]
  The Gerstenhaber square of the UMP vanishes: $\Sq(\Hclass{\Astr<d+2>})=0$.
\end{definition}

\begin{remark}
  Let $(A,\Astr)$ be a $d$-sparse minimal $A_\infty$-algebra, $d\geq1$. It
  follows from the $A_\infty$-morphism equations that its UMP of length $d+2$ is
  invariant under gauge $A_\infty$-isomorphisms,
  see~\cite[Proposition~4.4.4]{JKM22}.
\end{remark}

\begin{remark}
  The UMP of length $d+2$ is also considered in \cite[Ch.~3]{Sei15}. For $d=1$,
  it has been investigated for example in \cite{BKS04,Kad82} and plays a crucial
  role in \cite{Mur22,JKM22} as well. An operadic generalisation is also one of
  the key concepts used in~\cite{JM25}.
\end{remark}

\subsection{Hochschild--Massey cohomology}

Let $(A,\Astr)$ be a $d$-sparse minimal $A_\infty$-algebra. Recall from
\Cref{def:UMP} its UMP of length $d+2$
\[
  \Hclass{\Astr<d+2>}\in\HH[d+2]<-d>{A},\qquad \Sq(\Hclass{\Astr<d+2>})=0.
\]
The relations for the Gerstenhaber bracket yield
\[
  [\Hclass{\Astr<d+2>},[\Hclass{\Astr<d+2>},x]]\stackrel{\eqref{eq:Gerstenhaber_square-relations}}{=}[\underbrace{\Sq(\Hclass{\Astr<d+2>})}_{=0},x]=0,\qquad
  x\in\HH{A}.
\]
This observation serves as motivation for the following definition.

\begin{definition}[{\cite[Definition~5.2.5]{JKM22}}]
  \label{def:HMC}
  The \emph{Hochschild--Massey cochain complex of $(A,\Hclass{\Astr<d+2>})$} is
  the bigraded vector space with the components
  \begin{equation}
    \label{eq:HMC}
    \HMC[p]<q>{A}\coloneqq\HH[p]<q>{A},\qquad p\geq2,\ q\in\ZZ,
  \end{equation}
  which is equipped with the bidegree $(d+1,-d)$ \emph{Hochschild--Massey
    differential}
  \begin{align}
    \label{eq:HMd}
    \begin{split}
      \HMd\colon\HMC[p]<q>{A}&\longrightarrow\HMC[p+d+1]<q-d>{A}\\
      x&\longmapsto[\Hclass{\Astr<d+2>},x],
    \end{split}
  \end{align}
  except in bidegree $(d+1,-d)$ where it is given by\footnote{Recall that, when
    $\chark(\kk)\neq2$, we have $x^2=0$ for elements of odd total degree by the
    graded commutativity of the cup product in $\HH{A}$. In any case, the
    Hochschild--Massey differential always squares to zero,
    see~\cite[Remark~5.2.6.]{JKM22}.}
  \begin{align}
    \label{eq:HMd-2}
    \begin{split}
      \HMd\colon\HMC[d+1]<-d>{A}&\longrightarrow\HMC[2d+2]<-2d>{A}\\
      x&\longmapsto[\Hclass{\Astr<d+2>},x]+x^2.
    \end{split}
  \end{align}
  The \emph{Hochschild--Massey cohomology of $(A,\Hclass{\Astr<d+2>})$} is
  \begin{equation}
    \label{eq:HMH}
    \HMH{A}\coloneqq\H[\bullet,*]{\HMC{A}}.
  \end{equation}
\end{definition}

The following Kadeishvili-type theorem, which is implicit in \cite{Mur22} in the
case $d=1$, is our main motivation for considering the Hochschild--Massey
cohomology of a sparse minimal $A_\infty$-algebra. Its proof relies crucially on
the obstruction theory developed by the second-named author in \cite{Mur20b},
see also \cite[Section~5]{JKM22}. It is also worth mentioning that the
Hochschild--Massey cochain complex is embedded in an extension of a
Bousfield--Kan fringed spectral sequence introduced in \cite{Mur20b}, which
explains why its differential has a somewhat unusual expression in bidegree
$(d+1,-d)$, where it is not given by a linear map.

\begin{theorem}[{\cite[Theorem~B]{JKM22}}]
  \label{thm:B}
  Let $(A,\Astr)$ be a $d$-sparse minimal $A_\infty$-algebra, $d\geq1$. Suppose
  that
  \[
    \HMH[p+2]<-p>{A}=0,\qquad p>d.
  \]
  If $(A,m)$ is any minimal $A_\infty$-algebra structure with the same
  underlying associative algebra structure, $m_2=\Astr<2>$, and with the same
  UMP of length $d+2$,
  \[
    \Hclass{m_{d+2}}=\Hclass{\Astr<d+2>}\in\HH[d+2]<-d>{A},
  \]
  then there exists a gauge $A_\infty$-isomorphism
  $(A,\Astr)\stackrel{\sim}{\to}(A,m)$.
\end{theorem}

\subsection{Universal Massey products for minimal $A_\infty$-bimodules}

Let $(A,\Astr)$ be a $d$-sparse minimal $A_\infty$-algebra and $(M,\Astr[M])$ a
$d$-sparse minimal $A_\infty$-bimodule over it. Notice that the corresponding
$A_\infty$-algebra $(A\oplus M,\Astr[A\ltimes M])$ is also $d$-sparse and
minimal and hence it has a well-defined UMP of length $d+2$, namely the class
\[
  \Hclass{\Astr<d+2>[A\ltimes M]}\in\HH[d+2]<-d>{A\oplus M},
\]
where we regard $A\ltimes M=A\oplus M$ as an associative algebra with the
square-zero product $\Astr<2>[A\ltimes M]$. Although it would seem natural to
consider this class as the bimodule analogue of the UMP for sparse
$A_\infty$-algebras, in order to obtain the correct results we must modify the
ambient cohomology in which it lives. The reason for this is that the Hochschild
cohomology $\HH{A\ltimes M}$ relates to the obstruction theory for the existence
and uniqueness of \emph{arbitrary} $A_\infty$-algebra structures on $(A\oplus
M,\Astr<2>[A\ltimes M])$, which need not satisfy the conditions in
\Cref{prop:Ai-bimodule-vs-Ai-algebra}, and therefore need not correspond to
$A_\infty$-bimodule structures on the graded $A$-bimodule $(M,\Astr<2>[M])$.
This technical point is the \emph{raison d'être} for considering the above
cohomology class in the bimodule Hochschild cohomology of the pair $(A,M)$ and
not in the Hochschild cohomology of $A\ltimes M$, see \cite[Section~2]{JM25} for
more details on this.

\begin{definition}
  \label{def:bUMP}
  Let $(A,\Astr)$ be a $d$-sparse minimal $A_\infty$-algebra and $(M,\Astr[M])$
  a $d$-sparse minimal $A_\infty$-bimodule over it (recall that the sparsity
  condition is vacuous for $d=1$). The \emph{bimodule universal Massey product
    (bimodule UMP) of length $d+2$} of $(M,\Astr[M])$ is the cocycle
  \[
    \Hclass{\Astr<d+2>[A\ltimes
      M]}=\Hclass{\Astr<d+2>+\Astr<d+2>[M]}\in\RelBimHH[d+2]<-d>{A}{M}.
  \]
  Similar to the algebra case, the bimodule Gerstenhaber square of the
  bimodule UMP vanishes: $\Sq(\Hclass{\Astr<d+2>[A\ltimes M]})=0$.
\end{definition}

\begin{remark}
  \label{rmk:bUMP}
  We warn the reader of the following possible sources of confusion in
  \Cref{def:bUMP}. Firstly, while the cochain
  \[
    \Astr<d+2>[A\ltimes M]=\Astr<d+2>+\Astr<d+2>[M]\in\RelBimHC[d+2]<-d>{A}{M}
  \]
  is a cocycle, neither of the cochains $\Astr<d+2>$ and $\Astr<d+2>[M]$ need be
  cocycles in $\RelBimHC{A}{M}$, notwithstanding the fact that
  $\Astr<d+2>\in\HC[d+2]<-d>{A}$ is a cocycle Hochschild complex of $A$. The reason for this is the
  lower-triangular form of the bimodule Hochschild differential, see
  \Cref{eq:dRelBim-matrix}.
\end{remark}

\subsection{Massey bimodule cohomology}

Let $(A,\Astr)$ be a $d$-sparse minimal $A_\infty$-algebra and $(M,\Astr[M])$ a
$d$-sparse minimal $A_\infty$-bimodule over it. In view of \Cref{def:bUMP} and
the discussion preceding it, the following definition is not surprising.

\begin{definition}
  \label{def:RelBimHMC}
  The \emph{Massey bimodule cochain complex of $(M,\Astr[M])$} is the bigraded vector space with the
  components
  \begin{align}
    \label{eq:RelBimHMC}
    \BimHMC[n]<r>{M}\coloneqq\BimHH[n]<r>{M},\qquad n\geq1,\ r\in\ZZ,
  \end{align}
  which is endowed with the bidegree $(d+1,-d)$ \emph{
    Massey bimodule differential} given, without exceptions, by
  \begin{align}
    \label{eq:BimHMd}
    \begin{split}
      \BimHMd\colon\BimHMC[n]<r>{M}&\longrightarrow\BimHMC[n+d+1]<r-d>{M}\\
      x&\longmapsto[\Hclass{\Astr[A\ltimes M]},x].
    \end{split}
  \end{align}
  The \emph{Massey bimodule cohomology of $(M,\Astr[M])$} is
  \begin{equation}
    \label{eq:RelBimHMH}
    \BimHMH{M}\coloneqq\H[\bullet,*]{\BimHMC{M}}.
  \end{equation}
\end{definition}

We are now ready to state the key technical result that we leverage in our proof
of \Cref{thm:CY-correspondence}. It can be regarded as an almost-formality
theorem for minimal $A_\infty$-bimodules and should be compared with \Cref{thm:B} which
treats the case of minimal $A_\infty$-algebras. We remind the reader of \Cref{rmk:bUMP}.

\begin{theorem}[{\cite[Theorem~6.2.7]{JM25}}]
  \label{thm:B-bimodules} Let $(A,\Astr)$ be a $d$-sparse minimal
  $A_\infty$-algebra and $(M,\Astr[M])$ a $d$-sparse minimal $A_\infty$-bimodule
  over it, $d\geq1$. Suppose that
  \[
    \BimHMH[p+1]<-p>{M}=0,\qquad p>d.
  \]
  If $(M,m)$ is any minimal $A_\infty$-bimodule structure over $(A,\Astr)$ with
  the same underlying graded $A$-bimodule structure, $m_2=\Astr<2>[M]$, and such
  that
  \[
    0=\Hclass{\Astr<d+2>[M]-m_{d+2}}\in\BimHH[d+1]<-d>{M},
  \]
  then there exists a gauge $A_\infty$-isomorphism of $A_\infty$-bimodules
  $(M,\Astr[M])\stackrel{\sim}{\to}(M,m)$.
\end{theorem}

\begin{remark}
  Continuing the discussion in \Cref{rmk:bUMP}, in the context of
  \Cref{thm:B-bimodules}, it follows from \Cref{prop:RelBimHC-all} that the
  condition
  \[
    0=\Hclass{\Astr<d+2>[M]-m_{d+2}}\in\BimHH[d+1]<-d>{M},
  \]
  is equivalent to the more natural condition
  \[
    \Hclass{\Astr<d+2>[A]+\Astr<d+2>[M]}=\Hclass{\Astr<d+2>[A]+m_{d+2}}\text{ in
    }\RelBimHH[d+2]<-d>{A}{M}
  \]
  whenever the graded $\HH{A}$-bimodule $\BimHH{M}$ is graded-symmetric. This is
  the case when $M=A$ is the diagonal $A$-bimodule since $\HH{A}$ is
  graded-commutative and hence its diagonal bimodule is graded-symmetric.
\end{remark}

\begin{remark}
  Keeping \Cref{thm:B-bimodules} in mind, the introduction of the bimodule
  Hochschild--Massey cochain complex (\Cref{def:RelBimHMC}) is motivated by the
  following considerations, the details of which can be found in our companion
  article~\cite{JM25}. Firstly, associated with the $d$-sparse graded algebra $A$
  there is a space (Kan complex)
  \[
    \Map[\dgOp]{\opA}{\opEnd{A}}
  \]
  whose set of points is the set of minimal algebra $A_\infty$-structures on the
  underlying graded vector space of $A$; here, $\opA$ is the $A_\infty$-operad,
  $\opEnd{A}$ is the endomorphism (graded) operad of $A$, and the above space is
  the Dwyer--Kan mapping space in the model category of differential graded
  (non-symmetric) operads endowed with the (transferred) projective model
  structure~\cite{Lyu11,Mur11}. Using an appropriate cylinder object
  $I\opA$~\cite{Mur16}, it is not difficult to show that the set of connected
  components
  \[
    \pi_0(\Map[\dgOp]{\opA}{\opEnd{A}})
  \]
  identifies with the set of minimal $A_\infty$-algebra structures on $A$
  considered up to gauge $A_\infty$-isomorphism~\cite[Proposition~6.9]{Mur20b}.
  Secondly, given a $d$-sparse graded $A$-bimodule $M$, we consider the space
  \[
    \Map[\dgOp]{\opA}{\opLinEnd{A}{M}},
  \]
  whose set of points is now the set of minimal $A_\infty$-algebra structures on
  $A\oplus M$ that satisfy the conditions in
  \Cref{prop:Ai-bimodule-vs-Ai-algebra}; here, $\opLinEnd{A}{M}$ is the linear
  endomorphism (graded) operad of the pair $(A,M)$ introduced in~\cite{BM09}.
  Finally, the linear endomorphism operad is equipped with a canonical
  projection $\opLinEnd{A}{M}\twoheadrightarrow\opEnd{A}$ that induces a
  fibration
  \[
    \Map[\dgOp]{\opA}{\opLinEnd{A}{M}}\twoheadrightarrow\Map[\dgOp]{\opA}{\opEnd{A}}
  \]
  between the above mapping spaces. Given a point $\Astr$ in
  $\Map[\dgOp]{\opA}{\opEnd{A}}$, we consider the induced fibre sequence
  \[
    \Str{M}\rightarrowtail\Map[\dgOp]{\opA}{\opLinEnd{A}{M}}\twoheadrightarrow\Map[\dgOp]{\opA}{\opEnd{A}},
  \]
  where $\Str{M}$ is, by construction, the space of minimal $A_\infty$-bimodule
  structures on the underlying graded vector space of $M$ over the given minimal
  $A_\infty$-algebra $(A,\Astr)$. From this perspective, \Cref{thm:B-bimodules}
  provides sufficient conditions for two minimal $A_\infty$-bimodule structures
  $\Astr+\Astr[M],\Astr+\Astr[]\in\Str{M}$ to lie in the same connected
  component, that is for the minimal $A_\infty$-bimodules $(M,\Astr+\Astr[M])$
  and $(M,\Astr+\Astr[])$ to be gauge $A_\infty$-isomorphic. The bimodule
  Hochschild--Massey cochain complex then emerges as follows. There is a
  commutative diagram
  \[
    \begin{tikzcd}[row sep=small,column sep=small]
      \Str{M}\rar[tail]&\Map[\dgOp]{\opA}{\opLinEnd{A}{M}}\rar[two heads]&\Map[\dgOp]{\opA}{\opEnd{A}}\\[-2em]
      \vdots\dar[two heads]&\vdots\dar[two heads]&\vdots\dar[two heads]\\
      \Str[\opA[n]]<\Astr<\leq n>>{M}\rar[tail]\dar[two heads]&\Map[\dgOp]{\opA[n]}{\opLinEnd{A}{M}}\rar[two heads]\dar[two heads]&\Map[\dgOp]{\opA[n]}{\opEnd{A}}\dar[two heads]\\
      \vdots\dar[two heads]&\vdots\dar[two heads]&\vdots\dar[two heads]\\
      \Str[\opA[3]]<\Astr<\leq 3>>{M}\rar[tail]\dar[two heads]&\Map[\dgOp]{\opA[3]}{\opLinEnd{A}{M}}\rar[two heads]\dar[two heads]&\Map[\dgOp]{\opA[3]}{\opEnd{A}}\dar[two heads]\\
      \Str[\opA[2]]<\Astr<\leq
      2>>{M}\rar[tail]&\Map[\dgOp]{\opA[2]}{\opLinEnd{A}{M}}\rar[two
      heads]&\Map[\dgOp]{\opA[2]}{\opEnd{A}}
    \end{tikzcd}
  \]
  where $\opA[n]$ is the $A_n$-operad (whose algebras are truncated
  $A_\infty$-algebras with operations only up to and including arity $n$), whose
  rows are fibre sequences based at the apparent base points induced by the
  minimal $A_\infty$-algebra structure $\Astr$ and its truncations $\Astr<\leq
  n>$, whose columns are based towers of fibrations, and whose top row is
  obtained by passing to the homotopy limit of each of the towers. Each of these
  towers of fibrations gives rise to an extended Bousfield--Kan spectral
  sequence\footnote{In the case of the rightmost column, this spectral sequence
    was constructed by the second-named author for $d=1$ in~\cite{Mur20b} and
    for $d\geq1$ by the authors in~\cite{JM25}, see
    also~\cite[Sec.~IX.4]{BK72}. } for computing the homotopy groups of the
  homotopy limit at the top of the tower. These extended spectral sequences are
  fringed: they are not fully defined in the right half-plane, and some of their
  terms are possibly non-abelian groups while others are plain pointed sets. In
  the case of the leftmost column, the relevant part of the $(d+1)$-st page of
  the spectral sequence has the terms
  \[
    E_{d+1}^{n-1,r}=\BimHH[n]<r>{M}=\BimHMC[n]<r>{M},\qquad n\geq 2,\ r\in\ZZ,
  \]
  and the bidegree $(d+1,-d)$ differentials
  \[
    \dSS[d+1]\colon\colon E_{d+1}^{n-1,r}\longrightarrow E_{d+1}^{n+d,r-d}
  \]
  is given by $x\mapsto[\Hclass{\Astr[A\ltimes M]},x]$ (that one can compute
  far-away pages and differentials of the spectral sequence is a consequence of
  the $d$-sparsity of $A$ and $M$). The proof of \Cref{thm:B-bimodules} relies
  heavily on the use of the above spectral sequence, as it permits us to control
  the collapse of the $\lim\nolimits^1$ term in the Milnor short exact sequence
  of pointed sets~\cite[Theorem~IX.3.1]{BK72}
  \[
    *\rightarrow \lim\nolimits^1\pi_1(\Str[\opA[n]]<\Astr<\leq n>>{M})\rightarrow\pi_0(\Str{M})\rightarrow\varprojlim\pi_0(\Str[\opA[n]]<\Astr<\leq n>>{M})\rightarrow *
  \]
  Finally, notice that if the $\lim\nolimits^1$ term collapses, then two
  $A_\infty$-bimodule structures
  \[
    \Astr+\Astr[M],\ \Astr+\tilde{m}^M\in\Str{M}
  \]
  are gauge $A_\infty$-isomorphic if and only if all of their truncations to an
  $A_n$-bimodule structure are gauge $A_n$-isomorphic (which is in general a
  weaker statement). This turns out to suffice for proving
  \Cref{thm:B-bimodules}.
\end{remark}

\subsection{A Kadeishvili-type theorem for bimodule right Calabi--Yau
  $A_\infty$-algebras}

We state the following definition for the sake of completeness:

\begin{definition}[Kontsevich, see also~{\cite[Section~10.2]{KS09}}]
  \label{def:right-CY-as-bimodule-Ai}
  Let $(A,\Astr)$ be an $A_\infty$-algebra whose cohomology (with respect to
  $\Astr<1>$) is degree-wise finite dimensional and $n\in\ZZ$.
  \begin{enumerate}
  \item A \emph{bimodule right $n$-Calabi--Yau structure} on $(A,\Astr)$ is an
    $A_\infty$-quasi-iso\-mor\-phi\-sm of $A_\infty$-bimodules
    \[
      \varphi\colon(A(n),\Astr[A[n]])\stackrel{\sim}{\longrightarrow}(DA,\Astr[DA]).
    \]
  \item We say that $(A,\Astr)$ is \emph{right $n$-Calabi--Yau as a bimodule} if
    it admits a bimodule right $n$-Calabi--Yau structure.
  \end{enumerate}
\end{definition}

We remind the reader of \Cref{defprop:upsilon} and of the BV operator introduced
in \Cref{def:BV}.
 
\begin{proposition}
  \label{prop:CY-Ai-aux}
  Let $(A,\Astr)$ be a $d$-sparse minimal $A_\infty$-algebra, $d\geq1$, such
  that the graded algebra $\H{A}$ is unital. Suppose also that there exists an
  isomorphism of graded $A$-bimodules
  \[
    \varphi\colon A(n)\stackrel{\sim}{\longrightarrow}DA
  \]
  for some $n\in d\ZZ$. Then,
  \[
    0=\Hclass{\Astr<d+2>[A[0]]-\Upsilon(\Astr<d+2>[A[0]])}\in\BimHH[d+1]<-d>{A}\cong\HH[d+1]<-d>{A}
  \]
  if and only if
  \[
    0=\BV(\Hclass{\Astr<d+2>})\in\HH[d+1]<-d>{A}.
  \]
\end{proposition}
\begin{proof}
  We begin by recalling an important fact: By \cite[Théorème~2.1.1]{Lef03} or
  \cite[Lemma~2.1]{Sei08}, the minimal $A_\infty$-algebra $(A,\Astr)$ is gauge
  $A_\infty$-isomorphic to a strictly unital minimal $A_\infty$-algebra. Since
  the algebra UMPs and bimodule UMPs involved are invariant under gauge
  $A_\infty$-isomorphisms, we may assume $(A,\Astr)$ to be strictly unital.

  In order to prove the claim in the proposition, recall from \Cref{coro:kappa}
  that there is an isomorphism of Gerstenhaber algebras
  \[
    \H[\bullet,*]{\kappa}\colon\RelBimHH{A}{A}\stackrel{\sim}{\longrightarrow}\HH{A}<\varepsilon>/(\varepsilon^2),
  \]
  where $\varepsilon$ has bidegree $(1,0)$.
  This isomorphism maps
  \begin{align*}
    \Hclass{\Astr<d+2>[A[0]]-\Upsilon(\Astr<d+2>[A[0]])}&\longmapsto\Hclass{-\kappa(\Upsilon(\Astr<d+2>[A[0]]))}\stackrel{\eqref{rmk:BV-motivation}}{=}(-1)^{d}\BV(\Hclass{\Astr<d+2>})\cdot\varepsilon,
  \end{align*}
  since
  \[
    \kappa(\Astr<d+2>[A[0]])\stackrel{\eqref{eq:kappa}}{=}\sum_{i=1}^{d+2}\pm\underbrace{\Astr<d+2>(\id[A]^{i-1}\otimes1_A\otimes\id[A]^{\otimes(d+2-i)})}_{=0}\cdot\varepsilon=0
  \]
  already at the cochain level since the minimal $A_\infty$-algebra $(A,\Astr)$
  is strictly unital (\Cref{rmk:strict_unitality}). Since $\kappa$ is an
  isomorphism, it follows that
  \[
    0=\Hclass{\Astr<d+2>[A[0]]-\Upsilon(\Astr<d+2>[A[0]])}\in\BimHH[d+1]<-d>{A}
  \]
  if and only if
  \[
    0=\BV(\Hclass{\Astr<d+2>})\in\HH[d+1]<-d>{A},
  \]
  which is what we needed to show.
\end{proof}

We are ready to prove the following result, which can be regarded as a
Kadeishvili-type theorem for bimodule right Calabi--Yau algebras.

\begin{theorem}
  \label{thm:CY-Kadeishvili-Ai}
  Let $(A,\Astr)$ be a $d$-sparse minimal $A_\infty$-algebra, $d\geq1$, such
  that the graded algebra $\H{A}$ is unital. Suppose that the Massey
  bimodule cohomology of the diagonal $A_\infty$-bimodule satisfies
  \[
    \BimHMH[p+1]<-p>{A}<A\ltimes A(0)>=0,\qquad p>d.
  \]
  Suppose also that there exists an isomorphism of graded $A$-bimodules
  \[
    \varphi\colon A(n)\stackrel{\sim}{\longrightarrow}DA
  \]
  for some $n\in d\ZZ$. If
  \[
    0=\BV(\Hclass{\Astr<d+2>})\in\HH[d+1]<-d>{A},
  \]
  then $\varphi$ is the linear part of an $A_\infty$-isomorphism of $A_\infty$-bimodules
  \[
    (A(n),\Astr[A[n]])\stackrel{\sim}{\longrightarrow}(DA,\Astr[DA]),
  \]
  so that $(A,\Astr)$ is right $n$-Calabi--Yau as a bimodule.
\end{theorem}
\begin{proof}
  \Cref{thm:B-bimodules} and \Cref{prop:CY-Ai-aux} yield the existence of a gauge
  $A_\infty$-isomorphism of $A_\infty$-bimodules
  \[
    \gamma\colon(A(0),\Astr[A[0]])\stackrel{\sim}{\longrightarrow}(A(0),\Upsilon(\Astr[A[0]])).
  \]
  \Cref{prop:DM-shift} yields further (strict) $A_\infty$-isomorphisms
  \[
    \begin{tikzcd}
      (A(0),\Astr[A[0]])\rar{\gamma}\dar[dotted]&(A(0),\Upsilon(\Astr[A[0]]))\dar{\wr}\\
      ((DA)(-n),\Astr[(DA)(-n)])&(D(A(n)),\Astr[D(A(n))])\lar[swap]{\sim}
    \end{tikzcd}
  \]
  The linear part of the composite $A_\infty$-isomorphism
  \[
    (A(0),\Astr[A[0]])\stackrel{\sim}{\longrightarrow}((DA)(-n),\Astr[(DA)(-n)])
  \]
  is given by
  \[
    x\stackrel{\gamma}{\longmapsto}x\stackrel{\eqref{eq:prop:DM-shift:it:upsilon}}{\longmapsto} (-1)^{n|x|}\varphi(\s[n]x)\s[-n]\stackrel{\eqref{eq:prop:DM-shift:it:DM-shift}}{\longmapsto}\s[-n]\varphi(x).
  \]
  Applying \Cref{prop:shifting-Ai-morphisms}, we
  obtain the desired $A_\infty$-isomorphism of $A_\infty$-bimodules
  \[
    (A(n),\Astr[A[n]])\stackrel{\sim}{\longrightarrow}(DA,\Astr[DA]).\qedhere
  \]
\end{proof}

\begin{remark}
  \label{rmk:BimHMC-HMC}
  Let $(A,\Astr)$ be a strictly unital minimal $A_\infty$-algebra. Recall again
  from \Cref{coro:kappa} the isomorphism of Gerstenhaber algebras
  \[
    \H[\bullet,*]{\kappa}\colon\RelBimHH{A}{A}\stackrel{\sim}{\longrightarrow}\HH{A}<\varepsilon>/(\varepsilon^2).
  \]
  It follows that there is an induced isomorphism of bigraded vector
  spaces\footnote{Recall that, by definition,
    \[
      \HMC[1]{A}=0
    \]
    while
    \[
      \BimHMC[1]{A}<A\ltimes A(0)>=\BimHH[1]{A}.
    \]}
  \begin{equation}
    \label{eq:BimHMC-HMC}
    \BimHMC[n]<r>{A}<A\ltimes A(0)>\cong\HMC[n]<r>{A},\qquad n\geq2,\ r\in\ZZ,
  \end{equation}
  where the right-hand side is the Hochschild--Massey cochain complex of the
  $A_\infty$-algebra $(A,\Astr)$, and the left-hand side is the Massey
  bimodule cochain complex of the diagonal $A_\infty$-bimodule.
  Moreover, as in the proof of \Cref{thm:CY-Kadeishvili-Ai}, we notice that
  \[
    \kappa(\Astr<d+2>+\Astr<d+2>[A[0]])=\Astr<d+2>+\underbrace{\kappa(\Astr<d+2>[A[0]])}_{=0}=\Astr<d+2>
  \]
  by strict unitality, so that at the cohomological level the \emph{bimodule}
  UMP of the diagonal $A_\infty$-bimodule of $A$ corresponds to the \emph{algebra} UMP
  of $A$:
  \[
    \kappa\colon\Hclass{\Astr<d+2>+\Astr<d+2>[A[0]]}\longmapsto\Hclass{\Astr<d+2>}.
  \]
  It follows that the isomorphisms in \eqref{eq:BimHMC-HMC}
  are compatible with the corresponding differentials in all bidegrees with the
  exception, when $\chark(\kk)=2$, of the bidegree $(d+1,-d)$, where the
  Hochschild--Massey differential has a different expression, see
  \Cref{eq:HMd-2}. In particular, there are isomorphisms
  \[
    \BimHMH[p+1]<-p>{A}<A\ltimes A(0)>\cong\HMH[p+1]<-p>{A},\qquad p>d+1
  \]
  From the perspective of \Cref{thm:CY-Kadeishvili-Ai}, the crucial difference is then in the
  Massey bimodule cohomology
  \[
    \BimHMH[d+2]<-d-1>{A}<A\ltimes A(0)>,
  \]
  which can be identified, using $\kappa$, with the cohomology of the three-term
  complex
  \[
    \HH[1]<0>{A}\longrightarrow\HH[d+2]<-d-1>{A}\longrightarrow\HH[2d+3]<-2d-2>{A}
  \]
  in which both maps are given by the formula $x\mapsto[\Hclass{\Astr<d+2>},x]$.
\end{remark}

\subsection{A Kadeishvili-type theorem for bimodule right Calabi--Yau dg algebras}
\label{subsec:UMP-dg}

Any possible minimal model of a dg algebra admits a direct gauge
$A_\infty$-isomorphism to one obtained as in \Cref{thm:Markl}. In particular, the
following definition is consistent.

\begin{definition}
  \label{def:UMP-dgAs}
  Let $A$ be a cohomologically $d$-sparse dg algebra, $d\geq1$. The
  \emph{universal Massey product (UMP) of length $d+2$} of $A$ is defined as
  the UMP of any of its minimal models $(\H{A},\Astr[\H{A}])$:
  \[
    \Hclass{\Astr<d+2>[A]}\coloneqq\Hclass{\Astr<d+2>[\H{A}]}\in\HH[d+2]<-d>{\H{A}}.
  \]
\end{definition}

Similarly, minimal models of $A_\infty$-bimodules are also unique up to gauge
$A_\infty$-isomorphism, see~\cite[Remark~2.3.8]{JM25}.\footnote{This fact does not follow
  immediately from corresponding statement for $A_\infty$-algebras since one
  needs to guarantee that the gauge $A_\infty$-isomorphisms corresponds to an
  $A_\infty$-bimodule isomorphism, compare with
  \Cref{def:Ai-bimodule-morphism}.} Thus, the following definition is also
consistent.

\begin{definition}
  \label{def:UMP-dgBims}
  Let $A$ be a cohomologically $d$-sparse dg algebra and $M$ a cohomologically
  $d$-sparse dg $A$-bimodule, $d\geq1$. The \emph{universal Massey product
    (UMP) of length $d+2$} of $M$ is defined as the UMP of any of the minimal
  models $(\H{A},\Astr[\H{A}],\H{M},\Astr[\H{M}])$ of the pair $(A,M)$:
  \[
    \Hclass{\Astr<d+2>[A\ltimes
      M]}\coloneqq\Hclass{\Astr<d+2>[\H{A}\ltimes\H{M}]}\in\RelBimHH[d+2]<-d>{\H{A}}{\H{M}},
  \]
  where $\Astr<d+2>[\H{A}\ltimes\H{M}]=\Astr<d+2>[\H{A}]+\Astr<d+2>[\H{M}]$.
\end{definition}

\begin{remark}
  The notation in \Cref{def:UMP-dgAs,def:UMP-dgBims} conflicts slightly with our
  notation $\Astr+\Astr[M]$ for the square-zero product on $A\oplus M$. Since dg
  algebras have vanishing higher operations, this abuse of
  notation should not lead to confusion.
\end{remark}

The following result shows that the vanishing condition in
\Cref{thm:CY-equivalence}\eqref{it:thm:CY-equivalence:BV} is satisfied by all
bimodule right Calabi--Yau algebras with $d$-sparse cohomology.

\begin{proposition}
  \label{prop:CY-BV}
  Let $A$ be a cohomologically $d$-sparse dg algebra, $d\geq1$. Suppose that $A$
  admits a bimodule right $n$-Calabi--Yau structure
  \[
    \widetilde{\varphi}\colon A[n]\stackrel{\sim}{\longrightarrow}DA
  \]
  for some $n\in d\ZZ$. Then,
  \[
    0=\BV(\Hclass{\Astr<d+2>})\in\HH[d+1]<-d>{\H{A}};
  \]
  here, $\BV=\BV_\varphi$ is the BV operator associated with the induced
  isomorphism of graded $\H{A}$-bimodules
  \[
    \varphi\colon\H{A}(n)\stackrel{\sim}{\longrightarrow}D\H{A}.
  \]
\end{proposition}
\begin{proof}
  The isomorphism $\varphi$ is the linear part of an $A_\infty$-isomorphism of
  $A_\infty$-bimodules between minimal models, say
  \[
    (\H{A}(n),\Astr[\H{A}(n)])\qquad\text{and}\qquad(D\H{A},\Astr[D\H{A}])
  \]
  of the dg $A$-bimodules $A[n]$ and $DA$, respectively, over a compatible
  minimal model of the dg algebra $A$, see for example~\cite[Remark~2.3.8]{JM25}. It
  then follows from \Cref{prop:DM-shift} that there exists a gauge
  $A_\infty$-isomorphism of $A_\infty$-bimodules
  \[
    (\H{A}(0),\Astr[\H{A}(0)])\stackrel{\simeq}{\longrightarrow}(\H{A},\Upsilon(\Astr[\H{A}(0)])).
  \]
  Consequently, there is an equality of bimodule UMPs of length $d+2$
  \begin{multline*}
    \Hclass{\Astr<d+2>[\H{A}]+\Astr<d+2>[\H{A}(0)]}=\Hclass{\Astr<d+2>[\H{A}]+\Upsilon(\Astr<d+2>[\H{A}(0)])}\\\in\RelBimHH[d+2]<-d>{\H{A}}{\H{A}(0)}.
  \end{multline*}
  Forming their difference and applying \Cref{prop:CY-Ai-aux} yields the
  vanishing condition
  \[
    0=\BV(\Hclass{\Astr<d+2>[\H{A}]})\in\HH[d+1]<-d>{\H{A}},
  \]
  where $\BV=\BV_\varphi$, which is what we needed to prove.
\end{proof}

\subsection{The proof of \Cref{thm:CY-Kadeikshvili}}

We are ready to prove \Cref{thm:CY-Kadeikshvili}, which is a translation of
\Cref{thm:CY-Kadeishvili-Ai} to the language of dg algebras and dg bimodules. It
is also a partial converse to \Cref{prop:CY-BV}.

\begin{proof}[Proof of \Cref{thm:CY-Kadeikshvili}]
  Let $(\H{A},\Astr[\H{A}])$ be a minimal model for $A$; in view of the
  discussion in \Cref{ex:minimod-diagonal}, the corresponding diagonal
  $A_\infty$-bimodule is indeed a minimal model for the diagonal dg
  $A$-bimodule. Similarly, minimal models of the dg $A$-bimodules $A[n]$ and
  $DA$ are obtained by shifting and dualising the minimal model of the diagonal
  dg $A$-bimodule, respectively. \Cref{thm:CY-Kadeishvili-Ai} then yields an
  $A_\infty$-isomorphism of $A_\infty$-bimodules
  \[
    (\H{A}(n),\Astr[\H{A}(n)])\stackrel{\sim}{\longrightarrow}(D\H{A},\Astr[D\H{A}])
  \]
  with linear part the given graded $\H{A}$-bimodule isomorphism
  \[
    \varphi\colon\H{A}(n)\stackrel{\sim}{\to}D\H{A}.
  \]
  The existence of the above $A_\infty$-isomorphism of $A_\infty$-bimodules
  implies the existence of the desired isomorphism $A[n]\simeq DA$ in the
  derived category dg $A$-bimodules.\footnote{We also remind the reader
    of~\cite[Proposition~2.4.6]{JM25}, which says that quasi-isomorphic dg
    bimodules are distinguished by their minimal models.}
\end{proof}

\section{The proofs of
  \texorpdfstring{\Cref{thm:CY-equivalence,thm:CY-correspondence}}{Theorems C
    and D}}
\label{sec:Thm:CY}

\subsection{The $d$-sparse graded algebra $\bfLambda$}

\begin{setting}
  \label{setting:graded-lambda}
  We fix integers $d\geq1$ and $m\in\ZZ$ and set $n\coloneq dm$. We also fix a
  basic Frobenius algebra $\Lambda$, a ($\kk$-linear) algebra automorphism
  $\sigma\in\Aut{\Lambda}$, and introduce the graded algebra
  \[
    \bfLambda\coloneqq\frac{\Lambda\langle\imath^{\pm1}\rangle}{(\imath
      x-\sigma(x)\imath)_{x\in\Lambda}},\qquad |\imath|=-d.
  \]
  As suggested by the notation, the graded algebra $\bfLambda$ is not
  graded-commutative unless $d$ is even.
\end{setting}

\begin{remark}
  The graded algebra $\bfLambda$ is $d$-sparse in the sense that it is
  concentrated in degrees that are multiples of $d$. Indeed, the possibly
  non-zero homogeneous components of $\bfLambda$ are given by
  \[
    \bfLambda^{dk}=\imath^{-k}\Lambda=\set{\imath^{-k}x}[x\in\Lambda],\qquad
    k\in\ZZ.
  \]
  Moreover, by definition, for $x\in\Lambda$ and $k\in\ZZ$ there are equalities
  \begin{equation}
    \label{eq:imath-conjugation}
    \imath^k
    x\imath^{-k}=\sigma^k(x)\qquad\text{and}\qquad\imath^{-k}x\imath^{k}=\sigma^{-k}(x).
  \end{equation}
  In particular, the map
  \begin{equation}
    \label{eq:iso-imath-twbim}
    \bfLambda^{dk}=\imath^{-k}\Lambda\stackrel{\sim}{\longrightarrow}\twBim[\sigma^k]{\Lambda},\qquad
    \imath^{-k}x\longmapsto x,
  \end{equation}
  is an isomorphism of $\Lambda$-bimodules. These isomorphisms assemble into a
  further isomorphism of graded algebras
  \[
    \bfLambda\stackrel{\sim}{\longrightarrow}\gLambda\coloneqq\bigoplus_{dk\in
      d\ZZ}\twBim[\sigma^k]{\Lambda},
  \]
  where the algebra structure on the right-hand side is given by
  \[
    x*y=\sigma^{k}(x)y,\qquad |y|=dk.
  \]
  We prefer to work with the graded algebra $\bfLambda$ in what follows since
  there is no ambiguity on the degrees of its homogeneous elements.
\end{remark}

\begin{remark}
  The graded algebra $\bfLambda$ can be described as a graded skew group algebra
  over the abelian group of integers, see \cite[Remark~4.5.7]{JKM22} for details.
\end{remark}

The following result provides a criterion for the existence, for some $n\in\ZZ$,
of an isomorphism of graded $\bfLambda$-bimodules
\[
  \bfLambda(n)\stackrel{\sim}{\longrightarrow}D\bfLambda
\]
in terms of the Frobenius algebra $\Lambda=\bfLambda^0$.

\begin{proposition}
  \label{prop:graded-bimodule-iso-vs-pairing}
  Recall \Cref{setting:graded-lambda}. There is a canonical bijection between the following:
  \begin{enumerate}
  \item\label{it:graded-bimodule-iso} Isomorphisms of graded
    $\bfLambda$-bimodules
    \[
      \varphi\colon\bfLambda(n)\stackrel{\sim}{\longrightarrow} D\bfLambda.
    \]
  \item\label{it:pairing} Associative nondegenerate bilinear pairings
    \[
      \pairing<-,->\colon\Lambda\times\Lambda\longrightarrow\kk
    \]
    such that
    \begin{align}
      \pairing<y,x>&=\pairing<\sigma^m(x),y>&x,y&\in\Lambda,\label{eq:props-pairing-Nakayama}\\
      \pairing<\sigma^k(x),y>&=(-1)^{dk(dm+1)}\pairing<x,\sigma^{-k}(y)>&x,y&\in\Lambda,\ k\in\ZZ.\label{eq:props-pairing-shift}
    \end{align}
  \end{enumerate}
  In particular, when either of the above exist, $\sigma^m$ is a Nakayama
  automorphism of the Frobenius algebra $\Lambda$, that is
  $\twBim[\sigma^m]{\Lambda}\cong D\Lambda$ as $\Lambda$-bimodules. The
  correspondence \eqref{it:graded-bimodule-iso}$\to$\eqref{it:pairing}
  associates to an isomorphism $\varphi$ the bilinear pairing
  \[
    \pairing<-,->\colon\Lambda\times\Lambda\longrightarrow\kk,\qquad
    (x,y)\longmapsto\varphi(\s[n]\imath^{-m}x)(y).
  \]
\end{proposition}
\begin{proof}
  We make the following observations that are helpful for keeping track of
  degrees and Koszul signs throughout the proof. Recall that $n=dm$.
  \begin{itemize}
  \item For $k\in\ZZ$, we have
    \[
      \bfLambda(n)^{dk}=\bfLambda^{dk+dm}=\imath^{-(k+m)}\Lambda.
    \]
  \item For $k\in\ZZ$ and $x\in\Lambda$ we have
    \[
      \s*[n]{\imath^{-k}x}=\s*[n]{\imath^{-k}}x
    \]
    in $\bfLambda(n)$ and therefore we may write $\s[n]{\imath^{-k}x}$ without
    ambiguity. Furthermore,
    \[
      |\s[n]{\imath^{-k}x}|=dk-dm=d(k-m).
    \]
    In particular, $|\s[n]{\imath^{-m}x}|=0$.
  \item For $k\in\ZZ$ and $x,y\in\Lambda$ we have
    \[
      x\s[n]{\imath^{-k}y}=\s[n]{x\imath^{-k}y}
    \]
    since $x$ is homogeneous of degree $0$ (as an element of $\bf\Lambda$).
  \end{itemize}

  Let ${\varphi\colon\bfLambda(n)\stackrel{\sim}{\to}D\bfLambda}$ be an
  isomorphism of graded $\bfLambda$-bimodules and consider the induced
  nondegenerate associative bilinear pairing
  \[
    \pairing<-,->_\varphi\colon\bfLambda(n)\times\bfLambda\longrightarrow\kk,\qquad
    (s^n\imath^{-k}x,\imath^{-\ell}y)\longmapsto\varphi(\s[n]1)(\imath^{-k}x\imath^{-\ell}y);
  \]
  recall from \Cref{cor:graded-symmetric-pairing} that this pairing is
  graded-symmetric with respect to the degree in $\bfLambda$, that is
  \begin{equation}
    \label{eq:graded-symmetric-pairing-lambda}
    \pairing<\s[n]\imath^{-k}x,\imath^{-\ell}y>_\varphi=(-1)^{dkd\ell}\pairing<\s[n]\imath^{-\ell}y,\imath^{-k}x>,\qquad
    x,y\in\Lambda,\ k,\ell\in\ZZ.
  \end{equation}
  Since the restriction
  \begin{align*}
    \varphi\colon \imath^{-m}\Lambda=\bfLambda^n=\bfLambda(n)^0&\stackrel{\sim}{\longrightarrow}(D\bfLambda)^0=D\Lambda\\
    \s[n]{\imath^{-m}x}&\longmapsto(y\mapsto\varphi(\imath^{-m}x)(y))
  \end{align*}
  to the corresponding components degree $0$ is an isomorphism of
  $\Lambda$-bimodules, the induced bilinear pairing
  \[
    \pairing<-,->\colon\Lambda\times\Lambda\longrightarrow\kk,\qquad
    (x,y)\longmapsto\pairing<\s[n]\imath^{-m}x,y>_\varphi
  \]
  is nondegenerate and associative. Therefore, for $x,y\in\Lambda$ we have
  \begin{align*}
    \pairing<y,x>&=\pairing<\s[n]\imath^{-m}y,x>_\varphi\\
                 &\stackrel{\eqref{eq:graded-symmetric-pairing-lambda}}{=}\pairing<\s[n]x,\imath^{-m}y>_\varphi&|x|&=0\\
                 &=\pairing<\s[n]x\imath^{-m},y>_\varphi\\
                 &\stackrel{\eqref{eq:imath-conjugation}}{=}\pairing<\s[n]\imath^{-m}\sigma^m(x),y>_\varphi\\
                 &=\pairing<\sigma^m(x),y>.
  \end{align*}
  Similarly, for $x,y\in\Lambda$ and $k\in\ZZ$,
  \begin{align*}
    \pairing<\sigma^k(x),y>&=\pairing<\s[n]\imath^{-m}\sigma^k(x),y>_\varphi\\
                           &\stackrel{\eqref{eq:imath-conjugation}}{=}\pairing<\s[n]\imath^{-m}\imath^kx\imath^{-k},y>_\varphi\\
                           &=\pairing<\s[n]\imath^k,\imath^{-m}x\imath^{-k}y>_\varphi\\
                           &\stackrel{\eqref{eq:graded-symmetric-pairing-lambda}}{=}(-1)^{dk(dm+dk)}\pairing<\s[n]\imath^{-m}x\imath^{-k}y,\imath^k>_\varphi&&|\imath^{-m}|=dm,\ |\imath^{-k}|=dk\\
                           &=(-1)^{dk(dm+1)}\pairing<\s[n]\imath^{-m}x,\imath^{-k}y\imath^k>_\varphi\\
                           &\stackrel{\eqref{eq:imath-conjugation}}{=}(-1)^{dk(dm+1)}\pairing<\s[n]\imath^{-m}x,\sigma^{-k}(y)>_\varphi\\
                           &=(-1)^{dk(dm+1)}\pairing<x,\sigma^{-k}(y)>.
  \end{align*}
  This shows that the association
  \[
    \varphi\longmapsto\pairing<-,->_{\varphi}|_{\bfLambda(n)^0\times\bfLambda^0}
  \]
  yields a well-defined map
  \eqref{it:graded-bimodule-iso}$\mapsto$\eqref{it:pairing}.

  Conversely, suppose given an associative nondegenerate pairing
  \[
    \pairing<-,->\colon\Lambda\times\Lambda\longrightarrow\kk
  \]
  with the additional properties listed in statement \eqref{it:pairing} of the
  proposition. We claim that the isomorphism of graded vector spaces
  \[
    \varphi=\varphi_{\pairing<-,->}\colon\bfLambda(n)\stackrel{\sim}{\longrightarrow}D\bfLambda
  \]
  with the components
  \begin{align*}
    \varphi\colon\bfLambda(n)^{dk}=\imath^{-(k+m)}\Lambda&\stackrel{\sim}{\longrightarrow}D(\imath^{k}\Lambda)=(D\bfLambda)^{dk}\\
    \s[n]{\imath^{-(k+m)}x}&\longmapsto(\imath^{k}y\mapsto\pairing<\sigma^{-k}(x),y>)
  \end{align*}
  is in fact an isomorphism of graded $\bfLambda$-bimodules. Indeed, let
  $w,x,y,z\in\Lambda$ and $r,s,t\in\ZZ$; for readability we set $u\coloneqq
  r+s+t$. On the one hand,
  \begin{align*}
    \varphi(\imath^{-r}x\cdot\s*[n]{\imath^{-(s+m)}y}\cdot \imath^{-t}z)(\imath^{u}w)&=(-1)^{drdm}\varphi(\s[n]{\imath^{-(u+m)}\sigma^{s+t+m}(x)\sigma^{t}(y)z})(\imath^{u}w)\\
                                                                                     &=(-1)^{drdm}\pairing<\sigma^{m-r}(x)\sigma^{-r-s}(y)\sigma^{-u}(z),w>\\
                                                                                     &\stackrel{\eqref{eq:props-pairing-Nakayama}}{=}(-1)^{drdm}\pairing<\sigma^{-r-s}(y)\sigma^{-u}(z),w\sigma^{-r}(x)>\\
                                                                                     &=(-1)^{drdm}\pairing<\sigma^{-r-s}(y),\sigma^{-u}(z)w\sigma^{-r}(x)>.
  \end{align*}
  On the other hand,
  \begin{align*}
    (\imath^{-r}x\cdot\varphi(\s[n]{\imath^{-(s+m)}y})\cdot\imath^{-t}z)(\imath^{u}w)&=(-1)^{dr(ds+dt+du)}\varphi(\s[n]{\imath^{-(s+m)}y})(\imath^{-t}z\imath^{u}w\imath^{-r}x)\\
                                                                                     &=(-1)^{dr}\varphi(\s[n]{\imath^{-(s+m)}y})(\imath^{s}\sigma^{r-u}(z)\sigma^{r}(w)x)\\
                                                                                     &=(-1)^{dr}\pairing<\sigma^{-s}(y),\sigma^{r-u}(z)\sigma^{r}(w)x>\\
                                                                                     &\stackrel{\eqref{eq:props-pairing-shift}}{=}(-1)^{dr}(-1)^{dr(dm+1)}\pairing<\sigma^{-r-s}(y),\sigma^{-u}(z)w\sigma^{-r}(x)>\\
                                                                                     &=(-1)^{drdm}\pairing<\sigma^{-r-s}(y),\sigma^{-u}(z)w\sigma^{-r}(x)>.
  \end{align*}
  This shows that the construction
  \[
    \pairing<-,->\longmapsto\varphi_{\pairing<-,->}
  \]
  yields a well defined map
  \eqref{it:pairing}$\mapsto$\eqref{it:graded-bimodule-iso}.

  We now verify that the maps
  \eqref{it:graded-bimodule-iso}$\mapsto$\eqref{it:pairing} and
  \eqref{it:pairing}$\mapsto$\eqref{it:graded-bimodule-iso} defined above are
  inverse to each other. Indeed, suppose given a bimodule isomorphism $\varphi$
  as in \eqref{it:graded-bimodule-iso} and write
  $\pairing<-,->=\pairing<-,->_{\varphi}|_{\bfLambda(n)^0\times\bfLambda^0}$.
  Then, for $x,y\in\Lambda$ and $k\in\ZZ$,
  \begin{align*}
    \varphi_{\pairing<-,->}(\imath^{-(k+m)}x)(\imath^ky)&=\pairing<\sigma^{-k}(x),y>\\
                                                        &=\varphi(\imath^{-m}\sigma^{-k}(x))(y)\\
                                                        &\stackrel{\eqref{eq:imath-conjugation}}{=}\varphi(\imath^{-m}\imath^{-k}x\imath^k)(y)\\
                                                        &=\varphi(\imath^{-(k+m)}x)(\imath^ky).
  \end{align*}
  Similarly, suppose given a pairing $\pairing<-,->$ as in \eqref{it:pairing} and
  write $\varphi=\varphi_{\pairing<-,->}$. Then, for $x,y\in\Lambda$,
  \[
    \pairing<\s[n]\imath^{-m}x,y>_{\varphi}=\varphi(\s[n]\imath^{-m}x)(y)=\pairing<x,y>.
  \]
  This finishes the proof.
\end{proof}

\begin{remark}
  Suppose given a pairing
  \[
    \pairing<-,->\colon\Lambda\times\Lambda\longrightarrow\kk
  \]
  as in \Cref{prop:graded-bimodule-iso-vs-pairing}\eqref{it:pairing}. Then,
  according to \eqref{eq:props-pairing-Nakayama} we must have
  \[
    \pairing<\sigma^{m}(x),y>=\pairing<y,x>=\pairing<x,\sigma^{-m}(y)>\qquad
    x,y\in\Lambda.
  \]
  On the other hand, according to \eqref{eq:props-pairing-shift} for $k=m$ we also must
  have
  \[
    \pairing<\sigma^{m}(x),y>=(-1)^{dm(dm+1)}\pairing<x,\sigma^{-m}(y)>\qquad
    x,y\in\Lambda.
  \]
  Notice, however, that $dm(dm+1)$ is always even.
\end{remark}

When $m=1$, \Cref{prop:graded-bimodule-iso-vs-pairing} is much simpler.

\begin{corollary}
  \label{coro:graded-bimodule-iso-vs-pairing-trivial}
  Recall~\Cref{setting:graded-lambda}. Suppose that $m=1$, so that $n=d$. There
  is a canonical bijection between the following:
  \begin{enumerate}
  \item\label{it:graded-bimodule-iso-trivial} Isomorphisms of graded
    $\bfLambda$-bimodules
    \[
      \varphi\colon\bfLambda(n)\stackrel{\sim}{\longrightarrow} D\bfLambda.
    \]
  \item\label{it:pairing-trivial} Associative nondegenerate bilinear pairings
    \[
      \pairing<-,->\colon\Lambda\times\Lambda\longrightarrow\kk
    \]
    such that
    \begin{align}
      \pairing<y,x>&=\pairing<\sigma(x),y>&x,y&\in\Lambda.\label{eq:props-pairing-Nakayama-trivial}
    \end{align}
  \end{enumerate}
  In particular, when either of the above exist, $\sigma$ is a Nakayama
  automorphism of the Frobenius algebra $\Lambda$. The correspondence
  \eqref{it:graded-bimodule-iso-trivial}$\to$\eqref{it:pairing-trivial}
  associates to an isomorphism $\varphi$ the bilinear pairing
  \[
    \pairing<-,->\colon\Lambda\times\Lambda\longrightarrow\kk,\qquad
    (x,y)\longmapsto\varphi(\s[n]\imath^{-1}x)(y).
  \]
\end{corollary}
\begin{proof}
  In view of \Cref{prop:graded-bimodule-iso-vs-pairing}, it suffices to observe
  that, if $m=1$, then the integer $dk(dm+1)$ is always even and hence condition
  \eqref{eq:props-pairing-shift} reads
  \[
    \pairing<\sigma^k(x),y>=\pairing<x,\sigma^{-k}(y)>,\qquad x,y\in\Lambda,\
    k\in\ZZ.
  \]
  Clearly, it suffices to verify this condition for $k=1$, in which case we have
  \[
    \pairing<\sigma(x),y>\stackrel{\eqref{eq:props-pairing-Nakayama-trivial}}{=}\pairing<y,x>\stackrel{\eqref{eq:props-pairing-Nakayama-trivial}}{=}\pairing<x,\sigma^{-1}(y)>.\qedhere
  \]
\end{proof}

\begin{example}
  Let $\Lambda=\kk$ and $\sigma=\id$, the unique $\kk$-linear algebra
  automorphism of $\kk$. Every nondegenerate bilinear pairing $\kk\times \kk\to
  \kk$ is necessarily of the form
  \[
    \pairing<-,->\colon \kk\times\kk\longrightarrow\kk,\qquad (x,y)\longmapsto
    uxy,
  \]
  where $u\in\kk$ is non-zero. In particular, any such pairing is associative
  and satisfies \eqref{eq:props-pairing-Nakayama} since $\kk$ is commutative.
  However, condition \eqref{eq:props-pairing-shift} is satisfied if and only
  if either of the following conditions hold:
  \begin{itemize}
  \item $\chark(\kk)=2$.
  \item $d$ is even.
  \item $n=dm$ is odd.
  \end{itemize}
  Indeed, for condition \eqref{eq:props-pairing-shift} to hold for the above
  pairing it is necessary and sufficient to have
  \[
    (-1)^{dk(dm+1)}=1\in\kk\qquad k\in\ZZ.
  \]
  We conclude from \Cref{prop:graded-bimodule-iso-vs-pairing} that there exists
  an isomorphism of graded $\bfLambda$-bimodules
  \[
    \bfLambda(n)\stackrel{\sim}{\longrightarrow}D\bfLambda,\qquad n\coloneqq dm,
  \]
  if and only if either of the three conditions listed above are satisfied.
  Notice that $\bfLambda=\kk[\imath^{\pm1}]$ with $|\imath|=-d$ in this case.
\end{example}

\begin{example}
  \label{ex:dual_numbers-bimodule_isos}
  Let $\Lambda=\kk[\varepsilon]/(\varepsilon^2)$ be the algebra of dual numbers.
  The algebra $\Lambda$ is symmetric since it is commutative and self-injective,
  see for example~\cite[Proposition~IV.4.6]{SY11a}. The vector space
  \[
    \Hom[\Lambda^e]{\Lambda}{D\Lambda}\cong\Hom[\Lambda^e]{\Lambda}{\Lambda}\cong
    Z(\Lambda)=\Lambda
  \]
  is therefore $2$-dimensional: To a pair $a,b\in\kk$ we associate the morphism
  $\Lambda$-bimodules
  \[
    \varphi=\varphi_{a,b}\colon\Lambda\longrightarrow D\Lambda
  \]
  that is adjoint to the associative and symmetric bilinear pairing
  \[
    \pairing<-,->=\pairing<-,->_{a,b}\colon\Lambda\times\Lambda\to\kk
  \]
  uniquely determined by the requirement that
  \begin{align*}
    \pairing<1,1>&=a,&\pairing<1,\varepsilon>&=b,\\
    \pairing<\varepsilon,1>&=b,&\pairing<\varepsilon,\varepsilon>&=0;
  \end{align*}
  explicitly,
  \[
    \varphi(1)=a1^*+b\varepsilon^*\qquad{and}\qquad\varphi(\varepsilon)=b1^*.
  \]
  By the determinant criterion, such a pairing is nondegenerate if and only if
  $b\neq 0$, in which case the inverse of $\varphi=\varphi_{a,b}$ is the map
  \begin{align*}
    \psi=\psi_{a,b}\colon D\Lambda&\longrightarrow\Lambda,\\
    1^*&\longmapsto b^{-1}\varepsilon,\\
    \varepsilon^*&\longmapsto b^{-1}-ab^{-2}\varepsilon.
  \end{align*}  
  Fix now $a,b\in\kk$ with $b\neq 0$, and hence a $\Lambda$-bimodule isomorphism
  $\varphi=\varphi_{a,b}$ as above. Fix $d\geq 1$ and consider the algebra
  automorphism $\sigma\in\Aut{\Lambda}$ that is uniquely determined by the
  requirement that $\sigma(\varepsilon)=(-1)^d\varepsilon$ (compare
  with~\cite[Corollary~5.4.12]{JKM22}). Let $m\in\ZZ$ and set $n\coloneqq dm$.
  We analyse the conditions under which the equalities
  \eqref{eq:props-pairing-Nakayama} and \eqref{eq:props-pairing-shift} in
  \Cref{prop:graded-bimodule-iso-vs-pairing} are satisfied. This is clearly the case
  if $\chark(\kk)=2$, so let us assume that $\chark(\kk)\neq2$. In this case
  \begin{align*}
    \pairing<1,\varepsilon>&\stackrel{!}{=}\pairing<\sigma^m(\varepsilon),1>=(-1)^{dm}\pairing<\varepsilon,1>=(-1)^{dm}\pairing<1,\varepsilon>
  \end{align*}
  is satisfied if and only if $d$ or $m$ are even (recall that
  $\pairing<1,\varepsilon>\neq0$). Similarly, for the equality
  \begin{align*}
    (-1)^{dk}\pairing<\varepsilon,1>=\pairing<\sigma^k(\varepsilon),1>\stackrel{!}{=}(-1)^{dk(dm+1)}\pairing<\varepsilon,\sigma^{-k}(1)>=(-1)^{dk(dm+1)}\pairing<\varepsilon,1>
  \end{align*}
  to hold for all $k\in\ZZ$ in addition to the previous one requires that either
  $d$ is even and $m$ is arbitrary or that $d$ is odd and $m$ is even (hence in
  both cases $n=md$ is even). Finally, we must have
  \[
    a=\pairing<1,1>=(-1)^{dk(dm+1)}\pairing<1,1>=(-1)^{dk(dm+1)}a,\qquad k\in\ZZ,
  \]
  which forces that, in addition, $a=0$ whenever $d$ is odd. In summary,
  $\psi_{a.b}$ extends to a graded $\bfLambda$-bimodule isomorphism
  $\bfLambda(n)\stackrel{\sim}{\longrightarrow}D\bfLambda$ if and only if we are
  in the following cases (recall that always $b\neq0$):
  \begin{description}
  \item[Case 1] $\chark(\kk)=2$.
  \item[Case 2] $\chark(\kk)\neq 2$. In this case, there are the following subcases:
    \begin{itemize}
    \item $d$ is even and $m$ is arbitrary.
    \item $d$ is odd, $m$ is even and $a=0$.
    \end{itemize}
  \end{description}
\end{example}

\subsection{Proofs of \Cref{thm:CY-equivalence,thm:CY-correspondence}}

We are finally ready to give our proof of
\Cref{thm:CY-equivalence,thm:CY-correspondence}.

\begin{proof}[Proof of \Cref{thm:CY-correspondence}]
  We show first that the correspondence
  $\eqref{it:thm:CY-correspondence:dgAs}\to\eqref{it:thm:CY-correspondence:gLambda}$
  in \Cref{thm:CY-correspondence} is well defined. It is part
  of~\cite[Theorem~A]{JKM22} that the correspondence is well defined if we drop
  items \eqref{it:thm:CY-correspondence:dgAs:it:CY} and
  \eqref{it:thm:CY-correspondence:gLambda:it:CY} in
  \Cref{thm:CY-correspondence}; it is here that we need the assumption that the
  ground field is perfect, which is also needed in \emph{loc.~cit.} Suppose then
  that $A$ is a dg algebra as in
  \Cref{thm:CY-correspondence}\eqref{it:thm:CY-correspondence:dgAs} which
  therefore satisfies conditions \eqref{it:thm:CY-correspondence:dgAs:it:H0} and
  \eqref{it:thm:CY-correspondence:dgAs:it:dZ-CT} therein, and hence the pair
  \[
    (\Lambda,I)\coloneqq(\H[0]{A},\H[-d]{A})
  \]
  satisfies the conditions in \eqref{it:thm:CY-correspondence:gLambda:it:tp} and
  \eqref{it:thm:CY-correspondence:gLambda:it:I} in \Cref{thm:CY-correspondence}.
  Moreover, there is an isomorphism of graded algebras, which we treat as an
  identification in what follows,
  \[
    \bfLambda=\Lambda(\sigma,d)\stackrel{\sim}{\longrightarrow}\H{A},\qquad
    \imath\longmapsto u,
  \]
  that restricts to the identity on the degree $0$ part, and is otherwise
  determined by the choice of an isomorphism
  \[
    u\in\Hom[\DerCat[c]{A}]{A}{A[-d]}\cong\H{A}[-d]
  \]
  in the derived category of (right) dg $A$-modules (the latter exists since $A$
  satisfies conditions \eqref{it:thm:CY-correspondence:dgAs:it:H0} and
  \eqref{it:thm:CY-correspondence:dgAs:it:dZ-CT} in
  \Cref{thm:CY-correspondence}, see~\cite[Section~4.5.1]{JKM22} for details).
  Since the dg algebra $A$ is bimodule right $n$-Calabi--Yau by condition
  \eqref{it:thm:CY-correspondence:dgAs:it:CY} in \Cref{thm:CY-correspondence}, we
  may apply \Cref{prop:CY-BV} to obtain the vanishing condition
  \[
    0=\BV(\Hclass{\Astr<d+2>[\H{A}]})\in\HH[d+1]<-d>{\H{A}},
  \]
  where $\BV=\BV_\varphi$ is the BV operator associated with $\varphi$.
  Notice that the conditions on the class $\Hclass{\Astr<d+2>[\H{A}]}$ listed in
  \Cref{thm:CY-correspondence}\eqref{it:thm:CY-correspondence:gLambda:it:CY}
  uniquely characterise the UMP of $A$ up to $A_\infty$-isomorphism of its
  minimal models, see the proof of~\cite[Theorem~5.3.3]{JKM22}. Finally, it
  follows from \cite[Theorem~5.1.2]{JKM22} that the class of any short exact
  sequence of $\Lambda$-bimodules
  \[
    \eta\colon\quad 0\to I\to P_{d+2}\to\cdots\to P_3\to P_2\to P_1\to\Lambda\to
    0
  \]
  with projective(-injective) middle terms can be realised as the restricted UMP
  of a minimal model for a dg algebra in the quasi-isomorphism class of $A$,
  hence the previous argument shows that the vanishing condition in
  \eqref{it:thm:CY-correspondence:gLambda:it:CY} indeed holds for arbitrary
  $\eta$. This finishes the proof of the fact that the correspondence
  $\eqref{it:thm:CY-correspondence:dgAs}\to\eqref{it:thm:CY-correspondence:gLambda}$
  is well defined.

  We now conclude the proof of the theorem by proving that the correspondence
  $\eqref{it:thm:CY-correspondence:dgAs}\to\eqref{it:thm:CY-correspondence:gLambda}$
  is surjective; as explained in the introduction, the correspondence is
  injective since it is the restriction of the bijective correspondence
  in~\cite[Theorem~A]{JKM22}. Let $\Lambda$ and $\bfLambda$ as in
  \Cref{thm:CY-correspondence}\eqref{it:thm:CY-correspondence:gLambda}. In
  particular, there exists an isomorphism of graded $\bfLambda$-bimodules
  \[
    \varphi\colon\bfLambda(n)\stackrel{\sim}{\longrightarrow}D\bfLambda.
  \]
  Assuming that the conditions \eqref{it:thm:CY-correspondence:gLambda:it:tp}
  and \eqref{it:thm:CY-correspondence:gLambda:it:I} in
  \Cref{thm:CY-correspondence} hold, \cite[Theorems~5.1.2 and~5.3.3]{JKM22}
  guarantee the existence of a minimal $A_\infty$-algebra structure
  $(\bfLambda,\Astr[\eta])$ with the property that
  \begin{align*}
    j^*\colon\HH{\bfLambda}[\bfLambda]&\longrightarrow\HH{\Lambda}[\bfLambda]\\
    \Hclass{\Astr<d+2>[\bfLambda]}&\longmapsto j^*\Hclass{\Astr<d+2>[\bfLambda]}=\Hclass{\eta},
  \end{align*}
  where
  \[
    \eta\colon\quad 0\to\twBim{\Lambda}[\sigma]\to P_{d+2}\to\cdots P_3\to
    P_2\to P_1\to\Lambda\to 0
  \]
  is any given short exact sequence of $\Lambda$-bimodules with
  projective(-injective) middle terms. Since UMPs of length $d+2$ always satisfy
  $\Sq(\Hclass{\Astr<d+2>[\bfLambda]})=0$, by assumption we must also have
  \[
    0=\BV(\Hclass{\Astr<d+2>[\bfLambda]})\in\HH[d+1]<-d>{\bfLambda},
  \]
  where $\BV=\BV_\varphi$ is the BV operator. Let $A$ be a dg algebra
  strictifying $(\bfLambda,\Astr[\eta])$, so that the latter is a minimal
  $A_\infty$-model of $A$ (for example, we may take the dg endomorphism
  algebra of $(\bfLambda,\Astr[\eta])$ in its derived dg category). Now, as the
  graded algebra $\bfLambda$ is $d$-sparse by construction, we find ourselves in
  the setting of \Cref{thm:CY-Kadeikshvili} and hence we are able to conclude
  that the dg algebra $A$ is right $n$-Calabi--Yau as a bimodule as soon as we
  establish the following vanishing of the Massey bimodule cohomology of the
  diagonal $\bfLambda$-bimodule:
  \[
    \BimHMH!\bfLambda![p+1]<-p>{\bfLambda(0)}<\bfLambda\ltimes\bfLambda(0)>=0,\qquad
    p>d.
  \]
  We finish the proof of the theorem by establishing this vanishing. Keeping in
  mind the discussion in \Cref{rmk:BimHMC-HMC}, in what follows we identify the
  Massey bimodule cochain complex of the diagonal $\bfLambda$-bimodule with the
  bigraded vector space with the components
  \[
    \HH[p]<r>{\bfLambda},\qquad p\geq1,\ r\in\ZZ,
  \]
  that is endowed with the bidegree $(d+1,-d)$ differential
  $x\mapsto[\Hclass{\Astr<d+2>[\bfLambda]},x]$, where
  $\Hclass{\Astr<d+2>[\bfLambda]}$ is the (algebra) UMP of length $d+2$. From
  this point onwards the proof is very similar (in fact slightly simpler) to
  the proof of \cite[Proposition~5.2.9]{JKM22}, which is a key result in
  \emph{op.~cit.} Consider the bidegree $(d+2,-d)$ endomorphism of
  $\HMC{\bfLambda}$ given by
  \begin{equation}
    \label{eq:key-map}
    \HH[s]<t>{\bfLambda}\longrightarrow\HH[s+d+2]<t-d>{\bfLambda},\qquad
    x\longmapsto\cupp{\Hclass{\Astr<d+2>[\bfLambda]}}{x};
  \end{equation}
  unlike in the proof of~\cite[Proposition~5.2.9]{JKM22} where the algebra
  Hochschild--Massey cohomology is considered, there is no need to change the
  definition in bidegree $(d+1,-d)$ because the Massey bimodule differential is
  given by the formula ${x\mapsto[\Astr<d+2>[\bfLambda],x]}$ without exceptions,
  see \eqref{eq:BimHMd}. The same calculation as in
  \cite[Proposition~5.2.9]{JKM22} shows that the above maps indeed determine a
  cochain map, which is to say that
  \[
    [\Hclass{\Astr<d+2>[\bfLambda]},\cupp{\Hclass{\Astr<d+2>[\bfLambda]}}{x}]=\cupp{\Hclass{\Astr<d+2>[\bfLambda]}}{[\Hclass{\Astr<d+2>[\bfLambda]},x]}.
  \]
  We reproduce the proof of the latter equality for the convenience of the
  reader:
  \begin{align*}
    [\Hclass{\Astr<d+2>[\bfLambda]},\cupp{\Hclass{\Astr<d+2>[\bfLambda]}}{x}]&\stackrel{\eqref{eq:Gerstenhaber_relation}}{=}\cupp{[\Hclass{\Astr<d+2>[\bfLambda]},\Hclass{\Astr<d+2>[\bfLambda]}]}{x}+\cupp{\Hclass{\Astr<d+2>[\bfLambda]}}{[\Hclass{\Astr<d+2>[\bfLambda]},x]}\\
                                                                             &\stackrel{\eqref{eq:Gerstenhaber_bracket-HC}}{=}\cupp{2\underbrace{\Sq(\Hclass{\Astr<d+2>[\bfLambda]})}_{\stackrel{\ref{def:UMP}}{=}0}}{x}+\cupp{\Hclass{\Astr<d+2>[\bfLambda]}}{[\Hclass{\Astr<d+2>[\bfLambda]},x]}\\
                                                                             &=\cupp{\Hclass{\Astr<d+2>[\bfLambda]}}{[\Hclass{\Astr<d+2>[\bfLambda]},x]}.
  \end{align*}
  In bidegree $(d+2,-d-1)$, the `lowest' relevant bidegree from the perspective
  of the vanishing conditions required to apply \Cref{thm:CY-Kadeishvili-Ai},
  the situation looks as follows:
  \[
    \begin{tikzcd}[column sep=huge, row sep=large]
      \HH[1]<0>{\bfLambda}\rar{[\Hclass{\Astr<d+2>[\bfLambda]},-]}\dar{\cupp{\Hclass{\Astr<d+2>[\bfLambda]}}{?}}&\HH[d+2]<-d-1>{\bfLambda}\rar{[\Hclass{\Astr<d+2>[\bfLambda]},-]}\dar{\cupp{\Hclass{\Astr<d+2>[\bfLambda]}}{?}}\ar[dotted]{dl}[description]{\cupp{\Hclass{\delta_{/d}}}{?}}&\HH[2d+3]<-2d-2>{\bfLambda}\dar{\cupp{\Hclass{\Astr<d+2>[\bfLambda]}}{?}}\ar[dotted]{dl}[description]{\cupp{\Hclass{\delta_{/d}}}{?}}\\
      \HH[d+3]<-d>{\bfLambda}\rar[swap]{[\Hclass{\Astr<d+2>[\bfLambda]},-]}&\HH[2d+4]<-2d-1>{\bfLambda}\rar[swap]{[\Hclass{\Astr<d+2>[\bfLambda]},-]}&\HH[3d+5]<-3d-2>{\bfLambda}
    \end{tikzcd}
  \]
  We also know from \cite[Proposition~4.7.15]{JKM22} that the maps
  \eqref{eq:key-map} are isomorphisms in source bidegrees $(s,t)$, $s\geq2$ and
  surjective if $s\geq1$; hence, in particular, there are induced
  cohomology-level isomorphisms
  \begin{equation}
    \label{eq:key-vanishing}
    \BimHMH!\bfLambda![p+1]<-p>{\bfLambda(0)}\stackrel{\sim}{\longrightarrow}\BimHMH!\bfLambda![p+d+3]<-p-d>{\bfLambda(0)},\quad p>d.
  \end{equation}
  Still as in the proof of~\cite[Proposition~5.2.9]{JKM22}, the map
  \eqref{eq:key-map} is null-homotopic via the bidegree $(1,0)$ maps
  \[
    \HH[s]<r>{\bfLambda}\longrightarrow\HH[s+1]<r>{\bfLambda},\qquad
    x\longmapsto\cupp{\Hclass{\delta_{/d}}}{x},
  \]
  where $\delta\in\HC[1]<0>{\bfLambda}$ is the fractional Euler derivation
  \[
    \delta_{/d}\colon\bfLambda\longrightarrow\bfLambda,\qquad x\longmapsto
    \tfrac{|x|}{d}x,
  \]
  where we use that $|x|\in d\ZZ$ since the graded algebra $\bfLambda$ is
  $d$-sparse by assumption (we refer the reader to~\cite[Section~4.7]{JKM22} for
  more information on this derivation). This follows immediately from the last
  equation in~\cite[Proposition~4.7.8]{JKM22}, but we provide here some
  additional calculations for the convenience of the reader. For this, we follow
  the computations in \cite[Section~3]{Mur22} and observe first that
  \begin{align*}
    [\Hclass{\Astr<d+2>[\bfLambda]},\Hclass{\delta_{/d}}]\stackrel{\eqref{eq:shifted_Lie_algebra}}{=}-[\Hclass{\delta_{/d}},\Hclass{\Astr<d+2>[\bfLambda]}]=\Hclass{\Astr<d+2>[\bfLambda]},
  \end{align*}
  where in the first equality we use that the shifted total degree $|\s
  \delta_{/d}|=0$ and in the second equality we
  use~\cite[Proposition~4.7.4]{JKM22} for the vertical degree $-d$ of
  $\Astr<d+2>[\bfLambda]$. Combining this observation with the Gerstenhaber
  relation (keeping in mind that $|\s \Astr<d+2>[\bfLambda]|=1$ and
  $|\delta_{/d}|=1$) we obtain the required identity
  \begin{align*}
    \cupp{\Hclass{\Astr<d+2>[\bfLambda]}}{x}&=\cupp{[\Hclass{\Astr<d+2>[\bfLambda]},\Hclass{\delta_{/d}}]}{x}\\&\stackrel{\eqref{eq:Gerstenhaber_relation}}{=}[\Hclass{\Astr<d+2>[\bfLambda]},\cupp{\Hclass{\delta_{/d}}}{x}]-(-1)^{|\s\Astr<d+2>[\bfLambda]||\delta_{/d}|}\cupp{\Hclass{\delta_{/d}}}{[\Hclass{\Astr<d+2>[\bfLambda]},x]}\\
                                            &=[\Hclass{\Astr<d+2>[\bfLambda]},\cupp{\Hclass{\delta_{/d}}}{x}]+\cupp{\Hclass{\delta_{/d}}}{[\Hclass{\Astr<d+2>[\bfLambda]},x]}.
  \end{align*}
  Finally, since the maps \eqref{eq:key-map} are null-homotopic, we conclude
  that the induced isomorphisms \eqref{eq:key-vanishing} are null maps, and
  hence we must have
  \[
    \BimHMH!\bfLambda![p+1]<-p>{\bfLambda(0)}=0,\qquad p>d,
  \]
  which is what we needed to prove.

  Finally, the fact that a graded
  $\bfLambda$-bimodule isomorphism
  $\varphi\colon\bfLambda(n)\stackrel{\sim}{\to}D{\bfLambda}$ satisfies the
  vanishing condition in
  statement
  \eqref{it:thm:CY-correspondence:gLambda:it:CY}
  for some $\eta$ if and only if it is induced by a bimodule right
  $n$-Calabi--Yau structure on any choice of dg algebra $A$ satisfying the
  conditions in statement \eqref{it:thm:CY-correspondence:dgAs} follows from the
  previous argument and the conclusion of \Cref{thm:CY-Kadeikshvili}.
  This finishes the proof of the theorem.
\end{proof}

We conclude by proving \Cref{thm:CY-equivalence}.

\begin{proof}[Proof of \Cref{thm:CY-equivalence}]
  The implication
  \eqref{it:thm:CY-equivalence:CY}$\Rightarrow$\eqref{it:thm:CY-equivalence:BV}
  is a special case of \Cref{prop:CY-BV}, and only needs the dg algebra $A$ to
  be cohomologically $d$-sparse and bimodule right $n$-Calabi--Yau, $n=md$. To
  prove the converse implication, we recall that, under the assumptions of
  \Cref{thm:CY-equivalence}, the restricted UMP of length $d+2$,
  \[
    j^*\Hclass{\Astr<d+2>[\bfLambda]}\in\HH[d+2]<-d>{\H[0]{A}}[\H{A}]\cong\Ext[\H[0]{A}^e][\bullet]{\H[0]{A}}{\H[-d]{A}},
  \]
  can be represented by an exact sequence of $\H[0]{A}$-bimodules with
  projective(-injective) middle terms~\cite[Corollary~4.5.17]{JKM22}.
  \Cref{thm:CY-correspondence} applied to the pair $(\H[0]{A},\H[-d]{A})$ yields
  the fact that the dg algebra $A$ is bimodule $n$-Calabi--Yau (here we use
  again that any exact sequence $\eta$ as in
  \Cref{thm:CY-correspondence}\eqref{it:thm:CY-correspondence:gLambda:it:CY} can
  be realised as a restricted UMP for a dg algebra in the quasi-isomorphism
  class of $A$~\cite[Theorem~5.1.2]{JKM22}).

  Finally, the fact that a graded $\H{A}$-bimodule isomorphism
  \[
    {\varphi\colon\H{A}(n)\stackrel{\sim}{\longrightarrow}D\H{A}}
  \]
  satisfies the vanishing
  condition in statement \eqref{it:thm:CY-equivalence:BV} if and only if it is
  induced by a bimodule right $n$-Calabi--Yau structure on $A$ is part of
  \Cref{thm:CY-correspondence}. This finishes the proof of the theorem.
\end{proof}

\section{A non-enhanceable triangulated Calabi--Yau structure}
\label{sec:the-example}

In this section we provide an explicit example of an algebraic triangulated
category with a triangulated Calabi--Yau structure that cannot be lifted to a
bimodule right Calabi--Yau structure on any of its dg enhancements
(\Cref{thm:non-enhanceable-CY}). We fix the following setting throughout.

\begin{setting}
  \label{setting:dual-numbers-char2}
  Let $d\geq1$. Suppose that $\chark(\kk)=2$ and let
  $\Lambda=\kk[\varepsilon]/(\varepsilon^2)$ be the algebra of dual numbers
  (concentrated in degree $0$). It is well-known and easy to verify that
  $\Lambda$ is $(d+2)$-periodic. Indeed, there is a short exact sequence of
  $\Lambda$-bimodules
  \begin{equation}
    \label{eq:u-ses}
    0\longrightarrow\Lambda\stackrel{i}{\longrightarrow}\Lambda\otimes\Lambda\stackrel{\mu}{\longrightarrow}\Lambda\longrightarrow0,
  \end{equation}
  where $\mu$ is the multiplication map and
  $i(1)=1\otimes\varepsilon+\varepsilon\otimes1$. We set
  \[
    \bfLambda\coloneqq\kk[\imath^{\pm1}]\otimes\Lambda,\qquad|\imath|=-d.
  \]
  Compare with \Cref{setting:graded-lambda}.
\end{setting}

We begin by recalling the description of the Hochschild cohomology of $\Lambda$
as a graded algebra.

\begin{proposition}[{\cite[Theorem~7.1]{Hol00}}]
  \label{prop:HH-dual-numbers-char2}
  Recall~\Cref{setting:dual-numbers-char2}. There is an isomorphism of graded
  algebras
  \[
    \HH*{\Lambda}[\Lambda]\cong\kk[\varepsilon,u]/(\varepsilon^2),\qquad
    |\varepsilon|=0,\quad |u|=1.
  \]
\end{proposition}

\begin{remark}
  \label{rmk:u-derivation}
  In \Cref{prop:HH-dual-numbers-char2}, the class
  $u\in\HH*[1]{\Lambda}[\Lambda]$ is represented by the short exact sequence
  \eqref{eq:u-ses}. Alternatively, in terms of Hochschild cochains, the class
  $u$ corresponds to the derivation of $\Lambda$ that is uniquely determined by
  the requirement that $u(\varepsilon)=1$; the latter is a derivation since
  $\chark(\kk)=2$.\footnote{It is straightforward to verify that the
    derivation $u\colon\varepsilon\mapsto 1$ induces a functorial isomorphism
    \[
      ?\circ u\colon\Hom[\Lambda^e]{\Lambda}{M}\stackrel{\sim}{\longrightarrow}\operatorname{Der}(\Lambda,M),\qquad M\in\mmod*{\Lambda^e},
    \]
    where the target denotes the vector space of derivations $\Lambda\to M$,
    compare with~\cite[Section~4.2]{Wit19}. On the other hand, the short exact
    sequence \eqref{eq:u-ses} determines an isomorphism
    $\Omega_{\Lambda^e}^1(\Lambda)\stackrel{\sim}{\to}\Lambda$ in the stable
    category of $\Lambda$-bimodules, so that there are further functorial
    isomorphisms
    \[
      \sHom[\Lambda^e]{\Lambda}{M}\stackrel{\sim}{\longrightarrow}\Ext[\Lambda^e][1]{\Lambda}{M},\qquad\mmod*{\Lambda^e},
    \]
    induced by pushing out the short exact sequence \eqref{eq:u-ses} (we use
    here that $\Lambda^e$ is self-injective), see for
    example~\cite[Theorem~3.4.3]{Wei94} and its proof. For $M=\Lambda$, we
    obtain isomorphisms
    \[
      \operatorname{Der}(\Lambda,\Lambda)\stackrel{\sim}{\longleftarrow}\Hom[\Lambda^e]{\Lambda}{\Lambda}=\sHom[\Lambda^e]{\Lambda}{\Lambda}\stackrel{\sim}{\longrightarrow}\Ext[\Lambda^e][1]{\Lambda}{\Lambda},
    \]
    from which the claim readily follows. Since $\chark(\kk)=2$ by assumption,
    the equality between the centre and the stable centre of $\Lambda$ follows
    from the fact that the canonical quotient map has kernel the ideal generated
    by $(2\varepsilon)$, see for example the proof
    of~\cite[Corollary~5.4.12]{JKM22}.} In what follows, we interpret the class $u$ in either way, as
  is convenient from the context.
\end{remark}

\begin{setting}
  \label{setting:dual-numbers-char2-eta}
  Recall \Cref{setting:dual-numbers-char2}. We introduce the class
  \[
    \eta\coloneqq u^{d+2}\in\HH*[d+2]{\Lambda}[\Lambda],
  \]
  represented by the exact sequence obtained by splicing the short exact
  sequence \eqref{eq:u-ses} with itself $d+2$ times. Under the isomorphism
  \[
    \HH*[d+2]{\Lambda}[\Lambda]=\Ext[\Lambda^e][d+2]{\Lambda}{\Lambda}\cong\sHom[\Lambda^e]{\Omega_{\Lambda^e}^{d+2}(\Lambda)}{\Lambda},
  \]
  the class $\eta$ determines an isomorphism
  \[
    \eta\colon\Omega_{\Lambda^e}^{d+2}(\Lambda)\stackrel{\sim}{\longrightarrow}\Lambda
  \]
  in the stable category of $\Lambda$-bimodules.
\end{setting}

We wish to determine the class
\[
  \Hclass{\Astr<d+2>[\eta]}\in\HH[d+2]<-d>{\bfLambda}[\bfLambda],
\]
which we know can be realised as the UMP of length $d+2$ of a dg algebra $A$
with cohomology isomorphic to $\bfLambda$ and such that $A\in\DerCat[c]{A}$ is a
basic $d\ZZ$-cluster tilting object, see~\cite[Theorems~A and~5.1.2]{JKM22}. For
this, we need the description of the Hochschild cohomology of $\bfLambda$ as a
bigraded algebra.

\begin{proposition}
  \label{prop:HH-bfLambda-dual-numbers-char2}
  Recall \Cref{setting:dual-numbers-char2}. The Hochschild cohomology of
  $\bfLambda$ participates in a commutative diagram of morphisms of bigraded
  algebras
  \begin{equation*}
    \begin{tikzcd}[column sep=small]
      \HH{\bfLambda}[\bfLambda]\dar{\wr}\rar{j^*}&\HH{\Lambda}[\bfLambda]\dar{\wr}\\
      \HH*{\Lambda}[\Lambda]{[\imath^{\pm1},\delta]}/(\delta^2)\rar[two
      heads]\dar{\wr}&\HH*{\Lambda}[\Lambda]{[\imath^{\pm1}]}\dar{\wr}\\
      \kk[\varepsilon,u,\imath^{\pm1},\delta]/(\varepsilon^2,\delta^2)\rar[two
      heads]&\kk[\varepsilon,u,\imath^{\pm1}]/(\varepsilon^2)
    \end{tikzcd}
  \end{equation*}
  \begin{align*}
    |\varepsilon|&=(0,0),&|u|&=(1,0),&|\imath|&=(0,-d),&|\delta|&=(1,0).
  \end{align*}
  in which the vertical arrows are isomorphisms and the two bottom horizontal
  maps are the evident quotient maps with kernel the bigraded ideal generated
  by $\delta$. Here, $j\colon\Lambda\hookrightarrow\bfLambda$ is the inclusion
  of the degree $0$ part and the class $\delta$ is represented by the fractional
  Euler derivation $x\mapsto\frac{|x|}{d}x$. Moreover,
  \begin{align*}
    [\imath,\HH*{\Lambda}[\Lambda]]&=0&[\imath,\imath]=0\\
    [\delta,\HH*{\Lambda}[\Lambda]]&=0&[\delta,\imath]=\imath.
  \end{align*}
\end{proposition}
\begin{proof}
  Combined with \Cref{prop:HH-dual-numbers-char2}, this follows immediately
  from~\cite[Proposition 4.7.6, Proposition 4.7.12, Corollary~4.7.14 and
  Remark~4.7.15]{JKM22} taking $\sigma=\id$ (compare also
  with~\cite[Proposition~4.3.1]{JKM24}).
\end{proof}

\begin{remark}
  Recall \Cref{setting:dual-numbers-char2,setting:dual-numbers-char2-eta}. Under
  the isomorphisms
  \[
    \HH*[d+2]{\Lambda}[\Lambda]\cong\HH[d+2]<-d>{\Lambda}[\bfLambda]\stackrel{\ref{prop:HH-bfLambda-dual-numbers-char2}}{\cong}\imath\cdot\HH*[d+2]{\Lambda}[\Lambda],
  \]
  the class $\eta$ corresponds to $\imath\cdot\eta$.
\end{remark}

\begin{lemma}
  \label{lemma:Sq(u)}
  With the notation in \Cref{prop:HH-bfLambda-dual-numbers-char2}, we have
  \[ [u^p,u]=0\qquad\text{and}\qquad\Sq(u^p)=0,\qquad p\geq 1.
  \]
\end{lemma}
\begin{proof}
  The equality $[u,u]=0$ holds since $u$ has odd total degree (see
  \Cref{eq:shifted_Lie_algebra}), and a straightforward induction using the
  Gerstenhaber relation \eqref{eq:Gerstenhaber_relation} shows that $[u^p,u]=0$,
  $p\geq1$, also holds. Suppose that $\Sq(u)=0$ (we prove below that this holds)
  and $\Sq(u^p)=0$ for some $p\geq1$. Then,
  \[
    \Sq(u^{p+1})=\underbrace{\Sq(u^{p})}_{=0}\cdot
    u^2+u^p\cdot\underbrace{[u^p,u]}_{=0}\cdot
    u+u^{2p}\cdot\underbrace{\Sq(u)}_{\stackrel{!}{=}0}=0
  \]
  It thus remains to prove the base case of the induction: $\Sq(u)=0$. For this,
  observe that, as a derivation, $u$ satisfies
  \[
    \braces{u}{u}=u\circ_1u=0
  \]
  at the Hochschild-cochain level since necessarily $u(\id*[\Lambda])=0$ and
  $u(\Lambda)=\kk\cdot\id*[\Lambda]$, see~\Cref{rmk:u-derivation}.
\end{proof}

The following computation should be compared with~\cite[Corollary~4.3.2]{JKM24}.

\begin{proposition}
  \label{prop:ump-ieta}
  Recall \Cref{setting:dual-numbers-char2,setting:dual-numbers-char2-eta}. There
  is an equality
  \[
    \Hclass{\Astr<d+2>[\eta]}=\imath\cdot\eta\in\HH[d+2]<-d>{\bfLambda}[\bfLambda].
  \]
\end{proposition}
\begin{proof}
  The class $\Hclass{\Astr<d+2>[\eta]}$ is uniquely characterised by the
  conditions
  \[
    j^*(\Hclass{\Astr<d+2>[\eta]})=\eta\qquad\text{and}\qquad\Sq(\Hclass{\Astr<d+2>[\eta]})=0,
  \]
  see the proof of~\cite[Theorem~5.3.3]{JKM22}. Under the identifications in
  \Cref{prop:HH-bfLambda-dual-numbers-char2}, the condition
  $j^*(\Hclass{\Astr<d+2>[\eta]})=\eta$ means that
  $\Hclass{\Astr<d+2>[\eta]}=\imath\cdot\eta$ modulo $\delta$. Hence, we only
  need to show that $\Sq(\imath\cdot\eta)=0$. For this, we use
  \Cref{lemma:Sq(u)} to obtain
  \[
    \Sq(\imath\cdot\eta)\stackrel{\eqref{eq:Gerstenhaber_square-relations}}{=}\underbrace{\Sq(\imath)}_{=0}\cdot
    \eta^2+\imath\cdot\underbrace{[\imath,\eta]}_{\stackrel{\ref{prop:HH-bfLambda-dual-numbers-char2}}{=}0}\cdot\imath+\imath^2\cdot\underbrace{\Sq(\eta)}_{\stackrel{\ref{lemma:Sq(u)}}{=}0}=0;
  \]
  here, $\Sq(\imath)=0$ since $\braces{\imath}{\imath}=0$ for
  $\imath\colon\kk\to\Lambda$ is an operation of arity $0$.
\end{proof}

Having determined the class
$\Hclass{\Astr<d+2>[\eta]}\in\HH{\bfLambda}[\bfLambda]$, we now wish to compute
the obstruction class
\[
  \BV(\Hclass{\Astr<d+2>[\eta]})\in\HH[d+1]<-d>{\bfLambda}[\bfLambda],
\]
compare with \Cref{thm:CY-equivalence}. For this, we need to fix a
$\bfLambda$-bimodule isomorphism ${\bfLambda\stackrel{\sim}{\longrightarrow}D\bfLambda}$.

\begin{setting}
  \label{setting:dual-numbers-char2-pairing}
  Recall \Cref{setting:dual-numbers-char2,setting:dual-numbers-char2-eta}. Fix
  scalars $a,b\in\kk$ with $b\neq0$. Consider the isomorphism of
  $\Lambda$-bimodules
  \[
    \varphi=\varphi_{a,b}\colon\Lambda\stackrel{\sim}{\longrightarrow} D\Lambda
  \]
  that is adjoint to the associative and symmetric bilinear pairing
  \begin{equation}
    \label{eq:the-ab-pairing}
    \pairing<-,->=\pairing<-,->_{a,b}\colon\Lambda\times\Lambda\to\kk
  \end{equation}
  uniquely determined by the requirement that
  \begin{align*}
    \pairing<1,1>&=a,&\pairing<1,\varepsilon>&=b,\\
    \pairing<\varepsilon,1>&=b,&\pairing<\varepsilon,\varepsilon>&=0,
  \end{align*}
  see~\Cref{ex:dual_numbers-bimodule_isos}. Since $\chark(\kk)=2$, the
  $\Lambda$-bimodule isomorphism $\varphi$ extends to an isomorphism of graded
  $\bfLambda$-bimodules
  \begin{equation}
    \label{eq:the-ab-graded-iso}
    \varphi=\varphi_{a,b}\colon\bfLambda\stackrel{\sim}{\longrightarrow}D\bfLambda
  \end{equation}
  whose associated pairing
  \[
    \pairing<-,->=\pairing<-,->_\varphi\colon\bfLambda\times\bfLambda\longrightarrow\kk
  \]
  has the non-zero components
  \[
    \pairing<\imath^{-k}x,\imath^ky>=\pairing<x,y>,\qquad x,y\in\Lambda,\quad
    k\in\ZZ,
  \]
  see \Cref{prop:graded-bimodule-iso-vs-pairing} and
  \Cref{ex:dual_numbers-bimodule_isos}.
\end{setting}

The following computation shows that, in general, the vanishing of the
Batalin--Vilkovisky operator on the universal Massey product of length $d+2$ of
a cohomologically $d$-sparse dg algebra depends on the choice of pairing.

\begin{proposition}
  \label{prop:BV-non-zero}
  Recall
  \Cref{setting:dual-numbers-char2,setting:dual-numbers-char2-eta,setting:dual-numbers-char2-pairing}.
  Then,
  \[
    \BV(\Hclass{\Astr<d+2>[\eta]})\stackrel{\ref{prop:ump-ieta}}{=}\BV(\imath\cdot\eta)=\begin{cases}
      (ab^{-1}+a^2b^{-2}\varepsilon)\imath\cdot u^{d+1}&d\text{ is odd},\\
      0&d\text{ is even},
    \end{cases}
  \]
  where $\BV=\BV_\varphi$ is the Batalin--Vilkovisky operator (\Cref{def:BV})
  associated with the isomorphism of graded $\bf\Lambda$-bimodules
  \eqref{eq:the-ab-graded-iso}. In particular, $\BV(\imath\cdot\eta)=0$ if and
  only if $d$ is even or $d$ is odd and $a=0$.
\end{proposition}
\begin{proof}
  Recall that $\eta=u^{d+2}$ and that $u\in\HH*[1]{\Lambda}[\Lambda]$ is
  determined as a derivation by the requirement that $u(\varepsilon)=1$. We make
  the following recollections and observations:
  \begin{itemize}
  \item The Batalin--Vilkovisky operator satisfies the following identity on
    homogeneous elements~\cite[Theorem~18]{Men09}
    \[
      \BV(x\cdot y)=\BV(x)\cdot y+x\cdot\BV(y)+[x,y],\qquad
      x,y\in\HH{\bfLambda}[\bfLambda].
    \]
  \item In view of \Cref{prop:HH-bfLambda-dual-numbers-char2} and the previous
    item,
    \[
      \BV(\imath\cdot\eta)=\underbrace{\BV(\imath)}_{=0}\cdot\eta+\imath\cdot\BV(\eta)+\underbrace{[\imath,\eta]}_{\stackrel{\ref{prop:HH-bfLambda-dual-numbers-char2}}{=}0}=\imath\cdot\BV(\eta),
    \]
    where we use that $\BV(\imath)\in\HH[-1]<-d>{\bfLambda}[\bfLambda]=0$ since
    $\HH{\bfLambda}[\bfLambda]$ is concentrated in non-negative horizontal
    degrees.
  \item In view of \Cref{lemma:Sq(u)} and the first item,
    \[
      \BV(u^{p+1})=\BV(u^p)\cdot
      u+u^p\cdot\BV(u)+\underbrace{[u^p,u]}_{\stackrel{\ref{lemma:Sq(u)}}{=}0},\qquad
      p\geq1.
    \]
    A straightforward induction shows that
    \[
      \BV(\eta)=\BV(u^{d+2})=\begin{cases}
        u^{d+1}\BV(u)&\text{if $d$ is odd},\\
        0&\text{if $d$ is even}.
      \end{cases}
    \]
  \end{itemize}
  Combining the above we conclude that $\BV(\imath\cdot\eta)=0$ if $d$ is even.
  
  Suppose that $d$ is odd. We finish the proof by establishing the equality
  \[
    \BV(u)=ab^{-1}+a^2b^{-2}\varepsilon\in\HH[0]<0>{\bfLambda}[\bfLambda]=\Lambda.
  \]
  At the cochain level, $\BV(u)$ is uniquely characterised by the property that
  \[
    \pairing<\BV(u),x>=\pairing<u(x),1>,
  \]
  for all $x\in\bfLambda$, and both sides vanish on multiples of $\imath^p$ if
  $p\neq0$. Evaluating at $x=\varepsilon$ yields
  \[
    \pairing<\BV(u),\varepsilon>=\pairing<u(\varepsilon),1>=\pairing<1,1>=a,
  \]
  while evaluating at $x=1$ yields
  \[
    \pairing<\BV(u),1>=\pairing<u(1),1>=\pairing<0,1>=0.
  \]
  Consequently,
  \[
    \pairing<\BV(u),->=a\varepsilon^*\in D\Lambda
  \]
  and, therefore,
  \[
    \BV(u)=a(b^{-1}+ab^{-2}\varepsilon)=ab^{-1}+a^2b^{-2}\varepsilon,
  \]
  see \Cref{ex:dual_numbers-bimodule_isos}. This finishes the proof.
\end{proof}

\begin{remark}
  In $\chark(\kk)\neq2$, the Batalin--Vilkovisky operator
  \[
    \BV=\BV_{0,1}\colon\HH*[\bullet]{\Lambda}[\Lambda]\longrightarrow\HH*[\bullet-1]{\Lambda}[\Lambda]
  \]
  is computed in~\cite[Example~6]{Tra08a}.
\end{remark}

For the necessary terminology concerning algebraic $(d+2)$-angulated categories
and their dg enhancements, which is used in the statement below, we refer the
reader to the relevant sections in~\cite{JKM22}, see also~\cite{GKO13}. We only
mention that the notion of (algebraic) $3$-angulated category agrees with that
of (algebraic) triangulated category.

\begin{theorem}
  \label{thm:non-enhanceable-CY}
  Recall
  \Cref{setting:dual-numbers-char2,setting:dual-numbers-char2-eta,setting:dual-numbers-char2-pairing}.
  In particular, $\chark(\kk)=2$ and $\Lambda=\kk[\varepsilon]/(\varepsilon^2)$
  is the algebra of dual numbers. The following statements hold:
  \begin{enumerate}
  \item\label{it:dual_numbers-angulated} The pair $(\proj*{\Lambda},\id)$ admits
    an (essentially unique) algebraic $(d+2)$-angulated structure $\pentagon$
    determined by the exact sequence of $\Lambda$-bimodules
    \[
      0\to \Lambda\to\underbrace{\Lambda\otimes\Lambda\to\cdots\to
        \Lambda\otimes\Lambda\to\Lambda\otimes\Lambda\to
        \Lambda\otimes\Lambda}_{d+2\text{ terms}}\to\Lambda\to 0
    \]
    that represents the class $\eta\in\HH*[d+2]{\Lambda}[\Lambda]$. Moreover,
    the $(d+2)$-angulated category $(\proj*{\Lambda},\id,\pentagon)$ has a
    unique dg enhancement.
  \item\label{it:ab-Serre} Let $a,b\in\kk$ with $b\neq0$. Then, the graded
    $\bfLambda$-bimodule isomorphism
    \[
      \varphi_{a,b}\colon\bfLambda\stackrel{\sim}{\longrightarrow}D\bfLambda
    \]
    determined by the pairing \eqref{eq:the-ab-pairing} endows the algebraic
    $(d+2)$-angulated category $(\proj*{\Lambda},\id,\pentagon)$ with graded
    Serre duality isomorphisms
    \[
      \varphi_{a,b}\colon\Hom[\Lambda]{Q}{P}\stackrel{\sim}{\longrightarrow}D\Hom[\Lambda]{P}{Q},\qquad
      P,Q\in\proj*{\Lambda}.
    \]
  \item\label{it:BV-non-zero} In item \eqref{it:ab-Serre}, the graded Serre
    duality isomorphisms $\varphi_{a,b}$ are induced by a bimodule
    $0$-Calabi--Yau structure on a dg enhancement of
    $(\proj*{\Lambda},\id,\pentagon)$ if and only if $d$ is even or $d$ is odd
    and $a=0$.
  \end{enumerate}
  In particular, letting $d=1$ and $a\neq0$, we conclude that the pair
  $(\proj*{\Lambda},\id)$ can be endowed with the structure of an algebraic
  triangulated category with a triangulated $0$-Calabi--Yau structure that
  cannot be lifted to a bimodule right $0$-Calabi--Yau structure on any of its
  (quasi-equivalent) dg enhancements.
\end{theorem}
\begin{proof}
  Statement \eqref{it:dual_numbers-angulated} is contained
  in~\cite[Theorem~5.1.3 and Corollary~5.4.12]{JKM22}, while \eqref{it:ab-Serre}
  is a straightforward consequence of graded Morita Theory. To prove statement
  \eqref{it:BV-non-zero}, we first recall the following consequences of
  \cite[Theorem~5.1.2 and Corollary~4.5.18]{JKM22}:
  \begin{itemize}
  \item There exists a dg algebra $A$ such that $\H{A}\cong\bfLambda$ as graded
    algebras and
    \[
      \Hclass{\Astr<d+2>[A]}=\Hclass{\Astr<d+2>[\eta]}\stackrel{\ref{prop:BV-non-zero}}{=}\imath\cdot\eta\in\HH[d+2]<-d>{\bfLambda}[\bfLambda].
    \]
    Moreover, $A\in\DerCat[c]{A}$ is a basic $d\ZZ$-cluster tilting object.
  \item Let $B$ be a dg algebra such that $\H{B}\cong\bfLambda$ as graded
    algebras and
    \[
      j^*{\Hclass{\Astr<d+2>[B]}}\in\HH[d+2]<-d>{\Lambda}[\bfLambda]\cong\Ext[\Lambda^e][d+2]{\Lambda}{\Lambda}
    \]
    corresponds to an isomorphism
    $\Omega_{\Lambda^e}^{d+2}(\Lambda)\stackrel{\sim}{\to}\Lambda$ in the
    stable category of $\Lambda$-bimodules (compare with
    \Cref{setting:dual-numbers-char2-eta}). Then $A$ and $B$ are
    quasi-isomorphic as dg algebras, where $A$ is as in the previous item and,
    in fact, the above properties characterise the dg algebras in the
    quasi-isomorphism class of $A$. Moreover, given minimal models
    $(\bfLambda,\Astr[A])$ and $(\bfLambda,\Astr[B])$ of $A$ and $B$,
    respectively, there exists an $A_\infty$-isomorphism
    $f\colon(\bfLambda,\Astr[B])\stackrel{\sim}{\to}(\bfLambda,\Astr[A])$ (see also
    the proof of~\cite[Theorem~5.1.3]{JKM22}). In particular,
    \[
      \Hclass{(\Astr<d+2>[B])^\gamma}=\Hclass{\Astr<d+2>[A]},
    \]
    where $\gamma\coloneqq f_1\in\Aut{\bfLambda}$ and the action of
    $(\Astr<d+2>[B])^\gamma$ is defined as in \Cref{rmk:functoriality-isos-HC}.
  \item Let $\gamma\in\Aut{\bfLambda}$. Its restriction to
    $\Lambda$ is uniquely determined by the value
    $\gamma(\varepsilon)=c\varepsilon$, where $c\in\kk^\times$. As explained in
    \Cref{rmk:BV-under-algebra-isos}, we obtain a nondegenerate symmetric
    and associative pairing uniquely determined by
    \begin{align*}
      \pairing<\gamma(1),\gamma(1)>_{a,b}&=\pairing<1,1>_{a,b}=a,\\\pairing<\gamma(1),\gamma(\varepsilon)>_{a,b}&=\pairing<1,c\varepsilon>_{a,b}=cb\neq0.
    \end{align*}
    In particular, the scalar $a\in\kk$ is invariant under the action of
    $\Aut{\bfLambda}$.
  \end{itemize}
  We conclude that, for the purpose of establishing the validity of statement
  \eqref{it:BV-non-zero}, it suffices to consider the case of a dg algebra $A$
  such that $\H{A}\cong\bfLambda$ as graded algebras and
  \[
    \Hclass{\Astr<d+2>[A]}=\Hclass{\Astr<d+2>[\eta]}=\imath\cdot\eta.
  \]
  By \Cref{thm:CY-equivalence}, the graded $\bfLambda$-bimodule isomorphism
  $\varphi_{a,b}$ is induced by an isomorphism $A\cong DA$ in the derived
  category of dg $A$-bimodules if and only if
  \[
    0=\BV_{a,b}(\Hclass{\Astr<d+2>})\stackrel{\ref{prop:BV-non-zero}}{=}\begin{cases}(ab^{-1}+a^2{b^{-2}}\varepsilon)\imath\cdot
      u^{d+1}&d\text{ is odd},\\
      0&d\text{ is even},\end{cases}
  \]
  if and only if $d$ is even or $d$ is odd and $a=0$. This finishes the proof of
  the theorem.
\end{proof}

\begin{remark}
  \label{rmk:A-dual-numbers-periodicity}
  The algebra $\Lambda=\kk[\varepsilon]/(\varepsilon^2)$ can be described
  alternatively as the preprojective algebra of generalised Dynkin type
  $\mathbb{L}_1$~\cite{HPR80}, see also~\cite{BES07}, and is one of the algebras
  considered in~\cite[Section~9]{Ami07}, see also~\cite[p.~1176,
  Examples]{Dug12a}. For $d=1$, the algebraic triangulated structure on
  $\proj*{\Lambda}$ in \Cref{thm:non-enhanceable-CY} goes back to work of
  Auslander and Reiten (see also~\cite[Section~7.4]{Kel05}) and its essential
  uniqueness follows from the main result in~\cite{Kel18c}. That this algebraic
  triangulated structure admits a unique dg enhancement follows from work of the
  second-named author~\cite{Mur22}. An explicit representative of the
  quasi-isomorphism class of the dg algebra appearing in the proof of
  \Cref{thm:non-enhanceable-CY} appears already in~\cite[Remark~8]{MSS07}, see
  also~\cite[Example~5.4.13]{JKM22}. It is given by the graded algebra
  \[
    A\coloneqq\frac{\kk[e,t^{\pm1}]\langle h\rangle}{(h^2,ht+th,he+eh+1)},\qquad
    |e|=|h|=0,\quad |t|=1,
  \]
  endowed with the differential
  \[
    d(e)=0,\qquad d(t)=0,\qquad d(h)=e^2t.
  \]
  There is an isomorphism of graded algebras
  \begin{align*}
    \H{A}\stackrel{\sim}{\longrightarrow}\bfLambda,\qquad
    [e]\longmapsto\varepsilon,\qquad
    [t]\longmapsto\imath^{-1},
  \end{align*}
  induced by an equivalence of pairs
  \[
    (\DerCat[c]{A},[1])\stackrel{\simeq}{\longrightarrow}(\proj*{\Lambda},\id),\qquad
    A\longmapsto\Lambda,
  \]
  which is an equivalence of triangulated categories if we endow the target with
  its (essentially unique) algebraic triangulated structure. The dg algebra $A$
  enjoys the following strict form of periodicity: Let $n\in\ZZ$. Since $t^n\in
  A$ is a degree $n$ invertible central cocycle, the Leibniz rule implies that
  the isomorphism of graded $A$-bimodules (recall that $\chark(\kk)=2$)
  \[
    A\stackrel{\sim}{\longrightarrow}A[n],\qquad x\longmapsto \s[n]t^nx,
  \]
  is a morphism of dg vector spaces, and hence an isomorphism of dg
  $A$-bimodules. Combining this observation with
  \Cref{thm:non-enhanceable-CY}\eqref{it:BV-non-zero} (with $d=1$ and $a=0$), we
  conclude that $A$ admits a bimodule right $n$-Calabi--Yau structure for each
  $n\in\ZZ$.
\end{remark}

\begin{remark}
  The phenomenon described at the end of \Cref{rmk:A-dual-numbers-periodicity}
  holds also for $d\geq1$ and $n\coloneqq md$, $m\in\ZZ$. In this case, one
  considers instead the nondegenerate pairing
  \[
    \pairing<-,->\colon\bfLambda(n)\times\bfLambda\longrightarrow\kk
  \]
  whose non-zero components
  \[
    \pairing<\imath^{-(k+m)}x,\imath^ky>=\pairing<x,y>_{a,b},\qquad
    x,y\in\Lambda,\quad k\in\ZZ,
  \]
  are defined in terms of the pairing \eqref{eq:the-ab-pairing}. It is
  straightforward to verify that the formulas in \Cref{prop:BV-non-zero} also
  hold for the Batalin--Vilkovisky operator defined using this pairing (in
  particular, they are independent of $m$). Indeed, we have
  \begin{multline*}
    \pairing<\s[n]\BV(\imath\cdot
    u^{d+2})(x_1,\dots,x_{d+1}),x_{d+2}>\stackrel{\ref{def:BV}}{=}\\\sum_{i=1}^{d+2}\pairing<\s[n]u(x_i)\cdots
    u(x_{d+2})u(x_1)\cdots u(x_{i-1}),1>,
  \end{multline*}
  where $x_1,\dots,x_{d+2}\in\bfLambda$. Since $\bfLambda$ is
  (graded-)commutative, the right-hand side is the sum of $d+2$ equal terms and
  hence vanishes if $d$ is even. If $d$ is instead odd, the above equality
  simplifies to
  \begin{align*}
    \pairing<\s[n]\BV(\imath\cdot
    u^{d+2})(x_1,\dots,x_{d+1}),x_{d+2}>\stackrel{\ref{def:BV}}{=}\pairing<\s[n]u(x_1)\cdots
    u(x_{d+2}),1>.
  \end{align*}
  Keeping in mind that $u(\bfLambda)=\kk\cdot1_{\bfLambda}$, one directly
  verifies that
  \[
    \Delta(\imath\cdot
    u^{d+2})=a(b^{-1}+ab^{-2}\varepsilon)\imath\cdot u^{d+1},
  \]
  for the right-hand side satisfies the previous equality. Arguing exactly as in
  \Cref{thm:non-enhanceable-CY}, we conclude that any dg algebra $A$ such that
  $\H{A}\cong\bfLambda$ as graded algebras, and whose UMP of length $d+2$
  satisfies
  \[
    \Hclass{\Astr<d+2>[A]}=\Hclass{\Astr<d+2>[\eta]}=\imath\cdot\eta\in\HH[d+2]<-d>{\bfLambda}[\bfLambda],
  \]
  admits a bimodule right $n$-Calabi--Yau structure (since we can always choose
  $a=0$). In particular, in view of \Cref{thm:CY-correspondence}, the free dg
  $A$-module $A\in\DerCat[c]{A}$ is an example of a (basic) $d\ZZ$-cluster
  tilting object in an $n$-Calabi--Yau triangulated category ($n\in d\ZZ$) in
  which $[d]\cong\id$ as triangle functors.
\end{remark}


\begin{acknowledgements}
  The authors thank S.~Barmeier and Y.~Lekili for respectively bringing
  references~\cite{FM07} and \cite{Li24} to their attention, and B.~Keller for
  suggesting to us the terminology `Hochschild--Massey cohomology.' The authors
  also thank P.~Bodin and B.~Keller for their comments on a previous version of
  this article.
\end{acknowledgements}


\begin{financialsupport}
  G.~J.~was partially supported by the Swedish Research Council
  (Vetenskapsrådet) Research Project Grant 2022-03748 `Higher structures in
  higher-dimensional homological algebra.' Fernando Muro was partially
  supported by grants PID2020-117971GB-C21
  funded by MCIN/AEI/10.13039/501100011033, US-1263032 (US/JUNTA/FEDER, UE),
  P20\_01109 (JUNTA/FEDER, UE), and PID2024-157173NB-I00 funded by
  MCIN/AEI/10.13039/501100011033 and also by FEDER, UE.
\end{financialsupport}


\printbibliography

\end{document}